\documentclass[letterpaper,12pt,reqno]{amsbook}
\usepackage{amssymb}
\usepackage{cases}
\usepackage[applemac]{inputenc}
\usepackage{amsmath,amscd,centernot} 
\usepackage{mathrsfs}
\usepackage{eucal}
\usepackage{amsfonts}
\usepackage[all]{xy}
\usepackage{xcolor}
\usepackage{tikz}
\DeclareFontFamily{OT1}{pzc}{}
\DeclareFontShape{OT1}{pzc}{m}{it}{<-> [1.15] rpzcmi}{}
\DeclareMathAlphabet{\mathzc}{OT1}{pzc}{m}{it}

\usepackage[english]{babel}

\newtheorem{thm}{Theorem}[section]

\newtheorem{conds}[thm]{Conditions}

\newtheorem{lem}[thm]{Lemma}
\newtheorem{notation}[thm]{Notation}
\newtheorem{cor}[thm]{Corollary}
\newtheorem{hyp}[thm]{Stratification Hypothesis}
\newtheorem{str}[thm]{Strategy}
\newtheorem{fhyp}[thm]{Finiteness Hypothesis}\newtheorem{prop}[thm]{Proposition}

\newtheorem{const}[thm]{Construction}
\theoremstyle{definition}
\newtheorem{example}[thm]{Example}

\newtheorem{defn}[thm]{Definition}
\newtheorem{defns}[thm]{Definitions}

\newtheorem{conj}[thm]{Conjecture}

\newtheorem{rem}[thm]{Remark}
\newtheorem{rems}[thm]{Remarks}
\newtheorem{ex}[thm]{Example}

\numberwithin{section}{chapter}
\numberwithin{equation}{section}
\numberwithin{thm}{section}

\newcommand{\irr}{\text{\rm Irr}}

\newcommand{\wF}{{\widetilde F}}

\newcommand{\op}{{\text{\rm op}}}

\newcommand{\ssT}{\widetilde T}

\newcommand{\Ind}{\text{\rm Ind}}

\newcommand{\Bmod}{B\mbox{--mod}}

\newcommand{\Reps}{R[\epsilon]}

\newcommand{\rank}{\text{\rm rank}}

\newcommand{\ck}{{\text{\rm Coker}}}
\newcommand{\Ker}{\text{\rm Ker}}

\newcommand{\im}{\mbox{Im}}

\newcommand{\inff}{\text{Inf}}
\newcommand{\sH}{{\mathcal H}}
\newcommand{\bsH}{{\boldsymbol{\sH}}}

\newcommand{\tT}{{\widetilde T}}
\newcommand{\sA}{{\mathcal A}}
\newcommand{\sX}{{\mathcal X}}

\newcommand{\sZ}{{\mathcal{Z}}}

\newcommand{\sfB}{{\mathsf {B}}}

\newcommand{\Ext}{{\text{\rm Ext}}}

\newcommand{\x}{\text{\bf \it x}}

\newcommand{\htt}{\mathfrak{h}}

\newcommand{\sG}{{\mathscr G}}
\newcommand{\sS}{{\mathcal S}}

\newcommand{\Hom}{\text{\rm Hom}}

\newcommand{\End}{\operatorname{End}}

\newcommand{\ind}{\operatorname{Ind}}
\newcommand{\sE}{\operatorname{{\mathscr E}}}

\newcommand{\wt}{\operatorname{wt}}

\newcommand{\wA}{{\widetilde{A}}}
\newcommand{\wB}{{\widetilde{B}}}

\newcommand{\wN}{{\widetilde{N}}}

\newcommand{\wP}{{\widetilde{P}}}

\newcommand{\wK}{{\widetilde{K}}}

\newcommand{\wT}{{\widetilde{T}}}
\newcommand{\wL}{{\widetilde{L}}}
\newcommand{\wG}{{\widetilde{G}}}
\newcommand{\sD}{{\mathscr{D}}}

\newcommand{\St}{{\text{\rm St}}}

\newcommand{\bG}{{\bold G}}
\newcommand{\bT}{{\bold T}}
\newcommand{\bZ}{{\bold Z}}
\newcommand{\bB}{{\bold B}}
\newcommand{\bN}{{\bold N}}
\newcommand{{\bU}}{{\bold U}}

\def\AAA{{\scrA(\scrS)}}
\def\EEE{{\scrE(\scrS)}}
 \def\Af{{\scrA_\flat}}
 \def\Ef{{\scrE_\flat}}%theorem labels
\def\sJ{{\mathcal J}}

\def\sgn{{\text{\rm sgn}}}
\def\sQ{{\mathscr Q}}

\def\sO{{\mathcal O}}
\def\sM{{\mathcal M}}
\def\la{{\lambda}}

\def\sR{{\mathscr R}}

\def\scc{{\textsc c}}

\def\scrC{{\mathscr C}}
\def\scrK{{\mathscr K}}
\def\scrA{{\mathscr A}}
\def\scrB{{\mathscr B}}
\def\scrF{{\mathscr F}}
\def\scrL{{\mathscr L}}
\def\scrS{{\mathscr S}}
\def\scrE{{\mathscr E}}

\newcommand{\nsJ}{{\sJ}^\natural}

\newcommand{\Spec}{{\text{\rm Spec}}}

\def\dim{\text{dim}}

\def\Hom{\text{\rm{Hom}}}
\def\End{\text{\rm{End}}}

\def\Ext{\text{\rm{Ext}}}
\def\Res{\text{Res}}

\def\O{\mathcal{O}}
\def\C{\mathscr{C}}
\def\ker{\text{Ker}}
\def\im{\text{Im}}
\def\char{\text{char}}
\def\GL{\text{GL}}

\def\ind{\text{Ind}}
\def\sgn{\text{{\rm sgn}}}
\def\Ind{\text{Ind}}
\def\Res{\text{Res}}
\def\mod{\text{\,\rm{mod}\,}}

\def\b{\mathfrak{b}}
\def\e{\mathfrak{e}}

\def\u{\mathfrak{u}}
\def\St{\text{St}}

\def\cH{{\mathtt H}}

\DeclareRobustCommand{\SkipTocEntry}[4]{}
\makeindex

\def\GL{{\text{\rm GL}}}
\def\GU{{\text{\rm GU}}}
\def\SL{{\text{\rm SL}}}
\def\SU{{\text{\rm SU}}}
\def\O{{\text{\rm O}}}
\def\SO{{\text{\rm SO}}}
\def\Sp{{\text{\rm Sp}}}
  \def\Stab{{\rm Stab}}
\def\up{{t}}
\def\fS{{\mathfrak S}}
\newcommand{\ul}[1]{\underline{#1}}

\def\scB{{\jmath}}
\def\scC{{\imath}}

\def\bfG{{\textbf{G}}}
\def\bfB{{\textbf{B}}}
\def\bfN{{\textbf{N}}}
\def\bfW{{\textbf{W}}}
\def\wW{{\widetilde W}}

\def\La{{\Lambda}}

\def\sfJ{{\mathsf J}}
\def\sfj{{\mathsf j}}
\def\o{{\text{\rm o}}}
\def\fkt{{\mathfrak t}}
\def\lra{{\longrightarrow}}
\def\lmt{{\longmapsto}}
\def\bsS{{\boldsymbol{\sS}}}
\usepackage{mathtools}% http://ctan.org/pkg/mathtools

\begin{document}
%Comments: what is Theorem 4.9

\title[Hecke endomorphism algebras]{An exact category approach to Hecke endomorphism algebras$^*$}

 %\title[Hecke endomorphism algebras]{Extending Hecke endomorphism algebras at roots of unity}
 \author{Jie Du}
 \address{School of Mathematics and Statistics\\ University of New South Wales\\ UNSW Sydney 2052
}
 \email{j.du@unsw.edu.au {\text{\rm (Du)}}}

 \author{Brian Parshall$^\dag$}
\address{$^\dag$Department of Mathematics \\
University of Virginia\\
Charlottesville, VA 22903} 
%\email{bjp8w@virginia.edu {\text{\rm(Parshall)}}}
\author{Leonard Scott}
\address{Department of Mathematics \\
University of Virginia\\
Charlottesville, VA 22903} \email{lls2l@virginia.edu {\text{\rm
(Scott)}}}

\address{*{\bf with an appendix by Veronica Shalotenko}}
%\address{Department of Mathematics\\
%University of Virginia\\Charlottesville, VA 22903}
\date{\today}

\setcounter{tocdepth}{1}
\maketitle
\setcounter{tocdepth}{3}
\begin{abstract}  
 Let $G$ be a finite group of Lie type. In studying the cross-characteristic representation theory of
$G$, the (specialized) Hecke algebra $H=\End_G(\ind_B^G1_B)$ has played a important role.  In particular, when $G=GL_n(\mathbb F_q)$ is a finite general linear group, this approach led to the Dipper-James theory of $q$-Schur algebras  $A$.  These algebras can be constructed over $\sZ:=\mathbb Z[t,t^{-1}]$ as the $q$-analog (with $q=t^2$) of an endomorphism algebra larger than $H$, involving parabolic subgroups.  The algebra $A$ is quasi-hereditary over $\sZ$. An analogous algebra, still denoted $A$, can always be constructed in other types. However, these algebras have so far been less useful than in the $GL_n$ case, in part because they are not generally quasi-hereditary.  
 
   Several
years ago, reformulating a 1998 conjecture, the authors proposed (for all types) the 
existence of a $\sZ$-algebra $A^+$ having a stratified derived module category, with strata constructed via Kazhdan-Lusztig cell theory. The algebra $A$ is recovered as $A=eA^+e$ for an idempotent $e\in A^+$. A main goal of this monograph is to prove this conjecture completely.  The proof involves several new homological techniques using exact categories. Following the proof, we show that $A^+$ does become quasi-hereditary after the inversion of the bad primes. Some first applications of the result---e.g., to decomposition
matrices---are presented, together with several open
problems. \end{abstract}																 

\vspace{4.5in}
\thanks{\small This work was supported by grants from the ARC (DP120101436, Jie Du), and the Simons Foundation (\#359360, Brian Parshall; \#359363, Leonard Scott). The last author would like to thank UNSW for support from the Distinguished Research Visitor Scheme.}

\date{\today}

\maketitle
\pagenumbering{roman}
\setcounter{page}{6}
\tableofcontents

\mainmatter

 \sloppy \maketitle

\chapter*{Introduction} 
 
%{\color{blue}

Representations of the symmetric group $\mathfrak S_r$ are the first examples of group representations systematically investigated by the pioneers in representation theory such as F.G. Frobenius, A. Young, and W. Burnside. When I. Schur used the symmetric group representations to determine polynomial representations of the complex general linear group $\GL_n(\mathbb C)$, certain finite-dimensional algebras, known as Schur algebras, played a bridging role between the two. The well-known Schur--Weyl duality summarizes the relation between the representations of $\GL_n(\mathbb C)$ and $\mathfrak S_r$. 

A Schur algebra is the endomorphism algebra of a direct sum of certain permutation modules associated with Young (or parabolic) subgroups of $\mathfrak S_r$. These permutation modules are filtered by Specht modules (with multiplicities given by Young's rule), and such filtrations result in a standard stratification for the Schur algebra so that its representation category is a so-called highest weight category,  a notion naturally arising from the study of representations of Lie groups/algebras and algebraic groups. The Schur algebra itself is an example of a quasi-hereditary algebra, a notion introduced by E. Cline together with the last two authors of this monograph (CPS) in the late 1980s.

The endomorphism algebra of a permutation representation on a Borel subgroup $B(q)$, consisting of all upper triangular matrices, of a finite general linear group $\GL_r(q)$ is called an (Iwahori--)Hecke algebra which is a $q$-deformation of the group algebra of $\mathfrak S_r$. Such Hecke algebras and their generic versions are all defined for finite groups $G$ of Lie type with a parameter $q$ \cite{Iwa64, C85}. Often $q$ is a prime power and the associated Hecke algebras are $q$-deformations of the group algebra of the Weyl group $W$ of $G$. Now, using $q$-permutation modules associated with parabolic subgroups of $W$, the notion of Schur algebras is generalized not only to $q$-Schur algebras in the type $A$ case, but also to Hecke endomorphism algebras (Hecke endo-algebras) in general.

 More precisely, let $\sH$ be a generic Hecke algebra over $\sZ:=\mathbb Z[t,t^{-1}]$ associated to a finite group $G$ of Lie type with Weyl group $W$. For each subset $J$ of the set $S$ of fundamental reflections for $W$, form the associated
$q$-permutation module $x_J\sH$. Put $T:=\bigoplus_J x_J\sH$ and call the endomorphism algebra $A:=\End_\sH( T)$ a generic Hecke endo-algebra. If $G$ is a finite general linear group and $\scrK$ is a field,
$A_\scrK=\scrK\otimes A$ is (Morita equivalent to) a $q$-Schur algebra,\index{$q$-Schur algebra} as introduced by Dipper and James \cite{DJ89}, and used effectively in the study of cross-characteristic representation theory of
$G$ (so that $q$ is nonzero in $\scrK$). The definition makes sense for $G$ of any type, but $A$ itself, without modification, appears less useful (compared to $q$-Schur algebras), partly for structural and homological reasons. For simplicity, we will assume that $G$ is a fixed-point group of a Frobenius map on a {\it simple} algebraic group $\bG$ throughout the monograph.

The structure of $q$-Schur algebras is a natural quantization from the structure of Schur algebras, thanks to the existence of filtrations of $q$-permutation modules by $q$-Specht modules. What is a replacement of Specht or $q$-Specht modules in the general case?

In 1996, CPS developed a theory of stratifying endomorphism algebras. They considered abstract Specht modules and a set of stratification hypotheses with which the associated endomorphism algebra is a standardly stratified algebra, a notion that is slightly weaker than that of a quasi-hereditary algebra. For the aforementioned generic Hecke algebras, the authors used the modules associated with Kazhdan--Lusztig left cells as ``Specht modules'' to filter these $q$-permutation modules. This results in a standard stratification for $q$-Schur algebras, but a coarser stratification for a generic Hecke endo-algebra $A$ of other Lie types; see \cite{DPS98b, DPS98a}. However, in the rank-two case,  if we enlarge $T$ to $T^+=T\oplus X$ where $X$ is a certain module filtered by certain cell modules, we obtain a finer stratification in terms of two-sided cells.
In other words, if an enlargement $T^+$ of $T$ is used, then $A^+=\End_\sH(T^+)$ could be standardly stratified. This motivated the conjecture \cite[Conj. 2.5.2]{DPS98a}, later modified (slightly) in
\cite[Conj. 1.2]{DPS15}.

The conjecture is about the construction of an extended (generic) Hecke endo-algebra $A^+$ with good homological
properties. %was also conjectured in \cite{DPS98a}, later modified in \cite{DPS15}. 
It proposes  the existence of a $\sZ$-algebra $A^+$ with a nicely stratified derived module category, with all strata constructed via Kazhdan-Lusztig cell theory. It is required by the conjecture to extend the $\sH$-module $T$ to $T^+=T\oplus X$, for some right
$\sH$--module $X$ with a suitable dual left-cell (module) filtration. Note, this implies that $A=eA^+e$, where $e$ is the idempotent projection of $T^+$ onto $T$ along $X$. The conjecture also specifies ``standard modules," parametrized by left cells, that filter a projective generator for the category $A^+$--mod of finitely generated $A^+$-modules.

 More precisely, the $q$-permutation module $x_J\sH$, filtered by dual left-cell modules $S_\omega$, has a bottom section $S_{\omega^{0,J}}$, where $\omega^{0,J}$ is the left cell that contains the longest element $w_{0,J}$ of the parabolic subgroup $W_J$.
 For a left cell $\omega$ that satisfies $S_\omega\not\cong S_{\omega^{0,J}}$ for all $J\subseteq S$, we want to construct a right $\sH$-module $X_\omega$ filtered by dual left-cell modules and with bottom section $S_\omega$. %For example, truncating some $x_J\sH$ gives such an $X_\omega$.
 We require $X_\omega$ to satisfy a certain homological property, the Ext$^1$ vanishing property, which forms part of the Stratification Hypothesis \cite[(1.2.9)]{DPS98a}. 
 %This homological property can be easily verified for $x_J\sH$ under a certain invertibility condition (see \cite[Cor. 4.5]{DPS15}), but not for a truncation of $x_J\sH$. 
 Note that, for the symmetric group, every two-sided cell contains some $w_{0,J}$ and every $S_\omega\cong S_{\omega^{0,J}}$
 for some $J$. So, we take $X=0$ in this case. In general,
we may take $X$ to be a direct sum of such $X_\omega$. Thus, the proof of the conjecture relies on the existence of such $X_\omega$.

In \cite{DPS15}, we studied a local case of the conjecture, under the equal parameter assumption for $\sH$ in order to use a result from \cite{GGOR03}. In the DPS work, the module $X_\omega$ was constructed inductively, using a new order on left cells, to satisfy the required Ext$^1$ vanishing condition. This order is strictly compatible with the two-sided cell order $\leq_{LR}$, and is defined by a height function on left cells, which is called a sorting function in \cite{GGOR03}. Note that the height function is motivated by Lusztig's $a$-function on $W$ that takes constant values on two-sided cells. In this way, we proved a special case (\cite[Th. 5.6]{DPS15}) of the conjecture \cite[Conj. 1.2]{DPS15}.

To attack the conjecture in general, the authors developed in \cite{DPS17} an exact category approach to control the Ext$^1$ vanishing property when constructing certain modules (such as $X_\omega$) with some desired filtrations. Applying
this general construction to $\sH$-modules yields the exact category $(\scrA(\scrL^*),\scrE(\scrL^*))$, where $\scrA(\scrL^*)$ is a full (additive) subcategory of mod-$\sH$, and $\scrE(\scrL^*)$ is a subcategory of the category of short exact sequences in the category mod-$\sH$ of finite right $\sH$-modules. However, this general construction does not decide if the $q$-permutation modules, though belonging to $\scrA(\scrL^*)$, satisfy the required Ext$^1$ vanishing property. To overcome the hurdle, two more exact subcategories are constructed to obtain an exact category $(\scrA_\flat(\scrL^*),\scrE_\flat(\scrL^*))$ in which
$X_\omega$ may be constructed with a required filtration and both $X_\omega$ and $x_J\sH$ satisfy the Ext$^1$ vanishing property. Though the process is quite sophisticated, we eventually prove the conjecture \cite[Conj. 1.2]{DPS15}; it is a consequence of Theorem \ref{4.4} whose proof is outlined in Section 0.2 below.

%\medskip\noindent

  \medskip\noindent
%  \begin{center}
 \section{Further results and applications}
% \end{center}
 After preliminaries in Chapters 1, we begin in Chapter 2 with a construction of a sequence of exact categories, consisting of objects with filtrations of left-cell modules, inside the category $\sH$-mod of left $\sH$-modules. By applying a standard duality theory to the construction in Section 2.2, we obtain in Section 2.3 the exact categories $(\scrA_\flat(\scrL^*),\scrE_\flat(\scrL^*))$, required in the proof of the conjecture, of right $\sH$-modules with filtrations of dual left-cell modules and the homological property satisfied by $x_J\sH$.  Section 2.4 reviews and reformulates an abstract setting developed in \cite{DPS17} for the construction of these aforementioned extra modules $X_\omega$.
 
 The conjecture is 
  stated and proved in Chapter 3; see Theorem \ref{4.4}. The proof is technically quite difficult, so we
provide at the end of this introduction a conceptual framework for the way $A^+$ and $T^+$ evolve in these sections.
 
 Theorem \ref{4.4} involves {\it no mention of primes whatsoever}, good, bad, or otherwise.  Nevertheless, the theorem  proves  that the category $A^+$-mod has a natural (standard) stratifying system.   Section 3.4 is devoted to a ``uniqueness theory." 
 (See similar uniqueness results for quasi-hereditray covers, discussed by Rouquier in \cite[\S4]{Ro08}.) Its original motivation was to sharpen Theorem \ref{4.4}: A main step in the proof of that theorem involved the construction of certain $\sH$-modules $X_\omega$ associated to
 left cells $\omega$; see
 Construction \ref{const}.  By Theorem \ref{relative inj}, the $X_\omega$ have remarkable homological properties, which are
  required in realizing $A^+$ as an endomorphism algebra. Although no claim is made that the modules $X_\omega$ are unique, in Theorem \ref{11.5}, the
 following uniqueness result is proved: 
 {\it the Morita class of $A^+$ is independent of the choice of the modules $X_\omega$.} The uniqueness theory is formulated in such a way (see Theorem \ref{11.2}) that it might be useful for other suitably filtered endomorphism algebras. In Section 3.5, we turn the categorical property --- i.e., the stratifying system for $A^+$-mod, established in  Theorem \ref{4.4} --- into the structural property for $A^+$ by explicitly displaying a defining sequence, called a standard stratification, for $A^+$ with stratifying ideals as sections; see Theoren \ref{11.5}.
 
\begin{rem}[{\bf Good and bad primes}] \label{good bad}
  The %only property of 
  ``bad primes" 
  %we need is that they 
   are the only integer primes
  that divide the denominators of generic character degree formulas; see \cite[(4.1.12)]{DPS98a} or the discussion of display \eqref{Dchi} in the appendix to Section 4.1, where only the $^2F_4$ case is considered. They form a subset of $\{2,3,5\}$.
  We declare all such prime divisors to be ``bad" (for $W$, $G$, or the Lie series to which $G$ belongs) and all others to be ``good." Recall the classical definition that a prime is good for $\bG$ if it does not divide any coefficient of the maximal root in the root system of $\bG$. By \cite[4.1]{DPS98a}, all primes that are good for $\bG$ are good for $G$.
  \end{rem}

 Suppose that $A^{+\natural}:=\sZ^\natural\otimes_\sZ  A^+$  is obtained from $A^+$ by inverting all the bad primes for $W$.  Theorem \ref{quasi-hereditary} then proves that $A^{+\natural}$ is quasi-hereditary of separable type
 in all cases (and split quasi-hereditary except in type $^2F_4$, becoming split upon the adjunction of $\sqrt{2})$).
 Sections 4.1 and 4.2 prepare for the above results in Section 4.3. They also provide an interesting study in themselves of the base-changed version $\sJ:=\sZ\otimes \sfJ$ of the Lusztig asymptotic $\mathbb Z$--algebra $\sfJ$. There is a natural inclusion of the Hecke algebra $\sH$ into (a copy of) $\sJ$, and the latter may be regarded as an order in $\sH_{\mathbb Q(t)}$ containing $\sH$. 
One approach to understanding the representation theory of the generic Hecke algebra $\sH$ is to relate it to
 the simpler algebra $\sJ$.
 Any
 $\sJ$-module $M$ can be restricted to $\sH$, and our  Theorem \ref{cellmodule} says every left-cell module for $\sH$ arises this way. Once again there is no mention of primes. However, if bad primes are inverted in $\mathbb Z^\natural\subseteq\sZ^\natural$, we show the resulting order $\sJ^\natural$ is maximal (in the sense of Reiner \cite{Reiner03}). We then prove that every stratifying ideal arising from the defining sequence for $A^+$ in Theorem \ref{11.5} is, after base change to $\sZ^\natural$, a hereditary ideal of separable type (and of split type, if ${}^2F_4$ is excluded), completing the proof of Theorem \ref{quasi-hereditary}. In the rest of Section 4.3, we focus on a refinement of the defining sequence for $A^{+\natural}$ such that its sections are indecomposable split hereditary ideals, a notion used in \cite{Ro08}. In particular, standard modules and Rouquier's isomorphism are given in Corollary \ref{refine}. Corollary \ref{B'e}
brings in the idempotents that define these indecomposable split hereditary ideals. (Note that the approach developed in \cite{Ro08} is an idempotent-free approach.) Finally, in Corollary \ref{refine3}, we prove that $A^{+\natural}$-mod is a highest weight category in the sense of \cite{DS94} and \cite{Ro08}.

 Chapter 5 is devoted to applications to several related topics.
 In Section 5.1, we consider the algebras $A$ and $A^+$ from the point of view of standardly based algebras, as introduced by Du and Rui \cite{DR98}. For example, Theorem \ref{DM}
 proves that the algebras  $A^\natural$, $A^{+\natural}$ and $\sH^\natural$ are all standardly based. The latter case was previously known (see \cite{G07}), though our proof is new.
 %The related question of whether the $q$-Schur algebras $A^\natural$ are split quasi-hereditary in all %types
 %is apparently an open question.
  
  Sections 5.2 focuses on a new study of decomposition numbers of specialized Hecke endo-algebras relevant to those for finite groups of Lie type in cross-characteristic representation theory. It is not a new theme, but our results are quite general and often deeper than those found in the literature. Theorems \ref{12.1a} and \ref{bigDecom} identify a large collection of decomposition numbers which remain
  constant for all groups in a given series, independent of their $q$-parameter, the latter parameter, defined by \cite{C85}, is a general version of the $q$ in $SL_n(q)$ for the $SL_n$-series. The uni-triangularity of the associated decomposition matrix is an easy consequence of the split quasi-hereditary structure.
  
Section 5.3 links the decomposition matrix for a Hecke endo-algebra $A$ with that of the corresponding finite groups of Lie type. We prove in Theorem \ref{12.3} that, if $A$ is the Hecke endo-algebra associated with the standard finite Coxeter system $(W,S,L)$ and $G$ is a finite group in the Lie type series associated with $(W,S,L)$, then the decomposition matrix for the specialized  $A_{\sO}$  is part of the decomposition matrix of $\sO G$. In particular, if $r$ is good, then this part of the decomposition matrix is unitriangular. This provides evidence for possible improvement of Geck's conjecture \cite{G12} proved recently in \cite{BDT20}. See comments around footnote \ref{BDT} in Chapter 5.
  Part of our proof in the adjoint case is based on a note of Veronica Shalotenko; see Appendix A. We thank her for allowing us to include this work in our monograph. We also point out in a Remark \ref{2ndproof}(1) that her proof continues to work if the hypothesis on the connected center of $\bG$ is removed, due to a result of Dipper and Fleischmann \cite[Lem. (3.8)]{DF92}.
  
  Section 5.4 is motivated by Schur's original work, the Schur--Weyl duality, as mentioned at the beginning of this Introduction, and its quantum analog. In their study of quantum symmetric pairs of type AIII, H. Bao and W. Wang \cite{BW18} established a new type of the Schur--Weyl duality between $i$-quantum groups ${\bf U}^\jmath$ and ${\bf U}^\imath$ and Hecke algebras of types $B$ and $C$, respectively. Wu and the first author have partially lifted Bao--Wang's duality to the integral level (see \cite{DW22, DW23}). In these works, certain Hecke endomorphism algebras, which are centralizer subalgebras of $A^\ddagger$, played the bridging role. In particular, the standard basis theory in Section 5.1 and  the discovery about the decomposition matrix for $A^+$ or $A^\ddagger$ in Section 5.2 can be applied to representations of $i$-quantum groups and their hyperalgebras. See Theorems \ref{unitri} and \ref{integral}.
 
 In the final section of Chapter 5, we pose some problems, suggested by this monograph, on decomposition numbers for finite groups of Lie type.  
  Each of the problems can be viewed as  a step toward understanding the decomposition numbers in reductions
  mod $p$ of ordinary irreducible representations in the unipotent series.  The problems are formulated at a ``starter" level, assuming the underlying field characteristic is a good prime.  We hope this assumption might be eventually removed in some future reformulation. The framework of Theorem \ref{4.4}, which treats all primes alike, makes such a reworking of ``good prime" formulations  possible, and provides a general starting point.  Note that Theorem \ref{4.4} itself is a kind of reformulation of the ``good prime" result Theorem \ref{quasi-hereditary}.

  Finally, the monograph closes with several appendices which cover a number of 
 preparatory topics used in its main body.

 \medskip
%\begin{center} 
\section{Outline and guide for the proof of Theorem \ref{4.4}}
%\end{center}

Let $\Omega$ be the set of Kazhdan-Lusztig left cells in a finite standard Coxeter system $(W,S,L)$ as defined in \cite{Lus03}, using parameters from an ambient algebraic group
and a fixed-point group of a Frobenius map (see \cite[\S 1.17]{C85}).\footnote{See also footnote \ref{fn1} in Section 1.1 below.} In \cite{DPS17},  the authors constructed an early version $A^\dagger$ of $A^+$ based on 
  the existence of suitable modules $X_\omega$, for all $\omega\in \Omega$. Each 
$X_{\omega}$ has a filtration with sections isomorphic to dual left-cell modules $S_\nu$, $\nu\in\Omega$, and a bottom section $S_\omega$ (defined
precisely in display \eqref{dualleft}), %above Remark \ref{Specht} below), 
which, in addition, satisfy 
\begin{equation}\label{vanishing}
\Ext^1_{\scrE(\scrS)}(S_\tau,X_\omega)=0,\quad \forall\tau,\omega\in\Omega.\end{equation}
 Here $S_\tau$ and $X_\omega$ are objects in a suitable exact category $(\scrA(\scrS),\scrE(\scrS))$. The category
$\scrA(\scrS)$ is a full subcategory of the category of $\mathbb Z[t,t^{-1}]$-free, finite right $\sH$-modules. In addition, it  is required that
each object in $\scrA(\scrS)$ have a certain filtration with sections direct sums of dual left-cell modules.  
Also, $A^\dagger:=\End_{\scrA(\scrS)}
(T^\dagger)$, where  $T^\dagger:=\bigoplus _{\omega\in\Omega} X_\omega^{\oplus m_\omega}$ with each $m_\omega$ a positive integer. This definition of $T^\dagger$, together with the 
$\Ext^1_{\scrE(\scrS)}$--vanishing (\ref{vanishing}), implies that $\Ext^1_{\scrE
(\scrS)}(S_\tau, T^\dagger)=0$, for all $\tau\in\Omega$. This is a key property for managing filtrations of $A^\dagger$ that arise from filtrations of $T^\dagger$.

 Generally speaking, the pair $(A^\dagger, T^\dagger)$ has many desirable properties, except we want $T^\dagger$ to contain, as an $\scrA(\scrS)$-summand, each of the $q$-permutation modules $x_J\sH$, as well as their direct sum (denoted $T$). 
 This problem
 presents a major difficulty in passing from \cite{DPS17} to the set-up of this monograph. It is solved by introducing a new exact
 category $(\scrA_\flat,\scrE_\flat)=(\scrA_\flat(\scrS),\scrE_\flat(\scrS))$ inside $(\scrA(\scrS),\scrE(\scrS))$ with $x_J\sH\in\scrA_\flat$ and $\Ext^1_{\scrE_\flat}(S_\tau,x_J\sH)=0$, for all $\tau$. 
 
 The construction of $(\scrA_\flat,\scrE_\flat)$ is highly nontrivial. 
  It  begins in Sections 1.3 and  2.1 which study a dual category
 $(\scrA^\flat,\scrE^\flat)$ in the  category of $\sZ$-free left $\sH$-modules. Of course, it is possible to pass between right and left $\sZ$-free, finite $\sH$-modules by simply applying the dual functor $(-)^*=\Hom_\sZ(-,\sZ)$.  However, this process interchanges left exact functors with 
 right exact functors, while we need to use the left-module versions to insure the various functors sending an exact sequence $0\to M\to N\to P\to 0$  to the fixed-point sequences $0\to M^{\sH_J}\to N^{\sH_J}\ \to P^{\sH_J}$,
 $J\subseteq S$, are  left exact. This allows us to build an exact subcategory by focusing on the
 short exact sequences $0\to M\to N\to P\to 0$ whose associated fixed-point sequences are also all short exact sequences \cite[Lem. 3.1]{DPS17}. 
  A version of this approach is used in Proposition \ref{Prop2}  to create $\scrE^\flat=\scrE^\flat(\scrL)$ inside an exact category $(\scrA(\scrL),\scrE(\scrL)$ dual to $(\scrA(\scrS),\scrE(\scrS))$.
 The category $\scrA^\flat$ is then defined to consist of objects with certain ``internal" candidate exact sequences belonging to $\scrE^\flat$. Remarkably, the objects $\sH x_J$ belong to $\scrA^\flat$; see Corollary \ref{3.7a}. This implies that each dual object
$ x_J\sH$ belongs to $\scrA_\flat$. We want more than that, however. We want $x_J\sH$ to be
a direct summand of a new version of $T^\dagger$ with the key vanishing \eqref{vanishing} preserved. This forces us to prove that $\Ext^1_{\scrE_\flat}(S_\omega,x_J\sH)=0$, for each $\omega\in\Omega$ and $J\subseteq S$. This vanishing is proved in its dual
incarnation, Theorem \ref{ext0}.  The short exact  sequence  (\ref{F}) displayed in the proof of Theorem \ref{ext0} is the key ingredient. It is available from the exactness of $(-)^{\sH_J}$ there.

 Section 2.3 can be viewed as dualizing the constructions of Section 2.1. Of course, Section 2.1 was formulated with these duals in mind.
 
 A main step in completing the proof of the conjecture, at this point, is carried out in Section 2.4, namely, in Theorem \ref{relative inj}, which imitates the construction of $X_\omega$ in \cite[Th. 4.5]{DPS17}. 
 
 Section 3.1 is also preparatory, providing more detail on dual left-cell module filtrations of $q$-permutation
 modules. It refines the analysis of $q$-permutation module filtrations begun above Corollary \ref{hotcor}. The discussion
 draws from earlier work in \cite{DPS98a}.
 
 The main goal of Section 3.3 is to prove Theorem \ref{4.4}. This implies Conjecture \ref{conjecture}, equivalent to    \cite[Conj. 1.2]{DPS15}, a modification of \cite[Conj. 2.5.2]{DPS98a}.
 A further consequence of  the proof is the $\Ext^1$-vanishing result Corollary \ref{qperm}.  It is used later in Section 3.4 in constructing a uniqueness Morita theory complementary to Theorem \ref{4.4}, and discussed further below.  
 
 For further discussions in outline of this monograph, see the previous two sections.
 
 \vspace{1.5cm}
%\bigskip

\noindent
Sydney and Charlottesville,\hfill Jie Du\;\,\hspace{1.4cm}

\noindent
September 2022 \hfill Brian Parshall$^\dag$

\hfill Leonard Scott\;\,

 %%%%%%%%%%%%%%%%%
% \vspace{2cm}
\newpage
\begin{center}
{\bf Acknowledgements}
\end{center}
%\medskip\noindent
%\hfill$***$\hfill$***$\hfill$***$\hfill\null

\medskip

We would like to thank our colleague Weiqiang Wang for several discussions and special assistances during the writing of this monograph. Section 5.4 was heavily influenced by his work with H. Bao. %Jie Du extends his special thank to Weiqiang for  during his several visits.

Part of the work was written while Jie Du was visiting the University of Virginia in 2016 and 2022,
using Simons grants there. Leonard Scott visited the University of New South Wales in 2015 and 2019 with support there through an ARC grant and the Distinguished Research Visitor Scheme. Both authors want to thank the two universities for their support and hospitality.

 Finally, we wish to thank Karen Parshall (University of Virginia) for reading the entire manuscript and making many helpful comments.

\bigskip\noindent
%\hfill$***$\hfill$***$\hfill$***$\hfill\null

\begin{center}
{\bf A special note}
\end{center}

\medskip
 Our lifelong collaborator and friend, Brian Parshall, passed away on January 17, 2022. This was a few days after a Zoom meeting we had held, generally weekly, since mid 2020.  This last meeting was very successful and lasted 2 hours. The last question, one Jie had for Brian,  was about using ``Hecke endo-algebra'' as a short form for ``Hecke endomorphism algebra''. Brian answered ``it sounds good. I will think about it''. Unfortunately, these words from Brian became the last words we heard from him.  He passed away within a week of this last meeting. It was a shock and surprise for everyone, certainly for Brian’s collaborators. Most of the work on this monograph, however, had been done with Brian’s participation. Its remaining co-authors,  Jie Du and Leonard Scott, wish to take this opportunity to acknowledge Brian Parshall as being a co-author of this monograph. In a variation on the words above that Brian used in our last meeting,``we didn’t have to think about it.''

 \bigskip
 \hfill Jie Du and Leonard Scott\qquad\qquad

 \chapter{Preliminaries}
 
 In this chapter, we briefly review some notions that will be used throughout. Topics include finite groups of Lie type, generic Hecke algebras, exact categories and stratified algebras.
 
If $B$ is a ring, then $B$-mod (respectively, mod-$B$) will denote the category of all finite (i.e., finitely generated) left
(respectively, right) $B$-modules.   Usually, $B$ will be a finitely generated algebra over a commutative Noetherian ring
$\scrK$. In this case, and unless otherwise noted, objects in $B$-mod will be assumed to be finite $\scrK$-modules
(with a similar convention for mod-$B$).

\section{Finite groups of Lie type}
   We  review some basic notions concerning finite groups of Lie type. More material on this topic will be introduced  in Chapter 5. We closely follow the treatment given in Carter's book \cite{C85}. In particular, we will assume the reader is familiar with the notion \cite[pp.16--17]{C85} of a connected, reductive algebraic group over an algebraically closed field of positive characteristic. Such a group $\bG$ may be characterized as the internal product $\mathbf Z^\circ\bG'$, where $\mathbf Z^\circ$ is a torus, the connected component of the identity in the center $\mathbf Z=Z(\bG)$ of $\bG$, and $\bG'$ is a semisimple algebraic group, defined by Carter to be nontrivial and an ``almost direct product"  (a ``central product" in traditional terminology) of simple algebraic groups. The latter are defined as having a maximal normal closed subgroup which is finite. Note that
   $\bG'$ is  necessarily the commutator subgroup of $\bG$.

        Now let $\bG$ be a (connected) reductive algebraic group over an algebraically closed field
        of positive characteristic $p$.  Recall that a Frobenius map $F:\bG\to\bG$ in Carter's sense \cite[\S1.17]{C85} is a morphism,
         some power of which equals a standard Frobenius map. Fix a Frobenius map $F$, and let $G:=\bG^F$ 
be the group of fixed points of $F$. Necessarily, $G$ is finite.  It is called a {\it finite group of Lie type},\index{finite group of Lie type} and all finite groups of Lie type arise this way.\footnote{\label{fn1}Our notation differs slightly from that in \cite{C85}, where $\bG$ is sometimes written $G$. Also, we sometimes call a Frobenius map a Frobenius endomorphism or a Frobenius morphism. See  \cite[p. 31]{C85} for a detailed definition. A Frobenius endomorphism is sometimes called a Steinberg endomorphism in other sources (see, e.g., \cite{MT11}).
We remark from \cite[p.184, top]{MT11} that Carter's assertion near the bottom of \cite[p. 31]{C85} (just above the definition)
regarding  surjective endomorphisms is not correct. However, this does not affect the validity of his definition.}

The finite groups of Lie type most studied here arise in the context of the previous paragraph when 
$\bG$ is taken to be simple.\footnote{\label{simpleG}This assumption is convenient for the description of the Lie type series and weight function $L$ discussed below. Note also that sometimes it is useful to study the more general case where $\bG'$ is simple, either to place $\bG$ in a familiar context (e.g., $\bG$ is a general linear group), or,  after beginning a study using a simple group $\bG$, viewing it as a commutator group of a larger reductive group $\widetilde{\bG}$ with favorable properties, such as a connected center. See Lemma \ref{functor} in Chapter 5.}
In addition to 
the defining (algebraically closed) field, the (simple) algebraic group $\bG$ is determined by its root datum and its isogeny class (see \cite[\S\S1.9, 1.11]{C85}), and $\bG^F$ is discussed in \cite[\S1.19]{C85}.  

It is known that $\bG$ has an $F$-stable Borel subgroup $\bB$. In addition, $\bB$ contains a maximal torus $\bold T$ of $\bG$ which is $F$-stable. The Weyl group $\bold W:=N_{\bG}(\bold T)/{\bold T}$ is then part of a finite Coxeter system $(\bold W,\bold S)$ for a (finite) subset $\bold S$ of elements of order 2 and $\bf W$ is irreducible (i.e., the Coxeter diagram is connected) under the assumption that $\bG$ is simple. In addition, the pair $(\bold B,\bold N)$, where $\bold N=\bold N_{\bold G}(\bold T)$,
give $\bold G$ the structure of a BN-pair. The group $G=\bG^F$ also has a natural BN-pair structure $B,N$, where $B=\bold B^F$ and $N=\bold N^F$. Its Weyl group is $W=\bold W^F$. \index{$W$, a Weyl group} In fact, $W$ is generated by the set $S$ of all the long words in the various $\bold W_J$, as stated in \cite[p. 34]{C85} and proved in \cite[Lem. C.1]{MT11}.
%Here $\bold W_J$ is a subgroup of $\bold W$ generated by a subset $J$ of its fundamental reflections, that subset here being any orbit of a permutation associated to $F$ (discussed further below).
Moreover, the proof in \cite{MT11} shows that every element $w\in W$ can be written as a product of elements of $S$ with lengths additive. One can use these facts to complete Carter's sketch \cite[p.34]{C85} that $W$ is the Weyl group in a (split) BN-pair structure on $G$, so, in particular, it is a Coxeter group (with $S$ as its set of fundamental generators).

Let $\Phi$ be the root system of $\bold G$, and let $\Delta$ be the set of simple roots associated to $\bB$.
The Frobenius map $F$ on $\bG$ is closely  related to a certain permutation $\rho$ of $\Delta$. In fact, $\rho$ is a graph automorphism of  the undirected
Dynkin diagram associated to $\Delta$.
(See \cite[\S1.19]{C85}.)  Two finite groups $G_1$ and $G_2$ are said to be in the same {\it Lie type series}\index{finite group of Lie type! Lie type series of $\sim$} if there exist a connected reductive group $\bG$ and Frobenius morphisms $F_1$ and $F_2$, inducing the same graph automorphism $\rho$, and such that $G_i=\bG^{F_i}$, for $i=1,2$. By inspection, this amounts to saying that $G_1$ and $G_2$ are in one of the groupings
on \cite[pp. 39--41]{C85}; see also \cite[p.193]{MT11}.
Further, if $\rho$ is the trivial (or identity) graph automorphism, then $G$ is called a {\it split} (or untwisted) group of Lie type; \index{finite group of Lie type! split $\sim$} if $\rho$ is non-trivial, then $G$ is called a {\it quasi-split} (or twisted) group of Lie type. \index{finite group of Lie type! quasi-split $\sim$} Here, following Carter's classification in \cite[p.41]{C85}, we exclude the Ree and Suzuki groups in the quasi-split case.

% associated to a quasi-split $(W,S,L)$ is the fixed-point subgroup of $\bG$ of a Frobenius morphism $F$ that induces a non-trivial graph automorphism of its Dynkin diagram. In this case, $G$ is also called a {\it quasi-split} or {\it twisted group} of Lie type. If $F$ induces a trivial graph automorphism of the Dynkin diagram of $\bG$, then $G=\bG^F$ is also called a split or untwisted group of Lie type.\footnote{in \cite[p.41]{C85}, quasi-split finite groups of Lie type do not include the Ree and Suzuki groups.}

There is a natural action of $F$ on the character group $\bold X$ (respectively, cocharacter group $\bold Y$) of $\bT$. 
 The maximum of the absolute values of  the eigenvalues of $F$ as an operator on $\bold X$ is denoted $q$.
It is a (possibly fractional) power $q=p^d$, where $p$ is the characteristic of the defining field of $\bG$. For example, if $G$ is a classical finite (split) Chevalley group over a finite field $k$ of order $r$ (an integral power of the defining characteristic $p$),
then $q=r$.  In this case, $\rho=1$, the identity map.

Returning to the general case, a useful construction of $W$ as a subgroup of $\bold W$ using only the Dynkin diagram permutation $\rho$ associated to $F$
(rather than $F$ itself) can be based on \cite[p. 34]{C85}: First, $\bold W$ may be viewed as constructed from the Dynkin diagram $\Gamma$, with the fundamental reflections associated to the nodes of $\Gamma$. Thus, we can associate to each orbit $J$ of $\rho$ on $\Gamma$ a parabolic subgroup $\bold W_J$ of $\bold W$. Now, the fundamental reflections $s$ in $W$ may be defined as the longest words of the various subgroups $\bold W_J$. Next, define the ``weight function" $L$ on $W$ by taking it to be the restriction of the length function $\boldsymbol{\ell}$ on $\bold W$ to $W$. That is, $L(w) := \boldsymbol\ell(w)$, for
$w\in W\subseteq \bold W$. Then,
$(W,S, L)$ is a finite (irreducible) Coxeter system\index{Coxeter system} with weight function $L:W\to\mathbb Z$ in the sense of \cite[\S 3.1]{Lus03}. Thus, 
 if $w=xy$ with $\ell(w)=\ell(x)+\ell(y)$, then $L(xy)=L(x)+L(y)$, for $x,y,w\in W$. 
 
Within its Lie type series, the group $G=\bG^F$ is determined by 
    (and determines) a parameter $q$ \cite[p. 35]{C85} which is described as follows. 
      If $\delta$ is the smallest positive integer such that $F^\delta$ is a standard Frobenius endomorphism (as per \cite[p. 31]{C85}), then $q^\delta$ (rather than $q$ itself) is defined to be the power of $p$ associated to the latter endomorphism. The number $\delta$ is also the order of the Dynkin diagram graph automorphism $\rho$. If $G$ is a (split) Chevalley group, $\delta=1$ and $q$ can be any positive power of the defining characteristic $p$.  In the quasi-split cases which are not split, $\delta$ is the order of a non-trivial automorphism of the Dynkin diagram, and $q$ is still any positive power of $p$. However, in the Ree and Suzuki cases, $\delta=2$ and $q^2=p^{2n+1}$ with $p=2$ for the $^2B_2$ and $^2F_4$ cases and $p=3$ for the $^2G_2$ case. In particular, $q$ is a positive power of $\sqrt{p}$. See
 \cite[p. 41]{C85}.
 Remarkably, in all cases, $q$ together with the weight function $L$ gives
 \begin{equation}
 \label{weight} q^{L(s)}=[B:B\cap B^s],\end{equation} 
 where $B$ is a Borel subgroup of $G$ and $s$ is a fundamental reflection in its Weyl group $W$. Here $B=\bold B^F$, where $\bold B$ is an $F$-stable Borel subgroup of $\bold G$; see \cite[pp. 34-35, and p. 74]{C85} 
 for further details.
% Also, $L$ is the weight function as discussed in \S1.1; in particular, $L(s)$ is the length of $s$ in the Weyl group of $\bG$. 
This formula holds in the more general case where $\bG$ is connected and reductive. 
    (See the partial discussion in \cite[p. 74]{C85} and also  \cite[11.8]{St68}.)

 Throughout this work, $W$ will be, unless otherwise noted, the Weyl group of a finite group $G$
of Lie type (i.e., the Weyl group of the BN-pair associated to $G$). We will use the natural
positive weight function $L$, constructed above and  illustrated in \cite[Lem.16.2]{Lus03}. The resulting weighted Coxeter system \index{Coxeter system! weighted $\sim$} $(W,S,L)$ is irreducible and will be called {\it standard finite}.\index{Coxeter system! standard finite $\sim$}\footnote{The Hecke algebras associated to these triples are especially important
for the modular representation theory of unipotent blocks for finite groups of Lie type. According to a theory put forward by Bonaf\'e and Rouquier
\cite{BR03}, \cite{BR17}, the case of general blocks (for finite groups of Lie type) could, eventually, reduce to the unipotent case. In any event, all irreducible or indecomposable modules with nontrivial cohomology belong to the unipotent principal block.} \index{$W$, a Weyl group! $(W,S,L)$, a standard finite $\sim$}
We also ignore the rank one case, where $W$ has order 2, unless otherwise noted.\footnote{This connectedness assumption is a matter of convenience and is unnecessary in most cases, e.g., if no component is dihedral of order 16, the Weyl group of $^2F_4$.} 
 Note that standard finite Coxeter systems $(W,S,L)$ can also be divided into the {\it split case}, where the weight function $L$ equals  the length function $\ell$ on $W$, and the {\it quasi-split case}, where $L\neq \ell$. (The only non-trivial case for the Ree and Suzuki groups is the $^2F_4$ case, where $W$ is the dihedral group of order 16.) See \cite[Chs. 15\&16]{Lus03}.

%The next subsection goes here and comes from Sec. 5.3a later in the book.
\begin{example}\label{GL_n} Let $\mathbb F_q$ be the finite field of $q$ elements, and let $n$ be a positive integer. The finite general linear group $\GL_n(\mathbb F_q)$ is the fixed-point group of the following (standard) Frobenius map
\begin{equation}\label{F_q}
F_q:\GL_n(k)\longrightarrow \GL_n(k),\;(a_{i,j})\longmapsto (a_{i,j}^q),
\end{equation}
where $k$ is the algebraic closure of $\mathbb F_q$. Note that $\GL_n(k)$ is a connected reductive group and the subgroup
$\SL_n(k)$ of determinant 1 matrices is simple with $F_q$-fixed-point group $\SL_n(\mathbb F_q)$.

For any field $k$, $\GL_n(k)$ has the following split BN-pair:
$$\aligned
B_n(k)&:=\{\text{subgroup of all upper triangular matrices in }\GL_n(k)\}\\
N_n(k)&:=\{\text{subgroup of all monomial matrices in }\GL_n(k)\}.
\endaligned
$$
Let $T_n(k)=B_n(k)\cap N_n(k)$ be the subgroup of diagonal matrices.
If $n_i$ denotes the permutation matrix obtained by interchanging the $i$-th and $(i+1)$-th rows of $I_n$, then the Weyl group
$$\widetilde W:=N_n(\mathbb F)/T_n(\mathbb F)\cong \fS_n \text{ under the correspondence }\;n_i\mapsto \tilde s_i=(i,i+1).$$
See, e.g., \cite[1.6.9]{G03} for more details.
\end{example}
The other finite classical groups can be constructed by fixed-point groups of certain group automorphisms on $\GL_n(\mathbb F_q)$.  In Example \ref{SO_n}, we consider the symplectic and orthogonal groups.  They are relevant to our studies of $i$-quantum groups in Section 5.4. Unitary groups are taken up in Example \ref{SU_n}. These both help to complete the classical group picture and to provide examples of ``unequal parameter" weight functions.

\begin{example}\label{SO_n} For any field $k$ and $J\in \GL_n(k)$, consider the group isomorphism
\begin{equation}\label{vartheta}
\vartheta=\vartheta_J:\GL_n(k)\longrightarrow \GL_n(k),\;x\longmapsto J^{-1}(x^t)^{-1}J,
\end{equation}
and define orthogonal and symplectic groups as the fixed-point groups of some $\vartheta$:
$$\aligned
\O_n(k)&:=\{x\in\GL_n(k)\mid J_n=x^tJ_nx\},\\
\Sp_{2m}(k)&:=\{x\in\GL_{2m}(k)\mid J_{2m}^-=x^tJ_{2m}^-x\},
\endaligned$$
where 
\begin{equation}\label{JJ'}
J_{n}:=\begin{pmatrix}
0&0&\cdots&0&1\\
0&0&\cdots&1&0\\
&\cdots&\cdots&\cdots&\\
0&1&\cdots&0&0\\
1&0&\cdots&0&0\\
\end{pmatrix},\quad J^-_{2m}:=\begin{pmatrix}0&J_m\\-J_m&0\\\end{pmatrix}.
\end{equation}

Let $\SO_n(k)=\O_n(k)\cap\SL_n(k)$ and 
let $\bfG\subseteq\GL_n(k)$ be one of the classical groups:
$$
\bfG=\begin{cases}
\SO_{2m+1}(k), &n=2m+1,\text{char}(k)\neq 2;\\
\Sp_{2m}(k), &n=2m,\text{any characteristic};\\
\SO_{2m}(k), &n=2m,\text{char}(k)\neq 2.
\end{cases}$$
 These are known as split twisted groups
in \cite[Ch. 11]{St67}.\footnote{See \cite[p.8]{MT11} for the definition for char$(k)=2$.}

If $k$ is the algebraic closure of $\mathbb F_q$, then $\bfG$ is a connected simple group and has a split BN-pair $\bfB=\bfG\cap B_n(k)$, $\bfN=\bfG\cap N_n(k)$.

For the Frobenius map $F_q$ as defined in \eqref{F_q} and $J=J_n$ or $J_{2m}^-$, since $J=x^tJ x$ implies $J=(F_q(x))^tJ F_q(x)$, it follows that $F_q$ induces the (standard) Frobenius maps
$$F:\bfG\longrightarrow \bfG,$$
and $\bfG^F=\SO_n(\mathbb F_q)\; (n=2m+1, 2m)$ or $\Sp_{2m}(\mathbb F_q)$. Note also that $\bfG^F=\GL_n(\mathbb F_q)^\vartheta$ has BN-pair $\bfB^F=B_n(\mathbb F_q)^\vartheta$ and $\bfN^F=N_n(\mathbb F_q)^\vartheta$ since $\vartheta$ stabilizes the split BN-pair $(B_n(k), N_n(k))$.
In particular, $\vartheta$ induced a group automorphism 
$\sigma:\wW\longrightarrow\wW$, and there is a group isomorphism
$$\bfW:=\bfN/\bfB\cap\bfN \cong\wW^{\sigma}=N_n(\mathbb F_q)^\vartheta/T_n(\mathbb F_q)^\vartheta.$$
Note that the map induced by $F$ on $\bfW$ is the identity map. Note also that, for $\bfG=\SO_{2m+1}(k)$ or $\Sp_{2m}(k)$, $W:=\bfW^F=\bfW$ is the Weyl group of type $B/C$ with generators $S=\{s_1,\ldots,s_{m-1}, t\}$, where
$$\aligned
s_i&=\tilde s_{m-i}\tilde s_{m+1+i},  \text{ and } t=\tilde s_m\tilde s_{m+1}\tilde s_m, \text{ for }n=2m+1;\\
s_i&=\tilde s_{m-i}\tilde s_{m+i},\quad  \text{ and } t=\tilde s_m, \text{ for }n=2m.\endaligned
$$
Finally, $t$ is the homomorphic image of the elements in $N_n(k)^\vartheta$:
$$\begin{pmatrix}I_{m-1}&0&0\\ 0&J&0\\0&0&I_{m-1}\end{pmatrix},\text{ where }J=\begin{pmatrix}0&0&1\\0&-1&0\\1&0&0\end{pmatrix}\text{ or }
\begin{pmatrix}0&1\\-1&0\end{pmatrix}.$$
For more details, see, e.g., \cite[1.7.3]{G03}.
\end{example}

\begin{example}\label{SU_n}  Now, following \cite[\S4.1]{DVV19}, we combine $F_q$ in \eqref{F_q} with $\vartheta$ in \eqref{vartheta}, using the choice $J=J_n$, to obtain a new Frobenius morphism
$$F:=F_q\vartheta=\vartheta F_q:\GL_n(k)\longrightarrow \GL_n(k).$$
Since $\vartheta^2$ is the identity map, it follows that $F^{2}$ is the standard Frobenius map $F_{q^2}$ and 
$\GL_n(k)^{F}\subseteq\GL_n(\mathbb F_{q^2})$. We have 
$$\GU_n(\mathbb F_q):=\GL_n(k)^{F}=\{x\in\GL_n(\mathbb F_{q^2})\mid (F_q(x))^tJ_nx=J_n\}.$$ 
This is called the {\it general unitary group} associated with the hermitian form defined by $J_n$. Replacing G by S defines the special unitary group. (See also \cite[Ex. 21.48(2)]{MT11} or \cite[4.1.10(c)]{G03}.)

%For $\SU_3(\mathbb F_{q^2})$, the $F$-orbits of the simple roots $\alpha_1,\alpha_2, \alpha_3$ for $\SL_4(k)$ are $\{\alpha_1,\alpha_3\}$, $\alpha_2$, we have root subgroups

For $\SU_4(\mathbb F_{q^2})$, the $F$-orbits of the simple roots $\alpha_{1,2},\alpha_{2,3}, \alpha_{3,4}$ for $\SL_4(k)$, with associated root subgroups $U_{i,j}=\{I_4+aE_{i,j}\mid a\in \mathbb F_{q^2}\}$, are $\{\alpha_{1,2},\alpha_{3,4}\}$, $\alpha_{2,3}$.
The root subgroups $\SU_4(\mathbb F_{q^2})=\SL_4(k)^F$ are
$$(U_{1,2}U_{3,4})^F=\{I_4+aE_{1,2}-a^qE_{3,4}\mid a\in\mathbb F_{q^2}\},\quad U_{2,3}^F=\{I_4+aE_{2,3}\mid a+a^q=0\}.$$
The values $q^{L(s)}$ ($s\in S$) associated to the weight function $L$ on $\bfW^F=W(B_2)$ in this case are $q^2, q$. In general,
the values $q^{L(s)}$ for $\SU_{2m}(\mathbb F_{q^2})$ are listed in the Coxeter diagram:

\begin{center}
\begin{tikzpicture}[scale=1.5]
%\fill(-1,0) node {$B_m$:};
\fill (0,0) circle (1.5pt);
\fill (0,.25) node {$q^2$};
\fill (1,0) circle (1.5pt);
\fill (1,.25) node {$q^2$};
\fill (2,0) circle (1.5pt);
\fill (2,.25) node {$q^2$};
\fill (4,0) circle (1.5pt);
\fill (4,.25) node {$q^2$};
\fill (5,0) circle (1.5pt);
\fill (5,.25) node {$q$};
  \draw (0,0) node[below] {$_1$} --
        (1,0) node[below] {$_2$} -- (2,0)node[below] {$_3$}--(2.5,0);
\draw[style=dashed](2.5,0)--(3.5,0);
\draw (3.5,0)--(4,0) node[below] {$_{m-1}$};
\draw (4,0.05) --
        (5,0.05);
\draw (4,-0.05) --
        (5,-0.05) node[below]  {$_m$};
%\draw (4.4,0.1)--(4.6,0);
%\draw (4.4,-0.1)--(4.6,0);
\end{tikzpicture}
\end{center}

For $\SU_3(\mathbb F_{q^2})$, the $F$-orbits of the simple roots $\alpha_{1,2},\alpha_{2,3}$ for $\SL_4(k)$ are $\{\alpha_{1,2},\alpha_{2,3}\}$, and we have the root subgroup (cf. \cite[4.5.12(b)]{G03})
 $$(U_{1,2}U_{2,3}U_{1,3})^F=\bigg\{\begin{pmatrix}1&a&b\\0&1&-a^q\\0&0&1\end{pmatrix}\;\bigg|\;a,b\in\mathbb F_q, b^q+b+a^{q+1}=0\bigg\}.$$
 So, $q^{L(s)}=q^3$. In general, the values $q^{L(s)}$ ($s\in S$) for $\SU_{2m+1}(\mathbb F_{q^2})$ take the form:
 \begin{center}
\begin{tikzpicture}[scale=1.5]
%\fill(-1,0) node {$B_m$:};
\fill (0,0) circle (1.5pt);
\fill (0,.25) node {$q^2$};
\fill (1,0) circle (1.5pt);
\fill (1,.25) node {$q^2$};
\fill (2,0) circle (1.5pt);
\fill (2,.25) node {$q^2$};
\fill (4,0) circle (1.5pt);
\fill (4,.25) node {$q^2$};
\fill (5,0) circle (1.5pt);
\fill (5,.25) node {$q^3$};
  \draw (0,0) node[below] {$_1$} --
        (1,0) node[below] {$_2$} -- (2,0)node[below] {$_3$}--(2.5,0);
\draw[style=dashed](2.5,0)--(3.5,0);
\draw (3.5,0)--(4,0) node[below] {$_{m-1}$};
\draw (4,0.05) --
        (5,0.05);
\draw (4,-0.05) --
        (5,-0.05) node[below]  {$_m$};
%\draw (4.4,0.1)--(4.6,0);
%\draw (4.4,-0.1)--(4.6,0);
\end{tikzpicture}
\end{center}

\end{example}
% Part of Section 5.3a---make it a subsection. Maybe the end of 5.3a on $q$-Schur algebras could be moved forward? Chapter 4?

 \section{(Iwahori-)Hecke algebras}

We review the theory of Hecke algebras  in \cite{Lus03}.\index{Hecke algebra} Each weighted Coxeter system $(W,S,L)$ defines an algebra
$\sH$, called an {\it Iwahori-Hecke algebra} in \cite{Lus03} and called a {\it generic Hecke algebra } in this monograph. See  (\ref{relations}) and
(\ref{newtildebasis}) for more details.

  Many results in \cite{Lus03} are contingent on a series P1---P15 of 
conjectures (involving $(W,S,L)$ and $\sH$), all of which have been proved in the standard finite case. See \cite[Chs. 15, 16]{Lus03},
 revisited in \cite[Ch. 2]{GJ11} (especially \cite[2.4.1(b)]{GJ11} for the case $^2F_4$). Thus, we can assume Lusztig's conjectures
 P1--P15 hold for the triples $(W,S, L)$ considered in this monograph. 
 {\it Henceforth, we assume that $(W
 ,S,L)$ is  standard finite unless otherwise noted. In particular, $L$ is positive and Lusztig's conjectures
P1-P15 hold, as does 
condition (\ref{order}) below.} In fact, we mention that in 
Chapters 1--3 which include the main
 theorem, we need only assume conjectures P4 and P9, which, together with the rank 2 discussion 
 in \cite[p. 204]{DPS98a},\footnote{See footnote \ref{error} in Chapter 3 which corrects a typo.} imply that property (\ref{order}) holds.
  
 Let $\sZ={\mathbb Z}[t,t^{-1}]$ be the ring of integer Laurent polynomials in a variable
$t$. \index{$\sZ={\mathbb Z}[t,t^{-1}]$}
 Let $\sH$ be the generic Hecke algebra\index{Hecke algebra! generic $\sim$} over $\sZ$ associated with $(W,S,L)$. Recall 
that $\sH$ has basis $\{T_w\}_{w\in W}$ and defining
relations ($s\in S, \,w\in W$):
\begin{equation}\label{relations}T_sT_w=\begin{cases} T_{sw}, \quad sw>w;\\ t_s^2T_{sw}+ (t_s^2-1)T_w,\quad sw<w,\end{cases}
\end{equation}
where $t_s:=t^{L(s)}$ (and, more generally, $t_w=t^{L(w)}$). Also, for $x,y\in W$, $x<y$ denotes the Chevalley-Bruhat partial order.
The algebra $\sH$ is defined just using $t^2$, but, in practice, it is convenient to have its 
square root $t$ available.\footnote{\label{semisimpleH}It is well known that the algebra $\mathbb Q(t)\otimes_{\sZ}\sH$ is a semisimple
$\mathbb Q(t)$-algebra.  See \cite[Cor. 68.12]{CR81} as well as Remark \ref{Rem7.4}(b) below.}

We often use a different basis $\{\widetilde T_w\}$ also indexed by $W$ defined by 
putting 
\begin{equation}\label{newtildebasis}
\widetilde T_w:=t^{-L(w)}T_w,\quad w\in W.\end{equation}
 (Our notation here differs slightly from that in \cite{Lus03}, where our $\widetilde T_w$ is denoted $T_w$.)

The {\it bar  operator}\; $\bar{\ }:\sH\to\sH$ sending $t$ to $t^{-1}$ and $T_w$ to $(T_{w^{-1}})^{-1}$ is a ring involution on
$\sH$ (sending $t$ to $t^{-1})$, and
 $\sH$ admits a Kazhdan-Lusztig (or ``canonical" in the
sense of  \cite{Lus90}; see also \cite[\S0.5]{DDPW08}) basis $\{\scc_w\}_{w\in W}$. This basis is defined in \cite[Chap.5]{Lus03} using the formula
$$\label{rules} \scc_w=\sum_{y\leq w}p_{y,w}\tT_y.$$
In addition to the condition $\overline{\scc}_w=\scc_w$, for all $w$,  the Laurent polynomials $p_{y,w}\in\sZ$ satisfy (1) $p_{y,w}=0$ unless $y\leq w$, (2) $p_{w,w}=1$, and (3) $p_{y,w}\in\sZ_{<0}:=t^{-1}\mathbb Z[t^{-1}]$, if $y<w$.  The polynomials $p_{y,w}$ are determined by these properties. 
For example, 
\begin{equation}\label{Calpha}\scc_s=t_s^{-1}+\widetilde T_s= t_s^{-1}(1+ T_s).\end{equation}
(Note that the element $\scc_x$ is denoted $c_x$ 
in \cite{Lus03}.)

If $J\subseteq S$ is non-empty, let $W_J:=\langle s \rangle_{ s \in J}$, a parabolic subgroup of $W$;  if $J=\emptyset$, set $W_J:=\{1\}$. Then, $W_J$ has a unique
 longest (reduced) word, denoted $w_{0,J}$, and
\begin{equation}\label{longest} \scc_{w_{0,J}}=t^{-L(w_{0,J})}\sum_{y\in W_J}T_y=: t^{-L(w_{0,J})}x_J\end{equation}
(which defines $x_J$).
See  \cite[Cor. 12.2]{Lus03} (or \cite[(7.2.5)]{DDPW08} in the split case). In the literature, the endomorphism algebra
\begin{equation}\label{q-Schuralgebra} 
A:=\End_{\sH}\left(\bigoplus_{J\subseteq S}x_J\sH\right)\end{equation}
(or any algebra Morita equivalent to it) is called ``the" generalized generic $q$-Schur algebra or a {\it generic Hecke endomorphism algebra} ({or \it generic Hecke endo-algebra} for short) by analogy with our naming of $\sH$. 
%Here $q=t^2$. 
Various specializations of $A$ and its extensions 
will also be considered in later chapters. \index{$q$-Schur algebra! generic $\sim$}\index{generic Hecke endo-algebra} \index{$A$, generic Hecke endo-algebra} 
%q Schur algebra....

 Let $h_{x,y,z}\in\sZ$ denote the structure constants for $\sH$, defined by the equations 
\begin{equation}\label{structureconstants}\scc_x\scc_y=\sum_{z\in W}h_{x,y,z}\scc_z.\end{equation}
In particular, we have \cite[Th.~6.6]{Lus03}
\begin{equation}\label{CsCw}
\scc_s\scc_w=\begin{cases}\scc_{sw}+\sum_{x:sx<x<w}\mu^s_{x,w}\scc_x,&\text{ if }w<sw;\\
(t_s+t_s^{-1})\scc_w,&\text{ if }w>sw.
\end{cases}
\end{equation}
Here $\mu^s_{x,w}\in \sZ$ is fixed under the bar operator. In the case of split groups, $\mu^s_{x,w}$ is constant (in $\mathbb Z$). Also,
\eqref{CsCw} implies that the $\scc_s$, $s\in S$, generate the $\sZ$-algebra $\sH$.

%A relation $\leq$ on a set $\Lambda$ that is reflexive and transitive is called a {\it preorder}. \index{preorder} We will call $\Lambda=(\Lambda,\leq)$ a {\it quasi-poset}\index{quasi-poset} in the sequel. For such a preorder $\leq$, there is an associated equivalence relation $\sim=\sim_\leq$ ($x\sim y\iff x\leq y\, \& \, y\leq x$), and a strict version $<$ of $\leq$ (thus, $x<y\iff x\leq y\;\&\;x\not\geq y\iff x\leq y\; \&\; x\not\sim y$). If $\overline\Lambda=\Lambda/\sim$ denotes the set of equivalence classes, then $\leq$ induces a partial order, denoted again by $\leq$, on $\overline\Lambda$. {\color{blue}Call $(\overline \Lambda,\leq)$ the {\it poset associated with} $(\Lambda,\leq)$.}\index{quasi-poset! poset associated with $\sim$} Thus, for any $\la,\mu\in\Lambda$, if $\overline\lambda$ denotes the equivalence class containing $\lambda$, then we have $\overline\lambda\leq\overline\mu\iff\lambda\leq\mu$. For further details, see Appendix B. 

If $\scc_y$ appears in $\scc_s\scc_w$ with a nonzero coefficient, we set $y\leftarrow_Lw$, following \cite[8.1]{Lus03}.  Define a {\it preorder} (that is, a relation which is reflexive and transitive) on $W$ by setting,
for $w,w'\in W$, $w\leq_L w'$ if there exists $w=w_0,w_1,\ldots, w_n=w'$ in $W$ such that $w_{i-1}\leftarrow_R w_{i}$, for all $i=1,2,\ldots,n$, and define $w\leq_R w'$ if $w^{-1}\leq_L(w')^{-1}$.

We now consider the preorders $\leq_L$ and $\leq_R$ defined on
$W$, and the preorder $\leq_{LR}$ generated by $\leq_L$ and $\leq_R$ and their associated {\it quasi-poset} $(W,\leq_L)$ etc. See, e.g., \cite[Ch. 8]{Lus03}. Let $\sim_L, \sim_R$ and
$\sim_{LR}$ be the equivalence relations associated with $\leq_L, \leq_R$, and
$\leq_{LR}$, respectively. The equivalence classes for $\sim_L, \sim_R$, and
$\sim_{LR}$ are called the {\it left, right}, and {\it two-sided cells}, respectively. \index{cell} \index{cell! left $\sim$}
\index{cell! right $\sim$}\index{cell! two-sided $\sim$} 

For a quasi-poset, there is an associated poset (see Section B2 in Appendix B).
The poset associated with the quasi-poset $(W,\leq_L)$ is the set $\Omega$ of all left cells\index{$\Omega$, set of left cells}, while the poset associated with the quasi-poset $(\Omega,\leq_{LR})$ is the set $\overline{\Omega}$ of all two-sided cells.\index{$\overline{\Omega}$, set of two-sided cells}  Note that if a left cell $\omega\in\Omega$ contains an element $w_{0,J}$, then $J$ is a uniquely determined
subset of $S$. See \cite[Lem. 8.6]{Lus03} and the argument above Lemma \ref{w_0,J} in this monograph.

 We record that 
\begin{equation}\label{hxyz}
h_{x,y,z}\neq0\implies  z\leq_Ly, z\leq_R x.
\end{equation}
See \cite[13.1(d)]{Lus03}.  In particular, for any $w\in W$, 
\begin{equation}\label{H<}
\sH_{\leq_L w}:=\sum_{y\in W\atop y\leq_L w}\sZ\scc_y,\;\;\sH_{\leq_R w}:=\sum_{y\in W\atop y\leq_R w}\sZ\scc_y,\;\;\sH_{\leq_{LR} w}:=\sum_{y\in W\atop y\leq_{LR} w}\sZ\scc_y
\end{equation}
is a left, right, or two-sided ideal of $\sH$, respectively (see \cite[Lem.~8.2]{Lus03}.). We may also define similar ideals by using $<_L,<_R,<_{LR}$.

Each left cell $\omega\in\Omega$ defines a {\it left-cell module}\index{cell module! left $\sim$, $S(\omega)$} 
 $S(\omega):=\sH_{\leq_L\omega}/\sH_{<_L\omega}$---it is explicitly defined  in \cite[\S8.3]{Lus03} (but with a different notation). Here $\sH_{\leq_L\omega}=\sH_{\leq_L y}$ and $\sH_{<_L\omega}=\sH_{<_L y}$, for any $y\in\omega$. Similarly, we may define right- or two-sided-cell modules. Also, there are base-change versions associated with the Hecke alegbra $\scrK\otimes_\sZ\sH$ and the basis $\{1\otimes\scc_w\}_{w\in W}$ over any commutative ring $\scrK$ which is a $\sZ$-module.
 %and  is a {\it left} $\sH$-module.  
 
 In addition, there is a
corresponding
\begin{equation}\label{dualleft} \text{\rm{\it dual left-cell module \index{cell module! dual left $\sim$, $S_\omega$} $S_\omega:=S(\omega)^*=\Hom_\sZ(S(\omega),\sZ)$.}}
\end{equation}
  It is a {\it right}
$\sH$-module. Both $S(\omega)$ and $S_\omega$ are free $\sZ$-modules; see \cite[\S8.3]{Lus03}.
%\begin{color}{black}{Jie---the following remark is not quite correct. The Specht module is for the "specialized"
%Hecke algebra, where $t^2$ is specialized to a power of a prime. This needs to be explained. Also, we have %generators
%and relations for the generic Hecke algebra and similar ones for the standard specialized Hecke algebra. Finally, the
%specialized Hecke algebra is a $G(q)$-endomorphism algebra......}\end{color}

\begin{rem}\label{Specht} If $\sH$ is the generic Hecke algebra for a symmetric group $W=\mathfrak S_r$,
then it can be proved that the dual left-cell module $S_\omega$ is a ``classical'' Specht module for $\sH$, as defined more generally by Dipper-James; see \cite[p 34]{DJ86}, taking $R$  in the latter reference to be the ring $\mathbb Z$ of integers.  On the other hand, the right cell $\gamma:=\omega^{-1}$, defines a right-cell module
$K_\gamma$, which is related to $S_\omega$ by the isomorphism $\mathfrak D_\sH S_\omega\cong K_\gamma$,
$\gamma=\omega^{-1}$. (Here, $\mathfrak D_\sH$ is the contravariant duality functor on mod-$\sH$ defined in 
\cite[(1.2)]{DPS98c}.) We will not make use of any of the theory of Specht modules in this monograph.\end{rem}

The preorders $\leq_L,\leq_{LR}$ on $W$ induce partial orders $\leq_L,\leq_R$ on $\Omega$. Thus, for $\omega,\omega'\in\Omega$ and $X\in\{L,LR\}$, $\omega\leq_X \omega' \iff x\leq_X x'$, for some $x\in\omega,x'\in\omega'.$
Also, $\omega\in\Omega$ defines a unique two-sided cell $\bar\omega$ with the property that $\omega\subseteq\bar\omega$.  It is extremely important in our context here (see the second paragraph of Section 1.2 or Proposition \ref{prop1}  in Appendix B)  that, given $\omega_1,\omega_2\in\Omega$,
\begin{equation}\label{order}\omega_1<_L\omega_2\implies\overline\omega_1<_{LR}\overline\omega_2.\end{equation}
Thus, $\leq_L$ strictly dominates $\leq_{LR}$ in the sense of Definition \ref{dominates} in Appendix B.
 As discussed at the beginning of the section, this property holds when $(W,S,L)$ is standard finite, our standing assumption. %{\color{blue}In fact, it follows immediately from the conjectures \cite[P4,P9]{Lus03} which are proved in \cite[Chs 15, 16]{Lus03} for the split and quasi-split cases,}

For each two-sided cell $\mathbf c$ of $W$, there is an $\sH$-bimodule $M(\mathbf c)$.
The two-sided cell decomposition of $W$ induces a decomposition of $\sH_{\mathbb Q(t)}$ into two-sided ideals $M(\mathbf c)_{\mathbb Q(t)}$ and, hence, a decomposition of $\Lambda=\text{Irr}(\sH_{\mathbb Q(t)})$. 
%For $\lambda\in\Lambda$, let $E_\la$ denotes a representative from the class $\la$. 
Further, the partial order $\leq_{LR}$ on the set $\bar\Omega$ of two-sided cells induces a preorder on $\Lambda$ by setting, for $E,E'\in\Lambda$,
 \begin{equation}\label{Irr}
 E\leq_{LR}E'\text{ if }\mathbf c_E\leq_{LR}\mathbf c_{E'}.
 \end{equation}
  Here, $\mathbf c_E$ denotes the two-sided cell $\mathbf c$ such that $E|M(\mathbf c)_{\mathbb Q(t)}$
  in the sense that $E$ is a direct summand of the left $\sH_{\mathbb Q(t)}$-module defined by $M(\mathbf c)_{\mathbb Q(t)}$.

{We end this section with a brief discussion about the origin of the generic Hecke algebra $\sH$ associated with a standard finite Coxeter system $(W,S,L)$. Here $W=\mathbf W^F$ is the Weyl group of a finite group $G=\bG^F$ of Lie type that has
a Borel subgroup $B=\bB^F$; see Section 1.1.
The following result, which will not be used until Section 5.2, was first discovered by Iwahori \cite{Iwa64} for (split) Chevalley groups. 

\begin{lem}\label{H^o}  Let $\sH^o$ be the subalgebra over $\mathbb Z[t^2]$ generated by all $T_s\in \sH$, $s\in S$.\footnote{Thus, $\sH^o$ is a form for $\sH$
over $\mathbb Z[t^2]$.} Suppose $\sO$ is a commutative ring and $q\in\sO$,  and view $\sO$ as a $\mathbb Z[t^2]$-algebra by the homomorphism
$\mathbb Z[t^2]\to\sO$ sending $t^2\mapsto q$. 
Then, $$\sH^o_{\sO}\cong \End_{\sO G}( \sO G/B),$$ where $\sO G/B$ is the permutation module for the action of $G$ on the set of left cosets of $B$. 
\end{lem}

 %to define a generic Hecke algebras $\sH$. On the other hand, $(W,S,L)$ arises from a $BN$-pair of $G=\bG^F$.
 % We end this section with the fact that a certain specialisations of $\sH$ is isomorphic to the original Iwahori's Hecke algebra arising from the induced trivial module for $B$ to $G$.
 
 \begin{proof}
A proof of display (\ref{weight}) is outlined in \cite[p.74]{C85}, beginning with a general root group cardinality formula
 $$|X_\Psi^F|=q^{|\Psi|},$$
 where $\Psi$ is a certain subset of the positive roots $\Phi^+$ of $\bG$.  For  each fundamental reflection $s$ in the Weyl group $W$ of $G$, one of the root groups $X^F_\Psi$ can be seen, using \cite[Prop. 2.5.11]{C85},
 to satisfy $|B|=|X^F_\Psi(B\cap B^s)|=|X_\Psi^F|\cdot |B\cap B^s|
 $.  The set $\Psi$ here is the set of positive roots in $\Phi$ made negative by the action of $s$. Thus, $|\Psi|=L(s)$ for this $\Psi$, where $L(s)$ is the length of $s$ in the Weyl group $\mathbf W$ of $\bG$; see Section 1.1. Thus, we obtain the expression (\ref{weight}) above from the expression
 $$q^{L(s)}=q^{|\Psi|}=|B/B\cap B^s|=[B:B\cap B^s].$$
 (Carter actually gives in \cite[p. 74]{C85} an analogous formula for $q^{L(w)}$, for any $w$ in the Weyl group 
 of $\bG$.) 
 
 As in \cite[p. 188]{CPS73}, the elements
 \begin{equation}\label{isohecke}
 \frac{1}{|B|}{\underline{BwB}},\quad w\in W\end{equation}
 span a $\mathbb Z$-algebra $H$, well known to be the endomorphism algebra of the permutation module $\mathbb ZG/B$ for the action of $G$ on coset space $G/B$.  This algebra base changes well, so that we can replace $\mathbb Z$ by $\sO$, obtaining a basis for $\End_{\sO G}(\sO G/B)$. See, e.g., \cite[Exercise 9.3]{DDPW08}. In \cite[fn. 31]{CPS73}, a multiplication relation is showed to be analogous to the right-hand version of the relation $T_sT_w=T_{sw}$ in (\ref{relations}), when $\ell(sw)=1+\ell(w)$. Here, we think of $\frac{1}{|B|}{\underline{BwB}}$ as analogous to
 $T_w$ in (\ref{relations}). There is also the ``quadratic relation''
 $$\left(\frac{1}{|B| }{\underline{BsB}}\right)^2=a \frac{1}{|B|}{\underline{BsB}}+b1$$
 with $b=[B:B\cap B^s]$ and $a=b-1$. Such a  relation holds in any finite doubly transitive group, such as $\langle B,s\rangle$
 acting on the coset space $\langle B,s\rangle/B$. 
 
Hence $b=q^{L(s)}$ (and $a=b-1$).  The two types of relations are analogous to (\ref{relations}) with $w=s$, replacing the $t_s^2$ there with $q^{L(s)}$. The two types of relations we have considered clearly form special cases of (\ref{relations}), sufficient to provide generators and relations for $\sH^o_\sO$. Similarly, we have obtained from $\End_G(\sO G/B) $ analogs with enough information to do all 
multiplications of the basis elements $\frac{1}{|B|}{\underline{ BwB}}$.
 \end{proof} 
  }

\section{Exact categories and stratified algebras}
%\medskip\noindent\underline{\bf Exact categories.}
Beginning in this section, considerable use is made of exact categories $(\scrA,\scrE)$ \index{exact category} as formulated in Keller \cite[Appendix A]{K90} and revisited in B\"uhler's exposition \cite{Bu10}. The axioms given in these papers are equivalent to (but shorter than)  the
original axioms of Quillen \cite{Q73}. A version\footnote{There are several different versions of exact categories in the literature. In particular, following Quillen, we do not assume any kind of idempotent splitting.} is stated in \cite[Appendix]{DPS17} which follows \cite[Appendix A]{K90} quite closely (as quoted in \cite[Appendix A, Set 2]{DRSS99} written by Keller). Briefly, $\scrA$ is an additive category, while $\scrE$ consists of certain pairs $(i,d)$ of composable morphisms $i:X \to Y$
and $d:Y\to Z$ in $\scrA$, such that $i$ is the kernel of $d$ and $d$ is the cokernel of $i$. The pairs
$(i,d)$ which belong to $\scrE$ are required to satisfy certain axioms; see explicitly
\cite[Appendix A]{K90}, \cite[Appendix A, Set 2]{DRSS99},  \cite[Def. 2.1]{Bu10}, or  \cite[Appendix A]{DPS17}.  As in \cite{K90}, we sometimes refer to $i$ as an
{\it inflation}, $d$ as a {\it deflation}, and the pair $(i,d)$ as a {\it conflation}. Conflations $(i,d)$ are usually denoted
by $(X\to Y\to Z)$.\footnote{Such a sequence is denoted by $X\rightarrowtail Y\twoheadrightarrow Z$ in \cite{Bu10}. There are also other notations in the literature used with the same meaning.} As in \cite{Bu10}, we sometimes call conflations ``short exact sequences" in $\scrE$. When
it is convenient and clear from context,  $(X\to Y\to Z)$  is denoted $0\to X\to Y\to Z\to 0$ and called a short
exact sequence (in $\scrE$).

%\begin{center}\begin{tikzpicture} \node at (0,0) (a) {X};\node at (1.5,0) (b) {Y};\node at (3,0) (c) {Z};
%\node at (0,-1.5) (d) {X'};\path [draw,>->] (a) -- (b);\path [draw,->>] (b) -- (c);
%\path [draw,->] (d) -- (a);\end{tikzpicture}\end{center}

Notice that $\scrE$ is itself an additive category, using the usual mappings between 
the categorical diagrams formed by composable pairs of morphisms in $\scrA$; see \cite[Cor. 2.10]{Bu10}. As such, $\scrE$ is required to be closed
under isomorphisms as part of the exactness axioms.

One can consider a pair $(\scrA,\scrE)$ consisting of an additive category $\scrA$ together with a 
collection $\scrE$ of composable morphisms $(i,d)$ in $\scrA$ and ask if it defines an exact category. Of course, one needs
any $(i,d)$ in $\scrE$ to be a kernel-cokernel pair (in the terminology of \cite{Bu10}). Expanding on \cite{DPS17}, 
Chapters 2 and 3 present some important examples of exact categories $(\scrA,\scrE)$ inspired  by Hecke algebra
representation theory, especially involving the construction and control of endomorphism rings, as in Theorem 3.3.1 below. 
  Often $\scrA$ is a full subcategory of an additive category $\scrB$, associated to
an already introduced exact category $(\scrB,\scrF)$.  The theme of constructing new exact categories from old ones is continued in Sections 2.2 and 2.3 using duality.

Finally, for an exact category $(\scrA,\scrE)$, we can consider the abelian group $\Ext^1_{\scrE}(X,Y)$ for any
pair of objects $X,Y\in \scrA$.  In case $\scrA$ is an additive $\scrK$-category for a commutative ring $\scrK$, $\Ext^1_\scrE(X,Y)$
carries a natural structure as a $\scrK$-module; see \cite[Prop. A.2]{DPS17}. %Moreover, the proof of the latter
%result shows that $\Ext^1_{\scrE}(X,Y)$ is a $\scrK$-submodule of $\Ext^1_{\mathscr C}(X,Y)$ whenever
%$\scrA$ is a full subcategory of a $\scrK$-category $\mathscr C$. 
 The action of $a\in\scrK$ on $\Ext^1_{\scrE}(X,Y)$ is given by either pull-back of multiplication by $a$ on $X$ or (equally, as it turns out) by push-out of multiplication of $a$ on $Y$. 
 
 \medskip
 %\noindent
 {\bf 1.3A. Exact subcategories: some terminology and a new strategy.} 
 A note on exact category terminology: Let
$(\scrA,\scrE)$ and $(\scrB,\scrF)$ be exact categories. Suppose $\scrA$ is a subcategory of $\scrB$, and $\scrE\subseteq
\scrF$, viewing $\scrE$ and $\scrF$ as classes of composable pairs of morphisms in  $\scrA$ and $\scrB$, respectively.
We say that $(\scrA,\scrE)$ is an {\it exact subcategory}\footnote{\label{exact sub}Some authors note that the set $\scrE$ has a natural
category structure, using “morphisms” constructed
from one short exact sequence to another via
commutative diagrams. This does not, however,
affect our definition of an exact subcategory.} of $(\scrB,\scrF)$. It is a {\it full exact subcategory} if $\scrA$
is also a full subcategory of $\scrB$---that is, $\Hom_{\scrA}(A_1,A_2)=\Hom_{\scrB}(A_1,A_2)$ whenever $A_1,A_2$ are objects in $\scrA$. 

With this terminology in hand, we can now describe our new approach and some of its motivation.
 \begin{str} \label{rem2.2}
{\rm If $(\scrA,\scrE)$ is a full exact subcategory of $(\scrB,\scrF)$, then the fullness implies that the inclusion $\scrE\subseteq
\scrF$ induces an inclusion, i.e., an injective map, $\Ext^1_{\scrE}(A,A')\hookrightarrow\Ext^1_{\scrF}(A,A')$ for each
pair $(A,A')$ of objects in $\scrA$.\footnote{The members of exact category $\text{Ext}^1$ groups  are defined as equivalence classes of short exact sequences under a well-known equivalence relation. See \cite[p.249]{DPS17}.  In the presence of
fullness, each equivalence class in $\scrF$ intersects at most one of the $\scrE$ equivalence classes, resulting in the inclusions of this sentence.

     We take this opportunity to note that our notion of ``full exact subcategory" is quite different from the notion ``fully exact subcategory"
in \cite[Def.~10.21]{Bu10}, apparently aimed in a different direction.}  The inclusion is even $\scrK$-linear if $\scrK$ is a commutative Noetherian ring and the inclusion $\scrA\subseteq\scrB$ is also an inclusion of $\scrK$-categories (with $\scrA$ full in $\scrB$ as before.) See \cite[Prop. A.2]{DPS17}.

As in traditional homological algebra, the vanishing of a group $\Ext^1_{\scrF}(A_1, A_2)$ leads to the exactness
of $\Hom_{\scrF}(A_1,-)$ when applied to any given short exact sequence $\gamma$ in $\scrF$ that begins with $A_2$.  Assume this exactness is desirable for some set $\Gamma$ of such $\gamma$'s, but the proposed $\Ext^1$ vanishing is too much to ask for.  In that case, we can try to build  a full exact subcategory $(\scrA, \scrE)$ as above, but with  $\Gamma\subseteq\scrE$ and $\Ext^1_{\sE}(A_1,A_2)=0$.  This gives the desired exactness.}
 \end{str}

 Though the discussion above is only an outline, we are able to use it when constructing stratifying systems, as we now show.

\medskip%\noindent
{\bf 1.3B. Stratifying Systems.} \index{stratifying system} Let $\scrK$ be a Noetherian commutative ring.  Let $B$ be an algebra which
is finite (= finitely generated) and projective as a $\scrK$-module. \index{finite module}\index{finite module! $=$\;finitely generated module}
Let $\Lambda$ be a quasi-poset,\index{quasi-poset} i.e., a finite set $\Lambda$ with a preorder, denoted 
$\leq$ here. The associated poset is $\bar\Lambda,$\index{quasi-poset! poset associated with $\sim$} and $\bar\lambda$ is the element  of $\bar\Lambda$ associated to any given $\lambda\in\Lambda$.

 For $\lambda\in\Lambda$, assume that there is given both a $B$-module $\Delta(\lambda)$, which is finite and projective over $\scrK$, and a finite and projective $B$-module $P(\lambda)$, together with an epimorphism $P(\lambda)\twoheadrightarrow \Delta(\lambda)$.  The following
conditions are assumed to hold:

\begin{itemize}
\item[(SS1)] For $\lambda,\mu\in\Lambda$, 
$$\Hom_{ B}(P(\lambda), \Delta(\mu))\not=0\implies \lambda\leq\mu.$$

\item[(SS2)] Every irreducible $B$-module is a homomorphic image of some $\Delta(\lambda)$.

\item[(SS3)] For $\lambda\in\Lambda$, the $B$-module $P(\lambda)$ has a finite filtration by $ B$-submodules
with top section $\Delta(\lambda)$ and other sections of the form $\Delta(\mu)$ with $\bar\mu>\bar\lambda$.
\end{itemize}
When these conditions hold, the data 
\begin{equation}\label{strat} \{\Delta(\lambda),P(\lambda)\}_{\lambda\in\Lambda}\end{equation}
form (by definition) a  {\it  stratifying system}\footnote{Following \cite{DPS17}, we drop the word ``strict" in ``strict stratifying system" as used in
\cite{DPS98a}. I.e., all stratifying systems in this monograph are strict in the sense of \cite{DPS98a}, thus, we insist on
the strict comparison $\bar\mu>\bar\lambda$ in 
(SS3). }
for the category $ \Bmod$  of finite $B$-modules. 
Note that if (\ref{strat}) is a stratifying system for $\Bmod$, then $\{\Delta(\lambda)_{\mathscr K'}, P(\lambda)_{\mathscr K'}\}_{\lambda\in\Lambda}$ is a
 stratifying system for $ B_{\mathscr K'}$-mod {\it for any base change $\mathscr K\to \mathscr K'$}, provided $\mathscr K'$ is a Noetherian commutative 
ring. (See \cite[p. 231]{DPS15}. For more  details and further results, see \cite[Lem. 1.2.5]{DPS98a}.)

\begin{rem}\label{1.3.2}
 If $B$ has a stratifying system as explained above, it has a ``standard stratification" by idempotent ideals
$0\subseteq J_0\subseteq J_1\subseteq \cdots\subseteq J_n=B$ much like a defining sequence in the
theory of quasi-hereditary algebras. Thus, $B$ is a kind of generalization of a quasi-hereditary algebra, or, equivalently,
if $B$-mod has a stratifying system, then $B$-mod is a generalization of a highest weight category. We call $B$ a {\it standardly stratified algebra}; see Section 3.5 for more details. \index{standardly stratified algebra}
 We do not pursue this further here, but refer  to 
\cite[\S1.2]{DPS98a} and \cite{DPS18} for a more complete discussion.
\end{rem}
 
\medskip%\noindent
{\bf 1.3C. Exact categories and the stratification hypothesis.} Stratifying
systems can be constructed in an endomorphism-ring setting.  Originally, the construction was based on assuming a ``stratification hypothesis"  in \cite[Hyp. 1.2.9, Th. 1.2.10]{DPS98a}, which
required a difficult $\Ext^1$-vanishing condition (see \cite[Th. 2.3.9, Th. 2.4.4]{DPS98a}). This problem was
partly circumvented in \cite[Hyp. 2.5 \& Th. 2.6]{DPS17} by introducing ``designer" exact categories which had, a priori, smaller $\Ext^1$-groups, and thus were more likely to vanish. However, the new version was somewhat awkward to use in the Quillen axiom system. See \cite[Rem. 2.7]{DPS17}. The updated version below, given in the Stratification Hypothesis \ref{hypothesis},  removes this issue,
by replacing all $\Ext^1$-vanishings
with existence assertions for needed short exact sequences. As noted in Remark \ref{2.4}, these assertions are at least implied by a natural exact category $\Ext^1$-vanishing condition. This is the way they are verified  in this monograph.

Let $R$ be a finite and projective $\scrK$-algebra. Our aim will be to provide ingredients in mod-$R$, including a
right $R$-module $T$ with suitable filtrations, which enable the construction of a stratifying system for $\Bmod$,
where $B:=\End_R(T)$. Clearly, $T$ is naturally a left $B$-module. 

{\bf Construction.} We will construct $T$ as an object in
a full subcategory $\scrA$ of mod-$R$, with the following assumptions. First,
the subcategory $\scrA$ is part of an exact category $(\scrA,\scrE)$. Second, the exact sequences $(X\to Y\to Z)\in\mathscr E$ are among the short exact sequences $0\to X\to Y\to Z\to 0$ in mod-$R$.  Thus,
$(\scrA,\scrE)$ is a full exact subcategory of mod-$R$. (In our terminology, ``full exact" is just an abbreviation for
``full and exact.")
Third, we have a collection of objects $S_\lambda, T_\lambda\in {\mathscr A}$ indexed by the elements $\lambda$ of a finite quasi-poset
$\Lambda$, where $T_\lambda$ is a finite and projective $\scrK$-module.\footnote{This assumption is not often used
directly, but can be a step toward the required projectivity over $\scrK$ of the modules $\Delta(\lambda)$; see 
Remark \ref{2.5}(b). The assumption for $T$
often comes for free, since the $S_\nu$ are typically projective (or even free) $\scrK$-modules and each  $T_\lambda$ is constructed  to have a filtration with each section a direct sum of various $S_\nu$.} For each $\lambda\in\Lambda$, $S_\lambda$ is a subobject of $T_\lambda$, with
the inclusion $S_\lambda\hookrightarrow T_\lambda$ an inflation in $\scrE$.
For each $\lambda\in\Lambda$, fix a positive integer $m_\lambda$. Define
\begin{equation}\label{modules}
 T: =\bigoplus_{\lambda\in\Lambda}T_\lambda^{\oplus m_\lambda}.
\end{equation} 
% Note that $T$ is an object in $\scrA$, so that $B=\End_R(T)=\End_\scrA(T)$.   
 With this notation, we can now state the following:

\begin{hyp}\label{hypothesis} \index{stratification hypothesis}
{\rm Assume the set-up introduced in the previous paragraph. In particular, $(\scrA,\scrE)$ is an exact category, a full exact subcategory of mod-$R$, and $T\in\scrA$ with $B=\End_R(T)=\End_\scrA(T)$. There are also indexed objects $
S_\lambda, T_\lambda\, (\lambda\in \Lambda)$, and  the following statements are required as part of the hypothesis: 
\begin{itemize}

\item[(1)] For $\lambda\in\Lambda$, there is an increasing filtration of $T_\lambda$:  
$$0= T^{-1}_\lambda\subseteq T^0_\lambda\subseteq \cdots \subseteq T^{\ell(\lambda)}_\lambda=T_\lambda$$
in which each inclusion $T^{i-1}_\lambda\subseteq T_\lambda^i$ defines a conflation 
\begin{equation}\label{inclinf}(T^{i-1}_\lambda\longrightarrow T_\lambda^i \longrightarrow
T_\lambda^i/T_\lambda^{i-1})\in\scrE.\end{equation} 
In addition, $T^0_\lambda\cong
S_\lambda$, and, for $i>0$, 
the ``section"\footnote{Observe that $T^i_\lambda\to T^i_\lambda/T^{i-1}_\lambda$ is the unique (up to isomorphism) cokernel in $\scrA$ of the inflation
$T^{i-1}_\lambda\to T^i_\lambda$ and, thus, is a deflation. We generally leave such exact category exercises to the reader.  }  $T^i_\lambda/T^{i-1}_\lambda$ is a direct sum of various $S_\mu$, $\mu\in\Lambda$ and $\bar\mu>\bar\lambda$ (repetitions allowed).

\item[(2)] For $\lambda,\mu\in\Lambda$, $\Hom_{\mathscr A}(S_\mu, T_\lambda)\not=0\implies\lambda\leq \mu.$

\item[(3)] For all $\lambda\in\Lambda$ and integers $i\geq 0$, applying the contravariant ``diamond'' functor\footnote{This functor will also be used later in the book.} $(-)^\diamond:=\Hom_R(-,T)$ to \eqref{inclinf} gives the exact sequence
$$0\to (T_\lambda^i/T_\lambda^{i-1})^\diamond\longrightarrow (T^{i}_\lambda)^\diamond\longrightarrow 
(T^{i-1}_\lambda)^\diamond\to 0$$
of $B$-modules.
%(If $M\in\scrA$, $M^\diamond:=\Hom_\scrA(M,T)=\Hom_R(M,T)$, which has a natural $B$-module structure. Using it, $(-)^\diamond$
%is a left exact contravariant functor from $\scrA$ to $B$-mod.)
 
\end{itemize} }
\end{hyp}

 \begin{rem} \label{2.4} Notice that condition (3) {\it is implied by the vanishing of} $\Ext^1_\scrE(S_\lambda,T)$, for all $\lambda\in
 \Lambda$; see \cite[Prop. A.2]{DPS17}.  This is the way that we will use the (updated) Stratification Hypothesis \ref{hypothesis} in this monograph. 
 We mention, without proof, that it is implied by the less flexible earlier \cite[Hyp. 2.5]{DPS17}, which was modeled 
closely on \cite[Hyp. 1.2.9]{DPS98a} and which is harder to check, as written, in the Quillen exact category setting. 
 \end{rem}

%When the algebra $A$  arises as an endomorphism algebra $A:=\End_R(T)$, there is a useful theory for obtaining 
%a  stratifying system for $\Amod$. In fact, this is how such stratifying systems initially arose (see
%\cite{CPS96} and \cite{DPS98a}). This approach is followed in the proof of the main theorem in this monograph. For
%convenience, we summarize the sufficient conditions that will be used, all taken from \cite[Th. 1.2.10]{DPS98a}.
The following theorem updates \cite[Th. 1.3.3]{DPS98a} and \cite[Th. 2.6]{DPS17}.

\begin{thm}\label{thm2.5}{   Assume that the Stratification Hypothesis \ref{hypothesis} holds.   
For $\lambda\in\Lambda$, put
\begin{equation}\label{SH}  \begin{cases}
\Delta(\lambda):=\Hom_\scrA(S_\lambda,T)=\Hom_R(S_\lambda,T)\in \Bmod,\\
P(\lambda):=\Hom_\scrA(T_\lambda,T)=\Hom_R(T_\lambda,T)\in \Bmod.\end{cases}
\end{equation}
If each $\Delta(\lambda)$ is projective over $\scrK$, then $\{\Delta(\lambda), P(\lambda) \}_{\lambda\in\Lambda}$ is a
 stratifying system for $\Bmod$. }
\end{thm} 

\begin{proof} First, we check (SS1).  Observe that the contravariant functor $(-)^\diamond$ provides maps
$$\Hom_R(S_\mu, T_\lambda)\longrightarrow \Hom_B(P(\lambda),\Delta(\mu)).$$
Passing to $T=\bigoplus_{\la\in\Lambda} T_\lambda^{\oplus m_\lambda}$ on the left and $_BB=
\bigoplus_{\lambda\in\Lambda} P(\lambda)^{\oplus m_\lambda}$ on the right, we obtain an isomorphism
$$\Hom_R(S_\mu,T)\overset\sim\longrightarrow \Hom_B({_BB},\Delta(\mu))\cong \Delta(\mu).$$
Consequently, the original maps must also be isomorphisms.  Thus, (SS1) follows from Condition (2) (of Hypothesis \ref{hypothesis}).

Next, we will prove (SS2) assuming  (SS3) has been proved. By the construction of $B=\End_R(T)$, we have that $_BB\cong \bigoplus_{\lambda \in\Lambda}P(\lambda)^{\oplus m_\lambda}$. Using (SS3), $_BB$ has a filtration with sections various 
$\Delta(\nu)$, $\nu\in\Lambda$. Any irreducible $B$-module is, of course, a homomorphic image of $_BB$, and so
(using irreducibility) a homomorphic image of some $\Delta(\nu)$. This proves (SS2) in the presence of (SS3).

Now consider (SS3) itself for a given $\lambda$.  %For convenience, abbreviate $T_\lambda, T_\lambda^i$, etc. by $T, T^i$, etc. 
Let $n$ be an integer such that $T^n_\la=T_\la$. Let $j\geq i\geq -1$ be integers. Thus, $T^i_\la$ is a submodule of $ T^j_\la$.

\medskip\noindent
{\bf Claim.} The natural map $(T^j_\la)^\diamond\to (T^i_\la)^\diamond$ is a surjection.

\smallskip
To see this, observe that the inclusion $T^i_\la\hookrightarrow
T^j_\la$ factors into a composition of inclusions $T^i_\la\hookrightarrow T^{i+1}_\la\hookrightarrow T^{i+2}_\la\hookrightarrow
\cdots\hookrightarrow T^j_\la$. By Condition (3), for $i\leq a<j$, the inclusion $T^a_\la\hookrightarrow T^{a+1}_\la$ becomes 
a surjection
$(T^{a+1}_\la)^\diamond\twoheadrightarrow (T^a_\la)^\diamond$ upon applying $(-)^\diamond$. The Claim follows by taking compositions.

\medskip
We return to the proof of (SS3) for a given $\lambda$.
Let $n\geq j\geq 0$.  The maps 
$ (T^{n}_\la)^\diamond\to (T^{j}_\la)^\diamond $ and $(T^n_\la)^\diamond\to (T^{j-1}_\la)^\diamond$ are both surjective by the Claim.
This gives the two exact sequences forming the rows of the natural commutative diagram
\begin{equation}\label{diagg}
\begin{CD}
0@>>> (T^n_\la/T^j_\la)^\diamond @>>> (T^n_\la)^\diamond @>>> (T^j_\la)^\diamond  @>>> 0\\
@.  @V\alpha VV @| @V\beta VV @. \\
0 @>>> (T^n_\la/T^{j-1}_\la)^\diamond @>>> (T^n_\la)^\diamond @>>> (T^{j-1}_\la)^\diamond @>>> 0\\
 \end{CD}
 \end{equation}
 Take a sufficiently large integer $n$, so that $T^n_\la=T_\lambda$. Then, $P(\lambda)=T^\diamond_\la$ has a filtration by submodules of the form $(T_\la/T^j_\la)^\diamond$, and the map $\alpha$ above is an inclusion between adjacent filtration terms. The Snake Lemma shows that the section $\ck(\alpha)$ of these two filtration terms is isomorphic to
 ${\text{\rm Ker}}(\beta) = (T^j_\la/T^{j-1}_\la)^\diamond$. Thus, $P(\lambda)$ has a filtration with sections isomorphic to  various 
 $(T^j_\la/T^{j-1}_\la)^\diamond$. The module $(T^j_\la/T^{j-1}_\la)^\diamond$ is of the form %a direct sum of terms
  $S_\nu^\diamond=
 \Hom_R(S_\nu,T)=\Delta(\nu)$, where $\nu>\lambda$ if $j>0$, and $\nu=\lambda$ if $j=0$. In the $j=0$ case, $\ck(\alpha)$
 is the ``top" section of the given filtration on $P(\lambda)$, and it is isomorphic to ${\text{\rm Ker}}(\beta)=
 (T^0_\la)^\diamond=S_\lambda^\diamond=\Delta(\la)$.  This completes the proof of (SS3).
 
 Finally, because it is assumed that each $\Delta(\nu)$, $\nu\in\Lambda$, is a projective $\scrK$-module, it follows that each $P(\lambda)$, $\lambda\in\Lambda$, is also a projective $\scrK$-module.  Since $B$ is a direct sum of
 the $P(\lambda)$ (repetitions allowed),  $B$ is a projective $\scrK$-module.  This completes the proof.
  \end{proof}

\begin{rems}\label{2.5}

(a) Correspondingly, $\scrA$ can be one of  the categories $\scrA(\scrL^*)$ in Section 2.3  below (or a full
subcategory,
 such as $\scrA_\flat(\scrL^*)$ there.) A similar remark holds for the corresponding exact
category in Section 2.2; however, the notation above is chosen to be compatible with the applications in Section 2.4 as developed in
Section 2.3.

(b) If $\scrK$ is a regular ring of Krull dimension at most 2, then, in the statement of Theorem \ref{thm2.5}, 
the sentence ``Assume that each $\Delta(\lambda)$ is projective over $\scrK$" can be omitted, since $\Delta(\lambda)$ is automatically projective over $\sZ$. This is shown in the proof of  \cite[Cor. 1.2.12)]{DPS98a},
using earlier work of Auslander-Goldman together with the projectivity of $T$ as a $\scrK$-module (assumed in the preamble to the Stratification Hypothesis \ref{hypothesis}).   For a detailed account, see \cite[\S C.3]{DDPW08}. A similar result holds
if the stratifying system is ``defined over $Z$" where $Z$ is a regular ring of Krull dimension $\leq 2$. This means
that the stratifying system over $\scrK$ is obtained by base change from the regular ring $Z$.

(c) As an example in (b), we will be mainly interested in the special case in which $R=\sH$ is the Hecke algebra defined over $\sZ=\mathbb Z[t,t^{-1}]$ (and its localizations $\sZ^\natural$ at multiplicative subsets, not containing 0, of $\mathbb Z$). Since $\sZ$ is a regular ring of dimension 2, (b) above applies to show that each $\Delta(\lambda)$ is projective over $\sZ$.  In fact, by a result of Swan \cite[p.111, first paragraph]{Sw78},  $\Delta(\lambda)$ is then
$\sZ$-free. (Swan's paragraph also implies that any finite, projective $\sZ^\natural$-module is free. In addition, if
$ \sZ^{\natural\oplus m}\cong \sZ^{\natural\oplus n}$ for integers $m,n$, then $m=n$.) \end{rems}

\chapter{Some  exact categories}

Section 1.3 in Chapter 1 provides a general theory of stratifying endomorphism algebras using exact categories,  reformulating that developed in \cite{DPS17}. To apply the theory in the context of the conjecture \cite[Conj.~2.5.2]{DPS98a} by the authors, we need to construct a nested sequence of exact categories involving Hecke algebra modules with cell module filtrations. We eventually prove the Ext$^1$ vanishing condition (see Remark \ref{2.4}) in one of the exact categories and, then, prove the (slightly modified) conjecture in the next chapter.
%Since the modules considered in the conjecture \cite[Conj.~2.5.2]{DPS98a} by the authors are Hecke algebra modules which are filtered by dual left cell modules and satisfy a certain Ext$^1$ vanishing condition, we now construct several exact categories

\section{Three exact categories}
{\it In this section, we work exclusively with {\it left modules} for an algebra.} 
Let $\sH$ be the generic Hecke algebra over $\sZ$, associated to a standard finite Coxeter system $(W,S,L)$ (as defined in Section 1.1).

We begin with the following variation on \cite[Lem. 4.3]{DPS15}. If $V$ is a left $\sH$-module, the ``fixed-point module'' for the action of $\sH$ on $V$ is defined to be
$$V^\sH:=\{v\in V\,|\, \scc_sv=(t_s+t_s^{-1})v,\quad\forall s\in S\}\in \sZ{\text{\rm --mod}},$$
where $t_s,\scc_s$ are defined in \eqref{relations} and \eqref{Calpha}, respectively.
Equivalently,
$$v\in V^\sH\iff T_sv=t_s^2v,\text{ for all }s\in S.$$
The assignment $V\mapsto V^\sH$ is part of an evident left-exact endo-functor.
Similarly, $V^{\sH_J}$ is defined for any parabolic subalgebra $\sH_J$ of $\sH$ ($J\subseteq S$). 
The next result follows from \cite[Lem.4.3]{DPS15} and the above definition
of fixed points, using \eqref{relations} and \eqref{Calpha}.

\begin{lem}\label{onto}
Let ${\mathfrak I\subseteq\mathfrak K}$ be left ideals in the Hecke algebra $\sH$ over $\sZ$.  Assume that 
both $\mathfrak I$ and $\mathfrak K$ are spanned by the Kazhdan-Lusztig basis elements $\scc_x$, $x\in W$, that they contain.
Then, for any $J\subseteq S$, the natural map
\begin{equation}\label{strange} {\mathfrak K}^{\sH_J}\longrightarrow 
({\mathfrak K}/{\mathfrak J})^{\sH_J}\end{equation}
is surjective. \end{lem}

\begin{notation}\label{notation1}{\rm At this point, the reader should review the notion of height functions\index{height function} $\htt$
discussed in Appendix B below. 
In this section, we fix $$\htt=\htt_l:\Omega\to\mathbb Z,$$ a function defined on the set $\Omega$ of
left cells in $W,$ and compatible with the preorder $\leq_{LR}$; see display \eqref{preorder} in Appendix B for the meaning of ``compatible.'' 
For another example, see the order defined in display \eqref{preorder1}. The subscript $l$ is there to remind the reader of our left-module context.
Such a height function
 $\htt$ is also compatible with $\leq_L$, as formalized
in 
Corollary \ref{discussion} (which uses property (\ref{order}) via Proposition \ref{prop1}).  In particular, $\htt$ is constant on left cells which belong to the same two-sided cell.  
 It is also observed in Remark \ref{A7} that $\htt$ may be viewed as assigning heights to left-cell modules $S(\omega)$. \index{cell module! left $\sim$, $S(\omega)$}  Thus, if
$\omega\in\Omega$, $\htt$ assigns the left-cell module $S(\omega)$ height $\htt(\omega)$. 
Also, the value $\htt(\omega)=\htt(\bar\omega)$, where $\bar\omega$ is the two-sided cell containing $\omega$, may be assigned
without ambiguity as a ``height" to the irreducible constituents of the base-changed module $S(\omega)_{\mathbb Q(t)}$. In particular, this means that $\htt$ uniquely determines a height function on the set $\Lambda$ of irreducible
modules for the semisimple algebra $\sH_{\mathbb Q(t)}$. Since this height function also determines the original
$\htt:\Lambda\to\mathbb Z$ (see Remark \ref{A7}), we continue to denote it by $\htt$.  As explained in Example \ref{example1}, $\htt$
induces a canonical decreasing filtration 
\begin{equation}\label{lcf}
M=M_0\supseteq M_1\supseteq M_2\supseteq\cdots\supseteq M_{n-1}\supseteq M_n=0
\end{equation}
on each $\sH$-module $M$. We call this filtration the $\htt$-{\it filtration} or the {\it height filtration,} when $\htt$ is understood on $M$. Similar terminology can be used for right modules in the next three sections, although the height filtrations are increasing and the subscripts here become superscripts there. \index{height filtration}\index{$\htt$-filtration}

Finally, as a convenience, we 
assume in this section (unless noted otherwise) that $\htt=\htt_l:\Omega\to \mathbb Z$ only takes negative values. See Appendix B. Later, we make use of $\htt_r:=-\htt_l:\Omega\to\mathbb Z$, which is, in particular, used to assign positive height values to dual left-cell modules  
$S_\omega= S(\omega)^*=\Hom_\sZ(S(\omega),\sZ),\;\; \omega\in\Omega$ (see \eqref{dualleft}) which are naturally
right $\sH$-modules.  See Remark B.8.}
\end{notation}\smallskip

Relative to such a height function $\htt:\Omega\to\mathbb Z$ that induces $\htt:\Lambda\to\mathbb Z$ as above, there is an exact category $(\scrA,\scrE)$, as constructed in Example \ref{example1}. Thus, $(\scrA,\scrE)$ consists of the full subcategory $\scrA=\scrA_l$
of  $\sH$-mod whose objects are left $\sH$-modules, finite and torsion-free over $\sZ$, and the (full) ``subcategory'' $\scrE=\scrE_l$ (in the sense of footnote \ref{exact sub} in Chapter 1) of short exact sequences
$0\to M'\to M\to M''\to0$ in $\scrA$ satisfying the condition that $0\to M'_h\to M_h\to M''_h\to0$ is exact,  for all heights $h$.
%\footnote{If we view $\scrE$ as a category, then the definition for a full exact subcategory at the beginning of Subsection 1.3A can be interpreted as both $\scrA$ and $\scrE$ are full subcategories of $\scrB$ and $\scrF$, respectively.}
(Compare the right-module version given in \cite[Constr. 3.5]{DPS17} or Theorem \ref{A1} below.)

  Below, building on $(\scrA,\scrE)=(\scrA_l,\scrE_l)$, we present three more examples of exact categories:
$$\sH\text{-mod}\geq(\scrA,\scrE)\geq (\scrA(\scrL),\scrE(\scrL))\geq (\scrA(\scrL),\scrE^\flat(\scrL))\geq (\scrA^\flat(\scrL),\scrE^\flat(\scrL)),$$
which will be important in the sequel.  Here ``$\geq$'' indicates ``a full subcategory''. In particular, $\Hom_{\mathscr C}(M,N)=\Hom_\sH(M,N)$ for objects $M,N$ in $\mathscr C$ and every $\mathscr C=\scrA,\scrA(\scrL)$, etc.
  
\medskip%\noindent
{\bf 2.1A. The exact category $\boldsymbol{(\scrA(\scrL),\scrE(\scrL))}$.}  \index{exact category! $\sim$ $(\scrA(\scrL),\scrE(\scrL))$}
 Its construction involves the category $\sH$-mod
of (finite) left $\sH$-modules, but it closely follows the more abstract discussion in \cite[Assump. 3.2 and Constr. 3.8]{DPS17} which worked with the category mod-$H$ of right modules for a suitable algebra $H$ more general than the generic Hecke
algebra $\sH$.

For any integer $h$, let $\scrL_h$ be the full subcategory
of $\scrA$ consisting of modules $N$ which are a finite direct sum of modules isomorphic to left-cell modules $S(\omega)$ of height $h$.
 Let $\scrL$ be the set-theoretic union of the $\scrL_h$.
Next, let 
$\scrA(\scrL)$ be the full additive subcategory of $\sH$-mod with objects $M$ (necessarily in $\scrA_l$) whose $\htt$-filtration
\eqref{lcf}
%\begin{equation}\label{lcf}
%M=M_0\supseteq M_1\supseteq M_2\supseteq\cdots\supseteq M_{n-1}\supseteq M_n=0,
%\end{equation}
has the property that, for any integer $h$,  $M_{h-1}/M_h$ is an object in 
$\scrL_{-h}$.\footnote{The minus sign is chosen to be compatible with the duality theory developed later in Sections 2.2 and 2.3. In particular, we will consider dual left-cell modules $S_\omega$, for each $\omega\in\Omega$, and assign to each $S_\omega$ a positive height, negative to that of $S(\omega)$. See the beginning paragraphs of Section 2.3 and also (\ref{ht pair}). Note that, with Notation \ref{notation1}, this height filtration coincides with the one defined in Example \ref{example1}.}
We will also informally call such an $\htt$-filtration {\it a left-cell filtration}.\footnote{\label{LcellF}Clearly, every  $\htt$-filtration here may be refined to a filtration with sections isomorphic to left-cell modules. This refined filtration is a left-cell filtration  in the sense of \cite[\S2.2]{DPS98a}.} \index{cell filtration} \index{cell filtration! left $\sim$}
 
 By analogy with
\cite[\S3]{DPS17}, let $\scrE(\scrL)$ consist of those sequences $(A\to B\to C)$ of objects
in $\scrA(\scrL)$ such that, for every $h$, 
$$0\longrightarrow A_{h-1}/A_{h}\longrightarrow B_{h-1}/B_{h}\longrightarrow C_{h-1}/C_{h}\longrightarrow 0$$ is a {\it split}
short exact sequence in the additive category $\scrL_{-h}$ (or, equivalently, in the category $\sH$--mod). Then, 
\begin{equation}\label{display1}{\text{\rm{$(\scrA(\scrL),\scrE(\scrL))$ is an
exact category.}}}\end{equation}
This was proved in the more abstract setting of right $H$-modules in \cite[Th. 3.9]{DPS17} for an analogous
category $(\scrA(\scrS),\scrE(\scrS))$  associated to $H$. The proof in the present situation
is entirely similar. The result can also be deduced from the special case $H=\sH^{\text{op}}$ of the cited theorem.

 Note that left-cell modules are, by definition, finite and free over $\sZ$.  This implies that any object in 
 $\scrA(\scrL)$ is also finite and free over $\sZ$. Looking ahead to Section 2.2, this 
 implies that $\scrA(\scrL)\subseteq\scrA'_l$,  the
 left-hand version of the category  $\scrA'$ introduced in Theorem \ref{B2}. Thus, $\scrA'_l$ denotes the full additive subcategory of objects $M\in\scrA_l$ for which $M_h/M_{h+1}$ is projective as a $\scrK$-module, for each $h\in\mathbb Z$, and $\scrE'_l=\scrE_l|_{\scrA_l'}$.
 Also, $\scrE(\scrL)\subseteq\scrE'_l$, so
 that $(\scrA(\scrL), \scrE(\scrL))$ is a (full) exact subcategory of $(\scrA'_l,\scrE'_l)$; see the definition at the beginning of Subsection 1.3A.  
 
Note also that, for each object $M$ in $\scrA(\scrL)$ with the $\htt$-filtration given in display \eqref{lcf}, each inclusion $M_h\subseteq M_{h-1}$ gives a conflation $(M_h\to M_{h-1}\to M_{h-1}/M_h)\in\scrE(\scrL)$, for all $1\leq h\leq n$, so that its $\sZ$-dual $M^*$ satisfies the condition (1) in the Stratification Hypothesis \ref{hypothesis}. 
 
\medskip%\noindent
{\bf 2.1B. The exact category $\boldsymbol{(\scrA(\scrL),\scrE^\flat(\scrL))}$.} \index{exact category! $\sim$ $(\scrA(\scrL),\scrE^\flat(\scrL))$}
With the notation above, let $\scrE^\flat(\scrL)$ denote the full subcategory of $\scrE(\scrL)$ with objects
$(M\to N\to P)\in\scrE(\scrL)$ such that, for any integer $h$ and any subset $J$ of $S$, 
the sequence
\begin{equation}\label{shortie}0\longrightarrow (M_h)^{\sH_J} \longrightarrow (N_h)^{\sH_J} \longrightarrow (P_h)^{\sH_J}   \longrightarrow 0\end{equation}
of fixed points is exact. Note that, for any $h$, each $0\to M_h/M_{h-1}\to N_h/N_{h-1}\to P_h/P_{h-1}\to 0$ is split in
$\sH$-mod, since $(M\to N\to P)\in\scrE(\scrL)$, but this does not mean that the short exact sequence $0\to M_h\to N_h\to P_h\to 0$ is split. In particular, the exactness of (\ref{shortie}) is a nontrivial property.

We have:

\begin{prop}\label{Prop2} 
$(\scrA(\scrL),\scrE^\flat(\scrL))$
is an exact category.
\end{prop}

\begin{proof} Let $\mathscr C$ be the abelian category of all finite $\sZ$-modules. Put
$$F(-):=\bigoplus_{J,h}((-)_h)^{\sH_J}:\scrA(\scrL)\to \scrC.$$
Then, $F$ is an $\scrE(\scrL)$-left exact functor. Now we apply \cite[Lem. 3.1]{DPS17}, with $\scrE$ there
taken to be $\scrE^\flat(\scrL)$ as defined above.
\end{proof}

%\smallskip\noindent
%{\bf The left $\bold{\sH}$-{\bf modules} ${\sH x_J, \,\,J{\subseteq} S}$.} 
For $J{\subseteq} S$, if $W_J$ is the parabolic subgroup of $W$ defined by $J$, then $\sH x_J$ is the usual (left) $q$-permutation module determined by any $J\subseteq S$, where
\begin{equation}\label{xJ} x_J:=\sum_{w\in W_J}T_w\end{equation}
was defined in display \eqref{longest}. 
For future use in Section 5.2, we also define
\begin{equation}\label{yJ} y_J:=\sum_{w\in W_J}\sgn(w)(t^2)^{-L(w)}T_w,\end{equation}
where $\sgn(w):=(-1)^{\ell(w)}$ is the sign of $w$.  See \cite[pp. 324--326]{DPS98c}, taking $q=t^2$ in our current notation.

\begin{prop}\label{HxJ}The $\sH$-module $\sH x_J$ belongs to $\scrA(\scrL)$.  
\end{prop}
\begin{proof}%This may be proved as follows: 
Let $M$ denote $\sH x_J$. % From \cite[(2.3.3) \& Lem. (2.3.5)]{DPS98a}, 
Using display \eqref{hxyz}, we find
$M$ has a basis of elements $\scc_v$, with $v\in W$, in
a certain quasi-poset ideal $V$ for $\leq_L$, necessarily closed under $\sim_L$. In fact, we have
$V=\{w\in W\mid ws<w,\forall s\in J\}=\{w\in W\mid w\leq_Lw_{0,J}\}$; see \cite[(2.3.3)]{DPS98a} or Lemma \ref{w_0,J} in this monograph.
The quasi-poset ideal $V$ is, thus, a union of
left cells. For any integer $h$, let $V_{h-1}$ be the set of all $v\in V$ belonging to a left cell $\omega$ with 
$\htt(\omega)\leq -h$. Then, each $V_{h-1}$ is, by construction, a union of left cells, and is a quasi-poset ideal, since
$\htt$ is compatible with  $\leq_L$ on $\Omega$. (This was observed in Notation \ref{notation1} as a consequence of (\ref{order}).) The $\sZ$-span $M_{h-1}$ of $V_{h-1}$ is, thus, a left  
$\sH$-module, as follows from the definition of $\leq_L$ \cite[p. 43]{Lus03}.  Similar considerations identify
$M_{h-1}/M_h$ as a direct sum of left-cell modules of height $-h$, one summand for each left cell in  $ V$ of height  
$-h$. Thus, $M\in\scrA(\scrL)$, proving the assertion.
%Let $M$ denote $\sH x_J$. From \cite[(2.3.3) \& Lem. (2.3.5)]{DPS98a}, $M$ has a basis of elements $\scc_v$ with $v\in W$ in
%a certain quasi-poset ideal $V$ for $\leq_L$, necessarily closed under $\sim_L$. In fact, we have
%$V=\{w\in W\mid ws<w;,\forall s\in J\}.$
%The quasi-poset ideal $V$ is, thus, a union of
%left cells. For any integer $h$, let $V_{h-1}$ be the set of all $v\in V$ belonging to a left cell $\omega$ with 
%$\htt(\omega)\leq -h$. Then each $V_{h-1}$ is, by construction, a union of left cells, and is a quasi-poset ideal, since
%$\htt$ is compatible with  $\leq_L$ on $\Omega$. (This was observed in Notation \ref{notation1} as a consequence of (\ref{order}).) The $\sZ$-span $M_{h-1}$ of $V_{h-1}$ is, thus, a left  
%$\sH$-module, as follows from the definition of $\leq_L$ \cite[p. 43]{Lus03}.  Similar considerations identify
%$M_{h-1}/M_h$ as a direct sum of left cell modules of height $-h$, one summand for each left cell in  $ V$ of height  
%$-h$. Thus, $M\in\scrA(\scrL)$, proving the assertion. 
\end{proof}

%\begin{rem} 
We remark that improvements of Proposition \ref{HxJ} occur in several places in this monograph. For example, Corollary \ref{3.7a} gives the sharper result that, for $J\subseteq S$, $\sH x_J\in \scrA^\flat(\scrL)$. See 
Theorem \ref{huh?} for a right-module version involving the exact category $(\scrA_\flat(\scrL^*),\scrE_\flat(
\scrL^*))$.
%\end{rem}

%Later, in Corollary \ref{3.7a}, we will show $M$ belongs to a more refined category $\scrA^\flat(\scrL)$.

\begin{cor}\label{hotcor} Let $M=\sH x_J$, and consider  the decreasing filtration 
$$M=M_0\supseteq M_1\supseteq \cdots\supseteq M_{l_J}=0$$
defined by the height function $\htt$. For any $h$, the short exact sequence $0\to M_h\to M_{h-1}\to M_{h-1}/M_h\to 0$
belongs of $\scrE^\flat(\scrL)$.
 \end{cor}
\begin{proof}By the discussion at the end of the proof of Proposition \ref{HxJ}, each inclusion $M_{h}\subseteq M_{h-1}$ is an inflation. Thus, each
$(M_{h}\to M_{h-1}\to M_{h-1}/M_{h})\in\scrE(\scrL)$. By
 Lemma \ref{onto}, the  sequence
$0\to (M_{h})^{\sH_J}\to (M_{h-1})^{\sH_J}\to (M_{h-1}/M_{h})^{\sH_J}\to 0$ is exact for all $J\subseteq S$.
Hence,  by definition (\ref{shortie}), the short exact sequence $0\to M_h\to M_{h-1}\to M_{h-1}/M_h\to 0$
belongs of $\scrE^\flat(\scrL)$.
\end{proof}

In the following theorem, the $\Ext^1_{\scrE^\flat(\scrL)}$-group is calculated in the exact category $(\scrA(\scrL),
\scrE^\flat(\scrL))$. Subsection 2.1C below introduces a second exact category  $(\scrA^\flat(\scrL),
\scrE^\flat(\scrL))$, possibly setting up a conflict of notation. However, as pointed out in Proposition \ref{prop3} (or the proof of its corollary), the $\Ext^1_{\scrE^\flat(\scrL)}(M,N)$-groups can
be calculated in either exact category, as long as $M,N$ both belong to $\scrA^\flat(\scrL)$. We also remark that
the condition (\ref{strange}) plays an essential role in the proof of the theorem.

\begin{thm}\label{ext0} For any $J\subseteq S$, $\Ext_{\scrE^\flat(\scrL)}^1(\sH x_J,S(\omega))=0$, for all $\omega\in\Omega$.
\end{thm}
\begin{proof}Consider an exact sequence 
\begin{equation}\label{E}
0\longrightarrow S(\omega)\longrightarrow E\ \overset g\longrightarrow\sH x_J\longrightarrow 0
\end{equation} in $\scrE^\flat(\scrL)$. We want to prove that the sequence is split. By definition, we obtain an exact sequence
\begin{equation}\label{F}
0\longrightarrow S(\omega)^{\sH_J}\longrightarrow E^{\sH_J}\overset{g_J}\longrightarrow (\sH x_J)^{\sH_J}\longrightarrow 0.
\end{equation}
Consider the $\sH_J$-submodule $\sZ x_J$ of $(\sH x_J)^{\sH_J}$. The exactness of (\ref{F}) implies that there exists $e\in E^{\sH_J}$ such that $g_J(e)=x_J$. Thus, $\sZ x_J\cong\sZ e$ as $\sH_J$-modules. Since $\sZ e\subseteq E$, this induces an $\sH_J$-module homomorphism $f_J:\sZ x_J\to E$. By Frobenius reciprocity, there is an $\sH$-module homomorphism
$f:\sH x_J\to E$ such that $f(x_J)=e$. Hence, $gf(x_J)=x_J$, and, consequently, $gf$ is the identity map on $\sH x_J$. In other words, the sequence \eqref{E} is split.
\end{proof}

Again, we note that $\scrE^\flat(\scrL)\subseteq\scrE(\scrL)\subseteq \scrE'_l$, the latter defined in Section 2.2 (see
Theorem \ref{B2}). Thus,
 $(\scrA(\scrL),\scrE^\flat(\scrL))$ is a full exact subcategory of $(\scrA'_l,\scrE'_l)$. 
 
 Note that, in the exact category $(\scrA(\scrL),\scrE^\flat(\scrL))$, not every object satisfies the property in Corollary \ref{hotcor}. This motivates the definition of the following full subcategory.

\medskip%\noindent
{\bf 2.1C. The exact category $\boldsymbol{(\scrA^\flat(\scrL),\scrE^\flat(\scrL))}$.} \index{exact category! $\sim$ $(\scrA^\flat(\scrL),\scrE^\flat(\scrL))$}
Let $\scrA^\flat(\scrL)$ be the full subcategory of $\scrA(\scrL)$ with objects $M$
having the property that, for any integer $h$ and any $J\subseteq S$, the map
\begin{equation}\label{surjective}(M_{h-1})^{\sH_J}\longrightarrow (M_{h-1}/M_{h})^{\sH_J}\end{equation}
is surjective.  In other words, each conflation
$(M_{h}\to M_{h-1}\to M_{h-1}/M_{h})$ in $\scrE(\scrL)$ gives rise, for each $J\subseteq S$, to a short exact seqeunce
$0\to (M_{h})^{\sH_J}\to (M_{h-1})^{\sH_J}\to (M_{h-1}/M_{h})^{\sH_J}\to0$. Hence, $(M_{h}\to M_{h-1}\to M_{h-1}/M_{h})\in\scrE^\flat(\scrL)$.

\begin{prop} \label{prop3} Let $M,P\in\scrA^\flat(\scrL)$, and assume a given $(M\to N\to P)\in\scrE(\scrL)$.
Then, $N\in\scrA^\flat(\scrL)$ if and only if $(M\to N\to P)\in\scrE^\flat(\scrL)$.
In particular, the full subcategory $\scrA^\flat(\scrL)$ of $\scrA(\scrL)$ is closed under extensions in $(\scrA(\scrL),
\scrE^\flat(\scrL))$.  \end{prop}
\begin{proof} First, assume that $(M\to N\to  P)\in\scrE^\flat(\scrL)$. For an integer $h$ and a subset $J\subseteq S$, form the commutative diagram

\begin{equation}\label{Mac}\begin{CD}
(M_{h-1}/M_{h})^{\sH_J}@>>> (N_{h-1}/N_{h})^{\sH_J}@>>> (P_{h-1}/P_{h})^{\sH_J}\\
    @AAA@AAA @AAA \\
(M_{h-1})^{\sH_J}@>>> (N_{h-1})^{\sH_J}@>>> (P_{h-1})^{\sH_J}\\
    @AAA @AAA @AAA \\
(M_{h})^{\sH_J}@>>> (N_{h})^{\sH_J}@>>> (P_{h})^{\sH_J.}\end{CD}\end{equation}
Since $M,P\in\scrA^\flat(\scrL)$, columns 1 and 3 are short exact sequences. Since  $(M\to N\to P)\in\scrE^\flat(\scrL)$, the three rows are all short exact sequences as well. Row 1 is already, before taking fixed points, a split short
exact sequence of $\sH$-modules, because $(M\to N\to P)\in\scrE(\scrL)$. Thus, it remains exact upon taking $\sH_J$-fixed points. Hence, column 2 is also exact, using \cite[Ex. 2, p. 51]{Mac94}, so
that $N\in\scrA^\flat(\scrL)$, as required.

Conversely, if $N\in\scrA^\flat(\scrL)$, then all columns in (\ref{Mac}) are short exact sequences. 
 Row 1 is exact as in the previous paragraph. By induction on $h$, the bottom row is also a short exact sequence. By \cite[Ex. 2, p. 51]{Mac94} again,
row 2 is exact, proving that $(M\to N\to P)\in\scrE^\flat(\scrL)$, as required.
\end{proof}

If $\mathscr F$ is a category of composible morphism pairs from an additive category $\mathscr B$, and $\mathscr C$ is an additive
subcategory of $\mathscr B$, let $\mathscr F|_{\mathscr C}$ denote the category of pairs from $\mathscr F$ 
whose objects belong to $\mathscr C$. 
 In the corollary below and elsewhere, we abuse notation by writing
$(\scrA^\flat(\scrL),\scrE^\flat(\scrL))$ for $(\scrA^\flat(\scrL),\scrE^\flat(\scrL)|_{\mathscr A^\flat(\mathscr L)})$.

\begin{cor} \label{3.7a} With the above notation, $(\scrA^\flat(\scrL),\scrE^\flat(\scrL))$ is an exact category, and is a full exact subcategory of $(\scrA(\scrL),\scrE^\flat(\scrL))$.  Moreover, for every $J\subseteq S$, $\sH x_J\in \scrA^\flat(\scrL)$.
(In particular, Theorem \ref{ext0} holds in the full exact subcategory $(\scrA^\flat(\scrL),\scrE^\flat(\scrL))$.)
\end{cor}

\begin{proof} Observe that the category $\scrA^\flat(\scrL)$  is  closed under extensions in the exact category $(\scrA(\scrL), \scrE^\flat(\scrL))$ by Proposition \ref{prop3}. So $(\scrA^\flat(\scrL), \scrE^\flat(\scrL))$ is itself exact
by \cite[Lem. 10.20]{Bu10}. (The proof in \cite{Bu10} is largely left as an exercise, in part using the
diagram of \cite[Lem. 3.5]{Bu10}.)
 The assertion of the corollary regarding $\sH x_J$ follows from Lemma \ref{onto}. Alternately, it follows from a downward induction on $h$ by applying Proposition \ref{prop3} to Corollary \ref{hotcor}. 
 \end{proof}

 \begin{lem}\label{3.7} Let $M\in\scrA^\flat(\scrL)$, $h<a$ be integers, and $J\subseteq S$. Then, $M_h$ and $M_h/M_a$ belong to $\scrA^\flat(\scrL)$, and the natural map
  \begin{equation}\label{surj2} (M_h)^{\sH_J}\longrightarrow (M_h/M_a)^{\sH_J}\end{equation}
  is surjective. The same result is true if $M_h$ is replaced by $M$.\end{lem}
 
 \begin{proof} Obviously, $M_h\in\scrA^\flat(\scrL)$ by definition. Once the surjectivity of (\ref{surj2}) is proven, it follows that $M_h/M_a
 \in\scrA^\flat(\scrL)$, by writing the map $M_{h}\longrightarrow M_h/M_{h+1}$ as the composite
 $M_h\to M_h/M_a\to M_h/M_{h+1}$ and taking  $\sH_J$-fixed points. It remains to prove the surjectivity
 of (\ref{surj2}).  Of course, it holds for $a=h+1$ by the definition
 of $\scrA^\flat(\scrL)$. By induction on $a-h$, it holds if $h$ is replaced by $h+1$. The required surjectivity follows from, say,
 a diagram chase in $\sZ$-mod, using the commutative diagram
 $$\begin{CD}
 M^{\sH_J}_{h+1} @>>> M_h^{\sH_J} @>>> (M_{h}/M_{h+1})^{\sH_J}\\
 @VfVV  @VgVV @| \\  
( M_{h+1}/M_a)^{\sH_J} @>>> (M_h/M_a)^{\sH_J} @>>> (M_h/M_{h+1})^{\sH_J},
 \end{CD}$$
 whose rows are short exact sequences.  
 \end{proof}
  
  \begin{cor} \label{3.8} \begin{itemize}\item[(a)] If $M\in\scrA^\flat(\scrL)$ has the $\htt$-filtration \eqref{lcf}, %Suppose that $M\in\scrA^\flat(\scrL)$, 
  and $a\in\mathbb Z$ with $0\leq a\leq n$, then
  $$ (M_a\longrightarrow M\longrightarrow M/M_a)\in \scrE^\flat(\scrL).$$
In particular,  $(M_{h}\to M_{h-1}\to M_{h-1}/M_{h})\in\scrE^\flat(\scrL)$, for all $1\leq h\leq n$. 
  
  \item[(b)] Suppose $(M\to N\to P)\in\scrE^\flat(\scrL)$.  If $h\in\mathbb Z$, then both $(M_h\to N_h\to P_h)$ and
  $(M/M_h\to N/N_h\to P/P_h)$ belong to $\scrE^\flat(\scrL)$. \end{itemize}
  \end{cor}
  
  \begin{proof} (a) follows from Lemma \ref{3.7}, using $M$ for $M_h$.
  
  (b) Clearly, $(M_h\to N_h\to P_h)\in\scrE^\flat(\scrL)$. The corresponding statement for the quotients follows from the
  $3\times 3$-lemma for exact categories \cite[Cor. 3.6]{Bu10}, using (a) to obtain the columns.\end{proof}
 
 Note that $\scrA^\flat(\scrL)$ is, by definition, a full subcategory of $\scrA(\scrL)$ and, thus, of $\scrA'_l$. The latter is defined, along
 with $\scrE'_l$, in Section 2.2.
 Also, $\scrE^\flat(\scrL)\subseteq\scrE'_l$ so that $(\scrA^\flat(\scrL),\scrE^\flat(\scrL))$  is a full 
 exact subcategory of $(\scrA'_l,\scrE'_l)$.\footnote{See the discussion above (\ref{modules}) for the terminology.} By the discussion in Section 2.2 below Proposition \ref{dual ft} (including Remark \ref{dual ft1}), the contravariant
 dual $(\scrA^\flat(\scrL),\scrE^\flat(\scrL))^{\op}$ is  realized  as $(\scrA^\flat(\scrL)^*,\scrE^\flat(\scrL)^*)$
 inside $(\scrA^{\prime *}_l,\scrE^{\prime *}_l)=(\scrA'_r,\scrE'_r)=(\scrA',\scrE')$.
 
  Remarkably, it is this dual of the exact category in the corollary that interests us most. We pursue this in
  Section 2.3. 

\section{Some generalities on duality} 

This section uses the (more abstract) right-module setup of \cite[Assump. 3.2]{DPS17}; see Example \ref{example1} in Appendix B for the left-module setup.
Thus, $\scrK$ is a Noetherian integral domain, $K$ is its fraction field, and $H$ is a finite and torsion-free algebra 
over $\scrK$ such that $H_K$ is semisimple.   The (isomorphism classes of) irreducible {\it right} $H_K$-modules are indexed by a given set $\Lambda$, and $\htt:\Lambda\to\mathbb Z$ is a height function.

It was also assumed ``for convenience" in \cite[Assump. 3.2]{DPS17} that $\htt$ takes non-negative values.  We do not make that assumption here. The ``convenience'' was in being able to use 0 as a lower bound for indexing elements of $\Lambda$ via $\htt$. In no case is the validity of a stated result in \cite{DPS17} affected by dropping the requirement that $\htt$ takes non-negative values. One such general result is the following.

\begin{thm}[{\cite[Th. 3.6]{DPS17}}] \label{A1}Let $\scrA$ denote the category of right $H$-modules which are finite and torsion-free over $\scrK$. For each $M$ in $\scrA$, and $h\in\mathbb Z$, let $M^h$ denote $M\cap (M_K)^{\leq h}$, where $(M_K)^{\leq h}$ is the sum in $M_K$ of all irreducible right $H_K$-submodules indexed by $\lambda\in\Lambda$ with $\htt(\lambda)\leq h$. Let $\scrE$ denote the set of short exact sequences $0\to X\overset i\to Y\overset d\to Z\to 0$ in mod-$H$ of objects $X,Y,Z\in\scrA$, such that, for each $h\in\mathbb Z$, the induced sequence
\begin{equation}\label{short} 0\longrightarrow X^h/X^{h-1}\longrightarrow Y^h/Y^{h-1}\longrightarrow Z^h/Z^{h-1}\longrightarrow 0\end{equation}
is also exact in mod-$H$. Then, $(\scrA,\scrE)$ is an exact category.
\end{thm}

We also have

\begin{thm}\label{B2} Keep the notation above, but let $\scrA'$ denote the full additive subcategory of objects $M\in\scrA$ for which $M^h/M^{h-1}$ is projective as a $\scrK$-module, for each $h\in\mathbb Z$. If $\scrE'$ denotes $\scrE|_{\scrA'}$, the short exact sequences in $\scrE$ whose objects all belong to $\scrA'$, then $(\scrA',\scrE')$ is an exact category.
\end{thm}

\begin{proof}
The category $\scrA'$ is precisely the category $\scrA(\scrS)$ in \cite[Constr. 3.8]{DPS17} with $\scrS=\bigcup_{i\in\mathbb Z},\scrS_i$ in the special case that each $\scrS_i$ is the full additive category of $\scrA$ consisting of objects $S$ projective over $\scrK$ with $S_K$ a sum of irreducible right $H_K$-modules of height $i$, for all $i\in\mathbb Z$.

Note that $\scrA'$ is closed under extensions in $\scrA$, since the middle term of the exact sequence (\ref{short}) is projective over $\scrK$
when the other terms are projective over $\scrK$. Consequently, \cite[Lem. 10.20]{Bu10} implies that $(\scrA',\scrE')$
is an exact category.  
\end{proof}

Observe that the category $\scrE(\scrS)$, as defined in \cite[Constr. 3.8]{DPS17},  is not used in the proof of the above result.  Nevertheless, $(\scrA(\scrS),
\scrE(\scrS))$ is a (full) exact subcategory of $(\scrA',\scrE')$, while $(\scrA',\scrE')$ provides us with a general place
to park duals, arising from the construction below. Those duals of most interest to us will actually belong to 
$(\scrA(\scrS),\scrE(\scrS))$ (with $\scrS$ restricted to suitable direct sums of dual left-cell modules).

%\begin{rem} The same argument shows $(\AAA,\scrE')$ is an exact category for any choice of $\scrS$ as required in %\cite[Constr. 3.8]{DPS17}, provided all objects in $\scrS$ are, additionally, projective over $\scrK$. Notice that the description
%of  $M^h$ as an intersection in the statement of Theorem \ref{A1} is automatic in this
%case, just from the fact that all members of the various $\scrS_h$  are torsion-free,
%with distinct composition factors (for distinct choices of $h$) after base change to $H_K$.
%\end{rem}

\vspace{.3cm}
%\noindent
{\bf 2.2A. Duality.}  In Section 2.1, we emphasized the differences between left
and right modules, because of the historical use of left-cell modules
and their duals.  Here, we want to do the opposite and try to view left and right modules
on a more equal footing.

The discussion and results above in this section are formulated in terms of right $H$-modules. They all have analogs for left $H$-modules as may be seen from the right $H$-module case by substituting the opposite algebra $H^{\text{op}}$ for $H$. In a context where both left and right $H$-modules are used, it is useful to use a subscript $l$ or $r$ to distinguish the two cases. Thus, $\scrA_l$, $(\scrA_l,\scrE_l)$, $(\scrA_l',\scrE_l')$ are $\scrA$, $(\scrA,\scrE)$, $(\scrA',\scrE')$ analogs in the left $H$-module case, and a subscript $r$ could be used similarly to give a balanced treatment of the (original) right $H$-module case.

Next, let $(-)^*=\Hom_{\scrK}(-,\scrK)$ be the usual duality functor---actually, there are two functors, $\scrA_l\rightarrow
\scrA_r$ and $\scrA_r\rightarrow \scrA_l$. Both are contravariant and are denoted by the same
symbol. The functors apply also to $\scrA_l'$ and $\scrA_r'$ and induce equivalences of additive categories 
$$ \scrA_l'\overset\sim\longrightarrow \scrA_r^{\prime\text{op}}\quad{\text{\rm and}}\quad\scrA_r'{\overset{\sim}\longrightarrow \scrA_l^{\prime\text{\op}}}.$$
Similarly, we can apply $(-)^*$ to $\scrE'_l$ and $\scrE'_r$.  Finally, $(-)^*$ can be applied to the pairs $(\scrA'_l,\scrE'_l)$
and $(\scrA'_r,\scrE'_r)$ as well as the opposite categories.

Our notation hides the height function $\htt$, which affects the definition of $\scrA',\scrE',\cdots$. If we wish, we can always add a constant integer to any given $\htt$ function without affecting $\scrA',\scrE'\ldots$. We assume in the proposition below, and unless otherwise noted, that the dual of any irreducible right $H$-module indexed by $\la\in\Lambda$ is also indexed by $\la$ and that 
\begin{equation}\label{ht pair}\htt_l(\la)=-\htt(\la)=-\htt_r(\lambda),\end{equation} for all $\la\in\Lambda$.
\begin{prop} \label{dual ft}The duality functor $(-)^*$ induces a contravariant equivalence between $(\scrA'_l,\scrE'_l)$ and $(\scrA'_r,\scrE'_r)$.
\end{prop}
We leave the proof to the reader, except for noting that the projectivity (over $\scrK$) property of objects in $\scrA_l'$ is preserved by $(-)^*$ and also guarantees that $(-)^*$ sends $\scrE_l'$ to $\scrE_r^{\prime}$. Similar properties hold
with $r$ and $l$ interchanged. Also, $(-)^{**}=((-)^*)^*$ is naturally isomorphic to the identity functor. 

\begin{rem}\label{dual ft1} The functor $(-)^*$ above is more than a contravariant equivalence. It defines  an isomorphism of one category with a 
full representative subcategory of another. (Recall that a representative subcategory contains a member of every isomorphism class of objects in the larger category. If it is also full, its inclusion is an equivalence.)
\end{rem}

In view of the proposition, we may think of  $(\scrA'_r,\scrE'_r)$ as a dual of $(\scrA'_l,\scrE'_l)$, via the contravariant functor $(-)^*$. More interesting to note is the possibility of using $(-)^*$ to construct new exact categories inside $(\scrA'_r,\scrE'_r)$ as duals (i.e., strict images under $(-)^*$) of exact subcategories of $(\scrA'_l,\scrE'_l)$. More precisely, if $(\scrB,\scrF)$ is a full exact subcategory of $(\scrA_l',\scrE_l')$, let $\scrB^*$ denote the (strict) image of $\scrB$ under $(-)^*$ in $\scrA'_r$, and define
$\scrF^*\subseteq \scrE'_r$ similarly. Then, $(-)^*$ induces a contravariant equivalence $(\scrB,\scrF)\overset\sim\to
(\scrB^*,\scrF^*)$. That is,  the subcategory $(\scrB^*,\scrF^*)\subseteq(\scrA',\scrE')$ is equivalent to the contravariant dual
$(\scrB,\scrF)^{\text{\rm op}}$ of $(\scrB,\scrF)$.

\medskip

In general, the notation $M^h$ applies to both the left- and right-module case.
If it is really needed to distinguish the left case from the right,  subscripts
may be used, with the convention $M_h =M^{-h}$.  This notation is implicit in Section 2.1, and we use it below.

\vspace{.3cm}
%\noindent
{\bf 2.2B. Balanced subcategories.} A property we are looking for in exact subcategories of $(\scrA',\scrE')$ $(=(\scrA'_r,\scrE'_r))$ is {\it balance.} More generally, if $(\scrB,\scrF)$ is a full\footnote{\label{balanced}We could 
use a weaker condition than``full" in this definition. All we really need is that, for any $X,Y\in\scrB$,
if $f\in\Hom_{\scrA'}(X,Y)$ is an isomorphism in $\scrA'$, then
$f\in\Hom_{\scrB}(X,Y)$.  One could call the embedding $\scrB\to\scrA'$ {\it isomorphism-full} under such a condition.
Note that the strict closure $\overline{\scrB}$   of $\scrB$ in $\scrA'$ (throwing in all objects in $\scrA'$ isomorphic to an object in $\scrB$) makes sense as a category in the presence of this condition, using for 
morphisms $f:X'\to Y'$ between its objects those morphisms $f$ in $\scrA'$ for which there are in $\scrA'$ isomorphisms $e:X\to X'$ and $g:Y'\to Y$ with objects $X,Y\in \scrB$ and composite $efg$ a morphism in $\scrB$.
 As such, the inclusion of $\scrB$ into
$\overline{\scrB}$ is full and an equivalence of categories, and $\scrB$ is  a subcategory of $\scrA'$ closed under object
isomorphisms.} exact
subcategory of $(\scrA',\scrE')$, we say $(\scrB,\scrF)$ is {\it balanced} provided, for each $M\in \scrB$ and $h\in\mathbb Z$, the objects $M^h$ and $M/M^h$ are both isomorphic to objects in $\scrB$, and the sequence $(M^h\to M\to M/M^h)$ (which is obviously in $\scrE'$) is isomorphic to a member of $\scrF$. Here, we are using the defining properties of $\scrA'$ and $\scrE'$ introduced in Theorems \ref{A1} and \ref{B2} above. Of course, $(\scrA',\scrE')$ itself is balanced. 

We have not assumed in the definition above that $\scrB$ is closed under isomorphism in $\scrA'$. However, $(\scrB,
\scrF)$ is balanced if and only if its strict closure in $(\scrA',\scrE')$ is balanced. See footnote \ref{balanced} above
for further discussion.
   
   A similar discussion applies for $(\scrA'_l,\scrE'_l)$.
   
\begin{prop}\label{balanced1)} Suppose $(\scrB,\scrF)$ is a balanced exact subcategory of $(\scrA'_l,\scrE'_l)$. If $(\scrB^*,\scrF^*)$ denotes its strict image in $(\scrA_r',\scrE_r')$ under $(-)^*$, then
$(\scrB^*,\scrF^*)$ is balanced. 
 \end{prop}

We leave the proof of this result to the reader. 

We have already observed in Corollary \ref{3.8} and Proposition \ref{3.8bis} {\color{blue} below} that the exact categories $(\scrA^\flat(\scrL), \scrE^\flat(\scrL))$ and $(\scrA_\flat(\scrL^*), \scrE_\flat(\scrL^*))$ are balanced. As a consequence, the following proposition applies to both.   Note that $\scrA^\flat(\scrL)$ and $\scrA_\flat(\scrL^*)$ are, respectively, closed under isomorphism of objects in $\scrA'_l$ and $\scrA'_r$.
 
\begin{prop}\label{second} Let $(\scrA',\scrE')$ denote either $(\scrA'_r,\scrE'_r)$ or its left-hand counterpart. Suppose $(\scrB,\scrF)$ is a balanced exact subcategory of $(\scrA',\scrE')$.
 If $M\in\scrB$, $h\in\mathbb Z$, then the sequence
\begin{equation}\label{pull} M^{h-1}\longrightarrow M^h\longrightarrow M^h/M^{h-1}\end{equation}
belongs to $\scrE'$ and is
 isomorphic to a sequence in $\scrF$. In fact, the sequence belongs to $\scrF$, if $\scrB$ is closed under isomorphisms in $\scrA'$.
\end{prop}

\begin{proof} We treat the right-hand case $(\scrA',\scrE')=(\scrA'_r,\scrE'_r)$, which
also implies the left-hand version of the proposition. 

 Fix $M$ in $\scrB$ and let $h$ be an integer. Applying the definition of balanced twice, we obtain that  $M^h, 
(M^h)^{h-1},$ and $M^h/(M^h)^{h-1}$ all belong to $\scrB$. Also, $((M^h)^{h-1}\to M^h\to M^h/M^{h-1})\in \scrF$. Using the definition of
$\scrA$ in Theorem \ref{A1} we obtain the following commutative diagram in $\scrA$ with all vertical arrows isomorphisms.

\[
\begin{CD}
(M^h)^{h-1}       @>>>    M^h @>>> M^h/(M^h)^{h-1} \\
@V{\wr}VV        @|                 @V{\wr}VV \\
M^{h-1}         @>>>    M^h    @>>>             M^h/M^{h-1}.
\end{CD}
\]
Since $\scrA'$ is closed under isomorphism
(and full) in $\scrA$, this diagram and its isomorphic images are all in $\scrA'$.  The top row is in $\scrB$ and, in fact, belongs to $\scrF$.  The bottom rows belongs to $\scrE'$ and are isomorphic to the top row (as a diagram). This proves
the first assertion of the proposition.  

To prove the second, assume $\scrB$ is closed under isomorphism.
Then, the objects and morphisms in the bottom row all belong to $\scrB$, so the bottom row belongs to $\scrF$. 
(Recall that the conflations in any exact category are required to be closed under isomorphism as part of their definition.)
This completes the proof of the proposition.\end{proof}

The last part of the proof above implicitly used the fact that $\scrB$ is full, assumed in the definition of balanced. It would
have been enough to assume the weaker property in the footnote \ref{balanced} of this chapter.

\section{The dual category of $(\scrA^\flat(\scrL),\scrE^\flat(\scrL))$}
We now turn back to the generic Hecke algebra $\sH$ associated with a standard finite Coxeter system $(W,S,L)$.
We first consider the right $\sH$-module analog $(\scrA_r(\scrL^*),\scrE_r(\scrL^*))$ of $(\scrA_l(\scrL),\scrE_l(\scrL))$ defined in 
display (\ref{display1}) (see \eqref{display2} below).

Thus, $\scrA_r$ is the full
subcategory of right $\sH$-modules which are finite and torsion-free over $\scrK=\sZ$. We often
denote the category $\scrA_r$ simply as $\scrA$ when no confusion can arise.  

\begin{notation}\label{notation2} {\rm It will be convenient in this section to use the height function $\htt=\htt_r$, which is the negative of the height function $\htt_l$
introduced in Notation \ref{notation1}. Thus, the symbol $\htt$ in this section, without further comment, denotes
$\htt_r=-\htt_l$;  it takes positive values. It is also compatible with the preorder $\leq_{LR}^{\op}$ and partial order $\leq_L^{\op}$ on $\Omega$.}
\end{notation}

 For an integer $h$, let $\scrL^*_h$ be the full subcategory of $\scrA$ consisting of right $\sH$-modules $N$ such that $N$ is a finite direct sum of modules which are isomorphic to dual left-cell modules  % \index{cell module! dual left $\sim$, $S_\omega$} 
$S_\omega:=S(\omega)^*$ of height $h$. Then let $\scrL^*:=\bigcup_h\scrL^*_h$. We remind the reader that, by a slight abuse of notation, a left-cell module $S(\omega)$ has a negative height
 $h=\htt_l(\omega)< 0$, and the dual left-cell module $S_\omega$ has height $-h=\htt_r(\omega)>0$.

Let $\scrA(\scrL^*)$ be the full additive subcategory of mod-$\sH$ with objects $M$ whose $\htt_r$-filtration
\begin{equation}\label{dlcf}
0=M^0\subseteq M^1\subseteq M^2\subseteq\cdots\subseteq M^{n-1}\subseteq M^n=M
\end{equation}
has the property that, for any integer $h> 0$, $M^h/M^{h-1}$ is an object in $\scrL^*_h$. See \cite[\S3]{DPS17}, which
also allowed $h=0$ (but did not use this case). Like the refinement in footnote \ref{LcellF}, such a filtration will be informally called {\it a dual left-cell filtration}. \index{cell filtration! dual left $\sim$}

By analogy with
\cite[\S3]{DPS17}, let $\scrE(\scrL^*)$ consist of those sequences $(M\to N\to P)$
in $\scrA$ such that, for every $h$, 
$$0\to M^h/M^{h-1}\to N^h/N^{h-1}\to P^h/P^{h-1}\to 0$$
 is a split
short exact sequence in the additive category $\scrL^*_h$. Then, as a special case of
\cite[Th. 3.9]{DPS17}, we have
\begin{equation}\label{display2}{\text{\rm{$(\scrA(\scrL^*),\scrE(\scrL^*))$ is an
exact category.}}}\end{equation}

%\medskip\medskip

%At this point, the reader may wish to review section 2.2 on duality. In particular,
Recall from Section 2.2 the standard $\sZ$-duality functor $(-)^*$ discussed in the context of the exact categories used
here.  Thus, if $N\in\scrA_l(\scrL)$, then $N^*\in\scrA_r(\scrL^*)$ and $(N_{h-1}/N_h)^*\cong (N^*)^{h}/(N^*)^{ h-1}$,
for each $h\in\mathbb Z$. 
More precisely, if $N$ has a (left-cell) $\htt_l$-filtration as in \eqref{lcf}, then
$N^*$ has a (dual left-cell) $\htt_r$-filtration
$$0=(N^*)^0\subseteq(N^*)^1\subseteq (N^*)^2\subseteq\cdots\subseteq(N^{n-1})^*\subseteq (N^*)^n=N$$
as given in \eqref{dlcf}, where $(N^*)^i=(N/N_i)^*$ so that $(N^*)^h/(N^*)^{h-1}\cong (N_{h-1}/N_h)^*\in(\scrL_{-h})^*=(\scrL^*)_h$, for all $1\leq h\leq n$. Hence, $N^*$ is an object of $\scrA_r(\scrL^*)$. Similarly, for a conflation $(M\to N\to P)\in\scrE_l(\scrL)$, we have $(P^*\to N^*\to M^*)\in\scrE_r(\scrL^*)$,
since $\big((P^*)^h/(P^*)^{h-1}\to (N^*)^h/(N^*)^{h-1}\to (M^*)^h/(M^*)^{h-1}\big)$, which is isomorphic to
$\big((P_{h-1}/P_h)^*\to (N_{h-1}/N_h)^*\to (M_{h-1}/M_h)^*\big)$, is split exact, for all $1\leq h\leq n$.

Hence, the exact category $(\scrA(\scrL^*),\scrE(\scrL^*))$ is naturally isomorphic to the dual category
$(\scrA_l(\scrL)^*,\scrE_l(\scrL)^*)$ of $(\scrA_l(\scrL),\scrE_l(\scrL))$.
As in Section 2.2, $\mathscr A_l(\mathscr L)^*$ and $\mathscr E_l(\mathscr L)^*$ both denote strict images (closed under isomorphism).

Let $\scrE_\flat(\scrL^*)=\scrE^\flat(\scrL)^*$ and let $\scrA_\flat(\scrL^*)=\scrA^\flat(\scrL)^*$.  \index{exact category! dual $\sim$ $(\scrA_\flat(\scrL^*),\scrE_\flat(\scrL^*))$}
We easily prove the following key proposition, motivated by the weaker \cite[Cor. 4.5]{DPS15} and
\cite[Th. 2.3.9]{DPS98a}. For $J\subseteq S$, 
\begin{equation}\label{duality} x_J\sH\cong (\sH x_J)^*\end{equation}
by \cite[Lem.~2.3.9]{DPS98a}.  In particular, 
\begin{equation}\label{coset}  x_J\sH\,\,\, {\text{\rm belongs to $\scrA_\flat(\scrL^*)$ }} \end{equation}
by Corollary
\ref{3.7a}.

\begin{prop}\label{keyprop} With the notation above, $(\scrA(\scrL^*), \scrE_\flat(\scrL^*))$ is an exact category. 
Also, $(\scrA_\flat(\scrL^*),\scrE_\flat(\scrL^*))$ is a full exact subcategory of $(\scrA(\scrL^*), \scrE_\flat(\scrL^*))$ which is closed under extensions.
 Moreover, for any left cell $\omega\in\Omega$ and any $J\subseteq S$, there
is a vanishing
\begin{equation}\label{bigshot}\Ext^1_{\scrE_\flat(\scrL^*)}(S_\omega, x_J\sH)=0\end{equation}
in both of these categories.
\end{prop}

\begin{proof} The category $(\scrA(\scrL^*),\scrE_\flat(\scrL^*))$ is the (contravariant) dual
$(\scrA(\scrL),\scrE^\flat(\scrL))$. Hence, it is exact, since  $(\scrA(\scrL),\scrE^\flat(\scrL))$ is exact by Proposition
\ref{Prop2}. Similarly, $(\scrA_\flat(\scrL^*),\scrE_\flat(\scrL^*))$ is exact, by Corollary \ref{3.7a}. Clearly, $(\scrA_\flat(\scrL^*),\scrE_\flat(\scrL^*))$ is a full exact subcategory of $(\scrA(\scrL^*),\scrE_\flat(\scrL^*))$ and is closed under extensions by Proposition \ref{prop3}. The displayed vanishing, in the case of $(\scrA(\scrL^*),\scrE_\flat(\scrL^*))$, follows from Theorem \ref{ext0}
and duality.  However, we have already noted that $(\scrA_\flat(\scrL^*),\scrE_\flat(\scrL^*))$ is an exact subcategory of
$(\scrA(\scrL^*),\scrE_\flat(\scrL^*))$, and $\scrA_\flat(\scrL^*)$ is full in $\scrA(\scrL^*)$,  so the displayed vanishing also holds for $(\scrA_\flat(\scrL^*),\scrE_\flat(\scrL^*))$. This completes the proof of the proposition. \end{proof} 

 We conclude this section with the following proposition, a version of Corollary \ref{3.8} in our present dual setting.  The
 proposition is essentially equivalent to that corollary, and can be deduced from it using the duality functor $(-)^*$.
 We omit the details. An alternative argument can be given using the duality functors for part (a); then using 
 the $3\times 3$-Lemma.  See the proof of Proposition \ref{3.8}(b).  
 
\begin{prop}\label{3.8bis}  (a) If $M\in\scrA_\flat(\scrL^*)$ has the  $\htt$-filtration \eqref{dlcf}, and $a\in\mathbb Z$ with $0\leq a\leq n$. Then,
  $$ (M^a\to M\to M/M^a)\in \scrE_\flat(\scrL^*).$$
 In particular, each inclusion $M^{h-1}\subseteq M^h$ gives a conflation $(M^{h-1}\to M^{h}\to M^{h}/M^{h-1})\in\scrE^\flat(\scrL)$.

  (b) Suppose $(M\to N\to P)\in\scrE_\flat(\scrL^*)$.  If $h\in\mathbb Z$, then both $(M^h\to N^h\to P^h)$ and
  $(M/M^h\to N/N^h\to P/P^h)$ belong to $\scrE_\flat(\scrL^*)$.
  \end{prop}

%Note that, with mild changes of notation, we can view $\scrL^*$ as an example of an $\scrS$ used in  \cite[Th 4.7]{DPS17}. From now on. we simple take $\scrS=\scrL^*$, $(\scrA_\flat,\scrE_\flat)$ then becomes

\section{Another abstract setting}
Our main goal in this section is to prove Theorem \ref{relative inj}, which builds a kind of injective envelope for certain exact categories, improving upon \cite[Th. 4.7]{DPS17}.  This goal is achieved by 
 refining the abstract framework used in
 the earlier paper. At least some abstraction is desirable.  The category $(\scrA_\flat(\scrL^*), \scrE_\flat(\scrL^*))$ used in the main Theorem \ref{4.4} is not easy to investigate directly, but it inherits many properties
  from its ``dual" $(\scrA^\flat(\scrL),\scrE^\flat(\scrL))$. Those dual properties are often most easily checked
  in the abstract setting of an exact category satisfying Conditions \ref{condition} below.

 We follow the notation of \cite[Assump. 3.2 and Constr. 3.8]{DPS17}.\footnote{In \cite[Const.3.8]{DPS17} at the end of the first paragraph, it should be assumed that $M^j/M^{j-1}$ is isomorphic to, not necessarily equal to, an object in $\scrS$
for all $j$. This is not a problem usually, since one often requires that $\scrS$ be closed under isomorphism.}
 Thus, $\scrK$ is a Noetherian integral
domain with fraction field $K$. Also, $H$ is a finite and torsion-free $\scrK$-algebra such that $H_K$ is semisimple.
Thus, $H_K$ has a finite set $\Lambda$ of isomorphism classes of irreducible modules. In this section, ``modules" for
$H$ and $H_K$ are taken as right modules, We are given a ``height function"
 $\htt:\Lambda\to \mathbb Z$, as defined in Appendix B.  We will assume  $\htt$ takes positive values only on
 $\Lambda$, unless
 noted otherwise.

Often $\scrK$ is the algebra $\sZ=\mathbb Z[t,t^{-1}]$ of Laurent polynomials, and $H=\mathcal H$ is a generic Hecke
algebra (as defined in Section 1.2 above and used in Sections 2.1 and 2.3). To pass from the height function domain $\Omega$ used in Sections 2.1 and 2.3 to the set $\Lambda$ above, see Notation \ref{notation1} and Remark \ref{A7}.  We will later apply the abstract results of this section to Section 2.3, where the height function $\htt=\htt_r$ also takes positive values.

Let $(\scrA,\scrE)$ be the exact category given in Theorem \ref{A1}.  In particular, $\scrA$ is the full subcategory of mod-$H$ consisting of right $H$-modules which are finite and torsion-free
over $\scrK$.   For each integer $i$,  let $\scrS_i$  be a set of objects in a full additive 
subcategory of $\scrA$.  For each  $S\in\scrS_i$, it is required that $S_K$  be a direct sum of finitely many irreducible modules of the same height $i$. If $i$ is not in the image of $\htt$, set $\scrS_i=[0]$, a full subcategory of zero objects in $\scrA$. Let $\scrS:=\bigcup_i\scrS_i$.
  Finally, there is an exact category 
$(\scrA(\scrS),\scrE(\scrS))$ which is   constructed in \cite[Constr. 3.8]{DPS17}. 
(This construction with some details has already been
referenced in 
a  left-module setting in (\ref{F}). See also (\ref{display2}) for a right-module setting.) The construction shows that
$\scrA(\scrS)$ is a full subcategory of $\scrA$ (and thus of mod-$H$), and it shows that $\scrE(\scrS)$ is contained in $\scrE$. 

\begin{rem}\label{strict} For convenience, we will assume that each category $\scrS_i$, and thus $\scrA(\scrS)$, is closed under isomorphism in mod-$H$. \end{rem}

In particular, $(\scrA(\scrL^*),\scrE(\scrL^*))$, in the Section 2.3 setting,  is an example of an exact category $(\scrA(\scrS),
\scrE(\scrS))$.

We will be interested in exact categories $(\scrA_\flat,\scrE_\flat)$, which satisfy the following   
\begin{conds} \label{condition}  $(\scrA_\flat,\scrE_\flat)$ is a full exact subcategory of $(\scrA(\scrS),\scrE(\scrS))$, with
$\scrA_\flat\hookrightarrow\scrA(\scrS)$ an inclusion of $\scrK$-categories, such
that the following hold:
\begin{enumerate}
\item $\scrA_\flat$ is object-closed under isomorphisms in $\scrA(\scrS)$ (and thus, by Remark \ref{strict} above,
in mod-$H$).
  \item All objects of $\scrS$ are contained in $\scrA_\flat$.
\item If $M\in\scrA_\flat$ and $h$ is an integer, then $M^h$, $M^{h-1}$ and $M^h/M^{h-1}$ belong to $\scrA_\flat$, and $(M^{h-1}\to M^h\to M^h/M^{h-1})$  belongs to $\scrE_\flat$.
\item For $M,N,P\in\scrA_\flat$, if $(M\to N\to P)\in\scrE_\flat$, then $(M^h\to N^h\to P^h)\in\scrE_\flat$ for 
all integers $h$.
\end{enumerate}
\end{conds}

%We list several examples.

\begin{ex} (a) 
The exact category $(\scrA(\scrS),\scrE(\scrS))$, a full exact subcategory of itself, satisfies
the conditions of  Conditions \ref{condition}.  In particular, in the setting of Section 2.3, $(\scrA(\scrL^*),\scrE(\scrL^*))$
provides an example.

(b) Another example is $(\scrA_\flat(\scrL^*),\scrE_\flat(\scrL^*))$, constructed in
Section 2.3 and viewed as a full exact subcategory of $(\scrA(\scrL^*),\scrE(\scrL^*))$. Conditions (1) and (2) are clear,
while conditions (3) and (4) (for $H=\sH$) are explicit in Proposition \ref{3.8bis}. %Condition (3) follows from Proposition \ref{second} and the remarks above it.

(c) The exact category $(\scrA(\scrL^*),\scrE_\flat(\scrL^*))$ discussed in Proposition \ref{keyprop}
does {\bf not} satisfy the Conditions \ref{condition}. In fact, condition (3) fails to hold in general.\end{ex}

\medskip
We now state several lemmas on $\Ext^1_{\scrE_\flat}$-vanishing which hold for an exact category $(\scrA_\flat,\scrE_\flat)$ satisfying the above conditions. Lemma \ref{claim} is difficult, while the remaining ones are more routine generalizations of known results for  the exact categories $(\scrA(\scrS),\scrE(\scrS))$ studied  and proved in
\cite[\S4]{DPS17}.
Thus, we give a detailed proof of Lemma \ref{claim}, and just brief comments on the remaining ones. Everything
is pulled together in Theorem \ref{relative inj}.

\medskip
{\it Throughout the rest of this section, 
 $(\scrA_\flat,\scrE_\flat)$ denotes a fixed exact category which satisfies Conditions \ref{condition}. }
 
 \begin{lem}\label{dog}If $M,N\in \scrA_\flat$,
 then the natural maps
  $\Ext^1_{\scrE_\flat}(M,N)\rightarrow \Ext^1_{\scrE(\scrS)}(M,N)$ and $\Ext^1_{\scrE_\flat}(M,N) \rightarrow
  \Ext^1_H(M,N)$, as well as $\Ext^1_{\scrE(\scrS)}(M,N)\rightarrow\Ext^1_H(M,N)$, are inclusions of $\scrK$-modules.
 \end{lem}
 
 The above lemma follows from the discussion of Strategy  \ref{rem2.2}.

\begin{lem}  \label{claim}   Let 
$h$ be an integer. Let $S\in\scrS_h$ and $M\in\Af$, and suppose the map
$$\Hom_{\Af}(S, M^h/M^{h-1})\longrightarrow\Ext^1_{\Ef}(S, M^{h-1})$$
defined by pull-back is surjective. Then, $\Ext^1_{\Ef}(S,M)=0$. \end{lem}

\begin{proof} First, we claim it is sufficient to treat the case $M=M^h$: Given $(M\to E\to S)\in \scrE_\flat$,
$(M^h\to E^h\to S)\in \scrE_\flat$ by Condition \ref{condition}(4).   The commutative diagram

\begin{equation}\begin{CD}\label{bottom} M @>>>E @>>> S \\
@A\cup AA  @AAA  @| \\
 M^h @>>> E^h @>>> S \end{CD}
\end{equation}
identifies the original sequence $(M\to E\to S)$ as a push-out of $(M^h\to E^h\to S)$ in mod-$H$. Since we have full subcategory
inclusions
$\scrA_\flat\subseteq\scrA(\scrS)\subseteq
\scrA\subseteq{\text{mod-}}H$, each map in the diagram (\ref{bottom}) above belongs to $\scrA_\flat$.  Also, each row is a
conflation in $(\scrA_\flat,\scrE_\flat)$. Thus, the diagram (\ref{bottom}) defines the original sequence $(M\to E\to S)$ as a push-out in $\scrA_\flat$. 
That is, 
the top row $(M\to E\to S)$ is isomorphic to the push-out of the bottom row, regarded as a conflation in
$(\scrA_\flat,\scrE_\flat)$. Obviously, this push-out is 0 in $\Ext^1_{\scrE^\flat}(S,M)$ if it is a push-out of 0
in $\Ext^1_{\scrE_\flat}(S,M^h)$.
  Thus, from now on,
we can assume that $M=M^h$. (A somewhat shorter argument  can be made using \cite[Prop. 2.12]{Bu10}, which relates push-outs to rectangles like (\ref{bottom}) in general, in an exact category.)

Let $\chi\in \Ext^1_{\Ef}(S,M)$ and let $M\overset\iota\to E\overset \delta\to S$ be a representative conflation sequence for $\chi$. Without loss of generality, we can assume $M$ is a (right) $H$-submodule of $E$, with $i$ its inclusion. Also, $M^{h-1}$ identifies with $E^{h-1}$. Form the push-out diagram of the deflation $M=M^h\to M^h/M^{h-1}$ and the sequence $M\overset\iota\to E\overset \delta\to S$, to get a commutative diagram in the abelian category mod-$H$:
\begin{equation}\label{top row}\begin{CD}
M^h/M^{h-1} @>>> E/E^{h-1} @>>> S\\
@AAA @AAA @|\\
M @>\subseteq >>  E@>>> S \\
@AAA @AAA\\
M^{h-1} @>=>> E^{h-1} 
\end{CD}
\end{equation}
with exact rows and columns (after adding 0 maps at the beginning and end of each row). Also, all objects in the diagram are in $\Af$, and each-three term row or column is in $\Ef$. (The only non-obvious case, given our hypothesis on $(\Af,\Ef)$, is the top row. However, this row is also in $\Ef$, since the push-out of an inflation is an inflation; see \cite[A.1(a)]{DPS17} which is based on \cite{K90}.)

In particular, the top row is defined by an element of 
$$\Ext^1_{\Ef}(S,M^h/ M^{h-1})\subseteq\Ext^1_{\EEE}(S, M^h/M^{h-1})=0.$$
It is, thus, split by an $H$-module map $\alpha:S\to E/E^{h-1}$. Let $j$ denote the map $M^h/M^{h-1}\to E/E^{h-1}$ in the top row, and let $\beta:S\to M^h/M^{h-1}$ be any (right) $H$-module homomorphism. Fullness of $\Af$ in $\AAA$, and of the latter in mod-$H$, guarantees that $\beta$ is a map in the additive category defined by $\Af$. In particular, the pull-back through $\beta$ of the element of $\text{Ext}^1_\Ef(M^h/M^{h-1}, M^{h-1})$ defining the left-hand column is an element of $\Ext^1_{\Ef}(S, M^{h-1})$. Indeed, this process defines the natural homomorphism
$$\Hom_{\Af}(S, M^h/M^{h-1})\longrightarrow\Ext^1_{\Ef}(S, M^{h-1}),$$
which our hypothesis in the lemma guarantees to be surjective. Notice that the image of $\beta$ may also be described as the image of $j\circ \beta$ under a similarly defined map
$$\Hom_{\Af}(S, E/E^{h-1})\longrightarrow\Ext^1_{\Ef}(S, E^{h-1})=\text{Ext}^1_\Ef(S,M^{h-1}).$$
(The top left square in diagram (\ref{top row}) defines $M$ as a pull-back with top row $j$. Extend it to the left by a
pull-back square with top row $\beta$, and consider the composite $j\circ\beta$.)

Observe that the map $\alpha+j\circ\beta:S\to E/E^{h-1}$ is still a splitting map of the top row in (\ref{top row}). Now choose $\beta$ so that its image in ${\Ext}^1_\Ef(S,M^{h-1})$ is the negative of the image of $\alpha$ under the display above, using its identification of $E^{h-1}$ with $M^{h-1}$. With this choice, the image of $\alpha+j\circ\beta$ is 0. This implies that the map $\alpha+j\circ\beta:S\to E/E^{h-1}$ can be lifted to a map $\alpha':S\to E$ of right $H$-modules. That is, the composition of $\alpha'$ with $E\to E/E^{h-1}$ is equal to $\alpha+j\circ\beta$. Composing further with $E/E^{h-1}\to S$ gives the identity map of $S$ to itself.

We have thus constructed a splitting map of the map $E\to S$ in mod-$H$. That is, the short exact sequence defined by $\chi$ is 0 in $\Ext^1_H(S,M)$. It is, therefore, also 0 in $\Ext^1_\Ef(S,M)$. 

This completes the proof of the lemma.
\end{proof}

The following result is a generalization of \cite[Lem. 4.1]{DPS17}. See \cite[Prop. A.2]{DPS17} for the definition of Ext$^1_\scrE$-groups.

\begin{lem}\label{In particular} Let $M,N\in\scrA_\flat$ and $h\in\mathbb Z$. 
\begin{itemize}
\item[(a)] There is a
natural isomorphism
$\Ext^1_{\scrE_\flat}(N^h,M)\cong\Ext^1_{\scrE_\flat}(N^h, M^h)$.

\item[(b)] In particular, if $S\in{\mathscr S}_h,$ we have
$$\Ext^1_{\scrE_\flat}(S,M)\cong\Ext^1_{\scrE_\flat}(S, M^h).$$

\item[(c)] Assume that $S\in{\mathscr S}_h$. If $M=M^h$ and $M^{h-1}=0$, then $\Ext^1_{\scrE_\flat}(S,M)=0$.
\end{itemize}
\end{lem}

\begin{proof}Note that part (b) follows from part (a). Part (a)  follows from \cite[Prop. 4.1]{DPS17}   and its description of the isomorphism 
$ \Ext^1_{\scrE(\scrS)}(N^h,M)\cong\Ext^1_{\scrE(\scrS}(N^h,M^h)$, together with
Condition \ref{condition}(4) and Lemma \ref{dog}.  Since the proof will be quoted later, we provide some of the details. 

First, we give a map $$\Ext^1_{\scrE(\scrS)}(N^h,M)\overset\phi\longrightarrow\Ext^1_{\scrE(\scrS)}(N^h,M^h)$$
and show that it is an isomorphism by explicitly giving an inverse. The map $\phi$ is obtained by applying the
functor $(-)^h$ to a conflation.  The inverse of $\phi$ is defined by push-out
through $M^h\rightarrow M$.   The restriction of $\phi$ to $\scrE_\flat$ makes sense using Condition \ref{condition}(4) and Lemma \ref{dog}.

This push-out also make sense in the exact category
 $(\scrA_\flat, \scrE_\flat)$ 
 and agrees with the push-out  in $(\scrA(\scrS),\scrE(\scrS))$. (This is an elementary argument. The agreement of push-outs
also follows from \cite[Prop. 2.12]{Bu10}.)  Thus, the isomorphism $\phi$ induces a corresponding
isomorphism when $\scrE(\scrS)$ is replaced by $\scrE_\flat$ .

Finally,  
part (c) (which is independent of (a) and (b)) follows from \cite[Lem. 4.1(c)]{DPS17} and Lemma \ref{dog} above. \end{proof}

\begin{lem}\label{3lemma}
Let $M\in \mathscr A_\flat$. If $S\in{\mathscr S}_h$, then 
$$\Ext^1_{\scrE_\flat}(S,M^{h-1})\hookrightarrow\Ext^1_{\scrE(\scrS)}(S,M^{h-1})\cong\Ext^1_H(S,M^{h-1}).$$\end{lem}

\begin{proof}The containment $\Ext^1_{\scrE_\flat}\hookrightarrow\Ext^1_{\scrE(\scrS)}$ in the display 
is obtained from Lemma \ref{dog}. The isomorphism in the display is from \cite[Lem. 4.2]{DPS17}. 
\end{proof}

%\begin{lem}\label{4lema}  $S\in {\mathscr S}_h$,  let $M\in{\scrA_\flat}$, and let $j$ be a non-negative integer. There is a 6-term exact sequence
%$$\begin{aligned}
%0\to\Hom_{\scrA_\flat}(S, M^j)\to&\Hom_{\scrA_\flat}(S, M)\to \Hom_{\scrA_\flat}(S,M/M^j) \\ \overset{f}\to &
%\Ext^1_{\scrE_\flat}(S, M^j)\overset{g}\to\Ext^1_{\scrE_\flat}(S, M)\overset{}\to\Ext^1_{\scrE_\flat}(S,M/M^j).\end{aligned}$$
%It is compatible with the first 6 terms of the long exact sequence for the functor $\Hom_{{\mathscr A}_\flat}%(S,-)=\Hom_{H}(S,-)$ applied to
%the short exact sequence $0\to M^j\to M\to M/M^j\to 0$ of $H$-modules. This sequence belongs to ${\mathscr E}_\flat$.  \end{lem}

\begin{lem}\label{prevcor}Let $S\in\mathscr S_h$, and let $M\in\mathscr A_\flat$.
Suppose that $\Ext^1_{\scrE_\flat}(S,M^{h-1})$ is generated as a $\mathscr K$-module by $\epsilon_1, \ldots,
\epsilon_n$ (regarded as elements $\Ext^1_H(S, M^{h-1})$ using the inclusion above). Let $M^{h-1}\to N\to S^{\oplus n}$ represent the element of 
$\Ext^1_{H}(S^{\oplus n},M^{h-1})$ 
corresponding to $\chi:=\epsilon_1\oplus\cdots\oplus\epsilon_n$ above. Suppose there is
a commutative diagram in {\text{\rm mod-}}$H$
$$\begin{CD}
M^{h-1} @>>> N @>>> S^{\oplus n}\\ @| @VVV @VfVV \\
M^{h-1} @>>> M^h @>>> M^h/M^{h-1}.\end{CD}
$$
 Then, $\Ext^1_{\scrE_\flat}(S,M)=0$.
\end{lem}

\begin{proof}  
 By the construction of addition in $\Ext^1_H(S, M^{h-1})$, for $1\leq  j\leq n$, 
there is an inclusion $\iota_j:S\to S^{\oplus n}$, inducing by pull-back a map $\Ext^1_H(S,M)\longrightarrow
\Ext^1_H(S^{\oplus n}, M^{h-1})$ taking $\chi$ to $\epsilon_j$. This gives the top rectangle in the commutative
diagram below, with top row representing $\epsilon_j$.  
$$\begin{CD}
M^{h-1} @>>> N_j @>>> S \\ @| @VVV @VVV\\
M^{h-1} @>>> N @>>> S^{\oplus n}\\ @| @VVV @VfVV \\
M^{h-1} @>>> M^h @>>> M^h/M^{h-1}.
\end{CD}
$$

Thus, the conflation corresponding to $\epsilon_j$
may be obtained by a pull-back of $M^{h-1}\to M^h\to M^h/M^{h-1}$ using the evident composite $$g_j:S\to S^{\oplus n}\overset
f \longrightarrow M^h/M^{h-1}.$$ Consequently, the image of $g_j\in\Hom_{\scrA_\flat}(S, M^h/M^{h-1})\subseteq
\Hom_H(S, M^h/M^{h+1})$ under the connecting homomorphism to $\Ext^1_{\scrE_\flat}(S, M^{h-1})\subseteq
\Ext^1_H(S, M^{h-1})$ is $\epsilon_j$. Since $j$ was arbitrary, the connecting homomorphism is surjective, so $\Ext^1_{\scrE_\flat}(S,M)=0$ by Lemma \ref{claim}.
\end{proof}

The next result will assume the following  hypothesis for $\scrS$:
\begin{fhyp} \label{finite}{\rm
 For any integer $h$,  $\scrS_h$ is strictly generated as an additive category by a finite set $F(h)$, i.e., every object in $\scrS_h$ is isomorphic to a finite direct sum of objects chosen from $F(h)$.}
 \end{fhyp} 
\begin{thm}[{}] \label{relative inj}Suppose $(\scrA_\flat,\scrE_\flat)$ is a full 
exact subcategory of $(\scrA(\scrS),\scrE(\scrS))$ satisfying conditions (1)--(4) in Conditions \ref{condition}. Assume $\scrS$
satisfies 
Hypothesis \ref{finite}.
Then, given $M\in \scrA_\flat$, there exists an object $X=X_M$ in $\mathscr A_\flat$ and an inflation $M\overset{i}\to X$ 
in $(\scrA_\flat,\scrE_\flat)$ 
such that $\Ext^1_{\scrE_\flat}(S,X)=0$, for all
$S\in\mathscr S$.
In addition,
 if $h$ is chosen minimal such that $M^{h-1}\neq 0$, it may  be assumed that the inflation $M\overset{i}\longrightarrow X$ induces an isomorphism $M^{h-1}\cong X^{h-1}$.
\end{thm} 
\begin{proof}Without loss of generality, we can assume that $M\not=0$ and that $\Ext^1_{\scrE_\flat}(S, M)\neq
0$, for some $S\in \mathscr S$.  Choose an integer $h$ minimal with
such a non-vanishing occurring for some $S\in \mathscr S_h$. Note that
$M^{h-1}\neq 0$ by Lemma \ref{In particular}(c). We will enlarge
$M$ to an object $X$, closer to the $X$ required in the theorem.

  Let $S_1,\ldots, S_m$ be generators for $\mathscr S_h$. For each index $i$, let $\epsilon_{i,1},\cdots,\epsilon_{i,n_i}$ be a finite
set of generators for $\Ext^1_{\scrE_\flat}(S_i,
M^{h-1})$. Form conflations
 $0\to M^{h-1}\to Y_i \to S_i^{\oplus n_i}\to 0$ corresponding to
 $\chi_i:=\epsilon_{i,1}\oplus\cdots\oplus\epsilon_{i,n_i}
\in\Ext^1_{\scrE_\flat}(S_i^{\oplus n_i}, M^{h-1})$. Put $\chi:=\chi_1\oplus \cdots\oplus \chi_m$, and let
$\chi'\in\Ext^1_{\scrE_\flat}(M^h/M^{h-1}, M^{h-1}) $ correspond to the
conflation $M^{h-1}\to M^h\to M^h/M^{h-1}$. Put $Z:=\oplus_i S_i^{\oplus n_i}
\oplus M^h/M^{h-1}$, and let $M^{h-1}\to X^h\to Z$ correspond to
$\chi\oplus\chi'$. Observe that there is a  
commutative diagram
$$
\begin{CD}
M^{h-1}@ >>> Y_i @>>> S_i^{\oplus n_i}\\
@| @VVV @VVV\\
M^{h-1} @>>> X^h @>>> Z,\end{CD}
$$
in which the top row corresponds to $\chi_i$ and the bottom row to
$\chi\oplus\chi'$ .  Note that the bottom row belongs to $\scrE_\flat$.
Comparison with 
Lemma \ref{prevcor}, allowing for the differences in notation, shows 
$\Ext^1_{\scrE_\flat}(S_i,X^h)=0$, for all $i$. Thus, $\Ext^1_{\scrE_\flat}(S,X^h)=0$, for all $S$ in ${\mathscr S}_h$.    Note that we have
the same vanishing for $S\in{\mathscr S}_j$ with $j<h$,  by our
choice of $h$.  In all cases, we can replace $X^h$ with any $X'$
containing it with $(X')^h=X^h$ by Lemma \ref{In particular}.

     So far, we have not constructed an object $X$, only $X^h$. However,
the latter may be viewed as the middle term of an exact sequence of right
$H$-modules $0\to M^h\to X^h\to S'\to 0$, where $S':=\bigoplus
S_i ^{\oplus n_i}\in \mathscr S_h$.  This sequence clearly corresponds to a conflation in $\scrE_\flat$.
(Notice that $X^h\to S'$ is the composition of the deflation $X^h\to Z$ and the split deflation $Z\to S'$, both for
$\scrE_\flat$. It is therefore a deflation for $\scrE_\flat$. The morphism $M^h\to X^h$ belongs to $\scrA_\flat$ and is a kernel of $X^h\to S'$ in mod-$H$, hence is a kernel in $\scrA_\flat.$)  Applying a push-out construction using $M^h\to M$ (see
Lemma \ref{In particular}(a,b) and their proof), we obtain an object $X$ in
$\mathscr A_\flat$ which contains a copy of $M$ under an
inflation and has our constructed $X^h$ as its image under the
functor $(-)^h$. In addition, $X^j=M^j$ for $j\leq h-1$.

   As noted above, $\Ext^1_{\scrE_\flat}(-,X)$
vanishes on all objects in $\mathscr S_j$ with $j\leq h$. Now repeat
the argument with $X$ in the role of $M$. This requires a bigger
$h$, unless $\Ext^1_{\scrE_\flat}(S, X)$ already vanishes, for all $S\in
\mathscr S$. Eventually the process stops, at which point $X$
satisfies all requirements of the theorem.
\end{proof}

\begin{rem} Note that \cite[Th. 4.7]{DPS17} is the special case of the theorem with $\scrA_\flat=\scrA(\scrS)$ and $\scrE_\flat=\scrE(\scrS)$.   
The result was applied in the main theorem \cite[Th. 4.9]{DPS17}
there to build, for each dual left-cell module $S_\omega$, an object $X_\omega$ in $\scrA(\scrS)$ such that 
$S_\omega\subseteq X_\omega$ (part of the filtration structure required for $\scrA(\scrS)$) and such that
$\Ext^{1}_{\scrE(\scrS)}(S,X_\omega)=0$, for all $S\in\scrS$.   We can apply the same construction here.\end{rem}

In preparation for the next section, we now assume that $H=\sH$ is the generic Hecke algebra over $\sZ=\mathbb Z[t,t^{-1}]$ as discussed in Section 1.2.  Thus, the Noetherian domain $\scrK$ will be specialized to $\sZ$.

\begin{const}\label{const}
{\rm Let $(\scrA_\flat, \scrE_\flat)=(\scrA_\flat(\scrL^*),\scrE_\flat(\scrL^*))$, using a height
function $\htt=\htt_r$ as in Notation \ref{notation2}. (Thus, $\htt$ is compatible with the preorder $\leq_{LR}^{\text{\rm op}}$.) Let $\Omega$ be the set of (Kazhdan-Lusztig) left cells in $W$.

\begin{itemize}
\item[(a)]  Let $\Omega'$ be the set of left cells that do not contain the longest element $w_{0,J}$ in a
parabolic subgroup $W_J$, $J\subseteq S$. % {\color{red}in their associated two-sided cells}. 
For $\omega\in\Omega'$, define (as given in Theorem \ref{relative inj})
$X_\omega=X_M$, for $M=S_\omega$, and let 
$X:=\bigoplus_{\omega\in\Omega'}X_\omega$. %Also, if $\omega\in\Omega\backslash\Omega'$, let $T_\omega
%:=x_J\sH$, where $J$ is the unique subset of $S$ with $w_{0,J}\in\omega$.  
\item[(b)] Define right $\sH$-modules
    \begin{equation}\label{T^+}\begin{cases} T:=\bigoplus_{J\subseteq S} \x_J\sH=\bigoplus_{\omega\in\Omega\backslash\Omega'}T_\omega,\\ T^+:=T\bigoplus X,\end{cases}\end{equation}
and form the generic Hecke endo-algebra and its extended version
    \begin{equation}\label{endo} \begin{cases}A:=\End_\sH(T),\\ A^+:=\End_{\sH}(T^+).\end{cases}\end{equation}
    \index{generic Hecke endo-algebra! $\sim$ $A$}\index{generic Hecke endo-algebra! extended $\sim$ $A^+$}
%    For a left cell $\omega$, put $T^+_\omega:=X_\omega$ if $\omega\in\Omega'$; and, if $\omega\not\in\Omega'$, set  $T^+_\omega=T_\omega=x_J\sH$, as above.
\item[(c)]    Finally,  for $\omega\in\Omega\backslash\Omega'$, let $T^+_\omega:=x_J\sH$, where $J$ is the unique subset of $S$ with $w_{0,J}\in\omega$, and $T^+_\omega=X_\omega$, for $\omega\in\Omega'$. 
    Define the following left $A^+$-modules
    \begin{equation}\label{delta}\begin{cases}\Delta(\omega):=\Hom_{\sH}(S_\omega,T^+),\\
    P(\omega):=\Hom_\sH(T^+_\omega,T^+),\,\,{\text{\rm for any}}\,\, \omega\in\Omega.\end{cases}\end{equation}
\end{itemize} }\end{const}

As considered in \eqref{modules} and also for later application in Section 5.4, we define, for any
 poset ideal $\Theta$ of $(\Omega,\leq_L)$, 
 \begin{equation}\label{Tdagger}
T^\ddag_\Theta:=\bigoplus_{\omega\in\Theta}(T^+_\omega)^{m_\omega}\;\text{ and }\; A^\ddagger_\Theta=\End_\sH(T^\ddagger_\Theta),
%\in\scrA_\flat %\text{ with all $m_\omega>0$},
\end{equation}
%the multiply extended versions%, following \cite[(4.2)]{DPS18}, 
%\begin{equation}\label{Tdagger}
%T^\ddagger:=\Big(\bigoplus_{\omega\in\Omega\backslash\Omega'}T_\omega^{\oplus m_\omega}\Big)\oplus
%\Big(\bigoplus_{\omega\in\Omega'}X_\omega^{\oplus m_\omega}\Big),
%\end{equation}
and
\begin{equation}\label{Adagger}
A^\ddagger=\End_\sH(T^\ddagger_\Omega),
\end{equation}
where $\{m_\omega\}_{\omega\in\Theta}$ is any fixed set  of positive integers.\footnote{We view $T^\ddagger$ as a modification of $T^+$, rather than $T$.}

\index{generic Hecke endo-algebra! multiply extended $\sim$ $A^\ddagger$}

\chapter{The conjecture and its proof}

The conjecture in the title of this chapter is Conjecture \ref{conjecture}, which is
\cite[Conj.1.2]{DPS15} and also captures essentially the earlier conjecture \cite[Conj. 2.5.2]{DPS98a}, as
discussed in Remarks \ref{Rem7.4}.  A special local case \cite[Th. 3.6]{DPS15} of the \cite{DPS15} conjecture was proved in a Cherednik algebra context, using an Ext$^1$ vanishing result
\cite[Cor. 5.3]{DPS15} obtained in part from \cite{GGOR03}. The latter results seemed to be hard to generalize. Nevertheless, in Chapter 2 here, we have used
 exact categories to formulate by analogy a tractable global version of the Ext$^1$ vanishing condition, not
mentioning localization or Cherednik algebras. Using it, we are now ready to prove this new version of the conjecture.

\medskip
{\it For this chapter, we continue the notation in the previous chapter, but, as in Construction \ref{const},  return to the special case where $H$ is the generic Hecke algebra $\sH$ (so that $\scrK=\sZ$) 
associated to a standard finite Coxeter system $(W,S,L)$.  }
\medskip
\section{$q$-permutation modules and their filtrations}
\medskip
Let $\sH$ be a generic Hecke algebra over $\sZ=\mathbb Z[t,t^{-1}]$ associated to a standard finite Coxeter system $(W,S,L)$.  Let $\Omega$ be the set of left cells of $W$. The preorders $\leq_L$, $\leq_{LR}$ on $W$
induce partial order $\leq_L$ and preorder $\leq_{LR}$ on $\Omega$; see \eqref{order}.
 We take this opportunity to repeat and highlight the following result, mentioned after Proposition \ref{prop1}.

\begin{prop}\label{208} Each two-sided cell module
is a direct sum of left-cell modules. \end{prop}
\begin{proof}Let $\frak c$ be a two-sided cell of $W$, and set $M(\mathfrak c):=\sH_{\leq_{LR}\mathfrak c}/\sH_{<_{LR}\mathfrak c}$, the associated two-sided cell module, where $\sH_{\leq_{LR}\mathfrak c}=\sH_{\leq_{LR} w}$ 
 for $w\in\mathfrak c$  is defined in \eqref{H<}.
Write $\mathfrak c=\cup_{i=1}^m\omega_i$ (a disjoint union) with $\omega_i\in\Omega$. Then, the free $\sZ$-module $M(\mathfrak c)$ is the direct sum $\bigoplus_{j=1}^m S'(\omega_j)$, where each submodule $S'(\omega_j):=\sum_{y\in\omega_j}(\sZ\scc_y+\sH_{<_{LR}\mathfrak c})/\sH_{<_{LR}\mathfrak c}$ is a free $\sZ$-module. It remains to prove that each $S'(\omega_j)$ is isomorphic to the left-cell module $S(\omega_j)=\sH_{\leq_{L}\omega_j}/\sH_{<_{L}\omega_j}$. This is clear, since the inclusion $\sH_{\leq_{L}\omega_j} \subseteq \sH_{\leq_{LR}\mathfrak c}$ induces  by \eqref{order} an $\sH$-module homomorphism $\sH_{\leq_{L}\omega_j}\to M(\mathfrak c)$ with image $S'(\omega_j)$ and with kernel $\sH_{<_{L}\omega_j}=\sH_{\leq_{L}\omega_j}\cap \sH_{<_{LR}\mathfrak c}$. (Observe that,  for some $z\in \omega_j$,  by \eqref{order}, if $y<_Lz$, then $y<_{LR} z(\in\mathfrak c)$, and if $y\leq_Lz$ and $y<_{LR}w$ for some $w\in\mathfrak c$, then $y<_Lz$.)
\end{proof}
Observe that the argument for this result uses display (\ref{order}). A similar fact holds for right-cell modules or dual left-cell modules.
Also, the $q$-permutation modules $\sH x_J$ in Proposition \ref{HxJ} have natural sections  arising from $\mathfrak h_l$ which are direct sums of left-cell modules. Once more, this goes back to (\ref{order}), as discussed at the end of the proof of Proposition \ref{HxJ}. The (right) $q$-permutation modules $x_J\sH$ have, similarly, filtrations by the direct sum of dual left-cell modules $S_\omega:= S(\omega)^*$ defined as the $\sZ$-linear dual of the left-cell module $S(\omega)$; see Notation \ref{notation1}.

We now define a new preorder $\preceq=\preceq_\Omega$ \index{preorder $\preceq$} on $\Omega$ by setting, for, $\omega,\omega'\in\Omega$,
\begin{equation}\label{preorder1} 
\omega\preceq\omega'\iff \htt(\omega)<\htt(\omega'), \text{ or } \htt(\omega)=\htt(\omega')\text{ and }\omega\sim_{LR}\omega'.\end{equation} 
Here, the height function $\htt=\htt_r=-\htt_l$ is chosen so that
it is compatible with $\leq^\op_{LR}$ on $\Omega$. See Notation \ref{notation2} and display (\ref{preorder}) in Appendix B. 
In particular, 
$\htt$ is  constant on two-sided cells. 
 Thus, $\preceq$ above is well-defined on $\Omega$ and even induces a partial order on the set $\overline\Omega$
of two-sided cells.

It also follows from the definition in display \eqref{preorder1} that the equivalence classes of the preorder $\preceq$ are precisely the two-sided cells, 
satisfying a requirement in \cite[Conj. 1.2]{DPS15}. More is true along these lines: the preorders $\leq^\op_{LR}$ and $\preceq$ share  a compatible height function $\htt$. In addition, $\leq^\op_{LR}$ strictly dominates $\preceq$ (see Definition \ref{dominates} in Appendix B)\footnote{To check this, suppose that $\omega,\omega'\in\Omega$ and
$\omega<^\op_{LR}\omega'$. Since $\htt$ is compatible with $\leq^\op_{LR}$, we have $\htt(\omega)<\htt(\omega')$.
Thus, $\omega\preceq\omega'$. However, it is not possible for $\omega\sim_{LR}\omega'$, since $\htt$ is constant on  two-sided cells. So $\omega\prec\omega'$. Since we already know  for $\omega,\omega'\in \Omega$, that $\omega\sim_{LR} \omega'$ implies $\omega\sim_\prec\omega'$, we have proved that $\leq^\op_{LR}$ strictly 
dominates $\preceq$.} \index{preorder! strictly dominates} so that any height function compatible with $\preceq$ is also compatible 
with $\leq_{LR}^\op$ (as well as with $\leq_L^\op$).
See Proposition \ref{prop2} (and Proposition \ref{prop1}.) Again, the discussion traces back to (\ref{order}).

The {\it right set} of an element $w\in W$ is defined as the set $\sR(x)=\{s\in S\,|\, xs<x\}.$ Then,
if $x\leq_Ly$, it is known  (and elementary) that $\sR(x)\supseteq\sR(y)$ \cite[Lem. 8.6]{Lus03}. For example, if $J\subseteq S$, let $w_{0,J}$ be the longest
word in $W_J$. Clearly, $\sR(w_{0,J})=J$.   Thus, if $\omega\in\Omega$ contains $w_{0,I}$ and $w_{0,J}$, for subsets $I,J$ of $S$,
then, if $w_{0,I}\leq_Lw_{0,J}\leq_Lw_{0,I}$, so $I=J$.

 We have the following elementary result.

\begin{lem}\label{w_0,J}Let $J\subseteq S$, and let $w_{0,J}$ be the longest word in the parabolic subgroup $W_J$.
The following sets $A_J$, $B_J$, $C_J$ are equal:

\begin{itemize}
\item[(1)] $A_J=\{w\in W\,|\, \sR(w)\supseteq J\}.$

\item[(2)] $B_J=\{w=uw_{0,J}\,|\,\ell(w)=\ell(u)+\ell(w_{0,J})\}$.

\item[(3)] $C_J  = \{w\in W\,|\,w\leq_L w_{0,J}\}$.
\end{itemize}

\end{lem}
\begin{proof} Let $w\in A_J$, and write $w=uv$, where $v\in W_J$ and $u$ is a distinguished left $W_J$-coset
representative.  For $s\in J$, it follows that $ws<w$, so $vs<v$. Thus, $v=w_{0,J}$ and $w\in B_J$. Conversely,
if $w=uw_{0,J}\in B_J$, then it follows easily that $u$ is a distinguished left $W_J$-coset representative. So
$J\subseteq \sR(uw_{0,J})$, and $B_J\subseteq A_J$. Thus, $A_J=B_J$. Next, $w\leq_L w_{0,J}$ implies $\sR(w)\supseteq
\sR(w_{0,J})$, so $C_J\subseteq A_J$. Given $w=uw_{0,J}\in B_J$, let $s_{i_1}\cdots s_{i_r}$ be a reduced expression for
$u$. Because $\ell(w)=\ell(u)+\ell(w_{0,J})$, it follows that, if $1< m\leq r$ and $u_m:=s_{i_m}\cdots s_{i_r}$, then
$s_{-_{m-1}}u_mw_{0,J}> u_mw_{0,J}$, so that $s_{i_{m-1}}u_mw_{0,J}\leq_L u_mw_{0,J}$. Thus, $w\leq_Lw_{0,J}$
and $w\in C_J$.\end{proof}

\begin{thm}\label{huh?} Let $\omega\in\Omega$ and assume that, for some $J\subseteq S$,  $w_{0,J}\in\omega$,
where $w_{0,J}$ is the longest  
word in $W_J$. Then $x_J\sH$ has an increasing filtration in the exact category $(\scrA_\flat(\scrL^*), \scrE_\flat(\scrL^*))$ which satisfies Condition \ref{condition}. The  sections of the filtration are direct sums of  dual left-cell modules
$S_{\omega'}$, starting with $S_\omega\subseteq x_J\sH$. The other sections have summands $S_{\omega'}$
satisfying $\omega<_L\omega'$ (equivalently, $\omega' >^\op_L\omega$).  %An example of such a filtration is the height filtration for any height function $\htt$ compatible with $\leq_{LR}^\op$.
\end{thm}

\begin{proof}%The last line follows from the observation above Lemma \ref{w_0,J}. 
We consider the left module $\sH x_J$ and then apply the contravariant $\sZ$-duality $(-)^*$ in Section 2.2. We use
 the construction of the height filtration for $\sH x_J$ in the proof of Proposition \ref{HxJ}. 
The quasi-poset ideal $V$ used in the construction from \cite[Lem. 2.3.5]{DPS98a} is, by inspection of the latter lemma, the 
set $A_J$ in Lemma \ref{w_0,J}, which is also the set $C_J$. In particular, $\omega$ generates $V$ as 
a quasi-poset ideal with preorder $\leq_L$. The construction uses a height function which we will take to be
$\htt_l=-\htt$ for a given height function $\htt=\htt_r$ compatible with $\leq_{LR}^\op$. Thus, $\htt_l=-\htt$ is compatible with $\leq_{LR}$ and, therefore, with $\leq_L$.  Hence, $\omega$ has a unique maximal height among all
left cells in the construction.  The left-cell module $S(\omega)$ is the top section in this height filtration of $\sH x_J$, and the
left-cell modules $S(\omega') $, which must satisfy  $\omega'<_L\omega$, must appear in lower sections. The theorem
now follows by duality.
\end{proof} 

\section{The conjecture}
A main theorem in this monograph is Theorem \ref{4.4}, which, we will see, implies Conjecture \ref{conjecture} below. As noted in the introduction to this chapter, this latter conjecture was stated in \cite[Conj. 1.2]{DPS15}, which revised slightly\footnote{See Remarks \ref{Rem7.4}.} an earlier formulation \cite[Conj. 2.5.2]{DPS98a}.  In assertion (2), the ``strict" stratifying system in \cite{DPS15} is simply called a stratifying system here (with the same meaning).
 
 Recall that the set $\Omega$ of left cells (unequal parameter in the sense of Lusztig \cite{Lus03}) in the Weyl group $W$ has a natural preorder $\leq_{LR}^{\op}$ which partitions
$\Omega$ into its ``two-sided Kazhdan-Lusztig cells." That is, $\omega$ and $\omega'$ belong to the same two-sided cell provided 
$\omega\leq_{LR}^\op\omega'$ and $\omega'\leq_{LR}^\op\omega$. Similarly, any preorder $\leq$ on $\Omega$
also partitions $\Omega$ into ``cells." Let $\widetilde\Omega$ be the corresponding set of cells. We say that $\leq$ is {\it strictly compatible} with
$\leq_{LR}^\op$ provided $\overline\Omega=\widetilde\Omega$. The preorder $\preceq$ in display (\ref{preorder1}) is clearly
an example of such a preorder.

 Recall also that given a left cell $\omega\in\Omega$, $S_\omega$ denotes the corresponding dual left-cell module $S(\omega)^*=\Hom_\sZ(S(\omega),\sZ)$; see Notation \ref{notation1}. For the definition of the element $x_J$, see \eqref{xJ}.
% See also see \eqref{xJ} for the definition of the element $x_J$.
 
%For the notation $S_\omega$, see  Notation \ref{notation1}. For the definition of the element $x_J$, see \eqref{xJ}.

 \begin{conj}\label{conjecture} {\it There exist a preorder $\leq$ on the set $\Omega$ of left cells in $W$, strictly compatible with its partition $\overline \Omega$ into two-sided cells, and a right $\sH$-module $X$ such that the following statements hold, where $T=\bigoplus_{J\subseteq S}x_J\sH$:

\begin{enumerate}
\item[(1)] $X$ has a finite filtration with sections of the form $S_\omega$, $\omega\in\Omega$.

\item[(2)] Let $T^+:=\!T\oplus X$, $A^+:=\End_{\sH}(T^+)$ and, for any $\omega\in\Omega$, put 
$$\begin{cases} 
\Delta(\omega):=\Hom_{\sH}(S_\omega,T^+),\\
P(\omega):=\Hom_{\sH}(T^+_\omega,T^+),\;\text{as in \eqref{delta}.}\\
\end{cases}$$
\end{enumerate}
Then, for any commutative, Noetherian\! $\sZ$-algebra\!  $R$, the set $\{\Delta(\omega)_R,P(\omega)_R\}_{\omega\in\Omega}$ is a stratifying system for $A^+_R$-mod relative to the quasi-poset\!
$(\Omega,\leq)$. 
}
\end{conj}

In \cite[p. 231]{DPS15}, in the notion  of a stratifying system for $A^+$-mod, it is assumed, for each $\Delta(\omega)$, that there 
is a certain projective $A^+$-module $P(\omega)$ satisfying certain properties.       In the notation of the present monograph, reviewed in Section 1.3, the projective modules $P(\omega)$ are explicitly mentioned as part of the data for a stratifying system.
 The $P(\omega)$'s are required to  satisfy the same properties. Finally, as already noted above Remark \ref{1.3.2}, the notion of a stratifying system behaves well with respect to base change, so it suffices to prove the conjecture for $R=\sZ$ and, thus, for
 $A^+_R=A^+$. 

\section{Proof of the conjecture}
Before proceeding, we state and prove our main result, Theorem \ref{4.4}.  The preorder $\preceq$ in the statement of the theorem is defined explicitly in display (\ref{preorder1}).  It depends upon a suitable height function $\htt$.
Also, the exact category $(\scrA_\flat(\scrL^*),\scrE_\flat(\scrL^*))$ depends on the same height function, as discussed, as does Construction \ref{const}.

\begin{thm} \label{4.4} Let $(\scrA_\flat,\scrE_\flat)=(\scrA_\flat(\scrL^*), \scrE_\flat(\scrL^*))$, $T, X, T^+$,  and  $\htt=\htt_r$ be as
in Construction \ref{const}.  For $\omega\in\Omega$, let 
$$T^+_\omega=\begin{cases} X_\omega, &\text{if }\omega\in\Omega',\\ x_J\sH,&\text{if } w_{0,J}\in\omega.\end{cases}$$ 
%\begin{equation}\label{zoo1}
%\begin{cases}T:=\bigoplus_{J\subset S}x_J\sH, \\
%X:=\bigoplus_{\omega\in\Omega' }X_\omega,\\ 
%T^+:=T\oplus X \end{cases}
%\end{equation}
%be the right $\sH$-modules defined in Construction \ref{const}.
(Recall that $X_\omega$
is constructed from $S_\omega$, possibly non-uniquely, as in Theorem \ref{relative inj}. Here $\Omega'$ is defined in Construction \ref{const}(a).)

Form the  endomorphism algebra

\begin{equation} A^+:=\End_{\sH}(T^+).\end{equation}
Similarly, for $\omega
\in\Omega$,  $\Delta(\omega):=\Hom_{\sH}(S_{\omega},T^+)$ and $P(\omega):=
\Hom_{\sH}(T_\omega^+,T^+)$.  

Then $\{\Delta(\omega), P(\omega)\}_{\omega\in\Omega}$ is a stratifying system for the category $A^+$-\text{mod} with respect to the quasi-poset $(\Omega,\preceq)$. 
\end{thm}

\begin{proof} By Remarks \ref{2.5}(b)\&(c), the left $A^+$-modules $\Delta(\omega)$, $\omega\in\Omega$,  are automatically $\sZ$-projective, even $\sZ$-free.   If $\omega\in\Omega'$, $X_\omega$ is, by construction, a free $\sZ$-module. Also, for $\omega\in\Omega\backslash\Omega'$,
$T^+_\omega\cong x_J\sH$, for some $J\subseteq S$.  By Theorem \ref{huh?}, $T^+_\omega$ has a filtration by
dual left-cell modules $S_\tau$, $\tau\in\Omega$. (This follows already from (\ref{coset}).)
Each $S_\tau$ is, by construction, $\sZ$-free, so $T^+_\omega$ is
also $\sZ$-free. 

Now we must verify conditions (1), (2), and (3) in the Stratification Hypothesis \ref{hypothesis}.
Given $\omega\in\Omega'$, Condition (1) follows from Theorem \ref{relative inj}.  If $\omega\in \Omega\backslash\Omega'$, then Condition (1) follows from Theorem \ref{huh?}.  

Suppose that $\omega,\tau\in\Omega$, and $\Hom_{\scrA_\flat(\scrL^*)}(S_\omega, T^+_{\tau})\not=0$. 
Then, setting $K=\mathbb Q(t)$, $\Hom_{\sH_K}(S_{\omega,K},S_{\nu,K})\not=0$, so for some $\tau\in\Omega$, $\tau\preceq
\nu$. Since $S_{\omega,K}$ and $S_{\nu,K}$ share a composition factor, they belong to the same (dual) two-sided cell
module. Thus, $\omega$ and $\nu$ belong to the same two-sided cell for the preorder $\leq_{LR}$, or, equivalently
for $\leq_{LR}^{\op}$. Since $\preceq$ is strictly compatible with $\leq_{LR}^\op$, $\omega$ and $\nu$ belong to the same two-sided cell for $\preceq$. It follows that $\tau\preceq\omega$.
This proves (2).  

Next, we prove condition (3). In case $\omega\in\Omega'$, we can apply the 4-term exact sequence
in \cite[Prop. A.2]{DPS17} and the vanishing of $\Ext^1_{\scrE_\flat(\scrL^*)}(S,X)$ in Theorem \ref{relative inj}
to conclude that condition (3) holds in this case. If $\omega\in\Omega\backslash\Omega'$, we can argue in the same 
way, but using Theorem \ref{huh?} and Proposition \ref{keyprop}.\end{proof}

The following is a corollary of the proof of the theorem, and  Construction \ref{const}, using Theorem \ref{relative inj},  Theorem \ref{huh?}, Proposition
\ref{keyprop}, displays (\ref{coset}), and (\ref{preorder1}). The last assertion of the corollary follows from Theorem \ref{relative inj} ($\omega\in\Omega'$) and the last assertion of Theorem
\ref{huh?} (when $\omega\not\in\Omega'$).  

\begin{cor} \label{qperm}
Let $(\scrA_\flat,\scrE_\flat)=(\scrA_\flat(\scrL^*), \scrE_\flat(\scrL^*))$ and keep the notation of the theorem above.  Then, $S_\omega$ and
$T^+_{\tau}$ belong to $\scrA_\flat$, for all $\omega, \tau\in \Omega$. Also, 
$$\Ext^1_{\scrE_\flat}(S_\omega,T^+_{\tau})=0.$$
In addition, $S_\omega$ is a submodule (via an inflation in $\scrE_\flat$) of $T^+_\omega$ and is the lowest nonzero term in its height filtration. \end{cor}

If $\Theta$ is a poset ideal of $(\Omega,\leq_L)$, then $(\Theta,\preceq)$ is a quasi-poset co-ideal of $(\Omega,\preceq)$. Thus,
$T^\ddag_\Theta:=\bigoplus_{\omega\in\Theta}(T^+_\omega)^{m_\omega}\in\scrA_\flat$ with all $m_\omega>0$, as defined in \eqref{Tdagger}, satisfies Stratification Hypothesis \ref{hypothesis} relative to the exact category $(\scrA_\flat,\scrE_\flat)$. 
Theorems \ref{thm2.5} and \ref{4.4} immediately imply the following.

\begin{cor}\label{Adagger1}
Maintain the notation introduced above and form the endomorphism algebra $A^\ddag_\Theta=\End_{\sH}(T^\ddag_\Theta)$ and $A^\ddag_\Theta$-modules $\Delta^+(\omega)=\Hom_{\sH}(S_\omega, T^\ddag_\Theta)$ and
$P^+(\omega)=\Hom_{\sH}(T^+_\omega, T^\ddag_\Theta)$. Then, $\{\Delta^+(\omega), P^+(\omega)\}_{\omega\in\Theta}$ is a stratifying system for the category $A^\ddag_\Theta$-\text{mod} with respect to the quasi-poset $(\Theta,\preceq)$. 
\end{cor}
%Note that, for $\Theta=\Omega$, $T^\ddag_\Theta=T^\ddag$ and $A^\ddag_\Theta=A^\ddag$ are defined in displays \eqref{Tdagger} and \eqref{Adagger}.

%We can now complete the proof of the conjecture.

\medskip\noindent

{\bf We now complete the proof of Conjecture \ref{conjecture}.}
The above theorem proves
 part (1) of Conjecture \ref{conjecture}.  Part (1) of the conjecture deals with the filtration of $X$ by dual left-cell modules and follows easily from the construction given above. Our construction here is, in fact, much more informative than that envisioned by 
 part (1) of the conjecture. For example, $X_\omega$ is built as an object in $\scrA_\flat(\scrL^*)$.
As such, it has a height filtration with sections that are direct sums of dual left-cell modules. Also,  left cells associated to distinct sections always lie in distinct two-sided cells.

 Part (2) of the conjecture
is phrased in the language of base ring extensions of $\sZ$, but it is sufficient to check it over $\sZ$, by \cite[Lem. 1.2.5(4)]{DPS98a}. (Stratifying systems behave well under base change.) Now taking $R=\sZ$, part (2) of the conjecture follows from Theorem \ref{4.4}, provided we know the preorder $\preceq$ has the same equivalence classes
as $\leq_{LR}^\op$. The actual preorder $\preceq$ in Theorem \ref{4.4} has this property (and is constructed from
$\leq_{LR}^\op$, sharing the height function $\htt$ with it; see display (\ref{preorder1})).

This completes the proof. \qed

%\medskip
%{\color{blue} Let $T^\ddagger$ be the $\sH$-module defined by allowing multiplicities of $T_\omega^+$'s (see \eqref{Tdagger}) and consider the Hecke endomorphism algebra $A^\ddagger=\End_\sH(T^\ddagger)$.
%\begin{cor}The algebra $A^\ddagger$ is Morita equivalent to $A^+$. Moreover, for $\omega
%\in\Omega$,  let $\Delta^+(\omega):=\Hom_{\sH}(S_{\omega},T^\ddagger)$ and $P^+(\omega):=
%\Hom_{\sH}(T_\omega^\ddagger,T^\ddagger)$.  
%Then $\{\Delta^+(\omega), P^+(\omega)\}_{\omega\in\Omega}$ is a stratifying system for the category $A^\ddagger$-mod with respect to the quasi-poset $(\Omega,\preceq)$.
%\end{cor}
%\begin{proof}
%Since the Stratification Hypothesis \ref{hypothesis} is stated for an $\sH$-module $T$ that is a direct sum of $T_\lambda^{\oplus m_\lambda}$ with $m_\lambda>0$ (see \eqref{modules}), the assertion follows from Theorem \ref{thm2.5} and the proof of Theorem \ref{4.4}.
%\end{proof}}

\medskip

The following remarks discuss further the difference between the Conjecture proved in Theorem \ref{4.4} and its
 original version given in \cite{DPS98a}. It can be skipped by readers interested only in the more recent version
 as stated in Theorem \ref{4.4} or in \cite[Conj. 1.2]{DPS15}.  

\medskip\begin{rems}\label{Rem7.4}

(a) When \cite[Conj. 2.5.2]{DPS98a} was written and checked for all rank 2 cases,\footnote{\label{error}We correct some typos around the third display on \cite[p. 203]{DPS98a}: \dots Each module ${\widetilde S}_{\{u\}}^{\mathscr R}=S_{0'}\oplus S_{0''}$ and ${\widetilde S}_{\{v\}}^{\mathscr R}=S_{0}\oplus S_{0''}$ has a length two filtration with two dual left-cell modules ${\widetilde S}_{\omega k}$ as sections, corresponding to left cells $\omega=B, C_1, C_2$, and $D$ (in the notation of (3.5) below). The filtrations are of the form:
$$ {\widetilde S}_{\{u\}}^{\mathscr R}=\boxed{\begin{matrix}C_1\\D\end{matrix}}\;\text{ and }\;
 {\widetilde S}_{\{v\}}^{\mathscr R}=\boxed{\begin{matrix}B\\C_2\end{matrix}},
$$ and \dots}
 the best choice of a preorder $\leq$ seemed to be
$\leq_{LR}^\op$. In the revised conjecture \cite[Conj.1.2]{DPS15} (identical to Conjecture \ref{conjecture} above), we realized that this choice of
a preorder was too strong, and so it was weakened to require only a  preorder with the same equivalence classes as those of $\leq_{LR}^\op$. As
discussed in (a), the preorder $\preceq$ here in Theorem \ref{4.4} satisfies this condition and more,  sharing also a height function $\htt$ with $\leq_{LR}^\op$.  Moreover,  $\preceq$ is a preorder completely determined by its equivalence classes and the height function $\htt$, a specificity 
giving it some advantage over $\leq^\op_{LR}$. For example, suppose $\htt$ is given by Lusztig's $a$-function (see \eqref{afun} below),\index{$a$-function} a choice of height function we can certainly make, provided conjectures P4 and P11 hold; see \cite[13.6, 14.1]{Lus03}. If $\omega,\omega'$ belong to distinct 2-sided cells, to decide if $\omega
\preceq\omega'$, (equivalent here to $\omega\prec\omega'$),
 we need only check that $a(\omega)<a(\omega')$---which is a much more computable condition 
than $\omega\leq^\op_{LR}\omega'$ or $\omega<_{LR}^\op\omega'$.

(b) The other main difference between the conjecture in \cite{DPS98a} and its version in  \cite{DPS15} is that
we generally worked over $\mathbb Z[t^2, t^{-2}]$ in \cite{DPS98a} rather than $\sZ$. A case for using $\sZ$ rather than $\mathbb Z[t^2, t^{-2}]$ is that Lusztig's explicit isomorphism $\sH_{\mathbb Q(t)}\cong \mathbb Q(t)W$ holds without exception (for the standard finite Coxeter system, even when $G$ has type $^2F_4$). These Hecke algebras are all semisimple and, in fact, split semisimple except for the case $^2F_4$. We mention that the explicit isomorphisms can be obtained by applying \cite[18.12]{Lus03} to
the cases $R=\mathbb Q$ and $R=\mathbb Q(t)$. This works for any finite $W$ satisfying P1--P15, and semisimplicity follows. The splitting, except for the $^2F_4$-case, follows from the isomorphisms $\sH_{\mathbb Q(t)}\cong \mathbb Q(t)W$ above, and the fact that $\mathbb Q(t)W$ (in fact, $\mathbb QW$) is split semisimple; see \cite{B71}, \cite{BC71}. The latter reference discusses the Weyl groups arising in the ``split" (Chevalley group) cases,  also including those $W$ for the ``Steinberg twisted'' cases. The remaining groups, associated to the Ree and Suzuki groups, are finite dihedral groups, and the only new group $W$  arising, that has not already been considered, is
the dihedral group $I_2(8)$ of order 16 associated to $^2F_4$. It has $\mathbb Q[\sqrt{2}]$ as its unique minimal splitting field.

(c) \cite[Conj. 2.5.2]{DPS98a} left an ambiguity, removed in \cite[Conj. 1.2]{DPS15}. In detail, the second sentence of part (2) 
of \cite[Conj. 2.5.2]{DPS98a} contains an expression ``$\widetilde A_{\sZ'}= \End_{\widetilde H}({T}^+_{\sZ'})$" which
is a misprint, and should have been written ``$\wA_{\sZ'}=\End_{\widetilde H}(T^+)_{\sZ'}$.'' However, it is shown in 
\cite[Th. 2.4.4(b)]{DPS98a} that ``$\End_{\sH}(T)_{\sZ'}\cong\End_{\sH}(T_{\sZ'})$," not using the ``$+$" as a superscript on $T$. 
 Possibly the same isomorphism holds using $+$, but we only know this to be true if $\sZ'$ is flat over $\sZ$. 
 
 (d) Finally, we observe that while the methods of \cite{DPS98a} are useful in studying low rank examples of the
 Conjecture, they do not provide much help in the ultimate proof of the Conjecture!
\end{rems}

\section{A uniqueness theory} %This section treats some quite general uniqueness results. Our motivation is two fold:
%\smallskip\noindent (1) Although the uniqueness theory of quasi-hereditary algebras is fairly well developed---e.g., by Rouquier's ``quasi-hereditary cover'' theory
%\cite{Ro08}---there is no corresponding theory for stratified algebras. 
%\smallskip
%noindent
%(2) There has been no uniqueness tuned to constructions arising in an exact category context. 
%\smallskip
%Theorem \ref{11.2} (and its corollary) address both of these issues. In the interest of generality, the
% notation below is independent of previous sections, but has been chosen to be similar.
The main aim of this section is to provide a companion uniqueness theorem for (the Morita classes of) the algebras $A^+$ resulting from Construction \ref{const}.  
If $A^+$ is quasi-hereditary, at least one such uniqueness result is available in Rouquier's theory of quasi-hereditary covers; see \cite[\S4]{Ro08}. 

However, $A^+$ may not be quasi-hereditary. 
The method we use instead applies in a quite general, but natural, context  of endomorphism algebras of suitably filtered injective--like objects  in an exact category (e.g., the object $T^+$ in Theorem \ref{4.4}).  We begin with a description of this general context (with a notation that is independent of previous sections, though chosen to be suggestive).

Let $\scrK$ be a commutative Noetherian ring, and let $H$ be a finite and projective
$\scrK$-algebra.  We view $H$ and its modules as described by $\scrK$-module structure.
Let $(\scrA,\scrE)$ be a strict and full exact subcategory of mod--$H$. Recall this means that $\scrA$ is strict (closed under
isomorphisms) and full in mod--$H$ and that all conflations in $\scrE$ are also short  exact sequences in mod$-H$. Given $M\in\scrA$, a filtration 
$0=M^{-1}\subseteq M^0\subseteq\cdots\subseteq M^\ell=M$ in mod--$H$ is called an $\scrE$-filtration provided 
each short exact sequence $0\to M^{i-1}\to M^i\to M^i/M^{i-1}\to 0$ ($0\leq i \leq \ell$) in mod-$H$ is a conflation in $\scrE$. (In particular, each $M^i$ belongs to $\scrA$.)

\begin{lem}\label{11.1} Suppose $M\in\scrA$ has an $\scrE$-filtration, indexed   as  above, and that  $T$ is an object in $\scrA$ such that $\Ext^1_\scrE(N,T)=0$ for each section $N=M^i/M^{i-1}$, $0\leq i\leq\ell$. Then, 
$$0\to \Hom_H(M^j/M^i,T)\to\Hom_H(M^j,T)\to\Hom_H(M^i,T)\to 0$$
 is a short exact sequence, for all $0\leq i<j\leq \ell$.  In particular,  the map
$$\Hom_H(M,T)\to\Hom_H(M^0,T)$$
is surjective. \end{lem}

For the proof, see the Claim in the proof of Theorem \ref{thm2.5}.
\medskip

Now let $\scrS$ be a (strict) full additive subcategory of $\scrA$ generated by a finite set $\{S_\omega\}_{\omega\in\Omega}$ of
nonzero objects indexed by a finite set $\Omega$. That is, every object in $\scrS$ is a finite direct sum of objects
$S_\omega$, $\omega\in\Omega$. Define $\scrA_{\scrE}(\scrS)$ to be the subcategory of objects $M$ in $\scrA$ which have
an $\scrE$-filtration with every section $M^i/M^{i-1}$ in $\scrS$.

Note that, by Corollary \ref{qperm}, the theorem below applies in particular to the algebra $A^+$ in the conclusion of the main Theorem \ref{4.4}, using $A^+$ for
$A^\dagger$.

\begin{thm}\label{11.2}Assume that, for each $\omega\in\Omega$, there exists an object $X_\omega\in\scrA_\scrE(\scrS)$
such that

\medskip
\begin{itemize} \item [(1)] $X_\omega\in\scrA_\scrE(\scrS)$ has an $\scrE$-filtration with lowest section $S_\omega$ (and all other sections in $\scrS$), and

\item [(2)] $\Ext^1_{\scrE}(S_{\omega'}, X_\omega)=0$, for all $\omega,\omega'\in\Omega$. \end{itemize}

\medskip If $A^\dagger=\End_H(T^\dagger)$ where $T^\dagger=\bigoplus_{\omega\in\Omega}X_\omega$, 
then the Morita class of $A^\dagger$ is independent of the choices of the $X_\omega$ satisfying (1) and (2).
\end{thm}

As a corollary, we have the following result, whose statement will also be used in the proof of the theorem.

\begin{cor}\label{11.3} Suppose that, for each  $\omega\in\Omega$, there are given objects $X_\omega, Y_\omega\in\scrA_\scrE(\scrS)$ such that conditions (1) and (2) in Theorem \ref{11.2} hold for each
 $X_\omega$ and maintain the
notations $T^\dagger$
and $A^\dagger$. Assume also that each $Y_\omega$ satisfies condition (2) in the sense that $\Ext^1_{\scrE}(S_{\omega'}, Y_\omega)=0$, for all $\omega^\prime\in\Omega$.

\smallskip 
Put 
$T^{\dagger\prime}=\bigoplus_{\omega\in\Omega}(X_\omega\oplus Y_\omega)$ and $A^{\dagger\prime}:=\End_H(T^{\dagger\prime})$.
Then, $A^\dagger$ and $A^{\dagger\prime}$ are Morita equivalent.
\end{cor}
Note that, as a special case of this corollary, $A^\ddag:=A^\ddag_\Omega$ in \eqref{Adagger} is Morita equivalent to $A^+$.
\begin{proof}
It is clear that the theorem implies the corollary by using $X_\omega\oplus Y_\omega$ for $X_\omega$
satisfying (1) and (2). The filtration of the direct sum starts with that of $X_\omega$, then uses that of
$Y_\omega$.
However, it is also true that the corollary implies the theorem. In fact,  if $Y_\omega$ also
satisfies condition (1), then the roles of $X_\omega$ and $Y_\omega$ can be interchanged in the corollary, thus
giving Morita equivalences (designated here by $\sim$)
\begin{equation}\label{equivalences}
\left\{
\aligned
&\End_H(\bigoplus_\omega Y_\omega)\sim \End_H(\bigoplus_\omega(Y_\omega\oplus X_\omega))\\
\cong& \End_H(\bigoplus_\omega(X_\omega\oplus Y_\omega))\sim\End_H(\bigoplus_\omega X_\omega).
\endaligned\right.
\end{equation}
Thus, the corollary and the theorem are equivalent, so it remains to prove the corollary directly. We follow the
strategy in \cite[proof of Th. 4.1]{DPS18}. 

Let
$e\in A^{\dagger\prime}=\End_H(T^{\dagger\prime})$ be the idempotent projection of $T^{\dagger\prime}$ onto $T^\dagger\subseteq
T^{\dagger\prime}$ along $\bigoplus_\omega Y_\omega$. Since $A^\dagger\cong eA^{\dagger\prime}e$, it suffices to show that the projective
module $A^{\dagger\prime}e$ is a progenerator, or, equivalently, that every irreducible $A^{{+\prime}}$-module $L$ is a homomorphic image of $A^{\dagger\prime}e$. (A better-known criterion is that $A^{\dagger\prime}eA^{\dagger\prime}=
A^{\dagger\prime}$. However, if this fails, then the two-sided ideal generated by $e$ is contained in
a maximal left ideal $\mathfrak m$ of $A^{\dagger\prime}$. Take
$L=A^{\dagger\prime}/\mathfrak m$, and note that $eL=0$, so that $L$ cannot be a homomorphic image of $A^{\dagger\prime}e$.)

However, $A^{\dagger\prime}$ is filtered by modules
$\Hom_H(S_\omega,T^{\dagger\prime})$ by  Lemma \ref{11.1}, and each of these is a homomorphic image of $\Hom_H(T^\dagger, T^{\dagger\prime})$ again by Lemma \ref{11.1}. The corollary now  follows from the identification
$A^{\dagger\prime}e\cong \Hom_H(T^\dagger, T^{\dagger\prime})$ of $A^{\dagger\prime}$-modules. This completes the proof of the corollary and, as noted, the theorem. 
\end{proof}

\begin{rem}
The proof of Corollary \ref{11.3} can be used to give an explicit description of a Morita ($\scrK$-category) equivalence 
$A^\dagger\text{-mod} \overset\sim\to A^{\dagger\prime}\text{-mod}$, where $A^\dagger$ and $A^{\dagger\prime}$ are as stated in the corollary. First, identify $A^\dagger$-mod with $eA^{\dagger\prime}e$-mod as above. Then, use the well-known equivalence $$A^{\dagger\prime}\otimes_{eA^{\dagger\prime}e}(-):eA^{\dagger\prime}e\text{-mod}\longrightarrow A^{\dagger\prime}\text{-mod}.$$
\end{rem}

\begin{thm} \label{11.5} Let $A^+$ be a $\sZ$-algebra as constructed in the hypothesis of
Theorem \ref{4.4} and let $A^{+\prime}$ be a $\sZ$-algebra also constructed to satisfy this hypothesis. (Here we use the same
height function for $A^+$ and $A^{+\prime},$ but allow the $X_\omega$ to vary.)  Then $A^+$ is Morita equivalent to $A^{+\prime}$.\end{thm}

The proof is immediate from Corollary \ref{qperm} using Theorem \ref{11.2}. Further details on the Morita
equivalence can be obtained from the proof of the latter theorem and its corollary.

\begin{rems} We now mention several useful applications of Theorem \ref{11.2} and its corollary:

\medskip
(a) The identification $A^{\dagger\prime}e\cong\text{Hom}_H(T^\dagger,T^{\dagger\prime})$ in the proof of the corollary, together with the corresponding identification for $eA^{\dagger\prime}$ and the equality $A^{\dagger\prime}eA^{\dagger\prime}=A^{\dagger\prime}=\text{End}_H(T^{\dagger\prime})$, shows:
\begin{itemize}
\item[(i)] the trace of $T^\dagger$ is all of $T^{\dagger\prime}$. In particular,
\item[(ii)] any $H$-projective direct summand of $T^{\dagger\prime}$  is isomorphic to a direct summand of $(T^{\dagger})^{\oplus m}$, for a sufficiently large $m\in\mathbb N$.
\end{itemize}

(b) Part (ii) of (a) may be used to prove the claim made in \cite[Remark 4.10(a)]{DPS17} that $H=\mathcal H$ is a direct summand of $T^\dagger$ there in the conclusion of \cite[Th. 4.9]{DPS17}. (In the context of \cite[Remark 4.10(a)]{DPS17}, we are allowed to replace $T^\dagger$ here by $(T^\dagger)^{\oplus m}$.) In fact, it can also be proved that $T^\dagger$ on \cite[p. 245]{DPS17}, taking the multiplicities $m_\omega$ all equal to 1, always has the regular representation $H=\mathcal H$ as a direct summand. 

(c) Other Morita equivalences in the spirit of Corollary 
\ref{11.3} are developed in \cite[\S4]{DPS18}.  Conversely, Theorem \ref{11.2} can be used to prove \cite[Th. 4.3]{DPS18}. The theorem gives a Morita equivalence between localizations of stratified algebras in \cite{DPS17} with algebras in \cite{DPS15} which base change to algebras arising in a Cherednik category $\sO$ context (with equal parameters). 
\end{rems}

\section{The standard stratification of $A^+$}
We now use the stratifying system $\{\Delta(\omega), P(\omega)\}_{\omega\in\Omega}$ in the category $A^+$-mod of finite $A^+$-modules given in Theorem \ref{4.4} to construct a sequence of stratifying ideals in $A^+$ and to show that $A^+$ is a standardly stratified algebra.

Let $\scrK$ be a commutative Noetherian ring and let $B$ be a $\scrK$-algebra which is finite and projective as a $\scrK$-module.
\begin{defn}[{\cite[Defs.~3.1\&3.2]{DPS18}}]\label{SSA}  An ideal $J$ in $B$ is called a {\it standard stratifying ideal} if the following conditions hold:
\begin{itemize}
\item[(HI1)] $B/J$ is projective over $\scrK$;
\item[(HI2)]  $_BJ$ is $B$-projective as a left $B$-module;
\item[(HI3)] $J^2=J$.
\end{itemize}
 The algebra $B$ over $\scrK$ is called a {\it standardly stratified algebra} (SSA) provided there exists
a finite ``defining sequence'' of ideals in $B$:
\begin{equation}\label{SStratification}0=J_0\subseteq J_1\subseteq J_2\subseteq\cdots\subseteq J_m=B
\end{equation} 
such that, for $0<i\leq m$, $J_i/J_{i-1}$ is a standard stratifying ideal in $B/J_{i-1}$. We also call  \eqref{SStratification} a {\it standard stratification}.% or {\it stratifying chain}.
\end{defn}

We will see in Section 4.3 that stratifying ideals and standardly stratified algebras are weakened (more general) versions of heredity ideals and quasi-hereditary algebras.

We first observe the following nice properties.

\begin{prop}(1) Every ideal in the standard stratification \eqref{SStratification} is an idempotent ideal.

(2) Suppose that $J$ is a stratifying ideal in a $\scrK$-algebra $B$. If $M,N$ are $B/J$-modules which are also regarded as $B$-modules by inflation, then, for any integer $n\geq 0$, there is
a $\scrK$-module isomorphism
\begin{equation*}\label{coho}\Ext^n_{B/J}(M,N)\longrightarrow\Ext^n_B(M,N).\end{equation*}
\end{prop}
\begin{proof} A proof for (2) can be found in \cite[Appendix B]{DPS17}. We now prove (1).
 Since $J_1^2=J_1$ and $(J_2/J_1)^2=J_2/J_1$, it follows that 
$$J_2^2=J_2^2+J_1^2=J_2^2+J_1=J_2+J_1=J_2.$$
Hence, $J_2$ is an idempotent ideal. An induction argument proves the first assertion. 
\end{proof}
We now prove that the Hecke endomorphism $\sZ$-algebra $A^+$ is a standardly stratified algebra. There is a general proof of such a fact for any $\scrK$-algebra $B$ which is finite projective over $\scrK$, provided that  the category $B$-mod has a stratifying system; see \cite[\S3]{DPS18}. The proof for $A^+$ below is independent and more direct, laying down a foundation for a further proof of the quasi-heredity of a base-changed $A^{+\natural}$ in the next chapter. 

\begin{thm}\label{SSAA+}
If $A^+$ is the (extended) generic Hecke endo-algebra associated with a standard finite Coxeter system $(W,S,L)$ as defined in \eqref{endo},
then the  $\sZ$-algebra $A^+$ is a standardly stratified algebra.
\end{thm} 

\begin{proof}We maintain the notation of Theorem \ref{4.4}, so, in particular,
$A^+=\End_\sH(T^+)$ is a certain endomorphism algebra over $\sZ$, where $T^+=\bigoplus_{\omega\in\Omega}T^+_\omega$, and the category $A^+$-mod has a stratifying system $\{\Delta(\omega),
P(\omega)\}_{\omega\in\Omega}$.  % with $\Delta(\omega)=\Hom_\sH(S_\omega,T^+)$ and $P(\omega)=\Hom_\sH(T^+_\omega,T^+)$.
  {
Recall also from Theorem \ref{4.4} that, for each $\omega\in\Omega$, the left $A^+$-module{\color{black}s $\Delta(\omega)$ and} $P(\omega)$ are defined as {\color{black}$\Hom_\sH(S_\omega,T^+)$ and} $\Hom_\sH(T_\omega^+,T^+)${\color{black}, respectively,} with the action of $A^+:=\Hom_\sH(T^+,T^+)$ through compositions of morphisms (between right $\sH$-modules). 

Now, put $P:=\bigoplus_{\omega\in\Omega}P(\omega)$ and observe that $P\cong A^+$ as left $A^+$-modules. (Use the equality $\bigoplus_{\omega\in\Omega}T^+_\omega=T^+$ of right $\sH$-modules.)

Let $\htt:\Omega\to\mathbb Z$ and $\preceq$ be as discussed above Proposition \ref{208}, and used in Theorem \ref{4.4}.  Here $\Omega$ is the set of (Kazhdan--Lusztig) left cells of $W$, and $(\Omega,\preceq)$ is the quasi-poset used in Theorem \ref{4.4}. For more details on our set-up here, see the paragraph above the statement of Theorem \ref{4.4} and the internal references it provides.

Our immediate strategy is to construct a tractable ``defining sequence''  of ideals in $A^+$ (see Def. \ref{SSA})
\begin{equation}\label{Jseq}
0=J_0\subseteq J_1\subseteq\cdots\subseteq J_m=A^+
\end{equation}
such that all the conditions (HI1)--(HI3) in Definition \ref{SSA}, for every $J=J_i/J_{i-1}$, hold as an ideal of $B=A^+/J_{i-1}$.
%, except for condition (HI4), to say that the defining sequence provides $A^+$ with a structure of a {\it standardly stratified algebra}. \index{standardly stratified algebra} See op. cit. We will return this later in the proof. 
%Our approach will be to construct a fairly explicit defining sequence as above for $A^+$, then base change to $A^{+\natural}$ and re-examine property (HI4), using the based-changed defining sequence.

We construct the defining sequence for $A^+$ as follows. First, for any $\omega\in\Omega$ and $j\in \mathbb Z$, we define a submodule $P(\omega)_j\subseteq P(\omega)$ as the sum of all images in $P(\omega)$ of all maps in $\Hom_{A^+}(P(\mu),P(\omega))$ with $\htt(\mu)\geq j$. Using the ring-theoretic language of traces, we may write this as \index{trace of a module}
$$P(\omega)_j=\sum_{\htt(\mu)\geq j}\text{trace}_{P(\omega)}(P(\mu))=\text{trace}_{P(\omega)}\Big(\bigoplus_{\htt(\mu)\geq j}P(\mu)\Big)
:=\sum_{f\in\Hom_{A^+}(X,P(\omega))}\im(f),$$
where $X=\bigoplus_{\htt(\mu)\geq j}P(\mu)$.
Observe that $P(\omega)_{j+1}\subseteq P(\omega)_j$. It is proved\footnote{The method of proof reorders the sections $\Delta(\nu)$ in any given filtration of $P(\omega)$ that has sections given by various $\Delta(\nu)$'s with $\nu\in\Omega$. The multiplicities of any given $\Delta(\nu)$ does change in this process.} in \cite[Prop.~2.2]{DPS17} and \cite[Prop.~3.9]{DPS18} that $P(\omega)_j/P(\omega)_{j+1}$ is a direct sum of modules $\Delta(\nu)$ with $\nu\in\Omega$ and $\htt(\nu)=j$. Multiplicities are allowed. That is, isomorphic copies of the same $\Delta(\nu)$ may appear in the direct sum.

Next, we may repeat the above discussion of $P(\omega)_j$ with $P$ replacing $P(\omega)$. That is, we define
\begin{equation}\label{P_j}
P_j=\text{trace}_{P}\Big(\bigoplus_{\htt(\mu)\geq j}P(\mu)\Big).
\end{equation}
Clearly, $P_j=\bigoplus_{\omega\in\Omega}P(\omega)_j$, since $P=\bigoplus_{\omega\in\Omega}P(\omega)$ as a left $A^+$-module. From this we obtain $P_{j+1}\subseteq P_j$ and deduce also that $P_j/P_{j+1}$ is a direct sum of $A^+$-modules isomorphic to $\Delta(\nu)$, for some $\nu\in\Omega$, with $\htt(\nu)=j$. If the equality holds for a given $\nu$, then some $\Delta(\nu)$ does occur.

We are now ready to return to $A^+$ and its ideals. We have observed near the beginning of this proof that $P\cong A^+$ as left $A^+$-modules. For each $j\in\mathbb Z$, let $A_j^+$ denote the left ideal in $A^+$ identifying with $P_j\subseteq P$.
\smallskip
\begin{itemize}
%\noindent
\item[{\bf Claim.}] Each $A_j^+$ is a two-sided ideal of $A^+$. It is idempotent (allowing zero), and $A^+/A^+_j$ is free as a $\mathcal Z$-module.
\end{itemize}
\smallskip
To prove the claim, we first show that each $A^+_j$ is a two-sided ideal in $A^+$. By definition, it is the left ideal of $A^+$ corresponding to the left $A^+$-submodule $P_j$ of $P$ under the natural identification $P\cong A^+$ of left $A^+$-modules discussed above. Similarly, $A^+/A^+_j$ is isomorphic to $P/P_j$ as left $A^+$-modules. As such, it is filtered by copies of various modules $\Delta(\nu)$ with $\htt(\nu)< j$, and $A^+_j$ is similarly filtered by modules $\Delta(\tau)$ with $\htt(\tau)\geq j$. (In both cases, ``filtered by'' refers informally to a filtration by left $A^+$-modules, with the given $\Delta$'s as the only nonzero sections.)

If $A^+_j$ were not a two-sided ideal in $A^+$, it would not be a right ideal, and there would be an element $a\in A^+$ with $A^+_ja\not\subseteq A^+_j$. This implies that the composition
$$A^+_j\longrightarrow A^+_j a\longrightarrow (A^+_ja+A^+_j)/A^+_j\subseteq A^+/A^+_j$$
of left $A^+$-module homomorphisms would be non-zero. Combining this with the above paragraph, we find from this that there is a nonzero map $\Delta(\tau) \to\Delta(\nu)$ with $\htt(\tau)>\htt(\nu)$. However, this is impossible. It leads to a nonzero composite $P(\tau)\twoheadrightarrow \Delta(\tau) \to\Delta(\nu)$, implying $\tau\preceq\nu$, by the axiom (SS2) for stratifying systems. But  $\tau\preceq\nu$ and $\htt(\tau)>\htt(\nu)$ cannot both hold, by the compatibility of $\htt$ with $\preceq$; see \eqref{preorder1}. This contradiction shows that each $A^+_j$ is a two-sided ideal. 

Next, we show each ideal $A_j^+$ is idempotent. In fact, we will show it is generated by an idempotent element $e_j\in A^+_j$ (if $A^+_j\neq 0$). 

First, we need some refinements of our notation. Define  
$$P_{j,0}=\bigoplus_{\htt(\omega)\geq j}P(\omega)_j\;\;\text{ and }\;\; 
P_{j,1}=\bigoplus_{\htt(\omega)<j}P(\omega)_j.$$ Note that the trace definition $P(\omega)_{j}=\text{trace}_{P(\omega)}\Big(\bigoplus_{\htt(\mu)\geq j}P(\mu)\Big)$ implies $P(\omega)_j\subseteq P(\omega)$, for any $\omega\in\Omega$, and $P(\omega)_j=P(\omega)$ if $\htt(\omega)\geq j$. (In particular, $P_{j,0}=\bigoplus_{\htt(\omega)\geq j}P(\omega)$.) Recalling the identity $P_j=\bigoplus_{\omega\in\Omega} P(\omega)_j$, we have
$P_j=P_{j,0}\oplus P_{j,1}$. It is also useful to set $\widehat P_{j,1}=\bigoplus_{\htt(\omega)< j}P(\omega)$; so that $P_{j,1}\subseteq\widehat P_{j,1}$ and $P=P_{j,0}\oplus \widehat P_{j,1}$.

We are almost ready to define $e_j$. First, let $A^+,A^+_j, A^+_{j,0}, A^+_{j,1}, \widehat A^+_{j,1},\ldots$ correspond to $P,P_j,P_{j,0},P_{j,1}, \widehat P_{j,1},\ldots$ under our standard isomorphism $A^+\cong P$ of left $A^+$-modules. Next, let $f_j:P\twoheadrightarrow 
P_{j,0}\subseteq P_j\subseteq P$ be the obvious idempotent endomorphism of $P$ passing through $P_j$. Then, there is a unique idempotent $e_j\in A^+$ such that right multiplication by $e_j$ gives the map $A^+\twoheadrightarrow A^+_{j,0}\subseteq A^+_j\subseteq A^+$ corresponding to $f_j$. Thus, $e_j$ is idempotent (since $e_j^2=1\cdot e_j^2=f_j^2(1)=f_j(1)=1\cdot e_j=e_j$), and $e_j\in A^+_j$ (since $e_j=1\cdot e_j\in A^+_{j,0}\subseteq A_j^+$). 

To complete the proof that $A^+_j$ is an idempotent ideal in $A^+$, it is enough now (and even necessary and sufficient) to show $A^+e_j A^+=A^+_j$. This is actually fairly easy. The submodule $A^+ e_j$ of $A^+_j$ corresponds, by construction, to the submodule $P_{j,0}$ of $P$. Also, we have the equality $P_j=\text{trace}_P(P_{j,0})$ by definition. It follows, passing to $A^+$, that
\begin{equation}\label{AfA}
A^+_j=\text{trace}_{A^+}(A^+_{j,0})=\text{trace}_{A^+}(A^+e_j)=\sum_{f\in\Hom_{A^+}(A^+e_j,A^+)}\im(f)=A^+e_jA^+.
\end{equation}
Here the last equality follows from a well-known identification of $\Hom_{A^+}(A^+e_j,A^+)$ with $e_jA^+$, which leads to an identification of trace$_{A^+}(A^+e_j)$ with $A^+e_jA^+$. %Thus, $A^+_j=A^+e_jA^+$.

To prove the last remaining assertion of the claim that $A^+/A^+_j$ is free as a $\sZ$-module, it is enough to show $A^+_i/A^+_{i+1}$ is free for every $i<j$. However, we have already observed that the isomorphic module $P_i/P_{i+1}$ is a direct sum of $A^+$-modules isomorphic to $\Delta(\nu)$ for some $\nu\in\Omega$. Thus, it is projective (of finite rank) over $\sZ$, hence free, by Swan's theorem. See Remark \ref{2.5}(c). The claim is now proved. 

We are now able to construct a candidate for the defining sequence \eqref{Jseq} %$0= J_0\subseteq J_1\subseteq \cdots \subseteq J_m=A^+$
 previewed earlier in this proof.

Observe that $A^+_j=0$, for all sufficiently large integers $j$. In fact, it suffices to take $j>\htt(\omega)$, for all $\omega\in\Omega$, using the trace formula for $P_j$. Put $J_0=0$, and define $J_1=A^+_{j(1)}$, where $j(1)$ denotes the largest value of $\htt$ on $\Omega$. Continuing, define $J_2=A^+_{j(2)}$, where $j(2)$ denotes the largest value of $\htt$ on $\Omega\backslash\Omega_{j(1)}$ with $\Omega_{j(1)}=\{\omega\in\Omega\mid\htt(\omega)=j(1)\}$. This process continues until $j(1),j(2),\ldots,j(m)$ have been defined as all the elements of $\htt(\Omega)$, listed in decreasing order. For each integer $i$ with $1\leq i\leq m$, we set $J_i=A^+_{j(i)}$, generalizing our definitions for $J_1$ and $J_2$. We note that $A^+_{j(m)}=A^+$, since $P_{j(m)}=P$ by the trace formula. Thus, $J_m=A^+_{j(m)}=A^+$, the terminal member of our ascending sequence of $J_i$'s: % Notice that the inclusions 
\begin{equation}\label{j(i)}
0= J_0\subseteq J_1\subseteq \cdots \subseteq J_m=A^+\; \text{ with }\;j(1)>j(2)>\cdots>j(m).
\end{equation}
 (The general inclusion $P_j\subseteq P_{j-1}$, which holds for any $j\in\mathbb Z$,  also implies that $P_j\subseteq P_{j-a}$,   for any positive $a\in\mathbb Z$. Consequently,
$A^+_j\subseteq A^+_{j-a}$. So $A^+_{j(i)}\subseteq A^+_{j(i+1)}$, using $j=j(i)$ and $a=j(i)-j(i+1)$. Here, $i$ is an integer with $1\leq i\leq m-1$. Since we have defined  $J_i=A^+_{j(i)}$ for each $i$, as well as $i=m$, we have the desired sequence $0= J_0\subseteq J_1\subseteq \cdots \subseteq J_m=A^+$.)

Continuing the discussion above, we now wish to analyze more closely the left $A^+$-modules $J_1=J_1/J_0$ and $J_i/J_{i-1}$ for $2\leq i\leq m$. We start with $J_1=A^+_{j(1)}$. Recall that $j(1)$ is the largest value of the height function so that $A^+_j=0$, for $j>j(1)$. (Use the trace formula for $P_j$.) Thus, $A^+_{j(1)}=A^+_{j(1)}/A^+_{j(1)+1}$ is a direct sum of various left $A^+$-modules $\Delta(\omega)$ with $\htt(\omega)=j(1)$. 
These modules are projective for $A^+$, by the maximality of $\htt(\omega)$. (Just observe that $P(\omega)=\Delta(\omega)$ by axiom (SS3).)

We now consider $J_i/J_{i-1}$ for $2\leq i\leq m$ (the case $i=1$ having been just discussed). We note 
\begin{equation}\label{JiJi-1}
J_i/J_{i-1}=A^+_{j(i)}/A^+_{j(i-1)}=A^+_{j(i)}/A^+_{j(i)+1}.
\end{equation}
The right-hand equality holds because $j(i)<j(i-1)$ are adjacent values of the height function. Consequently, for any integer $j$ with $j(i)+1\leq j<j(i-1)$,  the $A^+$-module $A^+_j/A^+_{j+1}$ must be 0. (Use our description, earlier in this proof, of $P_j/P_{j+1}$ as a direct sum of modules
$\Delta(\omega)$ with $j=\htt(\omega)$.) But this gives $A^+_j=A^+_{j+1}$, with repeated application giving $A^+_{j(i)+1}=A^+_{j(i-1)}$, as desired.

We are now in a position to prove condition (HI2), the projectivity of $J:=J_i/J_{i-1}$ as a left $B$-module, for $B:=A^+/J_{i-1}$, $1\leq i\leq m$. (The case $i=1$ has already been treated.) We have $J=A^+_{j(i)}/A^+_{j(i)+1}$ by \eqref{JiJi-1}. %the previous paragraph. 
Since $j(i)$ is a height function value, we know that $J$ in this latter quotient form is the direct sum of various $A^+/A^+_{j(i)+1}$-modules isomorphic to $\Delta(\omega)$ with $\htt(\omega)=j(i)$. (Such a module may be regarded as a left $A^+$- or $A^+/A^+_{j(i)+1}$-module.) Recall that we have direct sum formulas $P=\bigoplus_{\omega\in\Omega}P(\omega)$ and $P_j=\bigoplus_{\omega\in\Omega}P(\omega)_j$ for each integer $j$, and variations. In particular, note that $P/P_{j+1}=\bigoplus_{\omega\in\Omega}P(\omega)/P(\omega)_{j+1}$.
When interpreted with $A^+$ replacing $P$ and $j=j(i)$, we find that the left-hand side is $B=A^+/A^+_{j(i)+1}$, and the right-hand side, viewed a left $A^+$-module, has $\Delta(\omega)$ as a direct summand whenever $\htt(\omega)=j(i)$, as is the case for $\Delta(\omega)$ above. Thus, $\Delta(\omega)$, as a left $B$-module, is projective. Since the $\Delta(\omega)$ also represents the general direct summand of $J$, we have proved $J$ is projective as a left $B$-module, as desired. (Note also the $A^+$ analog of the formula $P_j/P_{j+1}=\bigoplus_{\omega\in\Omega}P(\omega)_j/P(\omega)_{j+1}$. For $j=j(i)$, the left-hand side becomes $A^+_{j(i)}/A^+_{j(i)+1}=A^+_{j(i)}/A^+_{j(i-1)}=J_i/J_{i-1}$, which is $J$ in the current context. On the right, we obtain a direct sum of modules $\Delta(\omega)$ with $\htt(\omega)=j(i)$, all now known to be projective as left $B$-modules.)

This completes the proof of (HI2) for $B$ and $J$ and all choices of $i$ with $1\leq i\leq m$. Properties (HI1) and (HI3) hold as immediate consequence of our claim studied several paragraphs above. 
}This completes the proof of the theorem.
\end{proof}

In Chapter 4, we will consider the localization $\sZ^\natural$ of $\sZ$ at the multiplicative set generated by all bad primes for $W$; see Remark \ref{good bad}. By base change from $\sZ$ to $\sZ^\natural$, \eqref{Jseq} gives rise to a sequence of ideals of $A^{+\natural}:=\sZ^\natural\otimes A^+$:
\begin{equation}\label{JJseq}
0=J_0^\natural\subseteq J_1^\natural\subseteq\cdots\subseteq J_m^\natural=A^{+\natural}.
\end{equation}
We will prove in Section 4.3 that every ideal $J^\natural:=J_i^\natural/J_{i-1}^\natural$ ($1\leq i\leq m$) is a ``heredity ideal of split type'' in $B^\natural:=A^{+\natural}/J_{i-1}^\natural$. More precisely, every $\End_{B^\natural}(J^\natural)$ is a direct product of matrix algebras $M_{n}(\sZ^\natural)$. Consequently, if $W$ is not of type ${}^2F_4$, $A^{+\natural}$ is a (split) quasi-hereditary algebra. 

\begin{rem}\label{Adagger2}
Replacing $\Omega$ by a poset ideal $\Theta$ of $(\Omega,\leq_L)$, $T^+$ by the $\sH$-module $T^\ddagger_\Theta$, $A^+$ by $A^\ddagger_\Theta$ as defined in \eqref{Tdagger}, %and \eqref{Adagger}  
and using the stratifying system in Corollary \ref{Adagger1} (thus, replacing $\Delta(\omega),P(\omega), P$, etc. by $\Delta^+(\omega),P^+(\omega), P^+$, etc.) throughout the proof of Theorem \ref{SSAA+}, we may construct a standard stratification
\begin{equation}\label{J+seq}
0=J_0^+\subseteq J_1^+\subseteq\cdots\subseteq J_m^+=A^\ddagger_\Theta.
\end{equation}
for the standardly stratified algebra $A^\ddagger_\Theta$.
\end{rem}

\chapter{The quasi-heredity of $A^{+\natural}$} 

We proved in the last chapter that the enlarged Hecke endomorphism algebra $A^+$ is standardly stratified, a property reminiscent of quasi-hereditary algebras \cite{CPS88}. In this chapter, we will prove that its base change $A^{+\natural}:=\sZ^\natural\otimes_{\sZ}A^+$ over $\sZ^\natural$ (in which all bad primes are invertible) is {\it integral} quasi-hereditary; see, e.g., \cite{CPS90} (also \cite{DPS18}) for integral quasi-hereditary algebras.

\section[Maximal orders  $\sJ^\natural:=\sZ^\natural\otimes_{\mathbb Z} \sfJ$]{Maximal orders  $\sJ^\natural:=\sZ^\natural\otimes_{\mathbb Z} \sfJ$ arising from  Lusztig's asymptotic form $\sfJ$}
Let $(W,S,L)$ be a standard finite Coxeter system associated to a finite group $G$ of Lie type, and let $\sH$ be its generic Hecke algebra over $\sZ:=\mathbb Z[t,t^{-1}]$. As noted in Section 1.2, Lusztig's conjectures P1--P15 all hold for this situation. Also, $W$ is, of course, finite, so
the standing hypotheses of Lusztig used in \cite[\S18]{Lus03} hold.  {\it In this section, we temporarily exclude the case that $G$ is of
type $^2F_4$, i.e., we assume $W$ is not the dihedral group $I_2(8)$ of order 16.} We will treat this
case in a short appendix at the end of this section. (However, everything up through (\ref{themap}) holds for
$^2F_4$. Recall that the smallest splitting field of $I_2(8)$ is $\mathbb Q(\sqrt{2})$, the field generated by its irreducible character values.
The other Weyl groups have
splitting field $\mathbb Q$.)

Let $\sZ^\natural:=\sS^{-1}\sZ$, the localization of $\sZ$ at the multiplicative set $\sS$ generated by the 
bad primes of $W$.\footnote{As noted with more detail in Remark \ref{good bad}, one could use ``bad prime" in the traditional sense here, or, for our purposes, simply define it to be a prime divisor of the generic degree denominator $m$ for a principal series character.} Set $\sH^\natural:=\sZ^\natural\otimes_\sZ\sH$.  Of course, $\sZ^\natural$ is a unique factorization domain of Krull dimension 2 (which is regular). In particular, its height one
prime ideals $\mathfrak p$ are all principal and are given (in this case) as follows:
(i) $\mathfrak p= p\sZ^\natural$, where $p\in\mathbb Z$ is a good prime for $W$; and (ii) $\mathfrak p=f(t)\sZ^\natural$,
where $f(t)$ is an irreducible polynomial in $\mathbb Z[t]$, of positive degree and not equal to $\pm t$.
 This description follows easily from \cite[Example H, p. 74]{Mum88}, noting that the prime ideals of $\sZ$ correspond
 naturally to those of $\mathbb Z[t]$ which do not contain $t$.

Before going further we introduce some general terminology on lattices and orders over a  Noetherian integral
domain $\scrK$. We have $\scrK=\sZ$ or $\sZ^\natural$ in mind. Often it will be convenient to let $K$ be the fraction
field of $\scrK$, e.g., $K=\mathbb Q(t)$, if $\scrK=\sZ$ or $\sZ^\natural$.

\begin{defns}\label{orders}  (a) For any finite-dimensional $K$-space $V$, a {\it (full) $\scrK$-lattice} in $V$ is a finite
 $\scrK$-submodule $M$ in $V$ such that $K M = V$, where
$$K M :=\Big\{\sum r_im_i (\text{finite sum})\mid r_i\in K, m_i \in M\Big\}.$$
 Following \cite[p.~108]{Reiner03}, a {\it $\scrK$-order}\index{$\scrK$-order} is a $\scrK$-subalgebra $\sA$ of a finite-dimensional $K$-algebra $\mathfrak A$, having the same identity as $\mathfrak A$, such that $\sA$ is a full $\scrK$-lattice in $\mathfrak A$.  
  Unless otherwise noted, from now on, the term ``lattice" refers to a ``full lattice."    
  
  \smallskip
 (b) Also, by some abuse of notation, we shall refer to a module $M$ for a $\scrK$-order 
 $\sA$ as an $\sA$-lattice if it inherits a $\scrK$-lattice structure from the action of $\sA$.
  A $\scrK$-order $\sA$ in $\mathfrak A$ is a {\it maximal $\scrK$-order}\index{$\scrK$-order! maximal $\sim$} if it is not  properly contained in a larger $\scrK$-order in $\mathfrak A$; see \cite[p. 110]{Reiner03}. 
\end{defns}

To give a common example of a maximal order in a finite-dimensional $K$-algebra $\mathfrak A$, assume
that $\scrK$ is integrally closed and $\mathfrak A$ is a direct product $\prod_iM_{n_i}(K)$ of a finite number of matrix algebras $M_{n_i}(K)$.  Then $\sA:=
\prod_iM_{n_i}(\scrK)$ is a maximal $\scrK$-order in $\mathfrak A$. (See \cite[Th. 8.7]{Reiner03}.)

\medskip

We need some further notation, mostly taken from \cite{Lus03}.  For $w\in W$, let $\text{\sgn}(w)=(-1)^{\ell(w)}$.  Recall the notation $\ssT_w$ from display \eqref{newtildebasis}. The map
 \begin{equation}\label{daggerdefn} \ssT_w\longmapsto \ssT_w^\dagger:=\text{sgn}(w)\ssT_{w^{-1}}^{-1}\end{equation}
 defines a
$\sZ$-algebra involution $h\mapsto h^\dagger$ of $\sH$.  See \cite[3.5]{Lus03}.
 Next,
for $z\in W$, let $a(z)$ \index{$a$-function} be the smallest non-negative integer such that
\begin{equation}\label{afun}
t^{a(z)}h_{x,y,z}\in\mathbb Z[t],\;\; \text{for all }x,y\in W.
\end{equation}  Here $h_{x,y,z}
$ is the structure constant defined in
(\ref{structureconstants}). Then let $\gamma_{x,y,z}\in\mathbb Z$ be the coefficient of $t^{a(z)}$ in $h_{x,y,z}$.  

Obviously, $\sH$ is a $\scrK$-order in $\sH_{\mathbb Q(t)},$ using $\scrK=\sZ$. Next, we define two further orders in (isomorphic copies of) $\sH_{\mathbb Q(t)}$ in the sense of Definition \ref{orders}. The first one $\sJ:=\sZ\otimes_{\mathbb Z}\sfJ$ 
is the base change of a $\mathbb Z$-algebra $\sfJ$; the latter, defined below, is called {\it Lusztig's asymptotic form}. \index{$\sfJ$, asymptotic form} \index{$\sfJ$, asymptotic form! $\sJ:=\sZ\otimes_{\mathbb Z}\sfJ$} As we shall see, the order
$\sJ$  contains a copy
of $\sH$, as discussed below. The second order $\sJ^{\natural}:=\sZ^\natural\otimes_{\mathbb Z}\sfJ\cong\sZ^\natural\otimes_\sZ\sJ$, a $\sZ^\natural$-order, is obtained from the first by inverting all bad primes for $W$. \index{$\sfJ$, asymptotic form! $\sJ^{\natural}:=\sZ^\natural\otimes_{\mathbb Z}\sfJ$}
%As an abbreviation, with a slight abuse of notation, we write $\nsJ$
%for this $\sZ^\natural$-algebra. 

The algebra $\sfJ$ is a free $\mathbb Z$-algebra 
with $\mathbb Z$-basis $\{\sfj_x\,|\,x\in W\}$ and multiplication given by
\begin{equation}\label{mult} \sfj_x\sfj_y=\sum_{z\in W}\gamma_{x,y,z^{-1}}\sfj_z,\quad x,y\in W.\end{equation} 
Then $\sfJ$ is associative and unital \cite[18.3]{Lus03}. (Note that rank$(\sfJ)=|W|$.)

Recall that $W$ has a subset $\mathcal D$ of distinguished involutions \cite[P6, 14.1]{Lus03}.  Each left cell
 $\omega$   contains a unique element  $d\in {\mathcal D}$. Also, for $d\in\mathcal D$, an integer $n_d=\pm 1$ is defined in
 \cite[14.1, P5 \& P13]{Lus03}. Given $z\in W$, let $d\in\mathcal D$ be the unique distinguished involution
satisfying $z^{-1}\sim_L d$, and set $\hat{n}_z:= n_d$. 
By \cite[Th. 18.9]{Lus03}, there is a $\sZ$-algebra 
homomorphism 
$\phi:\sH\longrightarrow \sJ$ defined by
\begin{equation}\label{phiformula}
\phi(\scc_w^\dagger)=\sum h_{w,d,z}\hat{n}_z\sfj_z.\end{equation}
In this expression, $w\in W$ and the sum is over all pairs $(d,z)\in \mathcal D\times W$ such that $a(z)=a(d)$.
Then,  for any commutative $\sZ$-algebra $R$, $\phi$ induces an
$R$-algebra homomorphism $\phi_R:\sH_R\to \sJ_R$. By \cite[Prop. 18.12a]{Lus03}, the kernel of $\phi_R$ is a nilpotent ideal. Thus,  when $R$ is a field and $\sH_R$ is a semisimple algebra, $\phi_R$
is injective and, thus, an isomorphism by
dimension considerations. In particular, $\phi_{\mathbb Q(t)}$ is an isomorphism. Of course, this means that $\phi$ is an injection, so that 
$\sJ$ and $\nsJ$ both contain subalgebras isomorphic to $\sH$.
The case $R=\mathbb Q$ leads to a similar isomorphism $\sH_{\mathbb Q}\overset\sim\rightarrow \sJ_{\mathbb Q}$, regarding $\mathbb Q$ as a $\sZ$-algebra by means of the specialization $t\mapsto 1$. Note $\sJ_{\mathbb Q}\cong\sfJ_{\mathbb Q}$ by construction. For later use in the proof of the theorem below, note also that
$\sH_{\mathbb Q}\cong{\mathbb Q}W$, and so, $\sfJ_{\mathbb Q}\cong \sJ_{\mathbb Q}\cong\sH_{\mathbb Q}\cong \mathbb QW$.

\begin{rem}\label{adddagger} All the assertions of the above paragraph (below display (\ref{phiformula})) hold as written if the homomorphism $\phi$ is replaced
by the composition 
\begin{equation}\label{varpi}\varpi:=\phi\circ\dagger:\sH\longrightarrow\sJ, \end{equation}
where $\dagger$ is defined in display (\ref{daggerdefn}).\index{$\varpi$, the embedding $\sH\to\sJ$}
 This is not needed right away, but will be useful in showing in the next section that all left-cell modules
for $\sH$ extend to $\sJ$-modules. For the convenience of the reader, we write down explicitly the expression for the composite $\varpi=\phi\circ\dagger$:
 \begin{equation}\label{themap}\varpi(\scc_w)= 
 \sum_{z\in W}\sum_{d\in\mathcal D\atop a(d)=a(z)}h_{w,d,z}\hat n_z\sfj_z\quad (w\in W).\end{equation} 
We will also use $\varpi$ in the context of $\sH^\natural$ and $\nsJ$, in which case we define $\varpi^\natural$
to be the base-changed map
\begin{equation}\label{8.0.37}
\varpi^\natural:\, \sH^\natural=\sH_{\sZ^\natural}\overset{\varpi_{\sZ^\natural}}\longrightarrow\sJ_{\sZ^\natural}=\sJ^{\natural}.
\end{equation}
Moreover, the isomorphisms $\sfJ_{\mathbb Q}\cong\sH_{\mathbb Q}\cong \mathbb QW$ and
$\sfJ_{\mathbb Q(t)}=\sJ_{\mathbb Q(t)}\cong\sH_{\mathbb Q(t)}$ induce an isomorphism
$\mathbb Q(t)W\cong \sH_{\mathbb Q(t)}$ and, hence, a bijection
\begin{equation}\label{IrrWH}
\text{Irr}(\mathbb QW)\longrightarrow\text{Irr}(\sH_{\mathbb Q(t)}),\;E\longmapsto E_{\mathbb Q(t)}.
\end{equation}
If $G$ is of type ${}^2F_4$, $\mathbb Q$ should be replaced by $\mathbb Q(\sqrt2)$.
\end{rem}

Next, we pursue some of the remarkable properties of the $\sZ^\natural$-order $\nsJ$.  Here,
$\scrK=\sZ^\natural$ and $K=\mathbb Q(t)$.
 
Note that $\nsJ_{{\mathbb Q}(t)}$ identifies with $\sJ_{\mathbb Q(t)}$ which is isomorphic (via the inverse of  $\phi_{\mathbb Q(t)}$) to
$\sH_{{\mathbb Q}(t)}$.   The latter is semisimple by Remark \ref{Rem7.4}(b)  and even  split, since $\mathbb QW$ is split semisimple.
(Recall that we have postponed the case $^2F_4$ to an appendix at the end of this section.)
%, as commented at the beginning of this section.)  
Hence, there is a ``natural" isomorphism
\begin{equation}\label{iso1}\sJ^\natural_{{\mathbb Q}(t)}\overset\sim\longrightarrow\prod_i\End_{{\mathbb Q}(t)}(E_i),\end{equation}
where $E_i$ ranges over the distinct irreducible $\sJ_{{\mathbb Q}(t)}$-modules. 
The map is obtained from the action of $\nsJ_{\mathbb Q(t)}=\sJ_{\mathbb Q(t)}$ on  the various $E_i$.
We can assume that $E_i=KL_i$
where $L_i$ is a $\nsJ$-stable lattice in $E_i$. In fact, using Auslander-Goldman theory (see \cite[Cor. C.18]{DDPW08}), it can  also be assumed (after replacing
$L_i$ by its double dual $L_i^{**}$) that all the $L_i$
are projective over $\sZ^\natural$ .  Then, by the result of Swan mentioned in 
Remark \ref{2.5}(c)), the $L_i$ are free over $\sZ^\natural$.

Thus, (\ref{iso1}) may be obtained by base change from $\sZ^\natural$ to $\mathbb Q(t)$
of an injection 
\begin{equation}\label{bigmap} \psi:\nsJ\to \prod_i\End_{\sZ^\natural}(L_i),\end{equation}
 where each $L_i$ is $\sZ^\natural$-free. Note that each $\nsJ$-module $L_i$ is an $\sH^\natural$-module by restriction
 through the map $\varpi^\natural$ in (\ref{8.0.37}).
\medskip\medskip
 
We can now prove:

\begin{thm} \label{maxorder} Maintain the assumptions on $W$, $\sH$, etc. above. The map $\psi$ in (\ref{bigmap}) is an isomorphism
of $\sZ^\natural$-algebras.
 \end{thm}

\begin{proof}  
We have already noted that the map $\psi$ is  an injection (since its base change to ${\mathbb Q}(t)$ is obviously
an injection). Clearly, $\nsJ$ is a $\sZ^\natural$-order (and, in particular,  a full lattice) in $\nsJ_{\mathbb Q(t)}$. To prove the theorem, it suffices to prove that $\nsJ$ is a maximal $\sZ^\natural$-order in $\nsJ_{\mathbb Q(t)}$. This will show that $\psi$ is an isomorphism, since it is an inclusion of orders.   To do this, it is necessary and sufficient, by the Auslander-Goldman theorem (see \cite[Th. 11.4]{Reiner03}), to check two conditions: 
\begin{itemize}
\item[(a)] the $\sZ^\natural$-order $\Xi:=\nsJ$ is reflexive in the sense that $\Xi^{**}\cong \Xi$ as $\sZ^\natural$-modules;
\item[(b)] for any height one prime ideal $\mathfrak p$ in $\sZ^\natural$, $\nsJ_{\mathfrak p}$ is a maximal $\sZ^\natural_{\mathfrak p}$-order in $\nsJ_{\mathbb Q(t)}$.
\end{itemize}

But $\nsJ$ is $\sZ^\natural$-free, so it is certainly reflexive. Thus, (a) holds.

 For part (b), it suffices to show, for any height one prime ideal $\mathfrak p$ in $\sZ^\natural$, that 
  there is an abstract algebra isomorphism
  \begin{equation}\label{abstract}
\nsJ_{\mathfrak p}\cong\prod_iM_{n_i}(\sZ^\natural_{\mathfrak p})\end{equation}
for some finite list of positive integers $n_i$. See \cite[Ths. 8.7, 10.5]{Reiner03}.

Consider first the case in which ${\mathfrak p}=(p)$, where $p$ is a good prime
for $W$. Then \cite[Ths. 4.2.1, 4.2.2]{DPS98a} and their proofs show  that
\begin{equation}\label{dogformula}\sH_{\mathbb F_p(t)}\cong \prod_i\End_{\mathbb F_{p}(t)}(L_{i{\mathbb F_p(t)}})\end{equation}
with $L_i$ as above, and $L_{i{\mathbb F_p(t)}}=(L_i)_{\mathbb F_p(t)}$ denoting the 
base-changed module $\mathbb F_p(t)\otimes_{\sZ^\natural}L_i$. Note that $\mathbb F_p(t)$ is the residue field of $\sZ^\natural_{\mathfrak p}$. (Essentially, the argument in \cite{DPS98a} uses Schur matrix relations and the fact that good
primes avoid generic degree denominators \cite[(4.1.2)]{DPS98a}  to show that all matrix units on the right have preimages on the left.)
Since the action of $\sH_{\sZ^\natural}$ on $L_i$ is by restriction
through the algebra
$\nsJ$, we have a surjective composite map
$$(\sH_{\sZ^\natural})_{\mathfrak p}\hookrightarrow\nsJ_{\mathfrak p}\longrightarrow 
M_{n_i}(\sZ^\natural_{\mathfrak p})
\longrightarrow M_{n_i}({\mathbb F}_p(t)),$$
where the two matrix algebras are the full endomorphism rings of $(L_i)_{\sZ^\natural_{\mathfrak p}}$
and $(L_i)_{\mathbb F_p(t)}$, respectively.
Consequently, the  map 
$$\nsJ_{\mathfrak p}\longrightarrow\prod_iM_{n_i}({\mathbb F}_p(t))$$
is surjective, as is the map
$$\nsJ_{\mathfrak p}\longrightarrow\prod_iM_{n_i}(\sZ^\natural_{\mathfrak p})\qquad (\text{by Nakayama's Lemma}).$$
Since the latter map is an injection, it is also an isomorphism.

Finally, it is enough to show that an isomorphism (\ref{abstract}) holds whenever ${\mathfrak p}=(f(t))$ with $f(t)\in{\mathbb Z}[t]$
an irreducible polynomial of positive degree and $\not=\pm t$. (These are the remaining height one primes in $\sZ^\natural$.) In this case, no nonzero integer $z\not=0$ is divisible
by $f(t)$ in the unique factorization domain $\sZ^\natural$. Thus, $\mathbb Q\subseteq \sZ^\natural_{\mathfrak p}$.
  Since 
$$\sfJ_{\mathbb Q}\cong{\mathbb Q}W\cong\prod_jM_{n_i}({\mathbb Q})$$
for some positive integers $n_i$,\footnote{The first isomorphism is discussed above Remark \ref{adddagger}. It may be regarded as one of the many well-known results of Lusztig generalized in \cite{Lus03} to an unequal parameter setting.} we have
$$\aligned
\sJ^\natural_{\mathfrak p}:=\sZ^\natural_{\mathfrak p}\otimes_{\sZ^\natural}\sJ^\natural&\cong 
\sZ^\natural_{\mathfrak p}\otimes_{\sZ^\natural}\sZ^\natural\otimes_{\sZ}\sJ\cong
\sZ^\natural_{\mathfrak p}\otimes_{\sZ^\natural}\sZ^\natural\otimes_{\sZ}\sZ\otimes_{\mathbb Z}\sfJ\\
&\cong\sZ^\natural_{\mathfrak p}\otimes_{\mathbb Z}\sfJ\cong\sZ^\natural_{\mathfrak p}\otimes_{\mathbb Q}\mathbb Q\otimes_{\mathbb Z}\sfJ
\cong \sZ^\natural_{\mathfrak p}\otimes_{\mathbb Q}\sfJ_{\mathbb Q}
 \cong \prod_iM_{n_i}(\sZ^\natural_{\mathfrak p}),
 \endaligned$$
completing part (b) of the proof. %The theorem is proved.
\end{proof}

This has several consequences. The first corollary, which we call a theorem, follows either from the proof of Theorem \ref{maxorder} or from its statement. (Remember that $K=\mathbb Q(t)$ and $G$ is not of type $^2F_4$.)

\begin{thm} \label{maximalorder} %\begin{itemize}
{\rm(a)} The $\sZ^\natural$-algebra 
$\nsJ=\sJ_{\sZ^\natural}$ is isomorphic to a maximal order in $\sH_K$. 

{\rm(b)} Furthermore,
\begin{equation}\label{maximalorder1}
\nsJ=\sJ_{\sZ^\natural}\cong\prod_iM_{n_i}(\sZ^\natural),
\end{equation}
for a finite sequence $n_1,n_2, \cdots$ of positive integers. %\end{itemize}
\end{thm}

\begin{rems} \label{Corofmaxorder}
(1) The above theorem has an analogous version in which the $\sZ^\natural$-algebra $\sJ^\natural$ is replaced by
the $\mathbb Z^\natural$-algebra $\sfJ^\natural$.   In this context, another consequence of Theorem \ref{maximalorder}
is that $\sfJ^\natural$ is a maximal order in $\sH_{\mathbb Q}=\mathbb QW.$ Explicitly, $\sfJ^\natural$ is isomorphic to a
direct product of full matrix algebras over $\mathbb Z^\natural$. This follows by a base-change argument using the
specialization $\sZ^\natural\rightarrow \mathbb Z^\natural$ given by $t\mapsto 1$. Finally, all corollaries below
have an analogous version for $\sfJ^\natural$. They are easily derived through a similar proof, or through a base-change
argument together with the statements for $\sZ$.

(2) By the theorem, we have $\sJ_K=\sJ^\natural_K\cong\prod_{i=1}^sM_{n_i}(K)$, for some $s$, when $K=\mathbb Q(t)$. Thus, analyzing projections, there are a total of $s$ non-isomorphic irreducible $\sJ_K$-modules $L_1,\ldots,L_s$ such that, for each $i$ with $1\leq i\leq s$, $L_i\cong K^{n_i}$ as $\sJ_K$-modules, where $K^{n_i}$ is the irreducible left $M_{n_i}(K)$-module via matrix multiplication. The discussion of the map $\psi$, in and above Theorem \ref{maxorder}, implies that $L_i$ has a $\sJ^\natural$-lattice $L_{i,\sZ^\natural}$ isomorphic to $(\sZ^\natural)^{n_i}$. In particular, $L_{i,\sZ^\natural}$ is $\sZ^\natural$-free.

(3) In the equal parameter case, Lusztig first proved in \cite[(3.1j),(1.2a)]{Lus87d} that, for a large enough field $F$, $\sfJ_F$ is split semisimple. Further, in \cite[Remark 2.5]{G98}, Geck observed that this result applies to the case where the characteristic of $F$ is a good prime. See also \cite[Props. 1.5.7\&2.3.16]{GJ11} for a similar result in general over a sufficiently large field $\mathbb K$ of characteristic 0.
\end{rems}

We note that a maximal order copy of $\nsJ$ in $\sH_{K}$ is $\phi_K^{-1}(\nsJ)$, in the notation of (\ref{phiformula}). Another useful copy is obtained by applying to this copy the involution on $\sH_K$ induced
 by $\dagger$. The resulting copy may also be written $\varpi_K^{-1}(\nsJ)$, with $\varpi$ as in Remark \ref{adddagger}.

Recall that if $R$ is a commutative ring, an algebra $B$ over $R$ is separable if $B$ is a
projective $B\otimes_RB^{\op}$-module; see Auslander-Goldman \cite[p. 1]{AG60b}. 

\begin{cor}\label{yahoo}
%\begin{itemize}
{\rm(a)}
The algebra $\nsJ$ is separable over $\sZ^\natural$. 

{\rm(b)} Every finite module for $\nsJ$, which is projective over $\sZ^\natural$, is projective over $\nsJ$.
%\end{itemize}
\end{cor}

\begin{proof} The reader may directly check that the
direct product of a finite set of separable algebras is a separable algebra. Then, in view of Theorem \ref{maxorder}, part
(a) follows from \cite[Prop. 5.1]{AG60b}. Finally, part (b) follows from \cite[Th. 2.11]{Hat63}.
(See also \cite[Ex. 9, p. 106]{Reiner03}.)
\end{proof}

\begin{cor}\label{ress} Let $M$ and $M'$ be finite $\nsJ$-modules, projective over $\sZ^\natural$. If $M_K\cong M'_K$ as
$\nsJ_K$-modules, then $M\cong M'$ as $\nsJ$-modules.\end{cor}

\begin{proof} Note that $M$ and $M'$ are projective over $\nsJ$ by Corollary \ref{yahoo}(b).  By Theorem
\ref{maximalorder}, we can reduce to the case where $M$ and $M'$ are finite projective modules for a matrix algebra $M_n(\sZ^\natural)$
 which become isomorphic over the fraction field $K$. By Morita theory, we can assume $n=1$. (Let 
 $F:\text{mod-}M_m(\sZ^\natural)
 \rightarrow{\text{ mod-}}\sZ^\natural$ be an equivalence of $\sZ^\natural$-categories, e.g., a suitable idempotent
 projection. Note that $F$ induces by base change an equivalence $F_K:{\text {mod-}}M_m(K)\rightarrow \text{mod-}
 K$. For $X,Y\in{\text {mod-}}M_m(\sZ^\natural)$, we clearly have $X\cong Y$ if and only if $F(X)\cong F(Y)$, and $X_K\cong 
 Y_K$ in mod-$M_m(K)$ if and only if $F(X)_K\cong F(Y)_K$ in mod-$K$.)
 By Swan's theorem (see Remark \ref{2.5}(c)), we can assume that $M$ and $M'$ are free $\sZ^\natural$-modules which are isomorphic over $K$. Obviously, this means $M\cong M'$.
\end{proof}
Finally, we have:% the following result.

\begin{cor}\label{irred} If $M$ is a $\nsJ$-module which is finite and projective over $\sZ^\natural$, 
then
\begin{itemize}
\item[(a)] $M$ is a direct sum of $\nsJ$-modules which are projective over $\nsJ$ and which have irreducible base change to $\nsJ_K$-modules. 

\item[(b)] Moreover, $M$ is uniquely determined (up to isomorphism) by the multiplicity of each isomorphism type  of the irreducible summands of the $\nsJ_K$-module $M_K$. \end{itemize}\end{cor}

\begin{proof} We first prove part (a).  By Theorem \ref{maximalorder}, we may replace $\nsJ$ by one of its direct
factors $M_m(\sZ^\natural)$. Using Morita theory, we can even replace $M_m(\sZ^\natural)$ by $\sZ^\natural$.  Once again, applying Swan's theorem, the $\sZ^\natural$-module $M$ must be free, and so clearly is a direct sum of $\sZ^\natural$-modules with irreducible base change to $K$. This proves (a). 

For the proof of (b), the reduction is simply reversed. Further details are left to the reader.\end{proof}

\medskip \noindent
\subsection*{Appendix to Section 4.1: the case of the Ree groups of type $^2\!F_4$} %$\bold{^2F}_{\bold 4}({\bold 2^{{\bold 2{\bold n}+{\bold 1}}}})$.
 The Weyl group $W$ is the dihedral group $I_2(8)$ of order 16, generated by
reflections $s_1$ and $s_2$ with $s_1s_2$ having order 8. The weights are determined from the longest words in each
individual orbit of a graph ``symmetry";  see \cite[pp. 36-38]{C85} or by analogy with the construction \cite[16.1--16.2]{Lus03} formulated in the quasi-split case.  We can assume $L(s_1)=2$ and $L(s_2)=4$ for the values of the  weight function $L$ on fundamental reflections $s_1$, $s_2$. See \cite[pp. 489-490]{C85} for the character degrees of the unipotent characters in the principal series. (Note that the degree associated to $\rho_2$ should be corrected to 
$(1/2) q^4\Phi_8^2\Phi_{24}$.)\footnote{The character degree given in \cite{C85} does not appear in Lusztig's list of unipotent character degrees. The value we have given does appear in \cite[p. 374]{Lus84} but just as a unipotent character
degree. It can be seen to be a generic/principal series degree by direct calculation using \cite[pp. 
270-272]{GP00}.} There are 7 degrees in all, labeled $1,  \epsilon', \epsilon^{\prime\prime}, \epsilon$ (for the
degree 1 representations of ${\mathbb C}W$) and $\rho'_2, \rho_2^{\prime\prime}, \rho_2$ (for the degree 2 representations of ${\mathbb C}W$). The degrees of all the unipotent characters are in terms of polynomial functions, elements of
 $\mathbb Q(\sqrt{2})[q]$, evaluated with $q$ an odd power of $\sqrt{2}$. The generic degrees are obtained by replacing $q$ by the variable
 $t$. From the corrected list, we see that the generic degrees have the form
 \begin{equation}\label{Dchi}
 {\frac{1}{m}}t^a\Phi'_{b_1}(t)\cdots\Phi'_{b_r}(t),
 \end{equation}
 where $a$ is a non-negative integer. Also, $\Phi'_{b_i}(t)$ is a factor in $\mathbb Z[\sqrt{2}][t]$ of the corresponding cyclotomic polynomial $\Phi_{b_i}(t)$ in $\mathbb Z[t]$.  The integer $m$ is equal to 1, 2, or 4.
 
 The expression (\ref{Dchi}) is sufficiently similar to \cite[(4.1.2)]{DPS98a} that analogs of
 \cite[Th. 4.2.1]{DPS98a} and \cite[Th. 4.2.2]{DPS98a} may be proved with parallel arguments. It is only 
 necessary to replace $\mathbb Z$ by $\mathbb Z'=\mathbb Z[\sqrt{2}]$,  $\mathbb Q$ by $\mathbb Q'=\mathbb Q[\sqrt{2}]$, $\sH$ by $\sH'=\sH_{\mathbb Z'[t,t^{-1}]}$, $K$ by $K'=\mathbb Q'(t)=K[\sqrt{2}]$, etc. 
 
 It will also be
 useful to define $\sZ'=\mathbb Z'[t,t^{-1}]=\mathbb Z[\sqrt{2}, t, t^{-1}]$, $ \sZ^{\prime\natural}=\sZ'[1/\sqrt{2},t, t^{-1}]$,
 $\mathbb Z^\natural=\mathbb Z[1/2],$ and $\mathbb Z^{\natural\prime}=\mathbb Z[1/2][1/\sqrt{2}]=
 \mathbb Z[1/\sqrt{2}]$. Finally, defining $\sZ^{\natural\prime}=\mathbb Z^{\natural\prime}[t,t^{-1}]$, we have
 $\sZ^{\natural\prime}=\sZ^{\prime\natural}.$
 
 The only ``bad prime" is $p=2$  in $\mathbb Z$ (or $p=\sqrt{2}$ in $\mathbb Z'$) in the sense of dividing $m$
 in (\ref{Dchi}). Fortunately,
 $\mathbb Z'=\mathbb Z[\sqrt{2}]$ is a principal ideal domain. Using these facts, the argument in the proof of
 Theorem \ref{maxorder} establishing  the intermediate isomorphism (\ref{dogformula}) determines an  analogous
 isomorphism. In the analogy, $\sH'$ replaces $\sH$, and $\sZ^{\prime\natural}$-free $\mathcal J^{\prime\natural}$-modules 
 $L_i'$ replaces $L_i$, for each $i$. Also,  $\mathbb F_p(t)$ is replaced by $\mathbb F'_p(t)$, the residue field
 of $\sZ'_{\mathfrak p}$ (which equals the residue field of  $\sZ^{\prime\natural}_{\mathfrak p}.$) Here,
 $\mathfrak p$ is the ideal generated by $p\in\sZ'$, where $p$ is a prime not dividing $\sqrt{2}$.  With this analog
 of the isomorphism (\ref{dogformula}) in hand, an evident version of the isomorphism (\ref{abstract}) follows for each $\mathfrak p$, using the argument in the proof of Theorem \ref{maxorder}.
 Consequently, each $\mathcal J_{\mathfrak p}^{\prime\natural}$ is a maximal order. Thus, $\mathcal J^{\prime\natural}$, which is clearly reflexive, is also a maximal order, and the map 
  \begin{equation}\label{psianalog} \psi:\mathcal J^{\prime\natural}\to \prod_i\End_{\sZ^{\prime\natural}}(L'_i),\end{equation}
  analogous to (\ref{bigmap}), 
 is an isomorphism of $\sZ^{\prime\natural}$-algebras.  This proves the following theorem, an analog of Theorem
 \ref{maxorder} above, in the Ree group setting
 of this appendix.

 \begin{thm}\label{new8.3} The map $\psi$ in (\ref{psianalog}) is an isomorphism of $\sZ^{\prime\natural}$-algebras.
 \end{thm}
 
 Each of the endomorphism (``matrix") components of (\ref{psianalog}) is  a maximal order in some $M_n(K')$ by \cite[Th. 8.7]{Reiner03}, since $\sZ^{\prime\natural} $ is integrally closed in $K'=\mathbb Q'(t)$. The result below is a corollary of the above theorem, but given its importance, we upgrade it to a theorem. 
 
  \begin{thm}\label{new8.4} The $\sZ^{\prime\natural}$-algebra $\mathcal J^{\prime\natural}=\mathcal J_{\sZ^{\prime\natural}}$
 is a maximal order in $\sH_{\mathbb Q'(t)}$. %In particular, 
%$\nsJ=\sJ_{\sZ^\natural}$ is isomorphic to a maximal order in $\sH_K$.
 Moreover,
\begin{equation}\label{maximalorder2}
\sJ^{\prime\natural}=\sJ_{\sZ^{\prime\natural}}\cong\prod_iM_{m_1}(\sZ^{\prime\natural})\end{equation}
for a finite sequence $m_1,m_2, \cdots$ of positive integers. 
 
 \end{thm}

 \begin{rems} \label{moreremarks}(a) Similarly, Corollaries \ref{yahoo}--\ref{irred} above have analogs for $^2F_4$ in the $\sZ^{\natural\prime}=\sZ^{\prime\natural}$-setting.  We leave the details and proofs, which follow the original arguments, to the reader. Theorem \ref{new8.4} above is similarly analogous to Theorem \ref{maximalorder}.  
 
 (b) Also, Corollaries \ref{yahoo} to \ref{irred} and Theorem \ref{maximalorder}(a)
  remain true in the
 $^2F_4$-case, as stated, that is, using $\sZ^\natural:=\mathbb Z[{{1}\over{2}}, t, t^{-1}]$, and $\sJ^{\natural}
 :=\sJ_{\sZ^{\natural}}$. (However, Theorem \ref{maximalorder}(b) is not true without further modification; see display \eqref{decomp} below.)
 The proofs below, in this new setting for the $^2F_4$ case, are somewhat more involved.
 
  To prove Theorem \ref{maximalorder}(a) in this new setting, note that if
 $\nsJ$ were properly contained in a $\sZ^\natural$-order---call it $\sJ^\clubsuit$---then the $\sZ^{\natural\prime}$-order
 $\sJ^\clubsuit_{\sZ^{\natural\prime}}$ would properly contain $\sJ^{\natural\prime}$, contradicting Theorem \ref{new8.4} above.
 
 Next, to prove Corollary \ref{yahoo} in the new setting, observe first that part (b) follows from part (a), using the argument already
 given in the original context.   Part (a) follows from the following faithfully flat descent theorem: 
 \begin{thm}[{\cite[Th. 8.1.20(1)]{F17}}]\label{tjf} Let $R$ be a commutative ring,
 $B$ a finite $R$-algebra. Suppose $S$ is a commutative faithfully flat $R$-algebra such that $B\otimes_RS$ is separable over $S$. Then,
 $B$ is separable over $R$.\end{thm}
 
 For a more elementary proof of Corollary \ref{yahoo}(a),  first show, using the maximality of $\nsJ$ (Theorem \ref{maximalorder}(a)),
  that
  $\sJ^{\natural}$ is a direct product of $1\times 1$ and $2\times 2$ full matrix algebras over $\sZ^{\natural}$, and a
 $2\times 2$ full matrix algebra over $\sZ^{\natural\prime}$:
 \begin{equation}\label{decomp}
 \sJ^{\natural}\cong \sZ^{\natural}\oplus \sZ^{\natural}\oplus\sZ^{\natural}\oplus\sZ^{\natural}\oplus M_2(\sZ^{\natural})
 \oplus M_2(\sZ^{\natural\prime}).\end{equation}
 Here, the factor $M_2(\sZ^{\natural\prime})$ arises, using the field automorphism  $\sqrt{2}\mapsto-\sqrt{2}$, as a
 ``diagonal" $\sZ$-subalgebra of $M_2(\sZ^{\natural\prime})\oplus M_2(\sZ^{\natural\prime})$.
 It suffices to prove the separability of this factor over
 $\sZ^{\natural}$, which reduces to the separability of $\sZ^{\natural \prime}$ over $\sZ^{\natural}$, and then to the
 separability of $\mathbb Z^{\natural\prime}$ over $\mathbb Z^{\natural}$. A splitting of the (bimodule) multiplication surjection $\mathbb Z^{\natural\prime}\otimes_{\mathbb Z^{\natural}}\mathbb Z^{\natural^\prime}\longrightarrow \mathbb Z^{\natural\prime}$ is given by sending $1\in\mathbb Z^{\natural\prime}$ to $(1/4)(\sqrt{2}\otimes\sqrt{2})+ 
 (1/2)(1\otimes 1)$. This proves the needed separability and completes the elementary proof. 
 
 To prove Corollary \ref{ress} in the new setting, note that proving a similar statement
for each component algebra of $\nsJ$ in (\ref{decomp}) is equivalent. The original proof obviously works for all
components except for the $\sZ^{\natural}$-algebra $M_n(\sZ^{\natural\prime})$. Here,
it is necessary and sufficient to show (call it a claim) that  projective modules  $M,M'$ for this component, which are isomorphic
after base change to $K$, are themselves isomorphic. 
If we view $M_n(\sZ^{\natural\prime})$ as a $\sZ^{\natural\prime}$-algebra in this claim, and replace $K$ by $K'$, its validity follows from the proof
of the original Corollary \ref{ress}, in the spirit of (a) above.   However, the module categories
for this matrix component, whether viewed as a $\sZ^{\natural}$-  or $\sZ^{\natural\prime}$-algebra,
are the  same. Moreover, the respective base-change functors,
$K \otimes_{\sZ^{\natural}}(-)$  and $K'\otimes _{\sZ^{\natural \prime}}(-)$ are naturally isomorphic,
even in the larger category of $\sZ^{\natural\prime}$-modules. In view of these identifications,
the cited argument proves the claim, and the  version of Corollary \ref{yahoo} in new setting follows.
 
 Corollary \ref{irred} follows in the new setting from its current proof together with the decomposition of $\sJ^\natural$ given in (\ref{decomp}) above. The above identification of $K-$ and $K'-$base changes for
 $M_n(\sZ^{\natural\prime})$-modules is also useful here.
 \end{rems}
 
Note, too, that specializing $t\mapsto 1$ in the expression (\ref{decomp}) gives an analogous and very precise version of Remark \ref{Corofmaxorder} in the $^2F_4$ case.
 
 Finally, we alert the reader that Theorem \ref{tjf} will be needed in Section 4.3.

 %Similarly, the other results in \S8 have analogues in the setting of this Appendix. Let $\sZ'=\mathbb Z'[t,t^{-1}]$. The
% involution $\dagger$ on $\sH$ extends to an involution, still denoted $\dagger$, on $\sH'$ by base change from
 %$\sZ$ to $\sZ'$. Next, $\sJ':=\sJ_{\mathbb Z'}= \mathbb Z'\otimes_{\mathbb Z}\sJ$.  Also, 
 %$$\sJ_{\sZ'}=\sZ'\otimes_{\mathbb Z'}{\sJ'}\cong\sZ'\otimes_{\mathbb Z}\sJ.$$
 %We can also write $(\sJ_\sZ)'=\sJ_{\sZ'}$. Define
 %$$\begin{cases}
 %\mathbb Z^{\#\prime}={\mathbb Z}[\frac{1}{{2}}][\sqrt{2}]=\mathbb Z[\frac{1}{\sqrt{2}}]\\
%  %\sZ^{\#\prime}= \sZ[\frac{1}{\sqrt{2}}]. \end{cases}
  %$$
%and so $\mathbb Z^{\prime\#}=\mathbb Z^{\#\prime}$, $\sZ^{\prime\#}=\sZ^{\#\prime}$. Next, define
%$$\sJ^{\prime\#}:=\sJ^\prime_{\sZ^{\#\prime}}\cong ({\sJ_{\mathbb Z'})_{\sZ^{\#\prime}}}\cong \sJ_{\sZ^{\prime\#}}.$$
%and put $\sJ^{\#\prime}=\sJ^{\prime\#}.$

%In the discussion of \S8, $\mathbb Z$ is replaced by
%$\mathbb Z'$, $\mathbb Q$ by $\mathbb Q'$, $K$ by $K':=K[\sqrt{2}]$, $\sJ_\sZ$ by $\mathcal K_{\sZ'}$, $\sJ_\sZ$
%by $(\sJ_\sZ)'$, and $\sJ^\#$ by $\sJ^{\prime\#}$.

%\begin{rem}\label{newmaxorder} 

\section{Left cells and the order $\sJ:=\sZ\otimes_{\mathbb Z} \sfJ$} Recall from (\ref{varpi}) that  $\varpi:\sH\rightarrow \sJ$ is the $\sZ$-algebra homomorphism $\varpi=\phi\circ\dagger$, where $\phi:\sH\to\sJ$ is defined in (\ref{phiformula}) and $\dagger:\sH\to\sH$ is defined in  (\ref{daggerdefn}).\index{$\varpi$, the embedding $\sH\to\sJ$} Given
a $\sJ$-module $M$, there is a unique $\sH$-module $M^\varpi=\{m^\varpi\,|\, m\in M\}$ obtained by letting
$\sH$-act on $M$ through $\varpi$, i.e.,  $hm^\varpi=(\varpi(h)m)^\varpi$, for each $h\in\sH$, $m\in M$.  In other words, $M^\varpi$ is the restriction of the $\sJ$-module $M$ to
$\sH$ through $\varpi$. (Throughout this section, modules are taken to be left modules, unless otherwise indicated.)

One approach to understanding the representation theory
of a generic Hecke algebra $\sH$ is to relate it to the representation theory of the  ``simpler"
asymptotic algebra $\sJ$, by means of the inclusion $\varpi:\sH\rightarrow \sJ$.  For
example, as just explained, any
$\sJ$-module $M$ can be restricted to $\sH$, giving rise to a functor $M\mapsto M^\varpi$. 
Conversely, Theorem \ref{cellmodule} below shows that (under the standing hypotheses on $\sH$ of Section 1.2) for certain $\sH$-modules $S$, there exists a unique (up to isomorphism)
$\sJ$-module $V$ such that $V^\varpi\cong S$. In fact, we prove that this is true whenever $S$ is a finite
direct sum of left-cell modules for $\sH$. 

If bad primes are inverted as in Section 4.1, then $\sZ$ becomes $\sZ^\natural$, and $\sJ$ becomes a maximal (and separable) order $\sJ^\natural$ over $\sZ^\natural$.  See Theorem \ref{maximalorder}, and Remark
\ref{moreremarks}(b) in the case $^2F_4$. This implies that every finite and $\sZ^\natural$-projective $\sJ^{\natural}$-module is a direct sum of finite projective $\sJ^\natural$-modules, each of which becomes irreducible upon base change to $\sJ_{\mathbb Q(t)}$. Such a decomposition is unique up to
isomorphism.  Also, because $\sZ^\natural$ is regular of Krull dimension $\leq 2$ (it is dimension 2, in fact), every irreducible $\sJ_{\mathbb Q(t)}$-module is the base change of a $\sJ^{\natural}$-module which
is projective over $\sZ^\natural$. In fact, applying Swan's theorem, any such module is free
over $\sZ^\natural$. In effect, this provides a new refinement of the set of left-cell modules.

\begin{thm}\label{cellmodule} Let $S$ be a left-cell module (or even a finite direct sum of left-cell modules) for $\sH$.  Then, there exists a $\sJ$-module $M$, unique up to isomorphism, such that $S\cong M^\varpi$.  \end{thm}

\begin{proof}We use several results from \cite[Ch. 18]{Lus03}. Let $*$ be as defined
in \cite[18.10]{Lus03}; it is a left-module action of  $\sJ$ on $\sH$. Now define a new action $\boxdot$ of $\sJ$ on $\sH$ by putting $j\boxdot h:=(j*h^\dagger)^\dagger$, for $j\in \sJ$ and $h\in\sH$. For each integer $a$, let
\begin{equation}\label{sections}
\sH_a=\oplus_{w; a(w)=a}\sZ\scc^\dagger_w\quad{\text{and}}\quad\sH_{\geq a}=\oplus_{w; a(w)\geq a}\sZ\scc_w^\dagger\end{equation}
as in \cite[18.10]{Lus03}, where $a(-)$ is the Lusztig $a$-function defined in \eqref{afun}; see \cite[13.6]{Lus03}.\index{$a$-function} Lusztig observes that $\sfj_x *c^\dagger_w\in\sH_{a(w)}$, for all $x,w\in W$. 
In particular, $\sH_{\geq a}$ is stable under the $*$ action of $ \sJ$. 
He proves 
\cite[18.10(a)]{Lus03} that
$$h\scc_w^\dagger\equiv \phi(h)*\scc_w^\dagger\,\,{\rm mod} \,\,\sH_{\geq a(w)+1},\quad \forall w\in W.$$
Replacing $h$ by $h^\dagger$, and then applying $\dagger$ to both sides, we get 
\begin{equation}\label{starie}h\scc_w=(\phi(h^\dagger)* \scc^\dagger_w)^\dagger=\varpi(h)\boxdot\scc_w\quad{\rm mod}(\sH_{\geq a(w)+1})^\dagger.\end{equation}
 It follows from Lusztig's observation above that $(\sH_{\geq a})^\dagger$ is stable under the $\boxdot$-action of $ \sJ$,
as is $(\sH_{\geq q})^\dagger/(\sH_{\geq a+1})^\dagger$. Observe that $-a(-)$ is a height function in the sense of
Appendix B compatible with $\leq_{LR}$ by \cite[Ch. 14, {P4}]{Lus03}. By Corollary \ref{3.7a} above, the left $\sH$-module $\sH={_\sH\sH}$ has a height filtration with respect to the (compatible and negatively-valued) height function $-a(-)-1$. Note that $(\sH_{\geq a})^\dagger/(\sH_{\geq a+1})^\dagger$ is a section of
such a filtration. The equation (\ref{starie}) shows that this natural $\sH$-module is also the restriction through $\varpi$ of a module for
$ \sJ$, the latter acting via $\boxdot$ on the same $\sZ$-module. 
Since the $\sH$-module $_{\sH}\sH$
belongs to $\scrA(\scrL)$ using Corollary \ref{3.7a}, its filtration sections are direct sums of left-cell modules. Also, it is clear that every left-cell module
for $\sH$ appears as a direct summand of some such section. The theorem now follows from the elementary
Lemma \ref{general}
below, which proves uniqueness (see  part (c)) as well as completing the proof of existence (see  part (d)---the notation differs slightly).\end{proof} 

Temporarily, let $\scrK$ be any Noetherian integral domain, and let $K$ be its fraction field. The following lemma is inspired by the arguments in \cite[below (8.2), p. 109]{Reiner03}. Note the definition of $\mathfrak B$-mod$^f$ is given in part (a). A similar definition applies for $\mathfrak A$-mod$^f$.  Note also that part (b) will be used in the proof of Theorem \ref{quasi-hereditary}.

\begin{lem} \label{general}  Let $\mathfrak A$ be a finite and torsion-free $\scrK$-algebra and let $\mathfrak B$ be an order in the $K$-algebra $\mathfrak A_K$ containing $\mathfrak A$.  Then,

\begin{itemize}\item[(a)] The restriction  functor $\mathfrak B$-mod$\,^f\rightarrow \mathfrak A$-mod defines a fully faithful $\scrK$-linear functor from the $\scrK$-category 
$\mathfrak B$--mod$\,^f$ of $\scrK$-finite and $\scrK$-torsion-free $\mathfrak B$-modules 
to the $\scrK$-category of $\mathfrak A$-modules.

%(b) The restriction
%functor factors through $\mathfrak A$--mod$\,^f$

\item[(b)]  In particular, for any $M\in \mathfrak B$--mod$\,^f$, we have
$$\End_{\mathfrak B}(M)\cong\End_{\mathfrak A}(M|_{\mathfrak A})$$ as $\scrK$-algebras. Also, 
$M|_{\mathfrak A}\in\mathfrak A$--mod$\,^f$.

\item[(c)] Suppose that $M,M'\in \mathfrak B$-mod$\,^f$, and let $N,N'$ denote the $\mathfrak A$-modules which are their respective restrictions to $\mathfrak A$.  Then, $M\cong M'$ as $\mathfrak B$-modules if and only if $N\cong N'$ as $\mathfrak A$-modules.

\item[(d)] Suppose $M\in \mathfrak B$-mod$\,^f$. Assume that $N:=M|_{\mathfrak A}$ can be written as
a direct sum $N=N_1\oplus N_2$ in $\mathfrak A$-mod. Then, $M=M_1\oplus M_2$ in $\mathfrak B$-mod$^f$, where
$M_i|_{\mathfrak A}\cong N_i$, $i=1,2$. (Also, $N_1,N_2$ are in $\mathfrak A$--mod$^f$.)\end{itemize}\end{lem}   

\begin{proof} We first prove  part (a). Obviously, restriction is a $\scrK$-linear functor.  We have to prove it induces an isomorphism
on Hom-spaces between objects. 
For $M,M'\in\mathfrak B$--mod$\,^f$, note that $\Hom_{\mathfrak B}(M,M')$ 
identifies with the $\scrK$-submodule of $\Hom_{\mathfrak B_K}(M_K,M'_K)$ consisting of the elements of
the latter that send
$M\subseteq M_K$ to $M'\subseteq M'_K$. A similar identification holds for $\Hom_{\mathfrak A}(M|_{\mathfrak A},
M'|_{\mathfrak A})$ in $\Hom_{\mathfrak A_K}(M_K,M'_K)$. Since $\mathfrak B_K=K\mathfrak B=K\mathfrak A=\mathfrak A_K$, the isomorphism required in  part (a) holds.

Part (b) follows from the argument above with $M=M'$ (using the fact that restriction is a functor, hence it induces an algebra homomorphism).  The fact that $M|_\mathfrak A$ belongs to $\mathfrak A$--mod$^f$ is obvious.   Part (c) also follows from a functorial argument using (a). Finally, (d) follows from the functorial properties of  (a) and (b).  
It is also necessary to observe for  part (d) that idempotents split in $\mathfrak B$--mod$^f$. More precisely, in terms of the
argument for (a), if $e\in\End_{\mathfrak B}(M)$ is an idempotent, with $M\in\mathfrak B$--mod$\,^f$, then $eM$
also belongs to $\mathfrak B$--mod$^f$.
\end{proof}

Now return to the setting of Theorem \ref{cellmodule}, but consider modules over $\sZ^\natural$, rather than $\sZ$. For each $E\in\text{Irr}(\mathbb QW)$, let $E_{{\mathbb Q}(t)}$ be the corresponding irreducible
$\mathbb Q(t)W$-module, regarded as a $\sJ_{\mathbb Q(t)}$-module through the isomorphism
$\mathbb Q(t)W\cong \sJ_{\mathbb Q(t)}$. 
 (In the $^2F_4$ case, $E$ and $E_{\mathbb Q(t)}$ may not be absolutely irreducible.)  Thus, by Remark \ref{Corofmaxorder}(2), $E_{\mathbb Q(t)}\cong L_i$, for some $i$.
 Choose a $\sJ^\natural$-lattice $E_\natural$ in  
 $E_{\mathbb Q(t)}$ that is a projective as a $\sZ^\natural$-module
 (thus free by Swan's theorem \cite[first paragraph, p. 111]{Sw78}). For example, for a type other than $^2F_4$, we may choose $E_\natural\cong L_{i,\sZ^\natural}$, for some $i$, which is $\sZ^\natural$-free in the notation of Remark \ref{Corofmaxorder}(2).   
 Since $\sJ^\natural$ is projective over $\sZ^\natural$, the $\sJ^\natural$-module $E_\natural$ is projective over $\sJ^\natural$; see Corollary \ref{yahoo}(b) and its analog 
 in Remark \ref{moreremarks}(b).
 
 We remark that $E_\natural$ is uniquely determined (up to isomorphism) by Corollary \ref{ress} and
 its twisted $^2F_4$ analog in Remark \ref{moreremarks}(b). 
  Define 
 \begin{equation}\label{identify}
 S^\natural(E):=(E_\natural)^{\varpi^\natural}, \quad S^\natural_{E}:=\Hom_{\sZ^\natural}(S^\natural(E),\sZ^\natural),\end{equation}
 which belong to $\sH^\natural$-mod and mod-$\sH^\natural$, respectively.
% \begin{cor} Let $E$ be an irreducible ${\mathbb Q}W$-module. Then
Note that $S^\natural(E)$ and $S^\natural_{E}$  are natural $\sH^\natural$-forms for the $\sH_{{\mathbb Q}[t,t^{-1}]}$-modules $S(E)$ and $S_E$ defined in \cite[\S3]{DPS15}.  Note also that $S^\natural(E)_{\mathbb Q(t)}\cong E_{\mathbb Q(t)}\cong L_i$.

\begin{rem}\label{ss case}
Since $\{S^\natural(E)_K\mid E\in \text{Irr}(\mathbb QW)\}$, where $K=\mathbb Q(t)$, forms a complete set of irreducible left $\sH_K$-modules (see Remark \ref{Corofmaxorder}(2)), it follows that  $\{(S^\natural_E)_K\mid E\in \text{Irr}(\mathbb QW)\}$ is a complete set of irreducible right $\sH_K$-modules. Thus, since every  $(S_E^\natural)_K$ is an irreducible constituent of $T^+_K$, the set $$\{\Hom_{\sH^\natural_K}((S_E^\natural)_K,T^{+\natural}_K)\mid E\in \text{Irr}(\mathbb QW)\}$$ forms a complete set of  irreducible left $A^+_K$-modules. (Note that $A^+_K=\End_{\sH^\natural}(T^{+\natural})_K$. So, $A^+_K$ is split semisimple if $^2F_4$ is excluded.)
Now, if $\Delta(E):=\Hom_{\sH^\natural}(S_E^\natural,T^{+\natural})$, then $\Delta(E)\subseteq
\Hom_{\sH^\natural_K}((S_E^\natural)_K,T^{+\natural}_K)$ is an $A^{+\natural}$-submodule, which is a torsion-free $\sZ^\natural$-module. Hence, $\Delta(E)_K=K\Delta(E)$ is a (nonzero) $A^+_K$-submodule of $\Hom_{\sH^\natural_K}((S_E^\natural)_K,T^{+\natural}_K)$. Irreducibility (treated above) forces 
 $\Delta(E)_K=\Hom_{\sH^\natural_K}((S_E^\natural)_K,T^{+\natural}_K)$.
 % is a simple $A^+_K$-modules. (Note $\Delta(E)_K$ is a submodule of $\Hom_{\sH^\natural_K}((S_E^\natural)_K,T^{+\natural}_K)$ which is irreducible.) Hence, 
  
  Therefore, $\{\Delta(E)_K\mid E\in \text{Irr}(\mathbb QW)\}$ forms a complete set of irreducible left $A^+_K$-modules. To include the $^2F_4$ case, we may simply replace $\sZ^\natural$ by $\sZ^{\natural\prime}=\sZ^\natural[\sqrt2]$, $K$ by $K'$, and $\Delta(E)$ by $\Delta'(E):=\Delta(E)_{\sZ^{\natural\prime}}$ to get a complete set 
$\{\Delta'(E)_{K'}\mid E\in \text{Irr}(\mathbb Q'W)\}$  of irreducible left $A^+_{K'}$-modules. Here, the fields $K'$ , $\mathbb Q'$ are  obtained by adjoining $\sqrt 2$ to $K$, $\mathbb Q$, respectively.
\end{rem}

%In the next two results, we exclude the $^2F_4$ case for convenience. However, that case can be easily handled by replacing
%$\sZ$ by $\sZ'$ and $\mathbb Q$ by $\mathbb Q'$ as in the Appendix to \S8.

\begin{thm}\label{secondcellmodule} (a) Let $N$ be a finite direct sum of left-cell modules for $\sH$, and let $N^\natural
:=N_{\sZ^\natural}$ be the corresponding (base-changed) module for $\sH^\natural$. If $M$ is a direct summand of
the $\sH^\natural$-module $N^\natural$, then $M$ admits a direct sum decomposition
 $M=\oplus_{i=1}^r M_i$ in $\sH^\natural$-mod such that each $M_i\cong S^\natural(E_i)$, for some
$E_i\in\text{\rm Irr}({\mathbb Q}W)$. 
The modules  $M_i$ and their multiplicities in $M$ are determined up to isomorphism by $M_K$, where 
$K=\mathbb Q(t)$.

(b) Assume that $(W,S, L)$ is of type $^2F_4$ (so the base-change notations $(-)^{\natural\prime}$ and $(-)^\prime$ are defined).  Then, (a) also holds after
substituting the top row for the bottom row in the display (see the appendix to Section 4.1):
\begin{equation}\label{array}\begin{matrix} \sH'\, & \sZ^{\natural\prime}\,& \sH^{\natural\prime}\, &\mathbb Q'\, & K^\prime\,& N^{\natural \prime}\\ 
\sH\, & \sZ^\natural\,& \sH^\natural\, &\mathbb Q \, & K\, & N^\natural \end{matrix}\end{equation}
\end{thm}

\begin{proof} (a) First, we establish some projectivity properties. Note that all left-cell modules for $\sH$ are projective (even free) over $\sZ$.  Thus,  $N$  is projective over $\sZ$, and  both $N^{\natural}$  and its direct summand  $M$  are projective over
 $\sZ^{\natural}$. Consequently, if $M \cong X^{\varpi^\natural}$ (the restriction of $X$ to $\sH^\natural$ through $\varpi_{\sZ^\natural}$) for a given $\sJ^\natural$-module $X$, then $X$ is projective over $\sZ^\natural$. Consequently, $X$ is also projective over
$\sJ^\natural$, by Corollary \ref{yahoo}(b). 

By Theorem \ref{cellmodule}, $N$ is (isomorphic to) the restriction (through $\varpi$) to $\sH$ of a module for $ \sJ$. 
Thus, $N^\natural$ is the restriction to $\sH^\natural$ of a module $Y$ for the order $\nsJ$, and the analogous
statement holds for $M$ by Lemma \ref{general}(d).  (Use $\mathfrak A=\sH^\natural$ and let $\mathfrak B=\varpi_K^{-1}(
\sJ^\natural)$.) %Also, put $X=N^\natural$, $X_1=M,$ and $Y|_{\mathfrak A}=X$.)
Accordingly, let $X$ be a $\sJ^\natural$-module such that  $M$ is the restriction $X^{\varpi^\natural}$, and note that $X$ is projective by the paragraph above. (Alternately, Lemma \ref{general}(d) gives an $X$ which is a direct summand of $Y$, and $Y$ can also be shown to be projective over $\sJ^\natural$.)
 Write
$X_K \cong L_1\oplus\cdots\oplus L_r$, where the $L_i$ are irreducible $\nsJ_K$-modules. Then, $L_i=(L'_i)_K$, for
some $\sJ^\natural$-lattice $L_i'$ in $L_i$. In fact, we may take $L_i\cong (E_i)_{\mathbb Q(t)}$, $E_i\in\text{Irr}(\mathbb QW)$, and $L'_i\cong E_{i\natural}$ as discussed above the theorem. Note also from that discussion that 
$E_{i\natural}$ is projective over $\sJ^\natural$ and restricts though $\varpi^\natural$ to $S^\natural(E_i)$. The direct sum of all of the $E_{i\natural}$ remains projective and has the same base change to $\sJ_K^\natural$ as $X_K$. Using the projectivity of $X$ and Corollary \ref{ress}, we deduce that $X$ is the direct sum of all the
$E_{i\natural}$. Applying the functor $(-)^{\varpi^\natural}$ to $X$ and its summands gives the desired description of $M$ in part (a)---all but the last sentence, which
%Thus, $M_K\cong \oplus_iL_i\cong \oplus S^\natural(E_i)_K$ and $M\cong \oplus_i S^\natural (E_i)$ by Corollary \ref{ress}.The last sentence of part (a) 
is again a consequence of Corollary \ref{ress}.
 The $L'_i$ and their multiplicities are uniquely determined by Corollary \ref{ress} as well. 
 
 This completes the proof of (a), while the proof of (b) is similar, {\it mutatis mutandis}. Remark \ref{moreremarks}(a)
 can be used as a guide.
\end{proof}

\begin{cor}\label{9.4} The following statements hold for $\sH$-mod:
\begin{itemize}
 \item[(a)] Let $N_1,N_2$ be left-cell modules for $\sH$, and suppose $M_1,M_2$ are direct summands of $N_1^\natural,N_2^\natural$, respectively. Then, $M_1\cong M_2$ as $\sH^\natural$-modules if and only if
$(M_1)_{\mathbb Q(t)}\cong (M_2)_{\mathbb Q(t)}$ as $\sH_{\mathbb Q(t)}$-modules.  

\item [(b)] In the type $^2F_4$ case, (a) also holds if changes in notation  analogous to those in Theorem \ref{secondcellmodule}(b) are made. \end{itemize}  \end{cor}

Part (a)  is clear from the above theorem and its preamble, using the split semsimplicity of $\mathbb Q(t)W$ and
algebras isomorphic to it. %, using Lemma \ref{general}(a). 
Part (b) is similar, replacing $\mathbb Q(t)W$ by $\mathbb Q'(t)W$.
%In fact, one can take $N_1, N_2$ to 
%%finite direct sums of left-cell modules. The proof shows that, if $M$ is a direct summand of $N^\natural$, then the
%isomorphism class of the $\sH^\natural$-module $M$ is determined by that of $\sH_K$-module $M_K$.

\section{The quasi-heredity of $A^{+\natural}$} 
Quasi-hereditary algebras (over a field) were first introduced by CPS (E. Cline and the second and third authors) \cite{CPS88} in order to deal with highest weight categories arising naturally in the representation theory of semisimple Lie algebras and algebraic groups. 
%These algebras and their representations may be regarded as a Lie theoretic approach to the representation theory of finite dimensional algebras. 
Examples of quasi-hereditary algebras include
Schur (or, more generally, $q$-Schur) algebras and Auslander algebras which arise rather naturally. With the maximal order theory developed in the previous sections, we are now ready to prove that the base change from $\sZ$ to $\sZ^\natural$ of the standardly stratified algebra $A^+$ is, in fact, (integrally) quasi-hereditary.

We begin with some terminology about (integral) quasi-hereditary algebras.
A good source for much of this material is \cite[\S2]{DPS18}. 
\begin{defn}\label{QHA}
(1) Let $B$ be an algebra over a commutative Noetherian ring $\scrK$.  Assume that $B$ is projective and finite
over $\scrK$.  An ideal $J$ of $B$ is called a {\it heredity
ideal} in $B$ \index{heredity ideal} if the following four conditions hold:
\begin{enumerate}
\item[(HI1)] $B/J$ is projective over $\scrK$;

\item[(HI2)] $_BJ$ is $B$-projective;

\item[(HI3)] $J^2=J$; and

\item[(HI4)]  The endomorphism algebra $E:=\End_B({_BJ)}$ is $\scrK$-semisimple.\footnote{This means that, for all $\mathfrak p
\in\Spec(\scrK)$, $E({\mathfrak p})$ is a semisimple algebra over the residue field $\scrK({\mathfrak p}):=\scrK_{\mathfrak p}/\mathfrak p\scrK_{\mathfrak p}$. Equivalently,
any short exact sequence of left $E$-modules which is split over $\scrK$, is split over $E$. (See \cite[\S2]{CPS90}.)}
\end{enumerate}

(2) The algebra $B$, still assumed to be projective over $\scrK$, is called a  {\it quasi-hereditary algebra} (QHA) \index{quasi-hereditary algebra}
provided 
there exists a finite ``defining sequence" $$0=J_0\subseteq J_1\subseteq \cdots\subseteq J_m=B$$ of ideals in $B$
such that, for $0<i\leq m$, $J_i/J_{i-1}$ is a heredity ideal in $B/J_{i-1}$. 

(3) If, in addition,  every $J_i/J_{i-1}$ is a heredity ideal of {\it split}  type, \index{heredity ideal! $\sim$ of split type} in the sense that
the $\scrK$-semisimple algebra $\End_{B/J_{i-1}}(J_i/J_{i-1})$ is isomorphic to a direct product of $\End_\scrK(P)$, where $P$ is a finite (faithful) $\scrK$-projective module, we call $B$ a {\it split} quasi-hereditary algebra (or a quasi-hereditary algebra of {\it split} type). \index{quasi-hereditary algebra! split $\sim$} 

Notice that, since $P$ is necessarily a projective generator for $\scrK$-mod, $\End_\scrK(P)$ is Morita equivalent to $\scrK$.
\end{defn}

Clearly, a heredity ideal is a stratifying ideal and every quasi-hereditary algebra is a standardly stratified algebra. The ``split'' quasi-hereditary algebras defined above are quite popular and useful. Some mild variations have been collected in Remark \ref{semisplit}
below.

\begin{rem}\label{semisplit}
(1)
Recall that an algebra $E$ is {\it separable} over $\scrK$, provided the multiplication map $E\otimes_\scrK E^\op \rightarrow E$ splits as an $(E,E^\op)$-bimodule map.  By definition,  $E$ is {\it semisplit} if it is a direct product of algebras, each of which is separable with center $\scrK$ (i.e., an Azumaya algebra).  If the $\scrK$-semisimple algebra $E:=\End_B({_BJ})$ in (HI4) is separable (respectively, semisplit) over $\scrK$, we say that $J$ is of {\it separable} (respectively, {\it semisplit) type} over $\scrK$.  \index{heredity ideal! $\sim$ of separable type} \index{heredity ideal! $\sim$ of semisplit type} 

Then,  if each $J_i/J_{i-1}$ is of separable (respectively, semisplit) type over $\scrK$, we say that the quasi-hereditary algebra $B$ is of {\it separable} (respectively, {\it semisplit}) {\it type} over $\scrK$. See \cite{CPS90} for more details.

(2) Following \cite[Def. 4.1.4]{Ro08}, a heredity ideal $J$ of split type is said to be {\it indecomposable}, or $J$ is an {\it indecomposable split heredity ideal}, \index{heredity ideal! indecomposable split $\sim$} 
if (HI4) is replaced by (HI4)$'$:
\begin{itemize}
\item[(HI4)$'$] The algebra $E:=\End_B({_BJ)}$ is Morita equivalent to $\scrK$.
\end{itemize} 

Rouquier established a very nice relation between indecomposable split heredity ideals $J$ and ``standard (projective) $B$-modules'' $L$ that are $\scrK$-faithful and satisfy one of the equivalent conditions in \cite[Lem. 4.5]{Ro08}:

Given $L$, the coresponding $J$ satisfies the following $B$-$B$-bimodule isomorphism 
\begin{equation}\label{Roiso}
J\cong L\otimes_{\End_B(L)}\Hom_B(L,A),
\end{equation} while, given $J$, the corresponding $L$ satisfies the left $B$-module isomorphism
\begin{equation}\label{CPS}
L\cong J\otimes_{\End_B(J)}P,
\end{equation} where $P$ is a progenerator for $\End_B(J)$ with $\End_{\End_B(J)}(P)\cong\scrK$. 

The construction of $L$ from $J$ in \eqref{CPS} is also given in \cite[p.157]{CPS90}. This, in particular, implies that $\End_B(L)\cong \scrK$.
We will call the isomorphism \eqref{Roiso} {\it Rouquier's isomorphism}. \index{Rouquier's isomorphism}This isomorphism is useful in the construction of a standard basis for the split quasi-hereditary algebra; see Section 5.1.

(3) Note also that the definition of split quasi-hereditary algebras in \cite[4.1.4]{Ro08},  using indecomposable split heredity ideals (\cite[Def. 4.1]{Ro08}) and a poset, is stronger (at least in formulation) than Definition \ref{QHA}(3).

%\index{heredity ideal! split $\sim$}

%If, in addition, each factor is the endomorphism algebra of a finite projective $k$-module, then $E$ is said to be split. The conditions (separable, semisplit, and split) on $E$ each imply $k$-semisimplicity.

%{\color{blue}We remark that quasi-hereditary of {\it separable} or {\it semisplit} type may also be similarly defined. \index{quasi-hereditary algebra! semisplit $\sim$}  \index{quasi-hereditary algebra! separable $\sim$} See \cite{CPS90} for more details.}
\end{rem}

The following result is needed in the proof of the theorem below and is obtained by adapting an argument of Hochster \cite{Ho07} (compare also \cite[Prop.~11, p.23]{Bo72}).
\begin{lem}\label{Hoch} Let $\cH$ be a $\sZ$-algebra which is finite and projective as a $\sZ$-module, and let $M,N$ be finite left $\cH$-modules. If $M$ is a finitely presented $\cH$-module, then there exists a $\sZ^\natural$-module isomorphism
\begin{equation}\label{phiNN}
\phi^{\sZ^\natural,\cH}_{M,N}:\sZ^\natural\otimes_\sZ\Hom_\cH(M,N)\overset\sim\longrightarrow\Hom_{\cH^\natural}(M^\natural,N^\natural),
\end{equation}
defined, for $u\in\sZ^\natural$ and $f\in\Hom_\cH(M,N)$, by
$$\phi^{\sZ^\natural,\cH}_{M,N}(u\otimes f)(v\otimes m)=uvf(m)\;\;(v\in\sZ^\natural, m\in M).$$
Moreover, when $M=N$, $\phi^{\sZ^\natural,\cH}_{N,N}:\sZ^\natural\otimes \End_\cH(N)\to\End_{\cH^\natural}(N^\natural)$ is a natural isomorphism of $\sZ^\natural$-algebras.
\end{lem}
\begin{proof} Let $\phi=\phi^{\sZ^\natural,\cH}_{M,N}$. 
It is helpful to think of $\phi$ as a composition $\nu\circ\psi$ of two simpler maps $\nu$ and $\psi$. The map 
 $$\psi:\sZ^\natural\otimes_\sZ\Hom_\cH(M,N)\longrightarrow
\Hom_\cH(M,N^\natural)$$
 is given by $\psi(u\otimes f)(m)=uf(m)$ ($u\in\sZ^\natural$, $f\in\Hom_\cH(M,N)$, $m\in M$). We will discuss $\nu$ in a moment, but, for the next two paragraphs, we continue to focus on $\psi$.
  
 We will, however, let $M$ vary. First, take $M={}_\cH\cH$, the left $\cH$-module which is $\cH$ itself. We abbreviate $_\cH\cH$ to $\cH$ when usage is clear from context. There are well-known isomorphisms $\Hom_\cH(\cH,N)\cong N$ (here $N$ could be any left $\cH$-module) in which $f\in\Hom_\cH(\cH,N)$ corresponds to $f(1)\in N$. Equivalently, $n\in N$ corresponds to $f_n\in\Hom_\cH(\cH,N)$, where $f_n(h)=hn$ ($h\in\cH$). Using these isomorphisms and correspondences, we obtain a commutative diagram 
 $$\begin{CD}
 \sZ^\natural\otimes N @>=>> N^\natural\\ 
 @VVV @AAA \\
\sZ^\natural\otimes\Hom_\cH(\cH,N) @>\psi>> \Hom_\cH(\cH,N^\natural).\end{CD}
$$
with the top arrow equality and each vertical map an isomorphism. As a consequence, the map $\psi$ is an isomorphism, under the $M={}_\cH\cH$ hypothesis. This carries over to the weaker assumption that $M$ is a direct sum of finitely many copies of the $\cH$-modules $_\cH\cH$. Finally, $\psi$ is an isomorphism whenever $M$ is a finitely presented over $\cH$.

We sketch the argument for this last assertion, which follows the spirit of \cite{Ho07}: Let $E\to G\to M\to 0$ be a presentation of $M$ in terms of finite direct sums $E$ and $G$ of the $\cH$-modules $_\cH\cH$. Use it to form the columns of the diagram below
$$\begin{CD}
\sZ^\natural\otimes\Hom_\cH(E,N) @>>> \Hom_\cH(E,N^\natural)\\ 
 @AAA @AAA \\
\sZ^\natural\otimes\Hom_\cH(G,N) @>>> \Hom_\cH(G,N^\natural)\\
@AAA@AAA\\
\sZ^\natural\otimes\Hom_\cH(M,N)@>>> \Hom_\cH(M,N^\natural)\\
@AAA@AAA\\
0@.0
\end{CD}
$$
The columns are exact, by virtue of partial exactness properties of $\Hom_\cH$, together with flatness of $\sZ^\natural$ over  $\sZ$. The bottom row is $\psi$ and the other two rows are  the versions of $\psi$ obtained by using $G$ and $E$, respectively, in the role of $M$. The squares are 
both commutative, as can be checked with the formulas for $\psi$ which apply in each row. Finally, since the top two rows are isomorphisms, we conclude that the bottom row is an isomorphism (by the 5-Lemma), as asserted.  

We now return to the map $\nu:\Hom_\cH(M,N^\natural)\to\Hom_{\cH^\natural}(M^\natural,N^\natural)$ and define it as follows. First, define 
$$\tilde \nu: \Hom_\sZ(M,N^\natural)\longrightarrow\Hom_{\sZ^\natural}(M^\natural,N^\natural)$$
 on $f\in \Hom_\sZ(M,N^\natural)$ by $\tilde\nu(f)(v\otimes f)=vf(m)$\;\;($m\in M,v\in\sZ^\natural$). From this formula, it is easily seen that $\tilde\nu(f)\in \Hom_{\sZ^\natural}(M^\natural,N^\natural)$. Moreover, if $f$ belongs to the Hom space $\Hom_\cH(M,N^\natural)$, then $\tilde\nu(f)$ belongs to $\Hom_{\cH^\natural}(M^\natural,N^\natural)$. Starting from $f\in\Hom_{\cH}(M,N^\natural)$, we have $hf(m)=f(hm)$, for $h\in\cH,m\in M$. If $v\in\sZ^\natural$, we have
$$h\tilde\nu(f)(v\otimes m)=hvf(m)=vhf(m)=vf(hm)=\tilde\nu(f)(v\otimes hm)=\tilde\nu(f)h(v\otimes m).$$
Thus, $\tilde\nu(f)\in\Hom_{\cH}(M^\natural,N^\natural)$. We already know that $\tilde\nu(f)\in\Hom_{\sZ^\natural}(M^\natural,N^\natural)$, so now we have $\tilde\nu(f)\in\Hom_{\sZ^\natural\cH}(M^\natural,N^\natural)=\Hom_{\cH^\natural}(M^\natural,N^\natural)$. At this point, we can (and do) define $\nu$ to be the restriction of $\tilde\nu$ to $\Hom_\cH(M,N^\natural)$:
$$ \nu: \Hom_\cH(M,N^\natural)\longrightarrow\Hom_{\cH^\natural}(M^\natural,N^\natural).$$

Our next aim is to show that $\nu$ is an isomorphism, assuming $M$ is finitely presented (as we assumed when dealing with $\psi$). As before, the case $M={}_\cH\cH$ (which we abbreviate to $\cH$) leads to a standard identification $\Hom_\cH(\cH,N^\natural)\cong N^\natural$, and there is a similar identification $\Hom_{\cH^\natural}(\cH^\natural,N^\natural)\cong N^\natural$. A calculation, much as we carried out for $\psi$, leads to a commutative diagram
 $$\begin{CD}
  N^\natural @>=>> N^\natural\\ 
 @VVV @VVV \\
\Hom_\cH(\cH,N^\natural) @>\nu>> \Hom_{\cH^\natural}(\cH^\natural,N^\natural).\end{CD}
$$
We provide a few more details about the diagram and its commutativity. Here, the two vertical maps are ``standard'', sending $n^\natural\in N^\natural$ to the unique function in the Hom space drawn below it, and taking the value $n^\natural$ at 1. We let $f_{n^\natural}\in\Hom_\cH(\cH,N^\natural)$ be the unique function on the left, and $f_{n^\natural}^\natural\in \Hom_{\cH^\natural}(\cH^\natural,N^\natural)$ on the right (keeping $n^\natural$ the same). Then, $\nu(f_{n^\natural})(1\otimes 1)=f_{n^\natural}(1)=n^\natural$. Consequently, $\nu(f_{n^\natural})=f_{n^\natural}^\natural$, and the diagram (which has equality as its top row, and the map $\nu$ as its bottom row) is now shown to be commutative.
Since the columns are isomorphisms, as is the top row, it follows that the bottom row $\nu$ is also an isomorphism. 

We can also use a similar argument to show that $\nu$ is an isomorphism when $M$ is a direct sum $\cH^{\oplus r}$ of finitely many copies of $\cH$ (say $r$ of them). In the diagram used for $M=\cH$, replace the variable $\cH$ inside the left ``Hom'' term with $\cH^{\oplus r}$ and replace $\cH^\natural$ inside the right ``Hom'' term with $(\cH^\natural)^{\oplus r}$. Both terms $N^\natural$ in the top row are replaced by $(N^\natural)^{\oplus r}$. Remaining details, showing $\nu$ is an isomorphism in this $M=\cH^{\oplus r}$ case, are left to the reader.

We can now show that $\nu$ is an isomorphism whenever the associated $\cH$-module $M$ has a presentation $E\to G\to M\to 0$ by finite direct sums $G$ and $E$ of copies of $\cH$. The relevant diagram is displayed below.
$$\begin{CD}
\Hom_\cH(E,N^\natural) @>>> \Hom_{\cH^\natural}(E^\natural,N^\natural)\\ 
 @AAA @AAA \\
\Hom_\cH(G,N^\natural) @>>> \Hom_{\cH^\natural}(G^\natural,N^\natural)\\
@AAA@AAA\\
\Hom_\cH(M,N^\natural)@>>> \Hom_{\cH^\natural}(M^\natural,N^\natural)\\
@AAA@AAA\\
0@.0
\end{CD}
$$
The right column is obtained by applying $\sZ^\natural\otimes(-)$ and, then, $\Hom_{\cH^\natural}(-,N^\natural)$ to $E\to G\to M\to 0$. The left column is obtained similarly,  omitting the $\sZ^\natural\otimes(-)$ step. Both columns are exact. The rows are given by the maps $\nu$ associated with $M,G,E$ on their respective rows. By our previous discussion, this implies that the top two rows are isomorphisms. Both squares are commutative, as can be checked with the formulas for $\nu$ that apply. By the 5-Lemma, the bottom row $\nu$ is also an isomorphism.

It now follows that the composite 
$$\phi=\nu\circ\psi: \sZ^\natural\otimes_\sZ\Hom_\cH(M,N)\overset\sim\longrightarrow\Hom_{\cH^\natural}(M^\natural,N^\natural)$$ is an isomorphism, proving the first assertion of the lemma. 

We now assume $M=N$ and prove the last assertion for 
$\phi=\phi_{N,N}^{\sZ^\natural,\cH}$. We begin by showing that $\phi$, as a $\sZ^\natural$-linear map, induces a natural homomorphism of $\sZ^\natural$-modules between its domain $\sZ^\natural\otimes\Hom_\cH(N,N)$ and its image $\Hom_{\cH^\natural}(N^\natural,N^\natural)$. More precisely, for any $u\in\sZ^\natural$, we show that $\phi(u\otimes 1_N)$ is an element of 
$\Hom_{\cH^\natural}(M^\natural,N^\natural)=\Hom_{\cH^\natural}(\sZ^\natural\otimes N,\sZ^\natural\otimes N)$ that acts on $\sZ^\natural\otimes N$ via multiplication by $u\otimes 1_N$: If $v\in\sZ^\natural$ and $n\in N$, then
$$\phi(u\otimes 1_N)(v\otimes n)=uvn=(u\otimes 1_N)(v\otimes n),$$
proving the naturality assertion.
Since $v$ and $n$ were arbitrary, this equation proves the assertion regarding $\phi(u\otimes 1_N)$. 

To establish the isomorphism of $\sZ^\natural$-algebras, we need only show that $\phi$ preserves multiplication. Let $u_1,u_2\in\sZ^\natural$ and $f_1,f_2\in\Hom_\cH(N,N)$. Then, for $v\in\sZ^\natural$ and $n\in N$, $\phi((u_1\otimes f_1)(u_2\otimes f_2))(v\otimes n)=\phi(u_1u_2\otimes f_1f_2)(v\otimes n)=u_1u_2vf_1f_2(n)$, and 
$\phi(u_1\otimes f_1)\cdot\phi(u_2\otimes f_2)(v\otimes n)=\phi(u_1\otimes f_1)(u_2v\otimes f_2(n))=u_1u_2vf_1f_2(n)$, as desired. This completes our arguments for $\phi=\phi_{N,N}^{\sZ^\natural,\cH}$ and, hence, the proof of the lemma. 
\end{proof}

We now prove
a result which strengthens Theorem \ref{SSAA+} upon localizing $\sZ$ away from bad primes. 
In the $^2F_4$ case,
we take $2$ as the only bad prime (thus, $\sZ^\natural:=\sZ [1/2]$ in this case). We maintain the notation of Theorem \ref{4.4}, so, in particular,
$A^+=\End_H(T^+)$ is a certain endomorphism algebra over $\sZ$ which has a stratifying system $\{\Delta(\omega),
P(\omega)\}_{\omega\in\Omega}$. Now, base change to $\sZ^\natural$ or $\sZ^{\natural\prime} :=\sZ^\natural [\sqrt 2]$, and set $A^{+\natural}:=\sZ^\natural\otimes
A^{+}$, $A^{+\natural\prime}=\sZ^{\natural\prime}
\otimes A^+$. %  where $\sZ^{\natural\prime} :=\sZ^\natural [\sqrt 2].$ 

\begin{thm}\label{quasi-hereditary} %\begin{itemize}
{\rm(a)} The algebra $A^{+\natural}$ over $\sZ^\natural$ is quasi-hereditary of separable type in all cases.

{\rm(b)} Except in the case of 
 $^2F_4$, the algebra $A^{+\natural}$ over $\sZ^\natural$ is quasi-hereditary of split type. 
 
{\rm(c)} In the $^2F_4$ case, the algebra $A^{+\natural\prime}$ over $\sZ^{\natural\prime}$ is quasi-hereditary of split type.
%\end{itemize}
\end{thm}

\begin{proof} First, we prove (b). Thus, we are not in the case of $^2F_4$.  It has already been shown in Theorem \ref{4.4} that $A^+$ has a stratifying system $\{\Delta(\omega),P(\omega)\}_{\omega\in\Omega}$, where
$\Delta(\omega)=\Hom_\sH(S_\omega,T^+)$ and $P(\omega)=\Hom_\sH(T_\omega^+,T^+)$. Likewise,
Theorem \ref{SSAA+} gives us that $A^+$ is standardly stratified with the standard stratification
$$0=J_0\subseteq J_1\subseteq\cdots\subseteq J_m=A^+,$$
where $J_i=A^+_{j(i)}$ is an idempotent ideal, which is the image of the submodule $P_{j(i)}:=\text{trace}_P\big(\bigoplus_{\htt(\mu)\geq j(i)}P(\mu)\big)$ of $P=\oplus_{\omega\in\Omega}P(\omega)$ (see \eqref{P_j}) under the left $A^+$-module isomorphism $A^+\cong P$. Thus, each $J_i/J_{i-1}$ is an ideal of $A^+/J_{i-1}$ satisfying (HI1)--(HI3).

It remains to study (HI4) under the base change to $\sZ^\natural$ from $\sZ$, but we will not do that yet. Instead, we study some further interesting properties that do not require base change.

The first is the isomorphism
$$\Hom_{\sH}(S_\mu,T^+_\la)\overset\sim\longrightarrow\Hom_{A^+}(P(\la),\Delta(\mu)),$$
obtained (in greater generality) at the beginning of the proof of Theorem \ref{thm2.5}. Restricting $\la$ and $\mu$ to having the same height, we quickly find that $\Hom_{A^+}(P(\la),\Delta(\mu))\cong \Hom_{A^+}(\Delta(\la),\Delta(\mu))$. (Use (SS3) and (SS1).) Next, we observe an isomorphism $\Hom_{\sH}(S_\mu,S_\la)\cong\Hom_{\sH}(S_\mu,T^+_\la)$ from the exact sequence $0\to S_\la\to T^+_\la\to T^+_\la/S_\la\to0$ in mod-$\sH$ and the vanishing of $\sZ$-torsion-free module $\Hom_\sH(S_\mu,T^+_\la/S_\la)$. (Base change from $\sZ$ to $\mathbb Q(t)$ to get the vanishing. Then, observe that the kernel of the restricting map $\Hom_\sH(S_\mu,T^+_\la)\to\Hom_\sH(S_\mu,T^+_\la/S_\la)=0$ is the image of the injective map $\Hom_\sH(S_\mu,S_\la)\to\Hom_\sH(S_\mu,T^+_\la)$, which is, therefore, both injective and surjective.) Putting together isomorphisms and inverses of isomorphisms (which are also isomorphisms), 
we obtain the ($\sZ$-linear) composite isomorphism, for $\la,\mu$, of the same height:
$$\Hom_\sH(S_\mu,S_\la)\overset\sim\to \Hom_\sH(S_\mu,T^+_\la)\overset\sim\to \Hom_{A^+}(P(\la),\Delta(\mu)) \cong\Hom_{A^+}(\Delta(\la),\Delta(\mu)).$$
%$$ \Hom_{A^+}(\Delta(\la),\Delta(\mu))\cong  \Hom_{A^+}(P(\la),\Delta(\mu))\cong \Hom_\sH(S_\mu,T^+_\la)\cong \Hom_\sH(S_\mu,S_\la).$$
 Reading 
 from left to right, these isomorphisms fit in nicely with the commutative diagram
 $$\begin{CD}
 \Hom_\sH(S_\mu,T^+_\la) @>>> \Hom_{A^+}(P(\la),\Delta(\mu))\\ 
 @AAA @AAA \\
\Hom_\sH(S_\mu,S_\la) @>>> \Hom_{A^+}(\Delta(\la),\Delta(\mu)).\end{CD}
$$
Here, the top row and vertical columns are isomorphisms we have seen before, in which the map in the bottom row arises from the ($\sZ$-linear) contravariant functor $(-)^\diamond=\Hom_\sH(-,T^+)$ from mod-$\sH$ to $A^+$-mod. We leave it to the reader to check that the diagram is commutative.   

The remarkable consequence is that the bottom row map is an isomorphism. Preceding it with the isomorphism
$\Hom_\sH(S(\la),S(\mu))=\Hom_\sH((S_\la)^*,(S_\mu)^*)\cong\Hom_\sH(S_\mu,S_\la),$  it remains an isomorphism. In other words, the composite is a $\sZ$-linear isomorphism
$$\Hom_\sH(S(\la),S(\mu))\overset\sim\longrightarrow \Hom_{A^+}(\Delta(\la),\Delta_(\mu)),$$
now the specialization of a covariant $\sZ$-linear functor from $\sH$-mod to $A^+$-mod. (The functor in question is the composite $(-)^\diamond\circ(-)^*$,  which is being ``specialized'' in its domain category, to the full subcategory of $\sH$-mod formed by  the objects $S(\tau)$ with $\tau$ of the same height as $\la, \nu$.)

Now, return to the $A^+$-module $J_i/J_{i-1}$, which we have shown in \eqref{JiJi-1} to be isomorphic to $A^+_{j(i)}/A^+_{j(i)+1}$. In turn, the latter $A^+$-module is a direct sum, with various multiplicities, of modules $\Delta(\omega)$, with $\omega\in\Omega$ having height $\htt(\omega)=j(i)$. Ignoring uniqueness issues, fix a decomposition 
$$J_i/J_{i-1}\cong \bigoplus_{\omega\in\Omega,\, \htt(\omega)=j(i)}\Delta(\omega)^{\oplus n_\omega},$$
where $j(1)>j(2)>\cdots >j(m)$ by \eqref{j(i)}.
Using the same multiplicities $n_\omega$ of $\Delta(\omega)$ as a direct summand, as in the formula above, define a module $N$ for $\sH$ by the formula
\begin{equation}\label{LCmodule}
N:= \bigoplus_{\omega\in\Omega,\, \htt(\omega)=j(i)}S(\omega)^{\oplus n_\omega}.
\end{equation}
The reader will immediately notice that we have an isomorphism 
$\Hom_\sH(N,N)\overset\sim\to\Hom_{A^+}(J_i/J_{i-1},J_i/J_{i-1})$, via the functor $(-)^\diamond\circ(-)^*$, as discussed above. More is true, however. Because the isomorphism is induced by a $\sZ$-linear functor, it induces an isomorphism of $\sZ$-algebras 
\begin{equation}\label{NvzJ}
\End_\sH(N)\overset\sim\longrightarrow\End_{A^+}(J_i/J_{i-1})\cong\End_B(J),
\end{equation}
where, for notational simplicity in the sequel,  we set $J=J_i/J_{i-1}$ and $B=A^+/J_{i-1}$ when $i$ is fixed.
%The notation on the right may be modified in various ways. To review, $A^+$ may be replaced by $A^+/J_{i-1}$ (which is $B$) and $J_i/J_{i-1}$ may be replaced by $J$ (when $i$ is fixed).

We now make a base change from $\sZ$ to $\sZ^\natural$ to find a verifiable version of (HI4). Note that the base-changed version of (HI1), (HI2), and (HI3) all hold, since base change takes projectives over $\sZ$ (or $B$) to projectives over $\sZ^\natural$ (or $B^\natural$). Also, base change takes idempotent ideals to idempotent ideals.

It remains to consider the base-changed version of (HI4). Here, we have to use the flatness of $\sZ^\natural$ over $\sZ$ (which we have) and the existence of finite presentations of $N$ (and $J=J_i/J_{i-1}$) over $\sH$ (and $B=A^+/J_{i-1}$). We may apply Lemma \ref{Hoch}
to the $\sZ$-algebras $\cH=\sH$ and $B$, and the $\sH$-module $N$ and $B$-module $J$, respectively. We thus obtain algebra isomorphisms:
\begin{equation}\label{phiJJ}
\aligned
\phi_{N,N}^{\sZ^\natural,\sH}&:\sZ^\natural\otimes\End_\sH(N)\overset\sim\longrightarrow \End_{\sH^\natural}(N^\natural)\\
\phi_{J,J}^{\sZ^\natural,B}&:\sZ^\natural\otimes\End_B(J)\overset\sim\longrightarrow \End_{B^\natural}(J^\natural).
\endaligned
\end{equation}
Combining the two isomorphisms with the one in \eqref{NvzJ} gives  a string of $\sZ^\natural$-algebra isomorphisms:
\begin{equation}\label{stringiso}
\End_{B^\natural}(J^\natural)\overset\sim\longrightarrow\sZ^\natural\otimes\End_B(J)\overset\sim\longrightarrow
\sZ^\natural\otimes\End_\sH(N)\overset\sim\longrightarrow
\End_{\sH^\natural}(N^\natural).
\end{equation}
Reading from left to right, the first isomorphism is the inverse of $\phi_{J,J}^{\sZ^\natural,B}$ as given in \eqref{phiJJ},
%$$\phi_{J,J}^{\sZ^\natural,B}:\sZ^\natural\otimes\End_B(J)\overset\sim\to \End_{B^\natural}(J^\natural),$$
 the second is the inverse of the base change from $\sZ$ to $\sZ^\natural$ of \eqref{NvzJ}, %$\End_\sH(N)\overset\sim\to\End_B(J)$, 
 the latter map arising from our construction of $N$ using left-cell modules  (see \eqref{NvzJ}). The third and last isomorphism in the string is $\phi_{N,N}^{\sZ^\natural,\sH}$, the map in \eqref{phiNN}, for $M=N$. To complete the proof in progress of Theorem \ref{quasi-hereditary}(b), we must show that the $\sZ^\natural$-algebra $\End_{B^\natural}(J^\natural)$ on the left is a direct sum (categorical direct product) of (a finite number of) full matrix algebras over $\sZ^\natural$. By the string of isomorphisms above, it is enough to prove that property for the $\sZ^\natural$-algebra $\End_{\sH^\natural}(N^\natural)$ on the right.
 
 Applying Theorem \ref{cellmodule}, we find that $N^\natural$ is the restriction through $\varpi^\natural$ (see \eqref{8.0.37}) of a module $N^\natural_{\text{lift}}$ for $\sJ^\natural$. Of course $\End_{\sH^\natural}(N^\natural)$ is isomorphic to $\End_{\varpi^\natural\sH^\natural}(N^\natural_{\text{lift}})$, where the latter endomorphism algebra is isomorphic to $\End_{\sJ^\natural}(N^\natural_{\text{lift}})$ by the Reiner observation Lemma \ref{general}(b).

 Next, note from Theorem \ref{maximalorder}(b) that $\sJ^\natural$ is isomorphic to a direct product of full matrix algebras $M_{m_i}(\sZ^\natural)$, for a finite sequence $m_1,m_2,\ldots,m_s$ of positive integers; see \eqref{maximalorder1}. There is a corresponding decomposition of the identity $1\in\sJ^\natural$ into a sum $1=e_1+e_2+\cdots+e_s$ of pairwise orthogonal central idempotents $e_i$ of $\sJ^\natural$  with $e_i\sJ^\natural=\sJ^\natural e_i\cong
M_{m_i}(\sZ^\natural)$, $i=1,2,\ldots,s$.   There is also a decomposition of the $\sJ^\natural$-module $N_{\text{lift}}^\natural$ into $\oplus_ie_i N_{\text{lift}}^\natural$, each $e_i N_{\text{lift}}^\natural$ is a left $\sJ^\natural$-module, the latter acting through $e_i\sJ^\natural$ (and all $e_j\neq e_i$ acting as 0). Clearly, $\Hom_{\sJ^\natural}(e_iM,e_jM)=0$, for all left $\sJ^\natural$-modules $M$, whenever $i\neq j$. Thus,
$\End_{\sJ^\natural}(M)\cong\bigoplus_i\End_{e_i\sJ^\natural}(e_iM)$. (Here and below, we use $\bigoplus$ instead of $\Pi$, with some abuse of notation.) Taking $M=N_{\text{lift}}^\natural$ and using the previous paragraph, 
we have the reduction (all through $\sZ^\natural$-algebra isomorphisms)
\begin{equation}\label{Reiner}
\End_{\sH^\natural}(N^\natural)\cong \End_{\sJ^\natural}(N_{\text{lift}}^\natural)\cong\bigoplus_{i=1}^s\End_{e_i\sJ^\natural}(e_iN_{\text{lift}}^\natural).
\end{equation}

It is useful to note at this point that $N_{\text{lift}}^\natural$ is a projective $\sJ^\natural$-module, since it is projective (even free) as $\sZ^\natural$-module. See Corollary \ref{yahoo}(b). It follows easily that each $\sJ^\natural$-module $e_jN_{\text{lift}}^\natural$ is projective. Recall that the $\sZ^\natural$-algebra $e_i\sJ^\natural$ is isomorphic to the full matrix algebra $M_{m_i}(\sZ^\natural)$, well-known to be Morita equivalent to $\sZ^\natural$. 
A $\sZ^\natural$-linear Morita equivalence which achieves this is a multiplication by an idempotent $f_i\in e_i\sJ^\natural$, corresponding to an  idempotent matrix unit $e_{1,1}$ in $M_{m_i}(\sZ^\natural)$, on $e_i\sJ^\natural$-modules. %$M_{m_i}(\sZ^\natural)$-modules. 
Such an equivalence takes projectives to projectives and preserves endomorphism algebras up to $\sZ^\natural$-algebra isomorphism.

Since (finitely generated) projective $\sZ^\natural$-modules are free (see Remark \ref{2.5}(c)), the $\sZ^\natural$-module $f_ie_iN_{\text{lift}}^\natural$ is free and isomorphic to ${\sZ^\natural}^{\oplus n_i}$, for some integer $n_i\geq0 $, depending on $e_iN_{\text{lift}}^\natural$. 
We mention here that the ring $\sZ^\natural$ used as the base ring here is $e_{1,1} M_{m_i}(\sZ^\natural)e_{1,1}\cong f_ie_i\sJ^\natural f_i$, the target of the Morita equivalence given by left multiplication by $f_i$. We find now that the Morita equivalence  induces an isomorphism
\begin{equation}\label{Reiner1}
\End_{e_i\sJ^\natural}(e_iN_{\text{lift}}^\natural)\cong\End_{\sZ^\natural}({\sZ^\natural}^{\oplus n_i})\cong M_{n_i}(\sZ^\natural).
\end{equation}
 Finally, combining these isomorphisms with many others demonstrated above, tracing our way from the left to the right of \eqref{stringiso} and  \eqref{Reiner} and to the direct sum over $i$ of \eqref{Reiner1}, we obtain (again with some abuse of notation, using $\bigoplus$ instead of $\Pi$)
 \begin{equation}\label{Reiner2}
 \End_{B^\natural}(J^\natural)\cong \bigoplus_{i}M_{n_i}(\sZ^\natural).
 \end{equation}
 This establishes (HI4) %{\color{blue}Hence, $A^{+\natural}$ is quasi-hereditary of split type with the defining sequence (as promised in \eqref{JJseq})
% \begin{equation}\label{JJseq}
%$$0=J_0^\natural\subseteq J_1^\natural\subseteq\cdots\subseteq J_m^\natural=A^{+\natural}.$$}
%\end{equation}
and  completes the proof of part (b).

Note that (a) follows from (b) except in the $^2F_4$ case. Assume we are in the latter case. After making the same changes in the appendix to
Section 4.1 ($\sZ\rightarrow\sZ'$, $\mathbb Q\rightarrow \mathbb Q'$, etc.), the above argument shows, {\it mutadis mutandis,}
that $A^{+\natural\prime}$ is quasi-hereditary of split type. In particular, it is of separable type.  Therefore,
as a faithfully flat descent, Theorem \ref{tjf} implies that $A^{+\natural}$ itself is of separable type. (Note that each
$J_i^{\natural\prime}/J_{i-1}^{\natural\prime}$ is the base change of the projective module $J_i^\natural/J^\natural_{i-1}$, hence its endomorphism algebra is a corresponding base change of the latter module. Also, note  that
$\sZ^{\natural\prime}$ is faithfully flat over $\sZ^\natural$.) This not only completes the proof of (a) but also establishes (c).
 \end{proof}

We remark that the heredity ideals $J_i^{\natural}/J_{i-1}^{\natural}$ involved in the defining sequence above are not indecomposable in Rouquier's sense (see Remark \ref{semisplit}(2)).
The reader may ask if the defining sequence in the theorem can be refined so that only indecomposable split heredity ideals will occur. \index{heredity ideal! indecomposable split $\sim$} This is possible due to the fact that each $\Delta^\natural(\omega)$ is a direct sum of various $\Delta(E)$'s (Theorem \ref{secondcellmodule}), with the $E$'s all of the same height as $\omega$. Recall $\Lambda:=\irr(\mathbb QW)$ and the preorder $\preceq$ on
$\Omega$ given in \eqref{preorder1}, which induces a partial order $\preceq$ (using the same name) defined via \eqref{IrrWH} by
\begin{equation}\label{orderLa}
E\preceq E'\iff E_{\mathbb Q(t)}=E' _{\mathbb Q(t)}\quad {{\text or}} \quad \htt(E_{\mathbb Q(t)}) < \htt(E'_{\mathbb Q(t)}), \quad E,E'\in\Lambda.
\end{equation}
For each $E\in\Lambda$, $\Delta(E):=(S^\natural_E)^\diamond=\Hom_{\sH^\natural}(S_E^\natural,T^{+\natural}),$ where the $\sH^\natural$-module $S_E^\natural:=\Hom_{\sZ^\natural}(S^\natural(E),\sZ^\natural)$ is defined in the display \eqref{identify}.

The following corollary continues the argument from the proof of Theorem \ref{quasi-hereditary}(b). We leave the counterpart for the $^2F_4$ case in Theorem \ref{quasi-hereditary}(c) to the reader.
\begin{cor}\label{refine} The defining sequence $0=J_0^\natural\subseteq J_1^\natural\subseteq\cdots\subseteq J_m^\natural=A^{+\natural}$
 in the proof may be refined to a defining sequence of length $s=|\Lambda|$:
\begin{equation}\label{ids}
0=I_0\subseteq I_1\subseteq\cdots\subseteq I_{s-1}\subseteq I_s=A^{+\natural},
\end{equation}
which is called an {\sf indecomposable} defining sequence (partly our terminology).
Moreover, it may be assumed that, for every $1\leq j \leq s$, 
\begin{enumerate}
\item there exists an idempotent $e$ in $B':=A^{+\natural}/I_{j-1}$ such that $J':=I_j/I_{j-1}=B'eB'$.
\item $J'\cong\Delta(E)^{\oplus l_j}\cong\Delta(E)\otimes_{\sZ^\natural}\Hom_{B'}(\Delta(E),B')$, for some $E\in\Lambda$ and $l_j>0$. (We will see that the $l_j$ is equal to some $n_i$ in the display \eqref{Reiner1}).
\end{enumerate} 
Hence, $A^{+\natural}$ is also a split quasi-hereditary algebra in the sense of \cite[4.1.4]{Ro08}.
\end{cor}
\begin{proof}Continuing the argument from the proof of Theorem \ref{quasi-hereditary}(b), we refine each section $J^\natural_{h-1}\subseteq J_h^\natural$ ($1\leq h\leq m$). For a fixed $h$,
let $J^\natural=J_h^\natural/J^\natural_{h-1}$ and $B^\natural=A^{+\natural}/J^\natural_{h-1}$. Recall from Notation \ref{notation1} that the height function $\htt:\Omega\to\mathbb Z$ induces a height function $\htt:\Lambda\to\mathbb Z$. Recall also the idempotent decompsition $e_1+\cdots+e_s=1\in\sJ^\natural$, where $s=|\Lambda|$.
By displays \eqref{stringiso}, \eqref{Reiner}, and \eqref{Reiner1},\vspace{-1ex}
$$\mathcal E:=\End_{B^\natural}(J^\natural)\cong\End_{\sH^\natural}(N^\natural)\cong \End_{\sJ^\natural}(N_{\text{lift}}^\natural)\cong\bigoplus_{i=1}^{s_h}\End_{e_i\sJ^\natural}(e_iN_{\text{lift}}^\natural)\cong \bigoplus_{i=1}^{s_h}M_{n_i}(\sZ^\natural).$$
Here, after relabelling, we assume $n_i>0$, for all $1\leq i\leq s_h$, and $n_i=0$, for $s_h<i\leq s$. Note that the numbers $s_h$ are related to a partition (see \eqref{LCmodule})
\begin{equation}\label{LaPartition}
\Lambda=\bigcup_{i=1}^m\Lambda_{j(i)},\;\text{ where }%for the height $h'=j(h)$ 
\Lambda_{h'}:=\{E\in \Lambda\mid\htt(E)=h'\}\text{ and }|\Lambda_{j(h)}|=s_h.
\end{equation}

If $L:=N_{\text{lift}}^\natural$ and $L_i:=e_iN_{\text{lift}}^\natural$ ($1\leq i\leq s_h$), then $N^\natural=L^{\varpi^\natural}=\oplus_{i=1}^{s_h}L_i^{\varpi^\natural}$. Thus, denoting the canonical projection map $N^\natural\to L_i^{\varpi^\natural}\subseteq N^\natural$ by $e_i$ again, we have $e_i(N^\natural)=L_i^{\varpi^\natural}$ and the identity map $1_{N^\natural}=e_1+\cdots +e_{s_h}$ is decomposed into a sum of orthogonal central idempotents. Since the isomorphism $\phi:\End_{\sH^\natural}(N^\natural)\overset\sim\to\mathcal E$ is induced by the functor $(\;\,)^{*\diamond}=\Hom_{\sH^\natural}(-,T^{+\natural})\circ\Hom_{\sZ^\natural}(-,\sZ^\natural)$, the canonical projection $e_i$ induces the canonical projection $\tilde e_i=(e_i)^{*\diamond}:J^\natural\to  J^\natural$.
Hence, the element $1_{J^\natural}$ of $\mathcal E$ (i.e., the identity map on $J^\natural$) is decomposed into a sum of orthogonal central idempotents ${\tilde e}_i\in\mathcal E$: $1={\tilde e}_1+{\tilde e}_2+\cdots+{\tilde e}_s$, where ${\tilde e}_i=(e_i)^{*\diamond}$.  
Putting $J^{\natural(i)}={\tilde e}_iJ^\natural$, we have $J^\natural=\oplus_{i=1}^{s_h}J^{\natural(i)}$.

Note that $J^\natural$ is a $B^\natural$-$B^\natural$-bimodule. Also, each $J^{\natural(i)}$ is a left $B^\natural$-module, since the action of $\tilde e_i\in\mathcal E$ commutes with that of $B^\natural$.
We claim that each $J^{\natural(i)}$ ($1\leq i\leq s_h$) is also a right $B^\natural$-module. Indeed, suppose there exists $b\in B^\natural$ such that
$J^{\natural(i)}b\not\subseteq J^{\natural(i)}$. Then, the composition of the following homomorphisms
$$J^{\natural(i)}\longrightarrow J^{\natural(i)}b\longrightarrow (J^{\natural(i)}b+J^{\natural(i)})/J^{\natural(i)}\subseteq J^{\natural}/J^{\natural(i)}$$
is non-zero. This is impossible, since $\Hom_{B^\natural}({\tilde e}_iJ^{\natural},{\tilde e}_jJ^{\natural})={\tilde e}_j\mathcal E{\tilde e}_i=0$, for all $i\neq j$. Hence, each $J^{\natural(i)}$ is a $B^\natural$-$B^\natural$-bimodule.

 Let $J^\natural_{h-1,j}=\pi^{-1}(\oplus_{i=1}^jJ^{\natural(i)})$, where $\pi:J^\natural_h\to J_h^\natural/J^\natural_{h-1}$ is the obvious surjective homomorphism. Each $J^\natural_{h-1,j}$ is an ideal of $A^{+\natural}$ and  the section $J^\natural_{h-1}\subseteq J_h^\natural$ is refined as
$$J^\natural_{h-1}=J^\natural_{h-1,0}\subseteq J^\natural_{h-1,1}\subseteq J^\natural_{h-1,2} \subseteq\cdots\subseteq J^\natural_{h-1,s_h} = J_h^\natural.$$
Note that such a refinement can be done for any order on $\Lambda_{h'}$.
Combining these refinements for all $h$ results in a candidate defining sequence which has length $\sum_{h=1}^ms_h=s=|\Lambda|$. 

We now prove that it does give a defining sequence by checking (HI1)--(HI4) for each $J':=J^\natural_{h-1,i}/J^\natural_{h-1,i-1}\cong J^{\natural(i)}$ and $B':= A^{+\natural}/J^{\natural}_{h-1,i-1}\cong B^\natural/\oplus_{j=1}^{i-1}J^{\natural(j)}$.

  The proof for (HI1) and (HI2) is clear, since the $\sZ^\natural$-module
  $B'/J'\cong A^{+\natural}/J^{\natural}_{h-1,i}\cong A^{+\natural}/J^{\natural}_{h}\oplus J^\natural_h/J^{\natural}_{h-1,i}$ is
  a direct sum of two projective $\sZ^\natural$-modules, and $_{B'}J'={}_{B^\natural} J'$ is a direct summand of the projective $B^\natural$-module $_{B^\natural}J^\natural$.
(HI4) may be verified as follows:
\begin{equation}\label{sE'}
\End_{B'}(J')\cong\End_{B'}(\tilde e_i J^\natural)\cong\End_{B^\natural}(\tilde e_i J^\natural)\cong\tilde e_i\mathcal E\cong M_{n_i}(\sZ^\natural).
\end{equation}
For (HI3), we first note from the display \eqref{AfA} that the ideal $J_h^\natural=A_{j(h)}^{+\natural}$ is equal to $A^{+\natural}fA^{+\natural}$, for some idempotent $f$. Then the ideal $J^\natural=B^\natural\bar fB^\natural$, where $\bar f=f+J^\natural_{h-1}$ defines an idmpotent in $J^\natural$. If $\tilde f_i=\tilde e_i(\bar f)$, for each $i$ ($1\leq i\leq s_h$), then $\tilde f_i^2=\tilde e_i(\bar f)\tilde e_i(\bar f)=
\tilde e_i(\tilde e_i(\bar f)\bar f)=\tilde e_i(\tilde e_i(\bar f^2))=\tilde f_i$ and 
$J^{\natural(i)}=\tilde e_i(B^\natural\bar fB^\natural)=B^\natural\tilde f_iB^\natural$. Hence, $J'=B'\tilde f_iB'$ is idempotent, proving (HI3) and (1) as well.
%since $(\tilde e_iJ^\natural)(\tilde e_jJ^\natural)=0$ for all $i\neq j$,\footnote{\label{eiej}This is because $\tilde e_i(x)\tilde e_j(y)=\tilde e_j(\tilde e_i(x)y)=\tilde e_j(y_r(\tilde e_i(x)))=\tilde e_j((y_r\tilde e_i)(x))=\tilde e_j((\tilde e_iy_r)(x))=0$ for all $x,y\in J^\natural$, where $y_r:J^\natural\to J^\natural$ is the left $B^\natural$-module homomorphism $x\mapsto xy$.} it follows that
%$$(J^\natural)^2=(J^{\natural(1)})^2\oplus\cdots\oplus(J^{\natural(s_h)})^2.$$
%Hence, $(J^{\natural(i)})^2=J^{\natural(i)}$ for all $i$, proving (HI3).  This completes the proof of the first assertion and (2) of the corollary.

It remains to prove (2). The assertion that $J'$ is isomorphic to a direct sum of $\Delta(E)$, for some $E\in \Lambda$, can be seen as follows.   
 Since, by the proof of the theorem and recalling the notation $L_i=e_iN_{\text{lift}}^\natural$, $f_iL_i\cong \sZ^{\natural^{\oplus n_i}}$ (see the discussion above the diaplay \eqref{Reiner1}), it follows from the Morita equivalence (between $M_{n_i}(\sZ^\natural)$ and $\sZ^\natural$) that $L_i\cong E_\natural^{\oplus n_i}$, where $E=E_{h',i}\in \Lambda_{h'}:=\{E_{h',i}\mid 1\leq i\leq s_h\}$ for some height $h'=j(h)$ (see \eqref{LCmodule}), and the $\sH^\natural$-module $L_i^{\varpi^\natural}\cong  S^\natural(E)^{\oplus n_i}$; see \eqref{identify}. (Here we remind the reader that $E_\natural$ $(E\in\Lambda)$ is a $\sJ^\natural$-lattice of the irreducible $\sJ_{\mathbb Q(t)}$-module $E_{\mathbb Q(t)}$.) Hence, $(e_i(N^\natural))^{*\diamond}=(L_i^{\varpi^\natural})^{*\diamond}\cong
\Delta(E_{h',i})^{\oplus n_i}$.
Finally,  we have 
$$J^{\natural(i)}=\tilde e_i(J^\natural)=(e_i(N^\natural))^{*\diamond}\cong\Delta(E_{h',i})^{\oplus n_i}.$$ 
Since $\Hom_{B'}(J',B'/J')=0$ (see \cite[Lem. 4.4]{Ro08}), it follows that
$$\Hom_{B'}(\Delta(E_{h'i}),B')=\Hom_{B'}(\Delta(E_{h',i}),J')\cong (\sZ^\natural)^{\oplus n_i}.$$ 
This implies the last isomorphism in (2). 

It remains to verify Rouquier's definition \cite[4.1.4]{Ro08} for split quasi-hereditary algebras. By the definition in display \eqref{orderLa}, each coideal $\Gamma$ of $\Lambda$ has the form $\Gamma=\big(\cup_{h''>h'}\Lambda_{h''}\big)\cup\Phi$, where $\Phi\subseteq\Lambda_{h'}$ for some $h'=j(h)$.
Thus, each indecomposable defining sequence of the form \eqref{ids} defines a sequence of coideals $\{\Gamma_i\mid 1\leq i\leq s\}$ of $\Lambda$ such that the set of ideals $\{I_{\Gamma_i}:=I_i\mid 1\leq i\leq s\}$ satisfies the two conditions in \cite[4.1.4]{Ro08}.
%we may refine $J^\natural_{h-1}\subseteq J_h^\natural$ via $\Phi$ so that $I_\Gamma$ is defined as a term $I_j$ of an indecomposable defining sequence in \eqref{ids}. Clearly, if $\Gamma\subseteq\Gamma'$ are two coideals, then there is an indecomposable defining sequence of the form \eqref{ids} such that $I_\Gamma=I_i$ and $I_{\Gamma'}=I_j$ with $i<j$. 
This proves that $A^{+\natural}$ is a split quasi-hereditary algebra in the sense of Rouquier, completing the proof of the corollary.
%To see (2), we first note from \eqref{AfA} that the ideal $J_h^\natural=A^{+\natural}fA^{+\natural}$ for some idempotent $f$. Then the ideal $J^\natural=B^\natural\bar fB^\natural$, where $\bar f=f+J^\natural_{h-1}\in J^\natural$ is an idmpotent. Thus, $\bar f=\tilde f_1+\cdots+\tilde f_{s_h}$, where $\tilde f_i\in\tilde e_i(J^\natural)=J^{\natural(i)}$. Footnote \ref{eiej} shows that the $\tilde f_i$s are orthogonal idempotents. Hence, $B^\natural\bar fB^\natural=B^\natural\tilde f_1B^\natural\oplus\cdots\oplus B^\natural\tilde f_{s_h}B^\natural$. Consequently, each $B^\natural\tilde f_iB^\natural=J^{\natural (i)}$. Now, assertion (2) follows easily.
 %$\End_B(J^{\natural(i)})\cong\End_{\sJ^\natural}(L_i)\cong$
\end{proof}

\begin{cor}\label{B'e} Maintain the notations of Corollary \ref{refine}. There exists an idempotent $e_1\in B'$ such that $B'e_1\cong\Delta(E)$ as left $B'$-modules. 
\end{cor}
\begin{proof} By Corollary \ref{refine}(1) and \cite[Prop. 2.2]{DPS18}, we have
$$J'=B'eB'\cong B'e\otimes_{eB'e}eB',$$
 $\End_{B'}(J')$ is Morita equivalent to $eB'e$, and $eB'$ is a projective
$eB'e$-module.
By \eqref{sE'}, we see that $eB'e$ is Morita equivalent to $\sZ^\natural$.  Since projective finite $\sZ^\natural$-modules are free, it follows from the Morita equivalence that $eB'e\cong M_m(\sZ^\natural)$, for some $m\geq1$.

If $m=1$, then $eB'e\cong\sZ^\natural$ and $J'=B'eB'\cong B'e\otimes_{\sZ^\natural}eB'\cong B'e\otimes_{\sZ^\natural}\Hom_{B'}(B'e,B')$.
Thus, the proof of \cite[Prop.~4.7]{Ro08} and Corollary \ref{refine}(2) imply $B'e\cong\Delta(E)$. So, we take $e_1=e$ in this case.

Suppose now $m>1$. Then, there exist idempotents $e_1,\ldots,e_m$ (corresponding to the matrix idempotents $e_{1,1},\ldots,e_{m,m}$ in $M_m(\sZ^\natural)$) such that $e=e_1+\cdots+e_m$ and $e_ie_j=\delta_{i,j}e_i$. Thus, $B'e=B'e_1\oplus\cdots\oplus B'e_m$ and $\Hom_{B'}(B'e_i,B'e_j)\cong \sZ^\natural$. (The right action of $eB'e\cong M_m(\sZ^\natural)$ on $B'e$ may be interpreted as matrix multiplication.) Hence, $B'e\cong (B'e_1)^{\oplus m}$. Since the left $eB'e$-module $eB'$ is projective, it follows (from the Morita equivalence between $eB'e$ and $e_1B'e_1\cong \sZ^\natural$) that the $e_1B'e_1$-module $e_1B'=e_1(eB')$ is projective (over $\sZ^\natural$) and, hence, is free: $e_1B'\cong(\sZ^\natural)^{\oplus m'} $ for some $m'>0$. The Morita equivalence gives $eB'\cong ({\sZ^\natural}^m)^{\oplus m'}\cong C_1^{\oplus m'}$, where $C_1=M_m(\sZ^\natural)e_{1,1}$. Consequently,
$$J'\cong B'e\otimes_{eB'e}eB'\cong (B'e_1)^{\oplus m}\otimes_{M_m(\sZ^\natural)}C_1^{\oplus m'}\cong (B'e_1)^{\oplus m'}.$$
In other words, we may write $J'=P_1\oplus\cdots \oplus P_{m'}$, where $B'e_1\cong P_i$. Taking isomorphism
$\phi_i:B'e_1\to P_i$, we have
$P_i=\phi_i(B'e_1)=B'e_1\phi_i(e_1)\subseteq B'e_1B'$. Hence, $J'\subseteq B'e_1B'$. Clearly, we have $B'e_1B'=B'ee_1B'\subseteq B'eB'=J'$. This shows $J'=B'e_1B'$. By the argument for the $m=1$ case, we obtain $B'e_1\cong\Delta(E)$. The corollary is proven.
\end{proof}
 
Corollary \ref{refine} together with \cite[Th.~4.16]{Ro08} gives the following result.
\begin{cor} \label{refine3}If $\mathcal C:=A^{+\natural}\text{-mod}$ is the category of finite $A^{+\natural}$-modules,
then $(\mathcal C,\{\Delta(E)\}_{E\in\Lambda})$ is a highest weight category in the sense of \cite[Def.~4.11]{Ro08}, where $\Lambda$ is the poset with respect to the order defined in display \eqref{orderLa}.
\end{cor}
Note that the projective modules $P(E)$ used to form the stratifying system $\{\Delta(E),P(E)\}_{E\in\Lambda}$ may be recursively constructed; see the proof of \cite[Lem.~4.12]{Ro08}.

\medskip

Similarly to $A^{+\natural}$ and $A^{+\natural\prime}$
as defined above the statement of Theorem \ref{quasi-hereditary}, we define, for a poset ideal $\Theta$ of $(\Omega,\leq_L)$, $A^{\ddagger\natural}_\Theta:=\sZ^\natural\otimes_\sZ A^\ddagger_\Theta$ and $A^{\ddagger\natural\prime}_\Theta:=\sZ^{\natural\prime}\otimes_\sZ A^\ddagger_\Theta$, where $A^\ddagger_\Theta$ is defined in \eqref{Tdagger}. For later use in Sections 5.1 and 5.4, we need the quasi-hereditary structures of these algebras.

\begin{thm}\label{Adagger3}
All the statements {\rm(a)}, {\rm(b)}, and {\rm (c)} in Theorem \ref{quasi-hereditary} are true if $A^{+\natural}$ and $A^{+\natural\prime}$ are replaced by $A^{\ddagger\natural}_\Theta$ and $A^{\ddagger\natural\prime}_\Theta$, respectively. Moreover, the base-changed sequence of the defining sequence \eqref{J+seq} to $\sZ^\natural$ (or $\sZ^{\natural\prime}$) may be refined to a defining sequence of length $s=|\Lambda_\Theta|$, where $\Lambda_\Theta:=\{E\in \Lambda\mid E_{\mathbb Q(t)}|S(\omega)_{\mathbb Q(t)},\text{ for some }\omega\in\Theta\}$,
$$0=I_0^+\subseteq I_1^+\subseteq\cdots\subseteq I_s^+=A^{\ddagger\natural}_\Theta$$
such that $\bar I_j:=I_j/I_{j-1}$ is an indecomposable heredity ideal of $B_j:=A^{\ddagger\natural}_\Theta/I_{j-1}$. In particular, the highest weight category $A^{\ddagger\natural}_\Theta$-mod has standard objects
$$\Delta^+(E):=\Hom_{\sH^\natural}(S^\natural_E,T^{\ddagger\natural}),\;\;\forall E\in\Lambda$$
such that, if $\Delta^+(E)\otimes_{\sZ^\natural}\Hom_{B_j}(\Delta^+(E),B_j)\cong \bar I_j$, for some $1\leq j\leq s$, then $\Delta^+(E)\cong B_je$ for some idempotent $e\in B_j$ with $eB_je\cong\sZ^\natural$.
\end{thm}
\begin{proof} Note that, by Remark \ref{Adagger2}, Theorem \ref{thm2.5} is applicable in this case. Thus, the proof of  Theorem \ref{quasi-hereditary} can be easily modified by simply adding a $^+$ to the superscripts of the various objects throughout its proof to get a proof for the first assertion. (The first assertion also follows easily from the Morita equivalence discussed in Sec 3.4.)

The remaining assertions follow from arguments similar to the proofs of Corollaries \ref{refine} and \ref{B'e}.
\end{proof}
%{\color{red} Recall the endomorphism algebra $A^\ddag_\Theta$ associated with a quasi-poset ideal $\Theta$ of $(\Omega,\leq_L)$ defined in \eqref{Tdagger}.
%\begin{cor}\label{Adagger4}
%For any quasi-poset ideal $\Theta$ of $(\Omega,\preceq)$, the $\sZ^\natural$-algebra $A^{\ddag\natural}_\Theta:=\sZ^\natural\otimes A^\ddag_\Theta$ is quasi-hereditary.
%\end{cor}
%We remark that a defining sequence for $A^\ddag_\Theta$ may be constructed by following the proof of Corollary \ref{refine},  replacing $\Lambda=\irr(\mathbb QW)$ by 
%$$\Lambda_\Theta:=\{E\in \Lambda\mid E_{\mathbb Q(t)}|S(\omega)_{\mathbb Q(t)}\text{ for some }\omega\in\Theta\}.$$ }

Finally, in a different direction, the following remark  will be useful in the next chapter.
\begin{rem} \label{433}
Let $B$ be a split quasi-hereditary algebra over $\sZ$. (For example, $B =A^{+\natural}$ as in Theorem \ref{quasi-hereditary}(b) above.) 
{\it If $R$
is a commutative Noetherian $\sZ$-algebra, $B_R:=R\otimes_\sZ B$ is a split quasi-hereditary algebra over $R$.} This is verified by directly checking the
axioms at the start of this section. For example,  if $J$ is a heredity ideal in $B$, then $_BJ$ is a projective $B$-module. It follows easily, using this projectivity,  that 
the $R$-algebra 
$\End_{B_R}(_{B_R}J)$ is isomorphic to the $R$-algebra $R\otimes_{\sZ}\End_B(_BJ)$ and so,
is split over $R$. Thus, the operation $R\otimes_\sZ -$ carries a defining sequence for $B$ (exhibiting a splitting) into a similar defining sequence for $B_R$. The other steps in checking that
$B_R$ is split quasi-hereditary are straightforward.

The other parts (a) and (c) of Theorems \ref{quasi-hereditary} and \ref{Adagger3} similarly behave well under base change.
\end{rem}

   \chapter{Some related applications}
 We now present some applications of our theory to several related topics.  First, the split quasi-heredity of $A^{+\natural}$ gives rise to a full standard basis for $A^{+\natural}$, making its centralizer subalgebras such as $A$ and $\sH$ into standardly based algebras. Second, the quasi-heredity of $A^{+\natural}$ gives a unitriangular property of several decomposition matrices, including the unipotent principal blocks of the corresponding finite groups of Lie type in cross characteristic. Third, using a recent Schur duality of Bao--Wang for certain $i$-quantum groups and Hecke algebras of type $B/C$, we relate representations of finite orthogonal/symplectic groups to those of $i$-quantum groups. Finally, we end this chapter with some open problems.
 
  \setcounter{section}{-1}
 \section{Introduction}
 
 The usual decomposition number framework for  finite groups is the classical one of Brauer, involving integral representations over DVRs (discrete valuation rings, usually complete); see, e.g., \cite{Se77}. However, the theory is not adequate for generic Hecke algebras $\sH$ (or  $A$ or $A^+$), because the generic versions of these algebras are defined over
$\sZ=\mathbb Z[t,t^{-1}]$, and some authors introduce even more variables. Geck and Rouquier provided in \cite{GR97} a general theory which applies to $\sZ$ and many higher dimensional commutative rings.
The authors of this monograph introduced in  \cite{DPS98a} a decomposition number theory applicable to $\sZ$-algebras, which (though not applicable to commutative base rings of Krull dimension $>2$) is closer to the original Brauer framework.

We will use the  \cite{DPS98a} theory in this section. An exposition, with several new results, is provided in Appendix C.  We suggest that the reader look at this section first, then read some of Appendix C.  After that, both the present section and Appendix C may be read in tandem.

\medskip%\medskip
 %\noindent 
 {\bf  Regular local triples.} Consider a ``local triple" $(R, Q, F)$,\index{local triple} with a local Noetherian integral domain $R$
 (with maximal ideal $\mathfrak m$), fraction field $Q$, and residue field $F=R/\mathfrak m$ of $R$.  Several evident variations on this notational theme are useful. For example, if $R$ is just a domain, call $(R,Q,F)$ a ring triple. As before, $Q$ is the fraction field of $R$, and there is given an identification of $F$ with $R/\mathfrak m$ for an implicitly given maximal ideal $\mathfrak m$ of $R$. Also, if  $R$ is a DVR (automatically ``local"), then $(R,Q,F)$ is a ``DVR triple."\index{local triple! DVR triple}  If $R$ is  a regular ring (in the sense of
 \cite{M86}),
 then the ring triple $(R,Q,F)$ is a ``regular triple." Regular triples will play an important role in this discussion---and $R$ will usually be a regular local ring of Krull dimension $\leq 2$.
 
 {\it In case $R$ is a regular local ring of Krull dimension $d$, we call $(R,Q,F)$ a $d$-regular local triple.}\index{local triple! regular $\sim$}
 
% Note that every ring triple $(R,Q,F)$ has an associated local triple $(R_{\mathfrak m}, Q, F)$, where $\mathfrak m$ is %the maximal ideal for which $F$ identifies with $R/\mathfrak m$.
  
  Let $B$ be an $R$-free and finite algebra. Also, assume that $R$ is a regular local ring of Krull dimension $\leq 2$. Let
 $B_Q:=B\otimes_RQ$ be the corresponding finite-dimensional $Q$-algebra, and define $B_F$ similarly. 
  A basic result given in Theorem \ref{AG2} (and based on work of Auslander-Goldman \cite{AG60a,AG60b})  states that if $\mathfrak X$ is a finite $B_Q$-module, then $\mathfrak X\cong X_Q$, for some finite $R$-free $B$-submodule $X\subseteq \mathfrak X$.
 Furthermore, assume that $X_Q$ is irreducible. If $Y$ is a second $R$-free $B$-submodule such that $\mathfrak X\cong Y_Q$, then $X_F$ and $Y_F$ have the same composition factors (with the same multiplicities) as $B_F$-modules. This means the decomposition numbers of $B$ defined below depend only on $(R,Q,F)$.
 
 Thus, under the hypotheses and notation of Theorem \ref{AG2}, there is a natural way to relate the representation theory of 
 $B_F$ (often a {\it modular theory}, usually meaning that $F$ has positive characteristic,\footnote{For algebras over $\sZ$, if the order of $t$ in $F$ is finite (that is, the image of $t$ is a root of unity), that order is often called the quantum characteristic. ``Modular theory'' can occur in this case even when $F$ has characteristic zero!}) with that of $B_Q$ (often a {\it characteristic zero theory}).  In this regard,
 we are particularly interested in the {\it decomposition map} \index{decomposition map} in display (\ref{decomMap}). Note that the hypothesis that the Krull dimension of $R$ is at most 2 is often satisfied in the present monograph, since $R$ is usually closely related to the ring $\sZ$ of integral Laurent polynomials in a variable $t$.
  \medskip

% \noindent{\underline{\bf Decomposition matrices; the unitriangular property.} }  

{\bf Decomposition matrices.} Let $(R,Q,F)$ be a regular local triple, where $R$  has Krull dimension $\leq 2$. (See above.)
   If $B=B_R$ is a finite and $R$-free algebra, we can use the triple to associate decomposition maps, decomposition numbers, and decomposition matrices to $B$, imitating the classical constructions.  
    More explicitly, suppose that 
  $X$ is a $B$-module which is finite and projective (thus free) over $R$. Assume that 
  $X_{Q}$ is an irreducible $B_{Q}$-module. If $L$ is an irreducible $B_F$-module, then the decomposition number $d_{L,X_{Q}}$ is defined to be the multiplicity $[X_F:L]$ of $L$ as a composition factor of $X_F$. This number depends only on $X_{Q}$ not on the choice of $X$ \cite[\S1.1]{DPS98a}. Collectively, the decomposition numbers form the {\it decomposition matrix},\index{decomposition matrix} with rows indexed by ``all non-isomorphic" irreducible $B_F$-modules, and columns indexed 
  by ``all non-isomorphic" irreducible $B_K$-modules.\footnote{Many authors take the decomposition matrix to be the transpose
  of our decomposition matrix. Thus, the rows of the decomposition matrix index the irreducible $B_K$-modules. This causes no difficulty, provided the reader is aware of the different conventions.}

  %(Another notion of a decomposition number and decomposition matrix is due to Geck and Rouquier, and is described in \cite[\S3.1]{GJ11}. It is compatible with our approach, but applicable to more general base rings. This will not concern us here.) Also, some authors use a transpose decomposition matrix notation; see
  %\cite[\S3.1]{GJ11}.  
  
 Notice that, as seen above, these decomposition numbers for $B$ depend only on $B$ and the underlying triple $(R,Q,F)$. It is, however, common also to assume (although we will not do it) that 
 the irreducible modules indexing the rows and columns 
  are absolutely irreducible. Equivalently, both $B_F$ and $B_Q$ are split as algebras. In many cases, this
    can be achieved by enlarging 
  $R$ and replacing $B$ by $B_R$. See Propositions  \ref{D11} and \ref{D12}, as well as Proposition \ref{D9} regarding uniqueness. 
  
  Although we often speak of ``the" decomposition matrix, it is, in fact, only determined up to a
  permutation of its rows and a permutation of its columns, unless some specific (or restricted) ordering of irreducible modules is required. We will usually be somewhat informal about this and hope that it does not cause too much confusion. 
  
  \medskip
  
{\bf Decomposition submatrices and unitriangularity.}
  Given two (rectangular) matrices $D=(d_{ij})$ and $D'=(d'_{rs})$ (decomposition matrices or not), we say that $D$ is a submatrix of $D'$ if there are injective assignments $i\mapsto i'$ and $j\mapsto j'$ of the row (respectively, column) indices of $D$ such that
  $$ d_{i,j}=d'_{i',j'},$$
  for each row index $i$ and column index $j$ of $D$.
  
  A   rectangular matrix $D=(d_{ij})$ over any ring is said to be (upper) {\it unitriangular}\index{decomposition matrix! unitriangular $\sim$} if
  
  \medskip (1) $D$ has at least as many columns as rows;
  
  \smallskip
  (2) $d_{ij}=0$ if $i>j$;  and
  
  \smallskip
  (3) $d_{ii}=1$ for each row index $i$.

   \section{Standardly based algebras}
    In this section, we consider the algebras $A$ and $A^+$ from the point of view of the theory of standardly
   based algebras as developed in Du-Rui \cite{DR98} (see \cite{KM20} for a new development). Such algebras are quite close to the Graham-Lehrer theory of cellular algebras \cite{GL96}. For example, they have a similar family of bilinear forms (see (\ref{bilinform}) below), and all cellular algebras are standardly based. But, there is no requirement that a standardly based algebra have an involution as in \cite{GL96}.

  \begin{defn}
Let $\mathcal A$ be an algebra over a commutative ring $\scrK$.
We say that $\mathcal A$ is a {\it standardly based algebra} (SBA)\index{standardly based algebra} with respect to a poset $(\Lambda,\leq)$  if the following
conditions hold:
 
(a) For any $\lambda\in \Lambda$, there are index sets
$I(\lambda)$ and $J(\lambda)$ and subsets of $\mathcal A$,
$$\sfB^\la=\{a_{i,j}^\la\mid (i,j)\in I(\la)\times J(\la)\}$$
such that $ \sfB:=\bigcup_{\la\in\Lambda} \sfB^\la$ is a disjoint union and forms a $\scrK$-basis for $\mathcal A$.
 
(b) For any $a\in \mathcal A$, $a_{i,j}^\la\in \sfB$, we have
$$\aligned
& a\cdot a_{i,j}^\la\equiv \sum_{i'\in I(\la)} f_{i',\la}
(a,i) a_{i',j}^\la\mod (\mathcal A^{>\la})\cr
&a_{i,j}^\la\cdot a\equiv \sum_{j'\in J(\la)} f^\o_{\la,j'} (a,j)
a_{i,j'}^\la \mod(\mathcal A^{>\la}),\cr
\endaligned$$
where $\mathcal A^{>\la}$ is the $\scrK$-submodule of $\mathcal A$ generated by $ \sfB^\mu$ with $\mu>\la$, and the coefficients $f_{i' ,\la} (a,i)$ and $f^\o_{\la,j'} (a,j)\in\scrK$, defined by the display, are independent of $j$
and $i$, respectively.

We also call $\sfB$ a {\it standard basis} for $\mathcal A$. \index{standard basis}
%A based algebra $(\mathcal A,\sB)$ is said to be {\it standardly based} if $\sB$is standard. We shall also call $\sB$ a {\it defining base} of the standardly based algebra $A$.
\end{defn}

For any coideal $\Gamma$ of $\Lambda$, the submodule $A^\Gamma:=\text{span}(\cup_{\lambda\in\Gamma}\sfB^\lambda)$ is an ideal of $\sA$. For a finite $\Lambda$, if we order $\Lambda=\{\la_1,\la_2,\ldots,\la_m\}$ such that $\Gamma_i=\{\la_1,\la_2,\ldots,\la_i\}$ is a coideal for every $1\leq i\leq m$, then we obtain a sequence of ideals
\begin{equation}\label{std seq}
0=J_0\subseteq J_1\subseteq J_2\subseteq\cdots\subseteq J_m=\sA,
\end{equation}
where $J_i=\sA^{\Gamma_i}$. Thus, every section $A^{\la_i}:=J_i/J_{i-1}$ is an $\sA$-$\sA$-bimodule.

Given a SBA as above, we may introduce, for each $\la\in\Lambda$, a
function $f_\la:J(\la)\times I(\la)\to\scrK$ defined as follows:
for any $a_{ij}^\lambda, a_{i'j'}^\lambda \in \sfB^\la$, there is a unique
element $f_\la(j,i')\in\scrK$ (in other words, $f_\la(j,i'):=f_{i,\la}(a_{i,j}^\la,i')=f^\o_{\la,j'}(a_{i',j',}^\la,j)$) such that
$$ a_{ij}^\lambda a_{i'j'}^\lambda\equiv f_\la(j,i') a_{ij'}^\la
\mod(\mathcal A^{>\lambda}).$$

\begin{defn} For $\la\in\Lambda$, we define $\Delta(\la)$ (resp.
$\Delta^{\!\o}(\la)$) to be the left (resp. right) $\mathcal A$-module with $\scrK$-basis
$\{a_i^\la\mid i\in I(\la)\}$ (resp. $\{ b_{j}^\la \mid j\in J(\la)\}$)
and the module action defined by
$$aa_{i}^\la=\sum_{i'\in I(\la)} f_{i',\la}(a,i)a_{i'}^\la\quad (\text{resp.}, b_{j}^\la a=\sum_{j'\in J(\la)}f^\o_{\la, j'}(a,j)b_{j'}^\la.) $$

%Let $\nabla(\la)=\Delta^\op(\la)^\ast$. 
We shall call
the modules $\Delta(\la)$ (respectively, 
$\Delta^\o(\la)$) the {\it standard} modules\index{standard module} in the category $\mathcal A$-{mod} (respectively, {mod}-${\mathcal A}$).
\end{defn}
One sees easily from the definition that there is an $\sA$-$\sA$-bimodule isomorphism
$A^\la\cong \Delta(\la)\otimes \Delta^\o(\la)$.
%Assume now that $\mathscr K$ is a regular ring and $\mathcal O$ is a localisation of $\mathscr K$ at a prime ideal with %residue field $k$ and fraction field $K$. 

%\begin{lem}[{\cite[(2.4.6)]{DR98}}] If $\mathcal A$ is standardly based and $\mathcal A_K$ is semisimple, then the %decomposition matrix relative to $(K,\mathcal O,k)$ is upper unitriangular.
%\end{lem}

We define (compare \cite[2.3]{GL96}) a bilinear pairing
$\beta_{\la}: \Delta(\la) \times \Delta^\o(\la)\rightarrow \scrK$ by putting
\begin{equation}\label{bilinform}\beta_{\la}(a_i^\la, b_j^\la)=f_{\la}(j,i).\end{equation}
Clearly, $\beta_\lambda$ is ``associative" in a natural sense, see \cite[Lem. 2.3.1(a)]{DR98}.
A SBA $\sA$ is called a  {\it standardly full-based algebra}\index{standardly full-based algebra} if $\text{im}(\beta_\la)=\scrK$ for every $\la$.

When $\scrK$ is a field, these forms can be used to determine a classification of irreducible left $\sA$-modules. In this case, we first observe that $\Delta(\la)$ is a cyclic $\sA$-module if $\beta_\la\neq0$.
Let $\Lambda_1=\Lambda_{1,\scrK}=\{\la\in \Lambda\mid \beta_\la\neq0\}$ and, for $\la\in \Lambda_1$, let
$$L(\la):=\Delta(\la)/\text{\rm rad}\Delta(\la),\;\text{ where }
\text{\rm rad}\Delta(\la)=\{v\in\Delta(\la)\mid \beta_\la(v,y)=0\;\forall y\in\Delta^\o(\la)\}$$
\begin{lem}[{\cite[Th. 2.4.1]{DR98}}]\label{La1}Maintain the notation introduced above and assume that $\sA$ is a finite-dimensional standardly based algebra 
over a field $\scrK$ with respect to a given poset $\Lambda$. Then: 
\begin{itemize}
\item[(a)] For each $\la\in\Lambda_1$, $L(\la)$ is absolutely irreducible, and $\{L(\la)\}_{\la\in \Lambda_1}$ forms a complete set of all non-isomorphic irreducible $\sA$-modules.
\item[(b)] In particular, $|\Lambda_1|$ is the number of non-isomorphic irreducible $\sA$-modules.
\item[(c)] If $L(\la)$ is a composition factor of $\Delta(\mu)$, then $\la\leq \mu$.
\item[(d)] If $\la\in\Lambda_1$, then $[\Delta(\la):L(\la)]=1$.
\end{itemize}
\end{lem}

\begin{cor}\label{SFBA}
Let $(K,\sO,k)$ be a local triple in which $\sO$ is regular of Krull dimension $\leq 2$ and suppose that $\sA$ is a SBA over $\sO$ associated with the poset $\Lambda$ such that $\sA_K$ is finite dimensional and semisimple. Then the decomposition matrix $D$ of $\sA$ is unitriangular. If, in addition, $\sA$ is standardly full-based, then $D$ is also square.
\end{cor}
\begin{proof}Clearly, $\sA_K$ is a SBA with standard modules $\Delta(\la)_K$, $\lambda\in\Lambda$, and has a filtration of ideals obtained by base changing \eqref{std seq} to $K$ whose sections $I:=\sA_K^{\la_i}$
 %isomorphic to $(\Delta(\la)\otimes\Delta(\lambda)^\op)_K$. 
 are ideals of $\sA_K$ (since $\sA_K$ is semisimple). Thus, $I^2=I$ and, hence,  there exists an idempotent $e\in I$ such that $I=\sA_Ke\sA_K$. This implies $\Lambda_{1,K}=\Lambda$. Now the first assertion for the $|\Lambda_{1,k}|\times|\Lambda|$-matirx $D$ follows from the lemma above. If $\sA$ is standardly full-based, then $\Lambda_{1,k}=\Lambda$. Hence, $D$ is a square matrix.
\end{proof}

The next proposition follows from a general result proved in \cite[(4.2.7)]{DR98}. The assumption there that the base ring $k$ be a commutative local ring is not necessary here when $k=\sZ^\natural$ (or $\sZ^{\natural\prime}$ if we are in the ${}^2F_4$ case). This is because Swan's result \cite[p.111]{Sw78} implies that any finite projective $\sZ$-module is automatically free.\footnote{Note that $\sZ^\natural=\mathbb Z^\natural[t,t^{-1}]$, where $\mathbb Z^\natural$ is a PID. Similarly, $\sZ^{\natural\prime}=\sZ^{\prime\natural}=\mathbb Z^{\prime\natural}[t,t^{-1}]$, where $\mathbb Z^{\prime\natural}$ is a PID. Thus, Swan's result \cite[p.111]{Sw78} applies.}  For completeness, we provide a constructive argument for the proposition.

\begin{prop} \label{prop125} (a) Assume that we are not in the $^2F_4$ case. Then, the split quasi-hereditary ${\mathcal Z}^\natural$-algebra $A^{+\natural}$  is a standardly full-based algebra.

(b) In the case of $^2F_4$, the split quasi-hereditary $\sZ^{\natural\prime}$-algebra $A^{+\natural\prime}$ is a standardly full-based algebra. 
\end{prop}
\begin{proof}We only give the proof in case (a).
Let ${\mathcal A}:=A^{+\natural}$ and
let $\Lambda=\text{Irr}({\mathbb QW} )$. Then, by Theorem \ref{quasi-hereditary}(b), $\sA$ is a split quasi-hereditary algebra, and
$\Lambda$ becomes a poset with the partial order given in \eqref{orderLa} which  %\eqref{order}, 
is induced from the preorder \eqref{preorder1}. 
%For each $\la\in\Lambda$,  let $\Delta(\la)=\Delta^\natural(\la)$ and $\Delta(\la)^\op=\Hom_B(\Delta^\natural(\la),B)$ denote the standard left $\mathcal A$-modules and standard right $\mathcal A$-modules, which are $\mathcal Z^\natural$-projective. Hence, both are free $\mathcal Z^\natural$-modules.
Consider the defining sequence of length $n=|\Lambda|$ given in Corollary \ref{refine}: 
$$0=I_0\subset I_1\subset I_2\subset\cdots\subset I_{n}={\mathcal A}.$$
Assume that the chain is associated with the linear order $\la^{(1)},\la^{(2)},\ldots,\la^{(n)}$ on $\Lambda$, compatible with the partial order in $\Lambda$.
For each $i=1,\ldots,n$, ${\mathcal A}_i:={\mathcal A}/I_{i-1}$ is quasi-hereditary and $\bar I_i=I_{i}/I_{i-1}$ is a heredity ideal of $\sA_i$. The split condition then implies, by Corollary \ref{refine}(2), that there is an ${\mathcal A}_i$-${\mathcal A}_i$-bimodule isomorphism 
\begin{equation}\label{AeA}
\bar I_i\overset{\phi_i}\cong \Delta(E)\otimes_{\mathcal Z^\natural} \Hom_{\sA_i}(\Delta(E),\sA_i),
\end{equation}
where  $\la^{(i)}=E$. Since both
$\Delta(\la^{(i)}):=\Delta(E)$ and $\Delta^\o(\la^{(i)}):=\Hom_{\sA_i}(\Delta(E),\sA_i)$ are $\mathcal Z^\natural$-projective, they are free $\mathcal Z^\natural$-modules.

Now, we may construct a standard basis by induction. The case for $i=n$ is clear, since bases for  $\Delta(\la^{(n)})$ and $\Delta^\o(\la^{(n)})$ give rise to a standard basis for ${\mathcal A}_n$. Assume now, for $1<k\leq n$, that a standard basis $\sfB_k=\bigcup_{l=k}^n\{b_{i,j}^{\la^{(l)}}\mid i\in I(\la^{(l)}),j\in J(\la^{(l)})\}$, for $\mathcal A_k={\mathcal A}/I_{k-1}$ has been constructed.  Use \eqref{AeA} and the bases $\{a_i\}$ and $\{b_j\}$ for $\Delta(\la^{(k-1)})$ and $\Delta^\o(\la^{(k-1})$, respectively, to get a basis $\{a^{\la^{(k-1)}}_{i,j}=\phi_{k-1}^{-1}(a_i\otimes b_j)\}_{i,j}$ for $\bar I_{k-1}$. This together with a pre-image $\{a_{i,j}^{\la^{(l)}}\}_{l,i,j}$ of the basis $\sfB_k$
%$\{b_{i,j}^{\la^{(l)}}\}_{l,i,j}$ $(i\leq l\leq m$, $i\in I(\la^{(l)}), j\in J(\la^{(l)})$) 
under $\pi_k:\sA_{k-1}\to\sA_k$, the natural homomorphism, gives a basis 
$$\sfB_{k-1}=\bigcup_{l=i-1}^n\{a_{i,j}^{\la^{(l)}}\mid i\in I(\la^{(l)}),j\in J(\la^{(l)})\},$$ for $\sA_{k-1},$. Here, we require $\sfB_{k-1}$ to satisfy the condition
% that $\bigcup_{l=k-1}^k \{a_{i,j}^{\la^{(l)}}\mid i\in I(\la^{(l)}),j\in J(\la^{(l)})\}$ forms a basis for $J_l:=I_{l}/I_{l-1}$ for all $l=k-1,k,\ldots,n$ and 
 $(a_{i,j}^{\la^{(l)}}+I_{k-1})/I_{k-1}=b_{i,j}^{\la^{(l)}}$, for all $l=k,\ldots,n$. Now, by induction, $\sfB=\sfB_0$ (for $i=1$) is a desired standard basis. 

Finally, the basis $\sfB$ is a full-standard basis, since every heredity ideal $\bar I_i$ has the form $\bar I_i=\sA_ie_i\sA_i$ with $\Delta(\la^{(i)})\cong \sA_ie_i$, $\Delta^\o(\la^{(i)})\cong e_i\sA_i$ (see Corollary \ref{B'e}), and
$e_i\sA_ie_i\cong \sZ^\natural$ for some idempotents $e_i\in\sA_i$, where $i=1,2,\ldots,n$.
%$Ae$ is a standard module and $eA$ is a dual costandard module.
\end{proof}
We now use Theorem \ref{Adagger3} to get a more general version of the proposition.

\begin{cor}\label{Adagger4}
For a poset ideal $\Theta$ of $(\Omega,\leq_L)$,
the statements {\rm(a)} and  {\rm(b)} in Proposition \ref{prop125} are true if $A^{+\natural}$ and $A^{+\natural\prime}$ are replaced by the quasi-hereditary algebras $A^{\ddagger\natural}_\Theta$ and $A^{\ddagger\natural\prime}_\Theta$, respectively.
\end{cor}

The following lemma follows from the proof of the above proposition. The base ring $\sZ^\natural$ can be replaced by any commutative Noetherian ring with the property that any finite projective module is free.

\begin{lem}\label{fAf}
If $\mathcal A$ is a standardly based algebra over $\sZ^\natural$ (or 
$\sZ^{\natural\prime}$), and if $e\in\mathcal A$ is an idempotent, then $e\mathcal Ae$ is standardly based.  \end{lem}
\begin{proof}Suppose $\sfB=\bigcup_{\la\in\La}\sfB^\la$ is a standard basis for $\sA$ and assume $\La=\{\la^{(1)},\la^{(2)},\ldots,\la^{(m)}\}$ such that $i<j$ whenever $\la^{(i)}>\la^{(j)}$. Then, the two-sided ideals $J_i=\text{span}(\sfB^{\la^{(1)}}\cup\cdots\cup \sfB^{\la^{(i)}})$ of  $\sA$ gives a sequence in $\sA$
$$0=J_0\subseteq J_1\subseteq \cdots \subseteq J_m=\sA.$$
This, in turn, gives a sequence of two-sided ideals in $e\sA e$
$$0=eJ_0e\subseteq eJ_1e\subseteq \cdots \subseteq eJ_me=e\sA e.$$
Since all standard modules $\Delta(\la)$ and $\Delta^\o(\la)$ of $\sA$ are $\sZ^\natural$-free, it follows that $e\Delta(\la)$ and $\Delta^\o(\la) e$ are projective $\sZ^\natural$-modules. Hence, they are $\sZ^\natural$-free by Swan's result (if non-zero). Since $J_i/J_{i-1}\cong
\Delta(\la^{(i)})\otimes\Delta^\o(\la^{(i)})$ as $\sA_i$-$\sA_i$-bimodules, where $\sA_i=\sA/J_{i-1}$, it follows that
$eJ_ie/eJ_{i-1}e\cong e\Delta(\la^{(i)})\otimes\Delta^\o(\la^{(i)}) e$ for all $i=1,\ldots,m$.
By induction, we may now construct a standard basis for $e\sA e$ using the bases for $e\Delta(\la)$ and $\Delta^\o(\la)e$.
\end{proof}

%\begin{proof}
%The standard basis is constructed from the local bases by induction. Since $\mathcal A$ is quasi-hereditary, a heredity ideal has the form $J=\mathcal Ae\mathcal %A\cong \mathcal Ae\otimes_{e\mathcal Ae} e\mathcal A$ (a $\mathcal A$-$\mathcal A$-bimodule isomorphism). If $fJf$ is nonzero, then the isomorphism
%$J\cong \mathcal Ae\otimes e\mathcal A$ induces isomorphism $fJf\cong f\mathcal Ae\otimes e\mathcal Af$. Thus, bases for $f\mathcal Ae$ and $e\mathcal Af$ can be used to construct a standard basis for $\mathcal A$ by induction (see the proof of \cite[(4.3.2)]{DR98}.
%\end{proof}

%For a commutative local Noetherian ring $\scrK$, a $\scrK$-algebra $\mathcal B$ is split quasi-hereditary  if and only if it is a standardly full-based algebra.

As a corollary, we obtain the following theorem. 
\begin{thm}\label{DM} (a) Assume that we are not in the case of  $^2F_4$. The $\sZ^\natural$-algebras $A^\natural$, $A^{+\natural}$ and $\sH^\natural$ are all standardly based algebras.  

(b) The analogous result holds in type $^2F_4$ for the corresponding $\sZ^{\natural\prime}$--algebras. (See the appendix to Section 4.1.)
\end{thm}

This result for the $\sZ^\natural$-algebra $A^\natural$ (and for $A^{+ \natural}$, of course) is new.
Geck \cite{G07} obtained a stronger result in the case of  $\sH^\natural$, namely,  that it has a cellular basis in the sense of 
\cite{GL96}.   As mentioned at the beginning of this section, cellular bases  require the existence of an anti-involution which interacts with the algebra's basis and its bilinear forms.

By analyzing the structure of endomorphism algebras of tilting objects, G. Bellamy and U. Thiel proved that Hecke algebras associated to finite complex reflection groups are standardly based algebras over a field $K$. For real reflection groups, these bases are even cellular; see \cite[Th.~3.8, Ex.~3.25]{BT22}.

%%%%%%%%%%%%%%%%%%%%%
%%%%%%%%%%%%%%%%%%%%%
%%%%%%%%%%%%%%%%%%%%

 \section{Decomposition matrices, I} 
 
 This and the next section present some applications to decomposition numbers for representations of finite
groups of Lie type (generally in cross-characteristic) and associated algebras, such as the generic  Hecke algebra $\sH$ from Section 1.2 and the endomorphism algebras 
$A$ and $A^+$ of display \eqref{endo}.     

  \medskip%\noindent%\medskip\noindent{\underline
  {\bf  5.2A Decomposition matrices  associated with $A$, $A^+$, and $\sH$.} In this subsection, we begin our study of decomposition numbers of the endomorphism algebras $A$ and $A^+$, defined in display \eqref{endo}, relative to a local triple $(\sO,K,k)$. These algebras, as well as an underlying (generic) Hecke algebra $\sH$, are built from a standard finite Coxeter system $(W,S,L)$ arising. by the definition in Section 1.1, from a finite group $G$ of Lie type. The ``generic" aspect of this subsection comes from our naming of $\sH$  in Section 1.2 as a generic Hecke algebra.
  
  We do not directly use $G$ in this subsection, although its underlying Lie theory context is in the background, especially the defining characteristic $p$ of the ambient algebraic group. 
  
  In particular, finite groups of Lie type do appear informally throughout this section, so we digress slightly to introduce some further notation regarding them.  Our assumption (in Section 1.2) that the Coxeter diagram of $W$ is connected implies we may take $\bG$ simple in the sense of algebraic groups \cite[p, 17, top]{C85}. (Alternatively, we can take $\bG$ to be connected reductive, provided  its semisimple part is simple. See \cite[
\S1.9]{C85}. Unless otherwise noted, we take $\bG$ to be simple.)
  The group $G$ may then be viewed as part of a Lie-type series associated to $\bG$ and a permutation $\rho$ (a graph automorphism) of its Dynkin diagram.
  There is also associated to $G$ a power $q$ of $\sqrt{p}$ described explicitly on \cite[pp. 39--42]{C85}.\footnote{In the split Chevalley group case, the group $G$ itself may be written $\bold G(q)$. Within each series, the allowable parameters $q$ determine a unique group $G$.} 
  
  As previously remarked, the group $G$ plays a relatively minor role in the present subsection, but we will return to it in the next section.  We do take this opportunity to remark that, if two groups $G$ are in the same series, then they have the same standard finite Coxeter system $(W,S,L)$, as shown by the construction \cite[pp. 34--36]{C85}. We also note that the prime $p$ will remain fixed throughout this section (as the defining characteristic of $\bG$), and the defining parameter $q$ of $G$ is a positive integer power of $\sqrt{p}$.  In fact, if the Ree and Suzuki groups are excluded, $q$ is a positive  integer power
   of $p$.
   
  We now return to the main task of studying decomposition numbers of $A$, $A^+$, and $\sH$.  The reader may recall the definition of a 2-regular local triple
 $(\sO',K',k)$ from Section 5.0. We call such a triple {\it well-adapted} provided the following three conditions hold:
  
  \medskip
  (1) $\sZ$ is a subring of $\sO'$;
  
  (2) The field $k$ has positive characteristic $r$ distinct from $p$; and
  
  (3) $K'$ is a splitting field for $\sH_{K'}$. 
  
  \medskip 
Note that (1) implies that $\mathbb Q(t)$ is contained in $K'$, and so $\sH_{K'} $ is always semisimple, even split, by (3). The split condition holds automatically except in the ${}^2F_4$ case. 
See footnote \ref{semisimpleH} in Chapter 1 for the semisimplicity  of $\sH_{\mathbb Q(t)}$.

  Also, {\it the split 
  semisimplicity of $\sH_{K'}$ implies that of $A_{K'}$ and $A^+_{K'}$, as endomorphism algebras of modules for a split semisimple algebra. }
  
An example of a well-adapted 2-regular local triple is the Laurent triple associated with a DVR triple $(\sO,K,k)$,
where char$(K)=0$ and char$(k)=r$ different from the defining characteristic of $G$. Also assume $\sqrt2\in K$ if $G$ is of type $^2F_4$. See Section C3 in Appendix C.
  
  Finally, we mention that, in the context of (1) above, condition (3) is always true when $G$ is quasi-split or has rank 1.
  This covers all $G$ except those of type $^2F_4$, in which case the field $K'$ containing 
   $\mathbb Q(t)$ is a splitting field if and only if $\sqrt{2}\in K'$.  It is also the case that $\sqrt{2}\in K'$ if and only if $\sqrt{2}\in \sO'$. This can be proved using the fact that $\sO'$ is a regular local ring, hence a UFD (Auslander-Buchsbaum), hence integrally closed \cite[Ex. 1, p. 414]{Jac89}.
    In turn, $\sqrt{2}\in\sO'$ if and only if 
    \begin{itemize}
    \item[(1$'$)]
    $\sZ'\subseteq \sO'$, where $\sZ':=\mathbb Z'[t,t^{-1}]$ with $\mathbb Z'=\mathbb Z[\sqrt{2}].$
    \end{itemize}
     (This condition is analogous to condition (1), and we call it condition (1$'$).)
 Consequently, a 2-regular local triple $(\sO', K',k)$ is well-adapted when $G$ is not of type $^2F_4$, provided   conditions (1)  and (2) hold and provided (1$'$) and (2) hold when $G$ is of type $^2F_4$.                                                             
    
    The reader may notice that $p=2$ in the $^2F_4$ case, so that requiring $\sqrt{2}\in\sO'$ is the same as 
    requiring $\sqrt{p}\in \sO'$. Later (such as in Lemma \ref{decompMAT} below), we will need to consider similar requirements for other types, and even $\sqrt[4]{p}\in\sO'$ in some cases, but this will not be necessary in Theorem \ref{12.1a} below.
        
  In the subsection ``Laurent triples" of Appendix C,  we show how to construct many 2-regular local triples $(\sO', K',k)$,
  starting with a DVR triple $(\sO,K,k)$ and using for $\sO'$ a localization of $\sO[t,t^{-1}]$. If we take $k$ to have positive characteristic $r\not=p$, as in (2) above, and $K$ to have characteristic 0, then conditions (1) and (2) will be satisfied. Also, (1$'$) will be satisfied if $\sqrt{2}\in K$.

\begin{rem} \label{good r}Assume that the prime $r$ is good for the group $G$. Then $\sZ^\natural$  is contained in
$\sO'$, while in type $^2F_ 4$, the ring $\sZ^{\prime\natural}$ is a subring of $\sO'$. Thus, the algebra
$A^{+}_{\sO'}$ is a split quasi-hereditary algebra, as follows from the main result of the previous chapter, Theorem \ref{quasi-hereditary}, together with Remark \ref{433}.
 \end{rem}

  \begin{thm} \label{12.1a} Let $(\sO',K',k)$  be a well-adapted 2-regular local triple with
   $r$ ($=$ the characteristic of $k$) good for
  $G$.
  %In particular,
   %$A^+_{\sO'}$ is split quasi-hereditary. (See Theorem \ref{quasi-hereditary}.)

  Let $D_{\sO'}$ be the decomposition matrix of $A_{\sO'}$, {and let $D^+_{\sO'}$ be the decomposition matrix of $A^+_{\sO'}$.} Then, the following hold.
  
  \begin{itemize}
 \item [(a)] With respect to an ordering of the (non-isomorphic) irreducible modules for $A^+_{K'}$ and $A^+_k$, the decomposition
  matrix $D^+_{\sO'}$ of $A^+_{\sO'}$ is unitriangular and square.
  
 \item[(b)] Let $D^+_{\sO'}$ be arranged as in (a). Then, a version of $D_{\sO'}$ can be obtained by removing some (possibly none) of the rows of $D^+_{\sO'}$.  Also, this puts $D_{\sO'}$  in row echelon form with no row of zeros. 
  
\item[(c)] In particular, the indices of the columns of $D_{\sO'}$ in (b) above can be permuted so that $D_{\sO'}$ has the form $[D_1,D_2]$, where $D_1$ is a square unitriangular matrix and $D_2$ is a rectangular matrix with the same number of rows as $D_1$.  In particular, this version of $D_{\sO'}$ is itself unitriangular.
  
  \item[(d)] $D^+_{\sO'}$ and $D_{\sO'}$ have the same number of columns. Also, $A^+_{K'}$ and $A_{K'}$  have the same number
  of non-isomorphic irreducible representations as $\sH_{K'}$. 
  
  \item[(e)] Parts (b) and (c) hold if $D_{\sO'}$ is replaced by the decomposition matrix $D^\circ_{\sO'}$ of $\sH_{\sO'}$. Hence, there exist submatrices $D'$ and $D^{\circ\prime}$ of $D^+_{\sO'}$, $D_{\sO'}$, respectively, such that
 $$D^+_{\sO'}=\begin{pmatrix} D_{\sO'}\\ D'\end{pmatrix} \text{ and } D_{\sO'}=\begin{pmatrix}D^\circ_{\sO'}\\D^{\circ\prime}\end{pmatrix}.$$
   \end{itemize}
  \end{thm}

  \begin{proof}  We will give arguments for the split, quasi-split,  and some $^2F_4$ cases, leaving remaining $^2F_4$ arguments, and the rank 1  Ree and
  Suzuki groups to the reader.
  
   We first prove (a),\footnote{For a proof using standard basis theory, see Remark \ref{SBApf}(c) below.}  noting that the $\sZ^\natural$-algebra $A^{+\natural}$ in Corollary \ref{refine3} gives rise to highest weight category $(\mathscr C,\{\Delta(E)\}_{E\in\Lambda})$ in the sense of Rouquier \cite[Def. 4.11]{Ro08}. Here, $\mathscr C=A^{+\natural}\!\!-\!\!\text{mod}$, the poset $\Lambda$ is discussed in the paragraph surrounding \eqref{orderLa}, and the $\Delta(E)$'s (and $\Delta'(E)$'s) are discussed in 
Remark \ref{ss case}.  (The $\Delta'(E)$'s are variations on the $\Delta(E)$'s. They are used in the $^2F_4$ case.)

The $\Delta(E)$'s above, in Remark \ref{ss case}, agree with those in Corollary \ref{refine3}, as may be shown by tracing (backwards) through the discussion above Corollary \ref{refine}. (Here the discussion of the $^2F_4$ case is left to the reader.) A big advantage is that all $\Delta(E)$'s in Remark \ref{ss case} are shown to have $\Delta(E)_{\mathbb Q(t)}$ irreducible. Applying this, and then base changing, we may conclude that $\Delta(E)_{K'}$ is irreducible, in our $(\sQ', K',k)$ context, over $A^+_{K'}=(A^+_{\sO'})_{K'}$.

We begin with the highest weight category 
%$(\mathscr C,\{\Delta(E)\}_{E\in\Lambda})$ studied in  Corollary \ref{refine3}. The category 
$\mathscr C_{\sO'}:=A^{+\natural}_{\sO'}\!-$mod by base changing $\mathscr C$ above from $\sZ^\natural$ to $\sO'$.
Using the notation in Corollary \ref{refine3}, together with base-change properties such as \cite[Lem. 1.2.5(4)]{DPS98a} and \cite[Prop. 4.14]{Ro08},\footnote{Compare also Remark \ref{433}.} it is useful to define $V(\lambda)=\Delta(\lambda)_{\sO'}$ in $\mathscr C_{\sO'}$, modifying $E$ to $\lambda\in \Lambda$ by abuse of notation. Here $\Lambda=\text{Irr}(\mathbb QW)$ (or $\Lambda=\text{Irr}(\mathbb Q(\sqrt2)W)$ if $G$ is of type ${}^2F_4$), and $\Delta(\lambda)=(S^\natural_\lambda)^\diamond$ with $S^\natural_\lambda$ as defined in display \eqref{identify}. 
  %We first prove (a). By Remark \ref{good r}, the $\sO'$-algebra $A^+_{\sO'}$ is quasi-hereditary. Of course, $\sO'$ is local and Noetherian, so $A^+_{\sO'}$--mod is an $\sO'$-finite highest weight category in the sense of \cite[Prop. 2.1.1 and Defn. 2.1]{DS94}.\footnote{\color{red}The proof in the non-local split case mentioned in \cite{DS94} requires a similar discussion at the end of Section 2 in \cite{CPS90}. }
%  There are several consequences. 
Thus, passing from $\sO'$ to $k$, $A^+_k$--mod is a highest weight category in the original sense of \cite{CPS88} for the poset  $(\Lambda,\preceq)$, where $\preceq$ is defined in \eqref{orderLa}.  
In particular, $\Lambda$ is in bijection with the distinct irreducible modules
$\{L(\lambda)\}_{\lambda \in \Lambda}$ of $A^+_k$--mod, where $L(\lambda)$ is the head of 
 the corresponding standard modules $V(\lambda)_k$. 
 %The  $\overline{(-)}$  notation  is  used because $\bar V(\lambda):=V(\lambda)_k$  for the $\sO'$-free object $V(\lambda)$ in $A^+_{\sO'}$--mod. 
% See  \cite[Defn. 2.1]{DS94} for the general structure of an $\sO'$-finite highest weight category. 

Next, we note that each $V(\lambda)_{K'}$ is irreducible and that all irreducible objects in $A^+_{K'}$--mod have the form $V(\lambda)_{K'}$. (This follows from Remark \ref{ss case}.) Hence, all the decomposition numbers needed in (a) have the form $%[V(\lambda)_{K'}:L(\mu)]_{\sO'}=
[V(\lambda)_k:L(\mu)]$, for any $\lambda, \mu \in \Lambda$. This gives a $|\Lambda|\times|\Lambda|$ square(!) matrix as well as  %Alternatively, }the needed irreducibility of all $V(\lambda)_{K'}$ follows from the complete reducibility of all modules in $A^+_{K'}$--mod (in particular, of $V(\lambda)_{K’}$), together with the fact that the head of  $V(\lambda)_k$ is irreducible, and appears just once in any composition series of $V(\lambda)_k$.
 the unitriangular property, since, by \cite[Def. 2.1(i)]{DS94}, we have
    $$
[V(\lambda)_k:L(\mu)]=\begin{cases}1,&\text{ if }\lambda=\mu;\\
0,&\text{ unless }\mu\prec \lambda.\end{cases}$$
This completes the proof of (a). As noted, Remark \ref{SBApf}(c) gives an alternate proof.

%  Next, note that Remark \ref{good r} shows that $A^+_{\sO'}\cong \sO'\otimes_{\sZ^\natural}A^{+\natural}$ is the base-change of the split quasi-hereditary algebra $A^{+\natural}$. Now observe
%  that all the standard modules $\Delta(\la)_{\sO'},\la\in\La$, are irreducible after base change to $K'$ and have simple heads after base change $k$. Statement (a) follows.
  
 To prove  Part (b) (and Part (d)),\footnote{This proof is inspired in part by
arguments for  \cite[Th. 6.3.2]{DS00a}, but we include further details here.}  we first observe from the notation in Construction \ref{const} that $A^+\cong
 {\small\begin{pmatrix}\Hom_{\sH}(T,T)&\Hom_{\sH}(X,T)\\\Hom_{\sH}(T,X)&\Hom_{\sH}(X,X)\end{pmatrix}}$.
This isomorphism induces isomorphisms $eA^+e\cong \begin{pmatrix}A&0\\0&0\end{pmatrix}\cong A$, where
$e\in A^+$ is the idempotent corresponding to $\begin{pmatrix}1_T&0\\0&0\end{pmatrix}$. Clearly, by
noting that $T^+={T\choose X}$,
$e$ defines the natural projection $T^+\to T\subseteq T^+$.  Thus, by combining the isomorphisms above, we may now identify $A$ with $eA^+e$ and $A_\scrK$ with $eA^+_\scrK e$ for any commutative ring $\scrK$.
In the rest of the proof, we  will need only the cases $\scrK = k$ and $\scrK = K'$. These are fields, of course, allowing us to quote from
\cite[Section 6.2]{Gr80} for algebras $S$ over them such as $S=A^+_k$ and idempotents $e \in S$, taking, for instance, 
$eSe =eA^+_k e=A_k$,

  If $L$ is an irreducible $A^+_k$-module, then $eL$ is either 0 or an irreducible $A_k$-module (observing  that $A_k=eA^+_ke$). Further, the irreducible $A_k$-modules all arise this way. Also, if $L,L'$ are irreducible $A^+_k$-modules such that $eL\not=0\not= eL'$, then $eL\cong eL'$ if and only if $L\cong L'$. For these assertions, see \cite[Th. (6.2g)]{Gr80}; their analogs also hold with $K'$ replacing $k$.  In fact, the latter case can be used to prove Part (d).\footnote{See Remark \ref{ss case} for an alternate proof of Part (d).}  (Since  the module $\sH_{K'}=x_\emptyset\sH_{K'}$, %=(\sH_{K'})_{\sH_{K'}}
with $x_\emptyset$ defined in display \eqref{longest},  is a direct summand 
of both $T^+_{K'}$ and $T_{K'}$, we have, in this semisimple case, $eV(\lambda)_{K'}=e\Delta(\lambda)_{K'}\neq0$, for all $\lambda\in\Lambda$ since $\Hom_{\sH_K}((S_{\lambda}^{\natural})_K,\sH_K)\neq0$. Thus, $\{eV(\lambda)_{K'}\mid \lambda\in\Lambda\}$ forms a complete set of irreducible $A_{K'}$-modules.
Note that $\Lambda$ also indexes irreducible modules of $\sH_{K'}$; see \eqref{IrrWH}. Hence,  Part (d) follows.)
  
  Continuing now the proof of Part (b), note that if $L=V_i/V_{i-1}$ is a composition factor of $V(\lambda)_k$, where $V_{i-1}\subseteq V_i$ are submodules of $V(\lambda)_k$, and $eL\neq 0$, then $eL$ is irreducible and  \cite[(6.2a)]{Gr80} now implies that $eL\cong eV_i/eV_{i-1}$ is a composition factor of $eV(\lambda)_k$. Also, every composition factor of $eV(\lambda)_k$ has such a form. This shows that  a composition series $0=V_0\subseteq V_1\subseteq\cdots \subseteq V_m=V(\lambda)_k$ of  the left $A^+_k$-module $V(\lambda)_k$ gives rise to a composition series 
 $0=eV_0\subseteq eV_1\subseteq\cdots \subseteq eV_m=eV(\lambda)_k$ of the $A_k$-module $eV(\la)_k$ and, for every  $L$ with $eL\neq0$, $
[V(\lambda)_k:L]=
[eV(\lambda)_k:eL]$. 
 Hence, Part (b) follows except for its last assertion and that follows from Part (c).

%  may be obtained by applying $e$ to each factor of a composition series for the $A^+_k$-module $X_k$.  This argument, together with \cite[Th. 6.3.2]{DS00a}, proves Part (b)  except for its last assertion. The latter follows from Part (c), once (d) is proved. 
 
%To prove  (d), observe that $A^+_{K'},A_{K'}$, and $\sH_{K'}$ are all $\sH_{K'}$-endomorphism algebras of the $\sH_{K'}$-modules $T^+_{K'},T_{K'}$, and (the right regular module) $\sH_{K'}=(\sH_{K'})_{\sH_{K'}}$, respectively. The module $\sH_{K'}$ is a direct summand  of both $T^+_{K'}$ and $T_{K'}$. Since $\sH_{K'}$ is a split semisimple algebra (because ${K'}$ contains $\mathbb Q(t)$, and also
%contains $\sqrt{2}$ in the $^2F_4$ case), each of the endomorphism algebras is a direct product of full matrix algebras over $K'$.The number of matrix factors is the same in each case, namely the number of non-isomorphic irreducible $\sH_{K'}$-modules. (All irreducible $\sH_{K'}$-modules appear as summands of $T_{K'}, T^+_{K'}$ and $\sH_{K'}$.)
%If multiplication by
%$e$ killed any irreducible $A_K^+$-module, the algebras $A^+_K$ and $A_K=eA^+_Ke$ would have a different %number of non-isomorphic irreducible modules (and, hence, a different number of full matrix algebra factors). This proves  (d).

We leave the proofs of (c) and (e) to the reader. This completes the proof of the theorem.
   \end{proof}
   
%{\color{red} \begin{cor} For the decomposition matrices $D^+_{\sO'}$, $D_{\sO'}$, $D^\circ_{\sO'}$ of $A^+_{\sO'}$,  $A_{\sO'}$,  $\sH_{\sO'}$, respectively, there exist submatrices $A$ and $B$ of $D^+_{\sO'}$, $D_{\sO'}$, respectively, such that
 %$$D^+_{\sO'}=\begin{pmatrix} D_{\sO'}\\ A\end{pmatrix}, \text{ and } D_{\sO'}=\begin{pmatrix}D^\circ_{\sO'}\\B\end{pmatrix}.$$
% \end{cor}}
 
   \begin{rems}\label{SBApf}  (a) These results appear to be new, certainly for $A$ and $A^+$, but, because of our 
   ``generic" focus, also new for $\sH$.
   Notice that $\sH_{\sO'}$ is a local version of the {\it generic} Hecke algebra $\sH$, not one of its standard specializations (see Subsection 5.2B below), which typically require a replacement of the variable $t$ or $t^2$ with a power of $\sqrt{p}$. For a variation of the theorem
   for such a standard specialization of $\sH$, see
  \cite[Th. 4.4.1]{GJ11}. 
  
  %The same results for $D$ itself, however, appear to be new, and give evidence for \cite[Conj. 4.5.2]{GJ11}
  
    (b) Another variation of the theorem is obtained in \cite[(6.3.3)]{DS00a} in Lie type $B$, where a different version of $A^+$ is used, not involving Kazhdan-Lusztig cell modules. 
    
 (c) {\bf Second proof of Theorem \ref{12.1a}(a).}  The fact that $\sO'$-algebra $A^+_{\sO'}$ is quasi-hereditary of split type (see Remark \ref{good r}) implies that $A^+_{\sO'}$ is a standardly full-based algebra associated with a poset $\Lambda$ (see Proposition \ref{prop125}). Thus, every bilinear form associated with a cell module for $A^+_{K'}$ or $A^+_{k}$ is nonzero. Hence, we have $\Lambda_{1,K'}=\Lambda=\Lambda_{1,k}$ for both $A^+_{K'}$ and $A^+_k$ (see Lemma \ref{La1}(b) and  Corollary \ref{SFBA}). This proves that the decomposition matrix $D^+_{\sO'}$ is square.
The unitriangular property follows from Lemma \ref{La1}(c)\&(d).
    \end{rems}
    
   % \medskip\medskip\begin{color}{red} Jie:
    
   % We must explain the notation $\sO$ and $\sO'$. This comes between theorem 12.3 and stuff below on finite groups.
   % Integrate with half paragraph that begins with From .... std. specializations.  $A_\sO$ is the classical $q$-Schur alg. %with genuine $q$. \end{color}}}
    
    \medskip
    %\noindent
{\bf5.2B. From the generic case to standard specializations.} With the exception of the more general Lemma \ref{functor} below, we continue with the notations
    $\bG$ and $G$ of the previous subsection. In particular, $\bG$ is a simple algebraic group. We will fix $\bG$ and the Lie-type series to which $G$ belongs. In the setup
    of \cite[p. 31]{C85}, this means there is a ``Frobenius endomorphism"\footnote{Also called a ``Frobenoius map'' in Section 1.1. Thus, some power of $F$ is a standard Frobenius  endomorphism.}  $F$ of $\bG$ with $G=\bG^F$ and with $F$ inducing a fixed 
    permutation $\rho$ of the Dynkin diagram of $\bG$ \cite[pp.31--41]{C85}. Many of these---and other---details may
be found in Section 1.1 of this monograph---and others.
     %Within its Lie type series, the group $G=\bG^F$ is determined by 
   % (and determines) a parameter $q$ \cite[p. 35]{C85}, discussed further in the next paragraph.
    
%      
{\color{black}    We %keep $\bold G$ simple or semisimple, as context dictates, unless otherwise noted and 
retain the notations $F,G, \bold B, B$  introduced in Section 1.1, which may also be used in a more general context such as a connected reductive group equipped with a Frobenius endomorphism.}
     
     It is useful to consider connected reductive groups $\widetilde{\bold G}$ which are central products\footnote{See \cite[p.29]{Go80} for the definition of a central product.} $\widetilde\bZ\bold G$ of a torus $\widetilde\bZ$ with $\bold G$. We assume, in the structures we seek, that $F$ extends to a Frobenius endomorphism $\widetilde F$ of $\widetilde{\bold G}$ and write
    $\widetilde G=\widetilde{\bold G}^{\widetilde F}$. We also set $\bold{\widetilde B}=\widetilde\bZ{\bold B}$, 
    $\widetilde B=\widetilde{\bB}^{\widetilde F}$, and $\widetilde Z=\widetilde\bZ^{\widetilde F}$. Note that $\widetilde\bZ$ is the connected component of the identity in the center of $\widetilde \bG$, and, thus, is stable under $\widetilde F$.
    The lemma immediately below constructs functorially such a pair $\widetilde{\bold G}, \widetilde F$ in which  $\widetilde\bZ$ is the full center of
    $\widetilde{\bold G}$. (Compare \cite[pp.161--163]{Lus88} and \cite[\S\S1,2]{DF94}.)
    
    \begin{lem}\label{functor}  If $\bG$ is semisimple, then there exists a pair $\widetilde{\bG}, \widetilde F$ in the setup of the above paragraph, in which $\widetilde{\bG}=\widetilde\bZ\bG$ is connected reductive and $\widetilde\bZ$ is the full center of $\widetilde{\bG}$.
   (Equivalently, the center of $\widetilde\bG$ is connected.) Also, $\widetilde\bZ$ acts trivially on the coset space $\widetilde{\bG}/\widetilde{\bB}$. The $G$-sets $(\widetilde \bG/\widetilde{\bB})^{\widetilde F}$, $\widetilde G/\widetilde B$, $(\bG/\bB)^F$,
   and $G/B$ are all naturally isomorphic. 
    \end{lem}
    
    \begin{proof}
Observe that $Z(\bG)$ is finite and contained in all maximal tori $\bT$ of $\bG$. 
%(Notice we are not consistently using bold face when labelling maximal tori in the algebraic groups sense.) 
Take $\bT$ to be $F$-stable. (See \cite[p. 33]{C85}.) 
Let $\bT_1$ be a copy of $\bT$ independent of $\bG$. Let $Z(\bG)_1\subseteq \bT_1$ correspond to $Z(\bG)$.  Thus, the natural action of $F$ on $\bT$ gives rise to a corresponding action on $\bT_1$ and on $Z(\bG)_1\subseteq \bT_1$. The natural maps $Z(\bG) \cong Z(\bG)_1
   \subseteq \bT_1$ and $Z(\bG)\subseteq\bG$ give rise to a connected, reductive pushout  group: $\widetilde {\bG}=\bT_1 \bG$ with (connected) center $\bT_1$, and we now write $\widetilde\bZ=\bT_1$.
   Also,
   $F$ extends to a Frobenius endomorphism on $\widetilde \bG$, duplicating on $\bT_1$ its action on $\bT$.   
   
   Since the coset space $\widetilde\bG/\widetilde\bB$ identifies here, as a $\widetilde \bG$-set, with the set of $\widetilde\bG$-conjugates of $\widetilde\bB$, the action of $\bT_1=\widetilde\bZ:=Z(\widetilde\bG)$ is trivial.  The $G$-set identifications
    at the end of the lemma follow from elementary arguments and the Lang-Steinberg theorem, as illustrated in \cite[p. 33]{C85}. All the identifications are the ``obvious" candidates.  This completes the proof. \end{proof}
   
       \begin{rem}{\color{black} The lemma is also true if ``semisimple'' is replaced by ``connected reductive.'' We sketch a proof in this case. First, change notation so that the group in question is $\bG \bT_2$, where $\bG$ is semisimple and $\bT_2$ is a (central) torus. Apply the proof of the lemma for the group $\bG$ to show $\bT_1\bG$ has connected center $\bT_1$. Then construct $\bT_1\bG \bT_2$ and show it has center $\bT_1\bT_2$. The remaining details of the proof are left to the reader. (Later, we only use the lemma as stated, and only when $\bG$ is simple.)}
    \end{rem}

        {\it In the rest of subsection 5.2B},
       { let $(\sO,K,k)$ be a DVR triple (in some cases, a more general local triple could be used) in which}
      \begin{itemize}
       \item[$\bullet$] the field $k$ has positive characteristic $r$, distinct from the defining characteristic $p>0$ of $\bold G$;  
       \item[$\bullet$] $K$ has characteristic 0; and
       \item[$\bullet$] $\sqrt{2}\in \sO$ if $G$ is of type $^2F_4$.
       \end{itemize}
  These choices are natural when concerned with the representation theory of $kG$ (i.e., the
  cross-characteristic theory)---see specifically \cite[(1), p. 170]{DPS98a}.  
       
       We continue the assumption of this subsection
 that $G$ is the group ${\bold G}^F$ in a fixed Lie-type series  with $\bG$ simple (as discussed above). Associated with $G$ is a parameter $q$ (which determines $G$ in its series) and a parameter $\delta$ which equals $1$, $2$, or $3$, depending on the series, with $q^\delta$ a power of $p$ (an odd power, if the series is Ree or Suzuki).
  
  \medskip
  The following lemma, included for completeness, provides largely well-known ``standard specialization" properties of the generic Hecke algebra $\sH$. {\color{black}Recall the $\mathbb Z[t^2]$-form $\sH^o$ of $\sH$ considered in Lemma \ref{H^o}.

\begin{lem}\label{goodone}  Let $(\sO,K,k)$ be a DVR   triple as detailed above. Suppose $\sqrt{q}\in\sO$ is invertible in $\sO$; this gives $\sO$ a natural structure as a $\sZ$-algebra (via $t\mapsto\sqrt{q}$).  Then, the base-changed algebras $\sH_\sO$ and
$\sH^o_\sO$ are naturally isomorphic, and 
the base-changed algebras
 $\sH_K$ and $A_K$ are split semisimple.
\end{lem}

 \begin{proof}  The isomorphism $\sH_\sO\cong\sH^o_\sO$ follows from the hypothesis.
It remains to verify the final assertion of the lemma.  Theorem \ref{maxorder} shows that $\sJ^\natural$ is a direct product of matrix algebras over $\sZ^\natural$. Base change from $\sZ^\natural$ to $K$ gives that $\sJ_K$ is a similar direct product of full matrix algebras.  Taking $R=K$ in the discussion below display (\ref{phiformula}), we have $\sH_K\cong\sJ_K$.     (Note that $\sH_K\cong \sH^o_K$ is semisimple, as the endomorphism ring of a $KG$-module.) The isomorphism with a direct product of matrix algebras shows that $\sH_K$ is split  semisimple.

Finally, $A_K$ is the endomorphism algebra of an $\sH_K$-module by \cite[Th. 2.4.4(b)]{DPS98a}. Since $\sH_K$ is split semisimple, it follows that $A_K$ is also split semisimple.
 \end{proof}}

Recall the definition of a Laurent triple associated with a DVR triple in Appendix C. (See the paragraph above  Example \ref{example} in Appendix C.)

  \begin{lem}\label{decompMAT} Let $(\sO,K,k)$ be a DVR triple as detailed above. %with $K$ of characteristic zero  and $k$ of positive characteristic $r \not= p$ (the defining characteristic of $\bG$). 
  Also, assume $\sqrt{q}\in\sO$.\footnote{We note that the latter assumption is implied by the condition that $\sqrt[4]{p} \in\sO$. If the series of $G$ is not Ree or Suzuki, this may be weakened to the condition that $\sqrt{p}\in\sO$.}   
   Let $(\sO',K(t),k)$ be the Laurent triple associated to the DVR triple $(\sO,K,k)$ and the
  element $a=\sqrt{q}$. 
  
  (a) The decomposition matrix $D_{\sO'}$ of $A_{\sO'}$ (with respect to the triple $(\sO', K(t),k)$) 
  agrees with the decomposition matrix $D_{\sO}$ of $A_\sO$ with respect to $(\sO,K,k)$. 
  
  (b) Similarly, the decomposition matrix of $\sH_{\sO'}$ agrees with that of $\sH_\sO$.  
   
  \end{lem}
  
  \begin{proof}The reader may check that the diagram in display (\ref{setup2}) in Appendix C is interpreted as follows for $B=A$ or $\sH$ and with the indicated prime ideal $\mathfrak p$

    \begin{equation}\label{setup}
\begin{CD}
B @>>> B_{\mathfrak p} @>>> {B_Q}\\
@VVV @VVV\\
B/\mathfrak pB @>>>B(\mathfrak p)\\
@VVV \\
B_F 
\end{CD}
\end{equation}   
where in case (a), $B= A_{\sO'}$ and $\mathfrak p=(t-a)\in\sO'$, while $B(\mathfrak p)$ identifies with 
$B_K=B_{(O'/\mathfrak p)_{\mathfrak p}}$. Now apply Theorem \ref{cutie} and Lemma \ref{goodone}.

 A similar argument
 applies in the Hecke algebra case (b), where $B=\sH_{\sO'}$. This proves the lemma.\end{proof}

% By now, we have enough properties of $A$ to give it a name. 
The notation $G=\bG^F$ below continues as above. Recall the generic Hecke endo-algebra $A$ defined in \eqref{q-Schuralgebra}.

 \begin{thm} \label{bigDecom}  Suppose $G_1, G_2$ are finite groups of Lie type in the same Lie-type series as $G$ and let $q_1^\delta, q_2^\delta$ be the powers of $p$ associated to $G_1, G_2$, respectively (as in \cite[pp. 38--41]{C85}). 
 
  Let $(\sO,K,k)$ be a DVR triple as in (the first sentence of) Lemma \ref{decompMAT}, and assume that $\sqrt{p}\in\sO.$ If $G$ is of Ree or Suzuki type, assume
 $\sqrt[4]{p}\in\sO$.

Fix $B=A$ (the generic Hecke endo-algebra) or $B=\sH$ (the generic Hecke algebra).

   Suppose that $D'$ and $D^{\prime\prime}$ are the decomposition matrices of $B_{\sO'}$  and $B_{\sO^{\prime\prime}}$, with respect to the  Laurent triples $(\sO',K',k)$ and $(\sO^{\prime\prime}, K^{\prime\prime},k)$ associated to 
 $(\sO,K,k) $ with $a_1=\sqrt{q_1}$ and $a_2=\sqrt{q_2}$, respectively.\footnote{Thus,
 $(\sO',K',k)$ and $(\sO^{\prime\prime},K^{\prime\prime},k)$ are defined in the paragraphs above Example \ref{example} in
 Appendix C with $a_1$ and $a_2$ as the parameter $a$ there. In particular, $K'=K''=K(t)$.}
  Also, viewing $\sO$ as the $\sZ$-algebra $\sO_1$ or $\sO_2$ via the specializations $t\mapsto a_1$ and $t\mapsto a_2$, respectively, let $D_1, D_2$ be the decomposition matrices for $B_{\sO_1}, B_{\sO_2}$, respectively, relative to
   the DVR triples $(\sO_1,K_1,k)$, $(\sO_2,K_2,k)$.
  
  Then, the decomposition matrices  $D'$, $D^{\prime\prime}$, $D_1$, and $D_2$
   can all be identified with the same matrix $D$, where $D$ depends only on the choice $B=A$ or $B=\sH$.
  \end{thm}
 
 \begin{proof} 
First, assume $B=A$ and $G$ is not of Ree or Suzuki type. 

 Lemma \ref{decompMAT} gives an identification of $D_1$ with $D'$ and an identification of $D_2$ with
 $D^{\prime\prime}$, while Proposition \ref{old prop 10} identifies $D'$ and $D^{\prime\prime}$.  It is useful to note here that
 there is, in the notation of Example \ref{example}, a Laurent triple $({^a \sO'}, K',k)$ with $a=\sqrt{p}$, even if this value of $a$ is not
 $a_1, a_2$, or associated to any $G$ in the Lie-type series. So, the relations $(\sO',K',k)\leq( {^a\sO'}, K(t),k)$ 
and
 $(\sO^{\prime\prime}, K^{\prime\prime},k)  \leq ({^a\sO'}, K(t), k)$
 can be used with Proposition \ref{old prop 10} to identify $D^{\prime\prime}$
 with $D'$ as follows: First, note that there is, in the notation of Example \ref{example}, a Laurent triple $({^a\sO}',K(t),k)$ with $a=\sqrt{p}$, even if this value of $a$ is not $a_1,a_2$ or associated to $G$ in the Lie-type series. So, the relations $(\sO',K',k)\leq( {^a\sO'}, K(t),k)$ 
and
 $(\sO^{\prime\prime}, K^{\prime\prime},k)  \leq ({^a\sO'}, K(t), k)$ can be used with Proposition \ref{old prop 10} to identify $D^{\prime\prime}$ with $D'$.
 
 We leave the remaining cases, with either $A$ replaced  by $\sH$ or $G$ replaced by a group of Ree 
 or Suzuki type, to the reader. \end{proof}

  \section{Decomposition matrices, II} 
  
  We have seen from Theorem \ref{bigDecom} that the decomposition matrix $D$ is an invariant relative to the standard specializations arising from finite groups $G$ of Lie type in the same Lie-type series. We now prove that, for such a $G$ and DVR triple $(\sO,K,k)$, this matrix $D$ is part of the decomposition matrix of $\sO G$. This shows that $D$ is, in fact, an invariant associated with the Lie type series.  
% \medskip\noindent  {\bf5.2d. From standard specializations to some unipotent representations.}

     We continue with the notation in the previous section. The unipotent representations\index{unipotent representation} in this section are the irreducible constituents of the permutation module $\mathbb C G/B$,
where $B=\bB^F$ for an $F$-stable Borel subgroup $\bB$ of $\bG$.
     
     In the following theorem, as in Theorem \ref{bigDecom}, {\color{black}$(\sO,K,k)$ is a DVR triple with $\sO$ complete, $k$ is algebraically closed of characteristic $r\not=p$, and $K$ has characteristic 0} (see Appendix A, Main Theorem). Also, by Hensel's lemma, $\sO$ contains a primitive $p^{th}$-root of unity. Here $p$
     is the defining characteristic of $\bG$.  We also require, as in Theorem \ref{bigDecom}, that $\sqrt{p}\in \sO$ and
     $\sqrt[4]{p}\in\sO$ for $G$ of Ree or Suzuki type; thus, we always  have $\sqrt{q}\in\sO$.

     The proof of part (a) involves methods invented by Dipper many years ago
     (although we could not find an explicit statement in \cite{DG99}). Part (b) seems to be a remarkable fact, not considered before in the literature. It is based on our study (in Appendix C) of Krull dimension 2 decomposition numbers. 
Part (c) depends on our main theorem in Subsection 5.2B, and provides evidence for the conjecture \cite[Conj.~4.5.2]{GJ11}, which references work of Geck \cite{G90,G06} and Geck--Hiss \cite[Conj.~3.4]{GH97}. %The references to Geck include his 1990 Ph.D thesis and a later study in 2006.
%\cite[Conj. 4.5.2]{GH97} of Geck and Hiss ( Geck's 1990 Ph.D. thesis).   
 Geck's  conjecture \cite[Conj.~2.1]{G12}, originating from his 1990 Ph.D. thesis \cite{G90}, on unitriangular shape of decomposition matrices of unipotent blocks has been proved recently in \cite{BDT20}, under the assumption that the defining characteristic for $\bG$ is good and the non-defining characteristic $r$ is {\it very good}.\footnote{\label{BDT}Call $r$ a very good prime if $r$ is good and $r$ does not divide the order of $Z(\bG)_F$, the largest (finite) quotient of  $Z(\bG)$ on which $F$ acts trivially. See \cite[Th.~A]{BDT20}.} The following result suggests that, for certain unipotent blocks, these conditions could be weakened by simply assuming that $r$ is good. (Compare bottom line of \cite[p.587]{BDT20}.)

   \begin{thm} \label{12.3}  Let $G=\bG^F$ be as above.
  (Thus, $\bG$ is simple,\footnote{We continue to make this assumption for simplicity; see footnote \ref{simpleG} in Chapter 1. In fact, much of the argument applies if $\bG$ is just semisimple, or even connected reductive.} and we think of $G$ as a finite group in a given Lie-type series.)
    
   (a) Let the matrix $D$ be as described above in Theorem \ref{bigDecom} for the generic Hecke endo-algebra $A$.
   Then, $D$ is a submatrix of the decomposition matrix of $\sO G$.  The columns of $D$ are indexed by the irreducible
  constituents of the permutation module $\mathbb CG/B$.

(b) The same statements hold, with the same 
  matrix $D$, if $G$ is replaced by any group in the same finite Lie-type series.   
  
  (c) If $r$ is a good prime for $\bG$, then $D$ above is unitriangular. 
  \end{thm}

  \begin{proof} 
  
  First, assuming that (a) holds, (b) follows from Theorem \ref{bigDecom} and (c) follows from Theorem \ref
   {12.1a}(c) and 
  Theorem \ref{bigDecom}.
  
  It remains to prove (a).   This assertion contains two parts, the second of which, counting columns, follows from the argument for Theorem \ref{12.1a}(d). It remains to prove the first part of (a), which asserts an inclusion of decomposition matrices.  We divide the proof into two parts.
  
  We give the proof when $\bG$ is simple of adjoint type in Part 1. In this case, we may directly apply Theorem \ref{kerfactors} in Appendix A. We also supply the main ingredient in the argument for the general case. 
In Part 2, we prove a variation on Theorem \ref{kerfactors} for $\bG$ simple, reducing the general case to the validity of Theorem \ref{kerfactors} for the group $\widetilde{\bG}$ in Lemma
  \ref{functor} (which has connected center).

\medskip\noindent  
  {\bf Part 1. The adjoint-type case}
  
%  We give the proof when $\bG$ is simple of adjoint type, then supply the main ingredient in the argument for the general case. The ingredient we prove is a variation on Theorem A.1 in Appendix A for $\bG$ simple, reducing that case to the validity of Theorem A.1 for the group $\widetilde{\bG}$ in Lemma
 % \ref{functor} (which has connected center).
  
  Let $\bG$ be a simple algebraic group of adjoint type. Then, by the proof of \cite[Prop. 9.15]{MT11},  $\bG$ has a trivial center $Z(\bG)=1$; see also the proof of \cite[Prop.~8.1.2]{C85} on page 254 for this fact. Thus, Theorem \ref{kerfactors} applies in the current (adjoint) case.
  
  Let
  $$\mathcal M=\mathcal M_\sO:=\underset{J\subseteq S}{\bigoplus} R_{L_J}^G(\mathscr{S}_\sigma^{L_J}),$$ 
  using, on the right-hand side, the notation of the proof of Theorem \ref{kerfactors}.  Then, Theorem \ref{kerfactors} holds for %$M_k$ as does its analog for 
  $\sM_K$.  %(This follows from Theorem A.1 or its proof.)  
  This provides  a projective cover of $\sM$ satisfying 
          \cite[2.3]{Dip91}. Equivalently, in the context of \cite[\S4]{Dip90}, the cover satisfies \cite[2.6]{Dip90}; see \cite[Th.~4.7]{Dip90}. Thus, by \cite[Cor. 4.10]{Dip90}, %\cite[2.6iii]{Dip91}, 
          the decomposition matrix of $\End_{\sO G}(\sM)$ is a submatrix of the decomposition matrix of $\sO G$. (See the formulation of \cite[Th. 4.1.14]{GJ11} which uses a slightly different hypothesis on ($\sO,K,k$) but also references the same result \cite[Cor.~4.10]{Dip90}.) 
          %, we mention, does not assume completeness.) 
           Now, assertion (a) follows from the algebra isomorphism $\End_{\sO G}(\sM)\cong A_\sO$, which is proved below.%see \cite[Th. 2.24]{DJ89} for the general linear group case.
           
           For an $F$-stable maximal torus $\bT\subseteq\bB$, let $T=\bT^F$.        Let $M_\emptyset:=R_{T}^G\sO=\sO G/B$, where $\sO$ is the trivial $T$-module. 
         By remarks before the statement of the theorem, $\sO$ contains a  primitive $p^{th}$-root of unity. Thus, the 
       underlying regular character $\sigma$ of $\Gamma_\sigma$ is defined over $\sO$.  (The reader may wish to review properties of Gelfand-Graev characters and their underlying modules in \cite[\S8.1]{C85} or the helpful summary in \cite[\S4.3]{G17}.)
       As an $\sO G$-module, $\Gamma_\sigma$ is induced to $G$ from $\sigma$ on $U$, and the latter is not trivial on any fundamental root subgroup of $U$, nor are any of the conjugates of $\sigma$ by elements of $T$.
       We mention that \cite[\S2.1--\S2.6]{C85} gives an excellent treatment of root subgroups and root  systems in the broad  context of algebraic groups with split BN-pairs, suitable for our use here.
       So, if $w\in W$ is not the longest word, neither is its inverse, and there are no nontrivial $B$-homomorphisms from $\ind_U^B\sO_\sigma$ to $\ind_{^wB\cap B}^B\sO$. (The point is that $^wB\cap B$ must contain a fundamental root group $X$, on which $\sigma$ and its conjugates through elements of $T$ are all nontrivial.)
      Mackey decomposition arguments now show that $\Hom_{\sO G}(\Gamma_\sigma,M_\emptyset)\cong\sO$, generated by a map $\beta_\emptyset:\Gamma_\sigma\longrightarrow M_\emptyset$ with image $\sO$-pure.
        By \cite[4.1]{G17}, $\dim_K\Hom_{KG}(K\Gamma_\sigma,\St_K)=1$. As is well known, $\St_K$ is a submodule of $KM_\emptyset$. Thus, the image of $\beta_\emptyset$ satisfies $K\im\beta_\emptyset= \St_K\subseteq KM_\emptyset$. 
          So the kernel $\ker\beta_\emptyset$ is just the intersection of $\Gamma_\sigma$ with the sum of all irreducible $KG$-submodules of $K\Gamma_\sigma$ not isomorphic to $\St_K$. (Recall that $K\Gamma_\sigma$ is multiplicity free by 
          \cite[Th. 8.1.3]{C85}. This reference requires that $\bG$ be connected reductive with connected center, a condition that holds for $\bG$ with adjoint type and also for $\widetilde{\bG}$ in the more general case discussed in Part 2. Corresponding results for Levi subgroups also hold in each case. The multiplicity freeness may then be deduced for $\widetilde G$ and its Levi subgroups by using \cite[Lem. 2.3(a),(c)]{DF94}.)
 %         elementary induction argument, viewing all the       Gelfand-Graev modules as appropriate induced modules; cf \cite[Lem. 2.3(a)]{DF94}.)} 
           A similar kernel analysis applies to the surjective map $\Gamma_\sigma\longrightarrow\scrS_\sigma$ discussed above \eqref{surjhom}. Thus, the latter map has
          the same kernel as $\beta_\emptyset$. Consequently,
          $$\im \beta_\emptyset\cong \scrS_\sigma.$$
The remaining discussion in Part 1 does not assume that $\bG$ is of adjoint type.
         
 The Steinberg $KG$-submodule   of $KM_\emptyset$ can also be written as 
$y_S KM_\emptyset$, where 
$$y_S=\sum_{w\in W}\sgn(w)t^{-2L(w)}{T}_w=\sum_{w\in W}\sgn(w)t^{-L(w)}\widetilde{T}_w
\in \sH_K.$$ 
This is just the module-theoretic realization
of the correspondence between irreducible $KG$-constituents of $KM_\emptyset$ and irreducible characters of its
$KG$-endomorphism algebra. The Steinberg character corresponds to the sign 
character; see  \cite[Th. 2.1, Lem. 2.2]{G17}, and Lemma \ref{goodone} above. Note that $T_sx=-x$,  for $x=y_S$, and,
 consequently, for any $x\in y_S KM_{\emptyset}$. For the $x=y_S$ case, see \cite[Lem. 1.1(b)]{DPS98c}.

There is thus a unique pure $\sO G$-submodule of $M_\emptyset$ which base changes over $K$ to $\St_K$,
namely, $M_\emptyset\cap y_SKM_\emptyset$. The latter module may also be written as $\sqrt{y_SM_\emptyset}$ in the (now common)
terminology of \cite[p.~251]{Dip90}. We have proved $\scrS_\sigma\cong\sqrt{y_SM_\emptyset}$. 

The argument for this isomorphism has not used the assumption that $\bG$ is simple of adjoint type.
There are similar characterizations of the modules $\scrS_\sigma^{L_J}$ introduced above Theorem \ref{kerfactors}. We may write
$$\scrS_\sigma^{L_J}=\sqrt{y_JM^{L_J}_\emptyset} ,$$
 where $y_J=\sum_{w\in W_J}(-1)^{\ell(w)}(t_s^2)^{-\ell(w)}T_w$, and $\sqrt{y_JM_\emptyset}\cong R^G_T(\scrS_\sigma^{L_J})$, the $\sO G$-submodule of $M_\emptyset$ generated by $\sqrt{y_JM_\emptyset^{L_J}}=\scrS_\sigma^{L_J}$.
          
Applying \cite[Ths.~1.30\&1.36]{Dip98} with $R=\sO$, and $M_{\sO}=M_\emptyset$ in Theorem 1.30 there, we obtain% Here the notation $\sqrt{y_JM_\emptyset}$ follows the notation in \cite[Th. 2.14]{Dip91}. By \cite[Th. 2.14]{Dip91},
\footnote{See also \cite[Th. 1]{Ca98} and \cite[Th. 1.10]{DG99}.} %\cite[Th. 4.17]{Dip90} each of which includes a proof.
  $$\End_{\sO G}\left(\bigoplus_J\sqrt{y_JM_\emptyset}\right)\cong
  \End_{\sH_\sO}\left(\bigoplus_J{y_J\sH_\sO}\right).$$
  However, there is a well-known identification 
  $$\Hom_{\sH_\sO}(y_J\sH_\sO,y_I\sH_\sO)\cong\Hom_{\sH_\sO}(x_J\sH_\sO, x_I\sH_\sO)$$
  for each $I,J\subseteq S$. See \cite[Lem. 1.1(c)]{DPS98c}.
    The left-hand side of this display sums to the right-hand side of the display above it. The lower display's own  righth-hand-side sums to the algebra $A_\sO$, using the base-change result \cite[Th. 2.4.4]{DPS98a}.  The left-hand side of the higher display has decomposition matrix a submatrix of that of $\sO G$, as noted in our discussion of $\End_{\sO G}(M)$ in the opening paragraph of this proof.  However, by Theorem \ref{bigDecom} (or Lemma \ref{decompMAT}), the decomposition matrix $D$ of $ A_{\sO'}$ agrees with that of $A_\sO$. 
  The theorem is now completely proved for $\bG$ simple of adjoint type.
 
  \medskip\noindent
  {\bf{Part 2. The general case}}
  
  Next, we consider the general case where $\bG$ is simple, but possibly not adjoint. (Much of the argument applies if $\bG$ is just semisimple, or even connected reductive.) Possibly, $\bG$ does not have connected center. In particular, Theorem \ref{kerfactors} is not directly available, but does apply to the group $\widetilde\bG$ as constructed in Lemma \ref{functor}. The theorem can at least be (re)stated for $G$ (that is, for $G=\bG^F$), using a fixed regular character $\sigma:U\to K^\times$ and drawing on the notation and discussion in Appendix A, which in turn draws on \cite[\S4]{G17}. We next give the main details and then use the version of Theorem \ref{kerfactors} in Appendix A to prove the modified version.
  
  In particular, an $\sO G$-lattice $\Gamma_\sigma$ is defined as in \cite[4.1]{G17} via an idempotent in $\sO U\subseteq\sO G$ constructed from $\sigma$. (The homomorphism $\sigma^{-1}$ applied to the idempotent is $1$.) Equivalently, $\Gamma_\sigma$ is the $\sO G$-module induced by the rank 1 $\sO U$-module determined by $\sigma^{-1}$ (a module we call $\sO_\sigma$ by abuse of notation).
  Briefly, $\Gamma_\sigma:=\ind_U^G\sO_\sigma$. This definition makes sense for $G=\bG^F$ (or for any finite group of Lie type).
  
  Next, keeping the same level of generality, let $P_J=U_J L_J$ be the standard parabolic subgroup of $G$ corresponding to $J\subseteq S$. A specific analog $\Gamma_\sigma^{L_J}$ of $\Gamma_\sigma$ is constructed ``by restriction" in Appendix A. This makes sense here, since $U\cap L_J$ is the analog for $L$ of $U$, and the restriction of $\sigma$ to $U\cap L_J$ is regular. We write
  $$ \Gamma_\sigma^{L_J}:=\ind^{L_J}_{U\cap L_J}\Res^U_{U\cap L_J}\sO_\sigma.$$
  
  Continuing, note from \cite[4.1]{G17} that the Steinberg module $\St_K$ appears in $K\Gamma_\sigma$ with multiplicity
  $1$.  It follows (see \cite[Prop.~4.2]{G17}, attributed to Hiss, and its proof) that there is a unique pure
  $\sO G$-sublattice $\Gamma'_\sigma$ such that $K (\Gamma_\sigma/\Gamma'_\sigma)\cong \St_K$. Both
  \cite{G17} and Appendix A take this property to define $\mathscr S_\sigma:=\Gamma_\sigma/\Gamma'_\sigma$, and we make the same definition here.  As in Appendix A, we also make the parallel definition
  $$\mathscr S_\sigma^{L_J}:=\Gamma_\sigma^{L_J}/(\Gamma^{L_J}_\sigma)'.$$
  This does not require the connected center hypothesis and makes sense in the general context $G=\bG^F$ with $\bG$ simple (or even just connected reductive).
  
  To summarize, we have now defined for $\bG$ all of the ingredients needed to define $\beta_K$ for the formal statement of Theorem \ref{kerfactors}. 
  
 {\it We claim that Theorem \ref{kerfactors} holds for $\bG$ with this notation.} We intend to prove this claim by using the fact that it holds for $\widetilde{\bG}$ with obvious changes in notation from $G,L_J, \beta, \Gamma_\sigma^{L_J}, \mathscr S^{L_J}_\sigma,\St_K^{L_J}$ to $\widetilde G, \widetilde L_J, \widetilde\beta, \Gamma_\sigma^{\widetilde L_J}, \mathscr S_\sigma^{\widetilde L_J}, \St^{\widetilde L_J}_K$, for any $J\subseteq S$.
  
 %reduction in the first paragraph of the proof of the 
 Theorem \ref{kerfactors} as
currently stated requires that we prove that  %shows that it is enough to pove that 
the kernel and image of 
\begin{equation}\label{betaK}
\beta_K:\bigoplus_{J\subseteq S}R^G_{L_J}(K\Gamma_\sigma^{L_J})\lra \bigoplus_{J\subseteq S}R^G_{L_J}(St_K^{L_J})
\end{equation} have no common nonzero constituents, a property that we know (and will prove) to be true in our context.
%for $\widetilde\beta_K$ from either the statement of Theorem A.1 and the first paragraph of it proof, or from the remaining paragraphs, which directly prove the reduction).
  
  For each $J\subseteq S$, let $\gamma^J:K\Gamma_\sigma^{L_J}\to K\mathscr S_\sigma^{L_J}\cong \St_K^{L_J}$, discussed 
  %(but not named) 
  in the proof of Theorem \ref{kerfactors}. We take $L_J$ to be the standard Levi subgroup of $G$ associated to $J$, and all the terms $\Gamma_\sigma^{L_J}, \mathscr S_\sigma^{L_\sigma}, \St_K^{L_J}$ as discussed above. Similarly, we have a map $\widetilde\gamma^J$, replacing $G$ with $\widetilde G$, and using
  $\Gamma_\sigma^{\widetilde L_J}$, $\mathscr S_\sigma^{\widetilde L_J}$, and $\St_K^{\widetilde L_J} $ in place of their counterparts for $G$.
  
  We pause here to note that
  $$\beta_K=\bigoplus_{J\subseteq S} R^G_{L_J}(\gamma^J)\quad{\text{\rm and}}\quad 
   \widetilde\beta_K=\bigoplus_{J\subseteq S}R^{\widetilde G}_{\widetilde L_J}(\widetilde\gamma^J).$$
   Recall that $\widetilde G$ is defined in the notation and setting of Lemma \ref{functor} as $\widetilde G=\widetilde\bG^{\widetilde 
       F}$,  where $\widetilde \bG=\bG  \widetilde\bZ$. Here, $\widetilde\bZ$ is an isomorphic copy of the  ``maximally split" torus $\bT$ of $\bG$ in the terminology of \cite[\S1.18]{C85} and $\widetilde\bZ= Z(\widetilde \bG)$. The torus $\bT\widetilde\bZ$ is maximally split in $\widetilde \bG$, and in $\widetilde G=G(\bT\widetilde\bZ)^{\widetilde F}$, as may be shown using the Lang-Steinberg theorem.  There is also an $F$-stable Borel subgroup $\bB$ of $\bG$ containing $\bT$ in the setting of Lemma \ref{functor}. A closed subgroup $\bN$ normalizing $\bT$ may then be chosen so that $\bB$, $\bN$ give a split BN-pair structure to $\bG$ as per the axioms (specialized for algebraic groups) of \cite[p.23]{C85}, with $\bB\cap \bN=\bT$. (The choice in \cite[p.22]{C85} of $\bN$ is $\bN=N_\bG(\bT)$, but this is not part of the axioms. We note that this choice of $\bN$ is $F$-stable, a useful property.) We now have three additional candidates for split BN-pairs: the pair $\bB \widetilde\bZ, \bN\widetilde\bZ$ with intersection $\bT\widetilde\bZ$, the pair $B=\bB^F, N=\bN^F$ with intersection $T=\bT^F$, and, finally, the pair $(\bB \widetilde\bZ)^{\widetilde F}, (\bN \widetilde\bZ)^{\widetilde F}$ with intersection $(\bT\widetilde\bZ)^{\widetilde F}$. The first of these candidate pairs satisfies Carter's axioms by an easy inspection, while the last two are covered by \cite[pp. 33-34]{C85}, which also discusses their respective Weyl groups. 
   
   We now have four ``algebraic groups with split BN-pairs" in the sense of \cite{C85}, namely $\bG$, $\widetilde \bG=\bG \widetilde\bZ$, $G=\bG^F$\!, and $\widetilde G=\widetilde\bG^{\widetilde F}=G(\bT\widetilde\bZ)^{\widetilde F}$. The first two groups are connected reductive, and the second two are the respective finite fixed-point groups under Frobenius endomorphisms. According  to \cite[p. 61 bottom]{C85} these groups all ``satisfy the commutator relations" of Chevalley in the form displayed there. This is needed for the full Levi decomposition theory developed in \cite[\S2.6]{C85}, quoted in \cite[\S4]{G17}. The standard Levi subgroup $L_J$ of $G$ is introduced above \cite[Prop. 2.6.3]{C85}. Implicit in its definitions is the split BN-pair context provided by the pair $\bold B^F, \bN^F$ discussed above. If that pair is changed to $(\bB \widetilde\bZ)^{\widetilde F}, (\bN\widetilde\bZ)^{\widetilde F}$, the analogous definition gives the Levi subgroup $\widetilde L_J$ of $\widetilde G= (\bG \widetilde\bZ)^{\widetilde F}$.
   
   In more detail, $L_J$ is defined as the group generated by the BN-pair intersection $B\cap N=T$ and certain ``root groups" $X_{\alpha}$, $\alpha\in \Phi_J$, determined by the interaction of the unipotent radical $U=(\bU)^F$ of (the $B$-part of) the pair with its Weyl group, $W=N/T$. See \cite[p. 43, \S2.2, Prop. 2.3.1, p. 57]{C85}. We take the opportunity to observe that $W$ also identifies with the Weyl group $(\bN\widetilde\bZ)^{\widetilde F}/(\bT\widetilde\bZ)^{\widetilde F}$ of the split BN-pair $(\bB \widetilde\bZ)^\wF,$  $(\bN\widetilde\bZ)^\wF$. Briefly, using the Lang-Steinberg theorem and the observation that $\bN\cap \widetilde\bZ\subseteq \bG\cap \widetilde\bZ\subseteq \bT$, we have
   $$\aligned N/T &\cong\bN^F/\bT^F\cong (\bN/\bT)^F=(\bN/\bT(\bN\cap \widetilde\bZ))^F=(\bN/\bN\cap \bT\widetilde\bZ)^F\\
   &\cong (\bN\bT\widetilde\bZ/\bT\widetilde\bZ)^\wF\cong (\bN\widetilde\bZ)^\wF/(\bT\widetilde\bZ)^\wF.\endaligned$$
   Order considerations now imply that the natural map
   $$N/T\longrightarrow (\bN\widetilde\bZ)^\wF/(\bT\widetilde\bZ)^\wF$$
   is an isomorphism. (The map is injective, since 
   $$N\cap (\bT\widetilde\bZ)^\wF= (\bN\cap \bT\widetilde\bZ)^\wF=(\bT(\bN\cap \widetilde\bZ))^\wF
   \subseteq (\bT(\bG\cap \widetilde\bZ))^\wF=\bT^\wF=\bT^F=T.)$$ 
   Surjectivity gives the useful identity $(\bN\widetilde\bZ)^\wF=N(\bT\widetilde\bZ)^\wF$.
   
   Next, we note that the following statements all hold.
   \begin{itemize}
   
   \item[(1)] $T$ normalizes $U$.
   
   \item[(2)] If $V$ is any subgroup of $G$ or $\widetilde G$ normalized by $T$, then $^nV:=nVn^{-1}$ and $V^n:=n^{-1}Vn$ are also normalized
   by $T$, for any $n\in N$.
   
   \item[(3)] If $V_1$ and $V_2$ are subgroups of $G$ or $\widetilde G$, both normalized by $T$, then   $V_1\cap V_2$ and $\langle V_1,V_2\rangle$, the subgroup generated by $V_1,V_2$, are also normalized by $T$.
   
   \end{itemize}
   
  \noindent The proofs are all elementary, and the same properties (1), (2), (3) hold if $T$ is replaced by $(\bT\widetilde\bZ)^\wF.$ (To see that (1) holds in this case, note that $U=\bU^F$, and $\bT\widetilde\bZ$ normalizes $\bold U$.)
  
  The above paragraph contains an implicit recipe for computing root groups, in the sense of \cite[p. 50]{C85}, for $G$ or $\widetilde G$, using \cite[\S2.5]{C85}. The split BN-pair context we use for $G$ is $\bB^F, \bN^F$ and $(\bB \widetilde\bZ)^\wF, (\bN\widetilde\bZ)^\wF$ for $\widetilde G$. First, we use the first two lines of the display \cite[p. 50]{C85} to construct a ``fundamental" root group $X_i$ from $U$ and a fundamental reflection $s_i$ in $W$, represented by $n_i\in N$. A representative $n_0$ for any longest word is also needed. We then write
  $$X_i=U\cap U^{n_0n_i},$$
  which is the same group whether working with $G$ or $\widetilde G$. All the other root groups for $G$ have the form $nX_in^{-1}$ for some fundamental reflection $s_i$ and some element $n\in N=\bN^F$ \cite[Prop. 2.5.15]{C85}. A similar construction applies for $\widetilde G$, using $(\bN\widetilde\bZ)^{\wF}$ in place of $\bN^F$. However, it leads to the same root groups, since each $X_i$ is normalized by $(\bT\widetilde\bZ)^\wF$ and $(\bN\widetilde\bZ)^\wF=N(\bT\widetilde\bZ)^\wF$. 
  
  It will now be convenient to write $nX_in^{-1}$ above as $wX_iw^{-1}(={^wX_i}=X_i^{w^{-1}}$), where $w\in W$ is represented by $n\in N$.  More generally, if $X,Y$ are root groups and $Y=nXn^{-1}$, with $n$ representing $w$ ($n\in N$, $w\in W$), we write $Y=wXw^{-1}$ ($={^wX}=X^{w^{-1}}$). This gives a permutation action of $W$ on the set of root subgroups. This action agrees with the permutation action on the set of roots $\Phi$ in a ``root system"  constructed from $W$ as a finite Coxeter group in \cite[\S2.2]{C85}. See \cite[Prop. 2.5.15]{C85} and its proof, which gives the equivariant bijection. The notation $\Phi_J$ is introduced in \cite[Prop. 2.3.1]{C85}, using as $J$ a subset of the indices of the fundamental reflections $S$ of $W$. Our $J$ denotes the subset of these reflections themselves. With this modification, root groups $X_\alpha$ ``with $\alpha\in \Phi_J$" are those of the form $w(X_i)$ with $s_i\in J$ and
  $w\in W_J:=\langle J\rangle$. (The key to proving this is observing that the correspondence given in the proof of \cite[Prop. 2.5.15]{C85} is $W$-equivariant. It can then be applied to the definition of $\Phi_J=W_J(\Delta_J)$ in \cite[Prop. 2.3.1]{C85}, noting that $\Delta_J$ corresponds to the fundamental reflections in $J$.)
  
  We can now return to our discussion of $L_J$ and $\wL_J$ earlier in the proof. Let 
  $$X_J =\langle w(X_i)\,|\, s_i\in J\,\, \&\,\, w\in W_J\rangle.$$
  We now know 
  $L_J=\langle T, X_J\rangle$, from the definition in \cite[\S2.6]{C85} in terms of $\Phi_J$ and the reformulation above.  This is in the context of $G=\bG^F$ and its split BN-pair $\bB^F, \bN^F$. If instead we work with $\widetilde G=
  (\bG \widetilde\bZ)^\wF$, we obtain $\wL_J$ in place of $L_J$. Since the root groups for $G$ and $\widetilde G$ are the same, with the same Weyl group actions on them, we have 
  
  $$\wL_J=\langle(\bT\widetilde\bZ)^\wF, X_J\rangle=(\bT\widetilde\bZ)^\wF X_J=(\bT\widetilde\bZ)^{\wF}\bT^FX_J=(\bT\widetilde\bZ)^\wF L_J,$$
   taking into account the normalization of all the root  groups by both $T$ and $(\bT\widetilde\bZ)^\wF$.  We may also write
  $$
  \wL_J=L_J(\bT\widetilde\bZ)^\wF.$$
 Notice that the unipotent radical of $\widetilde B_J$ and  the unipotent radical of $B_J$ are the same. Hence, $\Gamma_{\sigma}^{\wL_J}\cong\text{Ind}_{L_J}^{\wL_J}(\Gamma_\sigma^{L_J})$ by definition.
  This is a good place to mention that $L_J$ (resp., $\widetilde L_J$) inherits a split BN-pair structure from $G$ (resp., $\widetilde G)$. See \cite[Prop. 2.6.3]{C85}. We denote the pair as $B_J, N_J$ (resp., $\widetilde B_J, \widetilde N_J$). We remark that $B_J\subseteq B$ and $N_J\subseteq N$ (resp., $\widetilde B_J\subseteq\widetilde  B$ and $\widetilde N_J\subseteq\widetilde N= \widetilde \bN^{\widetilde F}$, where $\widetilde  \bN=\bN\widetilde\bZ$).
  Also, we find, again using \cite[Prop. 2.6.3]{C85} with our discussions above (of root subgroups and of $W$), that 
  $$B_J\cap N_J=T,\quad \widetilde B_J\cap\widetilde N_J=(\bT\widetilde\bZ)^{\widetilde F},\quad N_J/T=W_J\subseteq W=N/T.$$
  The group $W_J$ in this display is the group generated by the reflections in $J$, in our notation, and by the reflections indexed by $J$ in the notation of \cite[Prop. 2.6(4)]{C85}. With some abuse of notation, we also have
 $$ \widetilde N_J/(\bT\widetilde\bZ)^{\widetilde F}=W_J\subseteq W=(\bN\widetilde\bZ)^{\widetilde F}/(\bT\widetilde\bZ)^{\widetilde F}.$$
 Technically, we should use groups $\widetilde W_J$, $\widetilde W$, naturally isomorphic to $W_J, W$, respectively, in the display above. Here, $\widetilde W_J=\overline \eta(W_J)$ and $\widetilde W=\overline\eta(W)$, where 
 $$\overline\eta: N/T\longrightarrow (\bN\widetilde\bZ)^{\widetilde F}/(\bT\widetilde\bZ)^{\widetilde F}$$
 is the natural isomorphism discussed above in this proof. It is induced by the natural inclusion map $\eta:N\longrightarrow (\bN\widetilde\bZ)^{\widetilde F}$ that sends any element of $N$ to itself.
 
 For use in the next paragraph, observe that if $w\in W=N/T$ is represented by $n_w\in N$, then $\overline \eta(w)\in\widetilde W=(\bN\widetilde\bZ)^{\wF}/(\bT\widetilde\bZ)^\wF$ is represented by the same element $n_w$, viewed as a member of $(\bN\widetilde\bZ)^{\widetilde F}$.  If $w\in W_J=N_J/T$, then $n_w\in N_J$, and $\overline\eta(w)\in \overline\eta(W_J)=\widetilde W_J$ is, as before, represented by $n_w$.  Since the isomorphism $\overline \eta$ takes $W_J$ onto $\widetilde W_J$, this process can be reversed, showing that any element ${\widetilde w}\in{\widetilde W}_J$ is representable by an element $n_w\in N_J$, where $\overline\eta(w)= \widetilde w$.  We remark that the element $n_w$ belongs to $\widetilde N_J$, since the latter group is 
 the preimage of $\widetilde W_J$ under the natural map $\widetilde N=(\bN\widetilde\bZ)^\wF\longrightarrow \widetilde W.$

 We next note, using the previous paragraph, that the Steinberg module $\St^{\wL_J}_K$ has $\St^{L_J}_K$ as its restriction to $L_J$, up to isomorphism: To see this, consider the description of $\St^{L_J}_K$ (and of $\St^{\wL_J}_K$ in \cite[Th. 2.1]{G17}, attributed to Steinberg): % According to this theorem, we have 
 $$\St^{L_J}_K=KL_J\mathfrak e\subseteq KL_J{{\underline{\mathfrak b}}},$$
where ${\mathfrak e}:=\sum_{w\in W_J}(-1)^{\ell(w)}n_w\underline{\mathfrak b}$
 and $\underline{\mathfrak b}$ is the sum of the elements in the ``B"-part of a given split BN-pair structure for $L_J$. 
 We use the pair $B_J, N_J$ 
 above for $L_J$, and $\wB_J, \wN_J$ for $\wL_J$
  and an analogous description of its Steinberg module $\St^{\wL_J}_K$ over $K$.  The elements $n_w$ are required to be in the ``N" part $N_J$ of the pair for $L_J$, with $n_w$ a preimage of $w\in W_J$ under the natural map 
 $N_J\twoheadrightarrow W_J$. Since $N_J$ is the full preimage of $W_J$ under the natural map $N\to W$, it is enough to know that $n_w\in N$ and its image $w$ lies in $W_J$. Passing to the formula for $\St_K^{\wL_J}$, we find that the expression $
  \sum_{w\in W_J}(-1)^{\ell(w)}n_w$ in the formula for $\St^{L_J}_K$ can be used in the same place for $\St^{\wL_J}_K$ in the formula for the analog for $\widetilde{\mathfrak e}$ of $\mathfrak e$. Let $\mathfrak c$ be a set of representatives for all the right cosets of ${\mathfrak b}:=B_J$ in $\widetilde{ \mathfrak b}:=\wB_J$, and let $\underline{\mathfrak c}$ denote the sum of all the elements in $\mathfrak c$ viewed as members of $K\wL_J$. Regard also the sums $\underline{\mathfrak b}$ and $\underline{\widetilde{\mathfrak b}}$ as taking place in $K\wL_J$. Then
 $$\underline{\mathfrak b}\,\cdot\underline{\mathfrak c}=\underline{\widetilde{\mathfrak b}}, \quad
 {\mathfrak e}\cdot\underline{
 \mathfrak c}=\widetilde{\mathfrak {e}}.$$
 Right multiplication by $\underline{\mathfrak c}$ thus provides, inside $K\wL_J$, a nonzero map of $KL_J$-modules 
 $\St^{L_J}_K\longrightarrow \St^{\wL_J}_K|_{L_J}$. The map is injective, since $\St^{L_J}_K$ is irreducible \cite[Th. 2.1(ii)]{G17}.  It is an isomorphism, since both domain and target have the same dimension $|B_J|_p=|\wB_J|_p$ by \cite[Th. 2.1(i)]{G17}.
(The two given $p$-parts represent the equal cardinalities of Sylow $p$-subgroups of the underlying groups $B$ and $\widetilde B$, and the quoted result implies these cardinalities give the dimensions of the relevant Steinberg modules.)

We now return to the task of proving that the kernel $\ker\beta_K$ and the image $\im\beta_K$ of $\beta_K$ have no nonzero constituents in common. (Recall that this will prove that Theorem \ref{kerfactors} holds for our $\bG$ here, i.e., one that may not satisfy the connected center hypothesis). Equivalently, we must show that $\ker R^G_{L_J}(\gamma^J)$ and $\im R^G_{L_{J^\prime}}(\gamma^{J'})$ have no common nonzero constituents, for each $J,J'$ contained in $S$. Recall that $\gamma^J$ is a surjective map
$K\Gamma_\sigma^{L_J}\twoheadrightarrow \St_K^{L_J}$ with analogous notation for $\widetilde\gamma^J$ (and $\gamma^{J'},\widetilde \gamma^{J'}$).
As such, it is unique up to nonzero scalar multiples. A similar uniqueness holds for its analogous maps.

We next {\bf claim} that $R^\wG_{L_J}(\ker\gamma^J)$ is isomorphic to a submodule of $\ker R^{\wG}_{\wL_J}(\widetilde \gamma^J)$. To see this, first induce the $KL_J$ exact sequences
$$0\longrightarrow \ker\gamma^J\longrightarrow K\Gamma^{L_J}_\sigma\overset{\gamma^J}\longrightarrow\St^{L_J}_K\to 0$$
to $K\wL_J$, obtaining
$$0\to \ind_{L_J}^{\wL_J}(\ker\gamma^J)\longrightarrow K\Gamma_\sigma^{\wL_J}\longrightarrow\St^{\wL_J}_K\otimes K(\wL_J/L_J)\to 0.$$
Here, we  view $\Gamma_\sigma^{L_J}$ as the induction of $\sigma$ to $L_J$ from the ``unipotent radical" of $B_J$, and $\Gamma_\sigma^{\wL_J}$ as its further induction to $\wL_J$. (We first learned this viewpoint from \cite[Lem. 2.3(a)]{DF94}.) The tensor product expression on the right is obtained from the well-known tensor identity, using the fact that $\St^{\wL_J}_K|_{L_J} \cong \St^{L_J}_K$ (proved above). Since the $K\wL_J$-module $K(\wL_J/L_J)$ has the trivial module $K$ as a homomorphic image, we have an epimorphism $\St^{\wL_J}_K\otimes K(\wL_J/L_J)\twoheadrightarrow \St^{\wL_J}_K$. After possibly multipling by a nonzero scalar, we can compose this map with the map $K\Gamma_\sigma^{\wL_J}\longrightarrow\St^{\wL_J}_K\otimes K(\wL_J/L_J)$ above to obtain $\widetilde\gamma^J$. Consequently, there is a natural inclusion 
$$\ind_{L_J}^{\wL_J}(\ker\gamma^J)\subseteq\ker\widetilde \gamma^J.$$ Next, apply $R^\wG_{\wL_J}$ to both sides (an exact functor).  The right-hand side becomes $R^\wG_{\wL_J}(\ker\, \widetilde\gamma^J)\cong\ker R^\wG_{\wL_J}(\widetilde \gamma^J)$. To understand the left-hand side, note first that the standard parabolic subgroups $P_J$ and $\wP_J$ associated to $J$ (and containing $L_J$ and $\wL_J$, respectively) have the same unipotent radical, called $U_J$ in \cite[\S2.6]{C85}.  As a consequence, inflation commutes with induction, in the sense that $\inff_{\wL_J}^{\wP_J}\circ\ind_{L_J}^{\wL_J}\cong \ind_{P_J}^{\wP_J}\circ\inff_{L_J}^{P_J}$ as functors on $KL_J$-modules. Apply this isomorphism of functors to the object $\ker\gamma^J$, and then apply $\ind_{\wP_J}^{\wG}$ to obtain
$$R_{\wL_J}^\wG(\ind_{L_J}^{\wL_J}(\ker\gamma^J))\cong R^{\wG}_{L_J}(\ker\gamma^J).$$
The {\bf claim} which began this paragraph follows.

As a corollary to the claim, using the validity of Theorem \ref{kerfactors} for $\widetilde{\bG}$ (and its consequences for $\widetilde\beta_K$ and 
the various $R_{\wL_J}^{\wG}(\widetilde\gamma^J)$, we now know
 that each $\ker R^{\wG}_{L_J}(\gamma^J)=R^{\wG}_{L_J}(\ker \gamma^J)$ has no nonzero constituents in common with each 
$$\im R^\wG_{\wL_{J'}}(\widetilde\gamma^{J'}) \cong R^{\wG}_{\wL_{J'}}(\St_K^{\wL_{J'}}).$$
Applying Frobenius reciprocity (for ordinary induction) and writing
$R^{\wG}_{\wL_J'}\cong\ind_{P_J}^{\wG}\circ{\rm Inf}_{L_J}^{P_J}$), where Inf is inflation,
 we find that $R^G_{L_J}(\ker\gamma^J) $ has no    nonzero constituents in common with the restriction

$$\begin{aligned} 
R^{\wG}_{\wL_{J'}}(\St_K^{\wL_{J'}})|_G
&\cong\Big( \ind^{\wG}_{\wP_{J'}}({\rm Inf}_{\wL_{J'}}^{\wP_{J'}}\St^{\wL_{J'}}_K)\Big)\Big|_G \\
                                       &\cong \ind^G_{\wP_{J'}\cap G}({\rm Inf}^{\wP_{J'}}_{\wL_{J'}}\St^{\wL_{J'}}_K)|_{{\wP_{J'}}\cap G}\\
                                         &\cong\ind^G_{P_{J'}}({\rm Inf}^{P_{J'}}_{L_{J'}}(\St^{\wL_{J'}}_K|_{L_{J'}}))\\
                                       &\cong\ind^G_{P_{J'}}({\rm Inf}^{P_{J'}}_{L_{J'}}\St^{L_{J'}}_K)
                                \cong R^G_{L_{J'}}(\St^{L_{J'}}_K).                                         
     \end{aligned}
  $$
     
    \medskip
     In this display, the first line holds by definition. The second line follows from the Mackey decomposition theory, using $\wG=\wP_{J'}G$. The third line uses $P_{J'}\subseteq\wP_{J'}\cap G\subseteq N_G(P_{J'})\subseteq P_{J'}$; see \cite[Prop.~2.1.6]{C85}. In particular, $P_{J'}=\widetilde P_{J'}\cap G$. The agreement of $P_{J'}$ actions inside the outer parentheses, in lines two and three, can be checked on $L_{J'}$ and $U_{J'}$. The resulting $L_{J'}$-module in both lines is $\text{St}^{\widetilde L_K^{J'}}|_{L_{J'}}$, and the $U_{J'}$-actions are both trivial. Passage from line three to four is achieved directly from the Steinberg module restriction result earlier in this proof. This concludes the proof of the 
     claimed variation on Theorem \ref{kerfactors}.
     
     The remaining details in the proof of the  general case of Theorem \ref{12.3} ($\bG$ is simple) parallel those given for $\bG$ of adjoint type and are left to the reader.
 \end{proof}
 
 \begin{rems}\label{2ndproof}
 (1) If we use a result of Dipper and Fleischmann, the proof for Theorem \ref{kerfactors} in Appendix A can be generalized to the general case where the center of $\bG$ is not necessarily connected. To show that the kernel and image of $\beta_K$ in \eqref{betaK} have no irreducible constituents in common, it suffices to show that
 $\dim \Hom_{KG}(P,M)=\dim\Hom_{KG}(M,M),$
 where $P=\bigoplus_{J\subseteq S}R^G_{L_J}(K\Gamma_\sigma^{L_J})$ and $M= \bigoplus_{J\subseteq S}R^G_{L_J}(St_K^{L_J})$. This can be easily reduced to prove that
 \begin{equation}\label{I-J}
 \dim\, \Hom_{KG}\big(R^G_{L_I}(K\Gamma_\sigma^{L_I}),R^G_{L_J}(St_K^{L_J})\big)
 =\dim\,\Hom_{KG}\big(R^G_{L_I}(St_K^{L_I}),R^G_{L_J}(St_K^{L_J})\big),
 \end{equation}
for all $I,J\subseteq S$. An application of Frobenius reciprocity and the Mackey decomposition theorem for Harish-Chandra inductions and restrictions as in the proof of Theorem \ref{kerfactors} requires the fact that ${}^*R^{^nL_I}_{^nL_I\cap L_J}({}^nK\Gamma_\sigma^{L_I})$ is isomorphic to a Gelfand--Graev module $K\Gamma_{\sigma'}^{^nL_I\cap L_J}$ for some regular character $\sigma'$, where $n=n_w\in N$ with $w\in\sD_{J,I}$, the set of $(W_J,W_I)$-double coset representatives in $W$. This result is proved in \cite[Lem. (3.6)]{DF92} via \cite[Lem. (1.8)]{DF92}. (Note that if the center of $\bG$ is connected, then we may replace $K\Gamma_{\sigma'}^{^nL_I\cap L_J}$ by $ K\Gamma_{\sigma}^{^nL_I\cap L_J}$, the case considered in Appendix A.) Thus, by applying Frobenius reciprocity again together with the fact that $^*R^G_{L_J}(St_K^G)\cong St_K^{L_J}$ (see \cite[(71.6)]{CR87}) and $[\Gamma_\sigma:St_K]=1$ (see \cite[4.1]{G17}),  we obtain $\text{LHS}=\#\sD_{J,I}$ in \eqref{I-J}. A similar argument shows that $\text{RHS}=\#\sD_{J,I}$ as well. This proves \eqref{I-J} and, hence, the generalized version of Theorem \ref{kerfactors} .

 (2) We conclude this section by again calling attention to Theorem \ref{12.3}(b), which shows that the decomposition matrix $D$ relative to a given cross-characteristic $r$ can be evaluated by working with the smallest groups (as measured by a power $q$ of $p$)
in any Lie-type series as parametrized by Carter in \cite[\S1.1.9]{C85}. %.
   \end{rems}

  \section{$i$-Quantum groups}  
The $q$-Schur algebras \cite{DJ89} are certain Hecke endo-algebras of type $A$ which play a central role in linking representations of quantum $\mathfrak{gl}_n$ and Hecke algebras of the symmetric groups as well as finite general linear groups (see, e.g., \cite{DD97}).
 In this section, we will look at how certain Hecke endo-algebras of type $B/C$ can be introduced to play a similar role in connecting representations of finite orthogonal/symplectic groups with those of the $i$-quantum groups introduced in \cite{BW18}.
For this purpose, we will modify the definition of the algebra $A$ to introduce $q$-Schur algebras of type $B/C$ with bidegree $(n,r)$, where $n$ is the rank of the $i$-quantum group and $r$ is the rank of the Hecke algebra.

\medskip 
%\noindent
{\bf 5.4A \bf The $q$-Schur algebras of type $B/C$.} Let $\mathbb F_q$ be the finite field of $q$ elements and $r$ a positive integer. (In this section, $r$ is not required to be the cross-characteristic associated with $G$.) Consider the following two types of the orthogonal/symplectic groups introduced in Example \ref{SO_n}:
%Let $SO(2r+1,\mathbb F_q)$ (resp., $Sp(2r,\mathbb F_q)$) be the special orthogonal\footnote{We exclude the special orthogonal groups of even degree here.}
% This is the type $D$ case; see \cite{FL}.  (resp., sympletic) group of degree $2r+1$ (resp., $2r$) over $\mathbb F_q$ defined by
%$$\aligned&SO(2r+1,\mathbb F_q)=\{x\in GL(2r+1,\mathbb F_q)\mid x^tJx=J\},\\(\text{resp., }&Sp(2r,\mathbb F_q)=\{x\in GL(2r,\mathbb F_q)\mid x^tJ'x=J'\}),\\\endaligned$$where 
%\begin{equation}\label{JJ'}J=J_{2r+1}=\begin{pmatrix}0&0&\cdots&0&1\\0&0&\cdots&1&0\\&\cdots&\cdots&\cdots&\\0&1&\cdots&0&0\\
%1&0&\cdots&0&0\\\end{pmatrix},\quad J'=J'_{2r}=\begin{pmatrix}0&J_r\\-J_r&0\\\end{pmatrix}
%\end{equation}
\begin{equation}\label{OSp}
G(r_*,q)=\begin{cases}\O_{2r+1}(\mathbb F_q),&\text{ if }*=\scB\\
\Sp_{2r}(\mathbb F_q),&\text{ if }*=\scC,\end{cases}\;\;\text{ where }\;\;r_*=\begin{cases}2r+1,&\text{ if }*=\scB\\
2r.&\text{ if }*=\scC.\end{cases}
\end{equation}
Then, $G(r_*,q)$ is the fixed-point subgroup of $\GL_{r_*}(\mathbb F_q)$ under the graph automorphism in \eqref{vartheta}:
$$\vartheta_{r_*}:\GL_{r_*}(\mathbb F_q)\longrightarrow \GL_{r_*}(\mathbb F_q),\;\;x\longmapsto J_{*}^{-1}(x^t)^{-1}J_{*},$$
where $J_{\scB}=J_{2r+1}$ and $J_{\scC}=J^-_{2r}$ as in \eqref{JJ'}. %Here the $r_*$ in $\GL_{r_*}(\mathbb F_q)$ or in $\fS_{r_*}$ and below should be understood as follows:
%$$r_*=\begin{cases}2r+1,&\text{ if }*=\scB;\\2r.&\text{ if }*=\scC,\end{cases}$$

The graph automorphism $\vartheta_{r_*}$ induces an automorphism $\sigma=\sigma_{r_*}$ on the Weyl group $\widetilde W^*$ of $\GL_{r_*}(\mathbb F_q)$, which is identified as the symmetric group $\fS_{r_*}$.
More precisely, we have an involution:
\begin{equation}\label{sigma}
\sigma:\fS_{r_*}\to\fS_{r_*}, (i,j)\mapsto(r_*+1-i,r_*+1-j),\text{ for all }i,j\in[1,r_*],
\end{equation}
where $[1,r_*].=\{1,2,\ldots,r_*\}$, such that the Weyl group $W^*$ of $G(r_*,q)$ is isomorphic to the fixed-point group $(\fS_{r_*})^\sigma$. 

If $S=\{s_1,s_2,\ldots,s_r\}$ is the set of the Coxeter generators for $W^*$, then we have the following identifications:

For $*=\jmath$, $W^\jmath=(\fS_{2r+1})^\sigma$ and
$$s_i=(i,i+1)(2r+2-i,2r+1-i)\;\;(1\leq i<r),\quad s_r=(r,r+1)(r+1,r+2)(r,r+1),$$
while, for $*=\imath$, $W^\imath=(\fS_{2r})^\sigma$ and
$$s_i=(i,i+1)(2r+1-i,2r-i)\;\;(1\leq i<r),\quad s_r=(r,r+1).$$
See Example \ref{SO_n}. Hence, $W^\jmath\cong W^\imath$ and we set $W_r=W^*$.

%For the parabolic subgroup $W_\la$, $\la\in\Lambda(n+1,r)$, if
%$$\widetilde \la=(\la_1,\ldots,\la_n,2\la_{n+1}+1,\la_n,\ldots,\la_1)\in\Lambda(2n+1,2r+1).$$ as in \eqref{wla},
%then $\sigma$ stabilises the Young subgroup $\fS_{\widetilde\la}$ and $W_\la=\fS_{\widetilde\la}^\sigma$ is the fixed-point subgroup.

%if we use the index set $[-r, r]:=\{-r,-r+1,\ldots -1,0,1,\ldots,r\}$ for $r_*=r_\scB$, and with 0 omitted for $r_*=r_\scC$, then the Coxeter generators for $\fS_{r_{\scB}}=\fS_{2r+1}$ consist of basic transpositions
%$$(-r,-r+1),\ldots, (-2,-1), (-1,0), (0,1), (1,2),\ldots,(r-1,r),$$
%and the same list $(-1,0), (0,1)$ replaced by $(-1,1)$ for $\fS_{r_\scC}=\fS_{2r}$. However, by embedding $\fS_{2r}$ into $\fS_{2r+1}$ such that the image consists of all permutations sending 0 to 0, the Weyl group $W$ of $G(r_*,q)$ in both cases is the fixed-point subgroup $W=(\fS_{2r+1})^\sigma$ of the induced graph automorphism
%\begin{equation}\label{sigma}
%\sigma:\fS_{2r+1}\to\fS_{2r+1},\;\;(i,j)\mapsto (-i,-j)\text{ for all }i,j\in[-r,r].
%\end{equation}
Note that $W_r$ is the subgroup of $\fS_{r_*}$ consisting of permutations 
\begin{equation}\label{Sr*}
\begin{pmatrix}1&2&\cdots &r&r+1&\cdots&r_*\\
i_1&i_2&\cdots&i_r&i_{r+1}&\cdots&i_{r_*}\end{pmatrix}
\end{equation}
satisfying $i_j+i_{r_*+1-j}=r_*+1$.

%Hence, $W=W(B_r)$ is isomorphic to the Weyl group of type $B_r$.
%In particular,  if
%$$s_0=(-1,0)(0,1)(-1,0)=(-1,1)=\begin{pmatrix}-r&\cdots&-2&-1&0&1&2&\cdots&r\\
%-r&\cdots&-2&1&0&-1&2&\cdots&r\end{pmatrix},$$
%and  $s_i=(i,i+1)(-i,-i-1)$, for $i=1,2,\ldots,r-1$, then 
%$$S:=\{s_0,s_1,\ldots,s_{r-1}\}$$
%forms the set of Coxeter generators for $W(B_r)$.

%$*\in\{\textsf{so,sp}\}$, let $\epsilon_*=\begin{cases}1,&\text{ if }*=\so;\\0,&\text{ if }*=\scC,\end{cases}$.
Let $\widetilde B(r_*,\mathbb F_q)$ be the subgroup of $\GL_{r_*}(\mathbb F_q)$ consisting of upper triangular matrices, where $r_*\in\{2r,2r+1\}$ and, for $*\in\{\jmath,\imath\}$, let
$$B(r_*,q)=G(r_*,q)\cap \widetilde B(r_*,\mathbb F_q).$$
Then, $B(r_*,q)$ is a Borel subgroup of $G(r_*,q)$ and, by Lemma \ref{H^o},% \cite[Ths 3.2\&5.1]{Iwa},
$$\End_{\mathbb ZG(r_*,q)}\Big(\text{Ind}_{B(r_*,q)}^{G(r_*,q)}1_{\mathbb Z}\Big)\cong\sH^o(W_r)_{\mathbb Z},$$
where $1_{\mathbb Z}$ is the trivial representation of $B(r_*,q)$, $\sH^o(W_r)_{\mathbb Z}$ is the specialization via $t^2\to q$ of the generic Hecke algebra of $W_r$ with equal parameters.
% and $L=\ell$,and $\sH(B_r)_{\mathbb Z}$ is obtained by specialising $t^2$ to $q$. %(We automatically assume $2\nmid q$ if $R=2r$.)

Let $\Lambda(m,r)=(\mathbb N^{m})_r$ be the set of sequences $(a_1,a_2\ldots,a_m)\in\mathbb N^m$ with entries summing to $r$, and let
\begin{equation}\label{Lanr}
%\aligned
\Lambda_\scB(n,r)=\Lambda(n+1,r),\quad
\Lambda_\scC(n,r)=\Lambda(n,r).
%\endaligned
\end{equation}
Note that we often regard $\Lambda_\scC(n,r)$ as a subset of $\Lambda_\scB(n,r)$ via the embedding $\la\mapsto(\la,0)$.
%$\Lambda_\scB(n,r)$ and $\Lambda_\scC(n,r)$ are identical as sets, but differ in labelling. 
Define
\begin{equation}\label{wtmap}
\aligned
\widetilde{\ }&:\Lambda_\scB(n,r)\longrightarrow\Lambda(2n+1,2r+1),\;\;\la\longmapsto\widetilde\la= (\la_{1},\ldots,\la_n,2\la_{n+1}+1,\la_n,\ldots,\la_{1})\\
\widetilde{\ }&:\Lambda_\scC(n,r)\longrightarrow\Lambda(2n,2r),\quad\la\longmapsto\widetilde\la=(\la_{1},\ldots,\la_n,\la_n,\ldots,\la_{1}).
\endaligned
\end{equation}
%If $\wLa_\scB(n,r)$ denotes the image of $\La_\scB(n,r)$, then
%$$\wLa_\scC(n,r)=\{(\la_{n},\ldots,\la_1,1,\la_1,\ldots,\la_{n})\mid(\la_1,\ldots,\la_{n})\in\La(n,r)\}.$$
For $\la\in \Lambda_*(n,r)$ with $*\in\{\scB,\scC\}$, let $\widetilde P_{\la}(r_*,q)$ be the standard parabolic subgroup of $\GL_{r_*}(\mathbb F_q)$ associated with $\widetilde\la$, consisting of upper triangular matrices with blocks of sizes $\widetilde\la_i$ on the diagonal.  Let
$$P_\la(r_*,q)=\widetilde P_{\la}(r_*,q)\cap G(r_*,q),$$
and define, for any commutative ring $R$ with $q=q1_R\in R$,
$$\mathcal E^*_R(n,r)=\End_{RG(r_*,q)}\bigg(\bigoplus_{\la\in\Lambda_*(n,r)}\text{Ind}_{P_\la(r_*,q)}^{G(r_*,q)}1_R\bigg),$$
where $1_R$ denotes the trivial representation of $P_\la(r_*,q)$. 

%This algebra has the following interpretation of the following Hecke endomorphism algebra.
%\underline{\bf The $q$-Schur algebras of type $B/C$.} 
The Weyl group of the Levi subgroup $L_\la(r_\scB,q)$ of $P_\la(r_{\scB},q)$ has the form
$$W_\la=W^\scB_\la=\langle S\backslash\{s_{\la_1+\cdots+\la_i}\mid 1\leq i\leq n\}\rangle,\quad(\text{Here }\la\in\La(n+1,r).)$$
%\langle s_0,s_1,\ldots,s_{\la_0-1}, s_{\la_0+1},\ldots,s_{\tilde\la_1-1}, \ldots, s_{\tilde\la_{n-1}+1},\ldots, s_{r-1}\rangle,$$
while the Weyl group of the Levi subgroup $L_\la(r_\scC,q)$ of $P_\la(r_{\scC},q)$ has the form
$$W_\la=W^\scC_\la=\langle S\backslash\{s_{\la_1+\cdots+\la_i}\mid 1\leq i\leq n\}\rangle.\quad(\text{Here }\la\in\La(n,r).)$$
%\langle s_{1},\ldots,s_{\la_1-1},s_{\la_1+1},\ldots,s_{\tilde\la_2-1},\ldots, s_{\tilde\la_{n-}+1},\ldots, s_{r-1}\rangle.$$
They are both the fixed-point subgroups of the corresponding standard Young subgroups $\fS_{\widetilde \la}$ under the graph automorphism $\sigma$ in \eqref{sigma}: $W_\la=(\fS_{\widetilde \la})^\sigma$.

 Let $\sH(W_r)$ be the Hecke algebra over $\sZ$ associated with $W_r$. It is generated by $T_i=T_{s_i}$, for $1\leq i\leq r$, with basis $\{T_w\}_{w\in W}$. The subalgebra generated by $T_1,\ldots,T_{r-1}$ is the Hecke algebra $\sH(\fS_r)$ associated with the symmetric group $\fS_r$.

 For %a DVR triple $(\mathcal O, K, k)$, 
$*\in\{\scB,\scC\}$ and $\la\in\Lambda_*(n,r)$, let $x_\la=\sum_{w\in W^*_\la}T_w$ (perhaps, $x^*_\la$ is more accurate than $x_\la$), $T^*_{n,r}:=\bigoplus_{\la\in\Lambda_*(n,r)}x_\la\sH(W_r)$, and define a Hecke endomorphism $\sZ$-algebra
of bidegree $(n,r)$:
\begin{equation}\label{Snr}
\sS^*(n,r)=\sS^*_{t,\sZ}(n,r)=\End_{\sH(W_r)}(T^*_{n,r}).
\end{equation}
The algebra $\sS^\scB(n,r)$ (resp., $\sS^\scC(n,r)$) is called the {\it $q$-Schur algebra of type $B$} (resp, {\it type }$C$). See the lemma below for a justification of the terminology.\index{$q$-Schur algebra! $\sim$ of type $B$}\index{$q$-Schur algebra! $\sim$ of type $C$}

\begin{rem}\label{Adag}
(1) Note that if we replace $T$ in displays \eqref{T^+} and \eqref{endo} by $T^\dag=\bigoplus_{J\subseteq S}(x_J\sH)^{m_J}$ with all $m_J>0$, and define $A^\dag:=\End_\sH(T^\dag)$, then, for $n\geq r$ and a suitable choice of $m_J$, $\sS^\scB(n,r)=A^\dag$.  However, all $\sS^\scC(n,r)$ or $\sS^\scB(n,r)$ for $n<r$ cannot be of  the form $A^\dag$, since $T^\imath_{n,r}$ does not contain $x_J\sH$ (as a direct summand) with $s_r\in J$ and $T^\jmath_{n,r}$ does not contain $\sH$ if $n<r$. But each of them has the form $eA^\dag e$ for some $A^\dag$ and some idempotent $e\in A^\dag$, i.e., is a centralizer subalgebra of $A^\dag$.
%for some poset ideal $\Theta$ of $(\Omega,\leq_L)$ (see \eqref{Tdagger}).

(2) The algebras $\sS^\jmath(n,r)$ ($n\geq r$) and $\sS^\imath(n,r)$ ($n\geq r$) are Morita equivalent to the $q$-Schur$^2$ algebra of parabolic and Young type, respectively (see \cite[p.4337]{DS00a}).
\end{rem}
%Later, we will call both $i$-quantum Schur algebras, consistent with $i$-quantum groups.

% for a certain set $\{m_\omega\}_{\omega\in\Omega}$ of positive integers, 

\begin{prop} \label{extend Iwahori}
Let $R$ be a commutative ring and assume $\sqrt q, q^{-1}\in R$.
By specializing $t$ to $\sqrt q$, there is an algebra isomorphism
$$\sS^*(n,r)_R\cong \mathcal E^*_R(n,r)\;\;(*\in\{\scB,\scC\}).$$
\end{prop}
\begin{proof}The isomorphism is an extension of the Iwahori isomorphism (see Lemma \ref{H^o}):
$$\sH(W_r)|_{t^2=q}\overset\sim\longrightarrow\End_{\mathbb Z'G^*}(\text{Ind}_{B^*}^{G^*}1_{\mathbb Z'}),\;\; T_w\mapsto (\underline{B^*}\mapsto \underline{B^*wB^*}),$$ where $\mathbb Z'=\mathbb Z[\sqrt q,{\sqrt q}^{-1}]$, $G^*=G(r_*,q)$, $B^*=B(r_*,q)$, and $\underline {X}=\sum_{x\in X}x$.
Compare \cite[Th. 4.37]{DDPW08}. Now, an argument similar to the proof of \cite[Th. 13.15]{DDPW08} proves the isomorphism.
\end{proof}

Recall the ring $\sZ^\natural$ with all bad primes invertible. For $*=\jmath,\imath$, let $\sS^{*}(n,r)^\natural:=\sZ^\natural\otimes_\sZ\sS^*(n,r)$. 
\begin{prop}There exist split quasi-hereditary $\sZ^\natural$-algebras $\sS^{\jmath+}(n,r)^\natural$ and $\sS^{\imath+}(n,r)^\natural$, which are endomorphism algebras of  $\sH(W_r)$-modules with dual left-cell filtrations,  such that $\sS^{\jmath}(n,r)^\natural$ (resp., $\sS^{\imath}(n,r)^\natural$) is a centralizer subalgebra of
$\sS^{\jmath+}(n,r)^\natural$ (resp., $\sS^{\imath+}(n,r)^\natural$).
\end{prop}

\begin{proof}
If $n\geq r$, then, for any $J\subseteq S$, there exists $\la\in\La(n+1,r)$ such that $W_\la=W_J$ (e.g., $W_{(1^r,0)}=W_\emptyset$ for $n=r$). Thus, 
$\sS^\scB(n,r)^\natural=A^{\dag\natural}$ with a suitable choice of $m_J,J\subseteq S$, as seen in Remark \ref{Adag}. 
Thus, in this case, $\sS^\scB(n,r)^\natural$ is a centraliser subalgebra of some $A^{\ddag\natural}$ defined in display \eqref{Adagger}) with base change to $\sZ^\natural$. Here the associated finite standard system $(W,S,L)$ is $W=W_r$ and $L=\ell$, the usual length function.  In this case, we choose $\sS^{\scB+}(n,r)^\natural=A^{\ddag\natural}$. Now, the split quasi-heredity of  $A^{\ddag\natural}$ follows from Theorem \ref{Adagger3} in our Chapter 4.

If $n<r$, let 
$$\Theta=\Theta_{n,r}=\{\omega\in\Omega\mid \omega\leq_L\omega_\lambda,\text{ for some }\lambda\in\Lambda(n+1,r)\},$$
where $\omega_\la$ is the left cell that contains the longest word of $W_\lambda$. Then, $\Theta$ is a poset ideal of $(\Omega,\leq_L)$ and $A^{\ddag\natural}_\Theta$ is a split quasi-hereditary algebra by Theorem \ref{Adagger3}. For a suitable choice $\{m_\omega\in\mathbb Z_+\mid\omega\in\Theta\}$,
 $\sS^\scB(n,r)^\natural$ is a centralizer subalgebra of $A^{\ddag\natural}_\Theta$. We then choose $\sS^{\jmath+}(n,r)^\natural:=A^{\ddag\natural}_\Theta$ as the desired algebra.%, again following  Theorem \ref{Adagger3}. 
 
By Remark \ref{Adag}, the $\imath$ case may be proved similarly as in the $n<r$ case above.
\end{proof}

 By Corollary \ref{Adagger4} and Lemma \ref{fAf} and noting the remark right above Lemma \ref{extend Iwahori}, we immediately have:

\begin{thm}\label{unitri}
The algebras $\sS^{\jmath}(n,r)^\natural$ and $\sS^{\imath}(n,r)^\natural$ %for $*\in\{\textsc{b,c}\}$ 
are standardly based algebras. In particular, their decomposition matrices are unitriangular.
\end{thm}

\medskip
%\noindent
{\bf 5.4B $i$-Quantum groups and the associated Schur--Weyl duality.}
We now introduce the $i$-quantum groups arising from the study of quantum symmetric pairs; see \cite{Le03}, \cite{BW18} and the references therein.

Symmetric pairs refer to pairs of the form $(\mathfrak g,\mathfrak g^\theta)$ or $(\mathcal U(\mathfrak g),\mathcal U(\mathfrak g^\theta))$, where $\mathfrak g$ is a Lie algebra with its enveloping algebra   $\mathcal U(\mathfrak g)$, $\theta:\mathfrak g\to\mathfrak g$ is an involution, and $\mathfrak g^\theta$ is the fixed-point subalgebra. Note that $\mathcal U(\mathfrak g^\theta)$ is a Hopf subalgebra of $\mathcal U(\mathfrak g)$. 

For example, if $\frak g=\mathfrak{gl}_{N}$ with basis consisting of matrix units $\{e_{i,j}\}_{1\leq, i,j\leq N}$,  and  
$$\theta:\mathfrak {gl}_N\longrightarrow\mathfrak {gl}_N,\;\;\;e_{i,j}\longmapsto e_{\tau(i),\tau(j)},$$
where $\tau(i)=N+1-i$, then (see, e.g., \cite[2.4.2, 3.5.1]{LZ})
 $$\mathfrak {gl}_N^\theta\cong\begin{cases}\mathfrak{gl}_{n+1}\oplus\mathfrak{gl}_{n},\;\text{ if }N=2n+1\\
\mathfrak{gl}_{n}\oplus\mathfrak{gl}_{n},\;\text{ if }N=2n.\end{cases}$$ 
Thus,
\begin{equation}\label{Ugth}
\mathcal U(\mathfrak {gl}_N^\theta)\cong\begin{cases}\mathcal U(\mathfrak{gl}_{n+1})\otimes\mathcal U(\mathfrak{gl}_{n}),\;\text{ if }N=2n+1\\
\mathcal U(\mathfrak{gl}_{n})\otimes\mathcal U(\mathfrak{gl}_{n}),\;\text{ if }N=2n.\end{cases}
\end{equation}

Note that this involution $\theta$ is different from the differentiation $\gamma$ of $\vartheta$ in \eqref{vartheta}. Here, for $*\in\{\jmath,\imath\}$,
$\gamma:\mathfrak {gl}_N\rightarrow\mathfrak {gl}_N,\;\; x\mapsto -J_*x^tJ_*$,
with
%,  and $\gamma(x)=-J'x^tJ'$ for all $x\in\mathfrak{gl}_{2n}$, where 
$J_\jmath=J_{2n+1},J_\imath=J_{2n}^-$ being given in \eqref{JJ'}, is used to define $\mathfrak o_{2n+1}$ and $\mathfrak{sp}_{2n}$. 

\def\bfU{{\mathbf U}}
We now define the quantum analogs of $(\mathcal U(\mathfrak {gl}_N),\mathcal U(\mathfrak {gl}_N^\theta))$.
\begin{defn}\label{Uglmn}
 The quantum enveloping algebra $\bfU(\mathfrak{gl}_{N})$
over $ \mathbb Q(\up)$  is generated by
$$E_{a}, F_{a},  K_{i}^{\pm  1},\;1\leq a,i\leq N,a\neq N.$$
%with a $\Z_{2}$-grading given by setting $\widehat{E}_{m}=\widehat{F}_{m}=1$,  $\widehat{E}_{a}= \widehat{F}_{a}=0$ for $a \neq m$, and $\widehat{K^{\pm 1}_{b}}=0$.
These elements are subject to the following relations:
\begin{enumerate}
\item[(QG1)]
$ K_iK_j=K_jK_i,\ K_iK_i^{-1}=K_i^{-1}K_i=1; $
\item[(QG2)]
$ K_iE_{a}=\up^{\delta_{i,a}-\delta_{i,a+1}}E_{a}K_i,\
K_iF_{a}=\up^{-\delta_{i,a}+\delta_{i,a+1}}F_{a}K_i; $
\item[(QG3)]
$ [E_{a},
F_{b}]=\delta_{a,b}\frac{\wK_a-\wK_a^{-1}}{\up-\up^{-1}}, 
$ where ${\widetilde K}_{a}=K_aK_{a+1}^{-1}$;
\item[(QG4)]
$ E_{a}E_{b}=E_{b}E_{a}, \  F_{a}F_{b}=F_{b}F_{a},$ if $|a-b|>1 ;$ and
\item[(QG5)]
$E_{a}^2E_{b}-(\up+\up^{-1})E_{a}E_{b}E_{a}+E_{b}E_{a}^2=0,
F_{a}^2F_{b}-(\up+\up^{-1})F_{a}F_{b}F_{a}+F_{b}F_{a}^2=0$, if $|a-b|=1$.
\end{enumerate}

\end{defn}

A Hopf algebra structure  on $\bfU(\mathfrak {gl}_{N})$ is
defined  (see \cite[p.~199]{Jan96}) by:
\begin{equation}\label{coalg}
\begin{aligned}
& \Delta(K_i)=K_i\otimes K_i,\\
&\Delta(E_a)=E_a\otimes {\widetilde K}_{a}^{-1} +1\otimes E_a, \quad
   \Delta(F_a)=F_a\otimes 1+{\widetilde K}_a\otimes F_a ,\\
& \varepsilon(K_i)=1,  \quad  \varepsilon(E_a)= \varepsilon(F_a)=0,\\
& S(K_i)=K_i^{-1}, \quad S(E_a)=-E_a {\widetilde K}_{a}, \quad
S(F_a)=-{\widetilde K}_{a} ^{-1}F_a.
\end{aligned}
\end{equation}
The quantum $\mathfrak{sl}_{N}$ is the Hopf subalgebra of $\bfU(\mathfrak{gl}_{N})$
$$\bfU_N=\bfU(\mathfrak{sl}_{N})=\langle E_i,F_i,\widetilde K_i\mid 1\leq i \leq N-1\rangle.$$

 We now follow \cite{BW18} to introduce two subalgebras in the cases $N=2n+1$ and $N=2n$ via the corresponding graph automorphism of the Dynkin diagram of $\mathfrak{sl}_{N}$. For this purpose, we label the vertices of the Dynkin diagrams in a more symmetric fashion, i.e., simple roots are labeled symmetrically about 0. See \cite[(1.1)]{BW18}.

For integers $m<n$, let
$$\aligned
&[m,n]=\{i\in\mathbb Z\mid m\leq i\leq n\},\\
&[m,n]_{\frac12}=\{i\in\mathbb Z+\frac12\mid m\leq i\leq n\}.
\endaligned$$
Note that  $[m,n]$ (resp., $[m,n]_{\frac12}$) consists of all endpoints (resp., midpoints) of intervals $[i,i+1]$ for $m\leq i<n$.

Let $V_N$ be the natural representation of $\bfU_N=\bfU(\mathfrak{sl}_{N})$. Following the notation used in \cite[Chs.~2\&6]{BW18}, we now relabel the generators of $\bfU_N$ by simple roots associated with $\mathfrak{sl}_{N}$.

For $N=2n+1$, we label the basis for $V_{2n+1}$ by%the natural representation $V$ of $\bfU_\up(\mathfrak{gl}_{N})$ by
\begin{equation}\label{Bbasis}
v_{-n},\ldots,v_{-2},v_{-1},v_0,v_1,v_2,\ldots,v_{n}
\end{equation}
and the corresponding simple roots 
$\alpha_i=\varepsilon_{i-\frac12}-\varepsilon_{i+\frac12}$ by all $i\in[-n,n]_{\frac12}$. Here, $\varepsilon_j, j\in[-n,n]$ are the fundamental weights for $\mathfrak{sl}_{2n+1}$.
 
For $N=2n$, we label the basis for $V_{2n}$ by
\begin{equation}\label{Cbasis}
v_{-n+\frac12},v_{-n+\frac12},\ldots,v_{-\frac12},v_{\frac12},v_{1+\frac12},\ldots,v_{n-\frac12}
\end{equation}
and the corresponding simple roots $\alpha_i=\varepsilon_{i-\frac12}-\varepsilon_{i+\frac12}$ by all $i\in[-n+1,n-1]$.

\begin{defn}\label{FirstPre}
%\begin{itemize}
{\rm(1)} The $\mathbb Q(t)$-subalgebra $\bfU^\jmath_n$ of $\bfU_{2n+1}$ is generated by
$$
{\bf k}_i=\wK_i\wK_{-i}^{-1},\;\;
{\bf e}_i=E_i+\wK_i^{-1}F_{-i},\;\; {\bf f}_i=F_i\wK_{-i}^{-1}+E_{-i},\;\;\forall i\in[0,n]_{\frac12}.
$$
{\rm(2)}  The $\mathbb Q(t)$-subalgebra $\bfU^\imath_n$ of $\bfU_{2n}$ is generated by, for $1\leq i\leq n-1$,
$$
{\bf k}_i=\wK_i\wK_{-i}^{-1},\;\;
{\bf e}_i=E_i+\wK_i^{-1}F_{-i},\;\; {\bf f}_i=F_i\wK_{-i}^{-1}+E_{-i},\;\;{\bf t}=E_0+\up F_0\wK_0^{-1}+\wK_0^{-1}.
$$
%\end{itemize}
\end{defn}
\begin{rem} Note that $\bfU^\jmath_n,\bfU^\imath_n$ are not Hopf subalgebras, but coideal subalgebras, and
 $(\bfU_{2n+1},\bfU^\jmath_n)$, $(\bfU_{2n},\bfU^\imath_n)$  form  some quantum symmetric pairs,\footnote{$\bfU^\jmath_n$, $\bfU^\imath_n$ are denoted $\bfU^\jmath, \bfU^\imath$, respectively, in \cite{BW18}.} which are the quantum analogs of the symmetric pairs $(\mathcal U(\mathfrak{gl}_N),\mathcal U(\mathfrak{gl}_N^\theta))$ with $N=2n+1,2n$. Thus, they are the quantum analog of $\mathcal U(\mathfrak{gl}_{N}^\theta)$, $N\in\{2n+1,2n\}$, and will be called  $i$-{\it  quantum group}\index{$i$-quantum groups}. \index{$i$-quantum group! $\sim$ $\bfU^\jmath_n$, $\bfU^\imath_n$}
 
 Note that, by \eqref{Ugth}, there is another quantization of $\mathcal U(\mathfrak{gl}_{N}^\theta)$ as a tensor product of two quantum $\mathfrak{gl}_m$ which cannot be embedded into $\mathcal U(\mathfrak{gl}_{N})$. 
\end{rem}
We summarize the two cases in the following table:
\vspace{1ex}
\begin{center}
\begin{tabular}{|c|c|c|c|}\hline
$\bfU_N=\bfU(\mathfrak{sl}_N)$&Basis for $V_N$&Simple roots&$i$-Quanutm groups\\\hline
$N=2n+1$&$v_i,i\in[-n,n]$&$\alpha_i,i\in[-n,n]_{\frac12}$&$\bfU^\jmath_n\subset \bfU_{2n+1}$\\ \hline
$N=2n$&$v_i,i\in[-n,n]_{\frac12}$&$\alpha_i,i\in[-n+1,n-1]$&$\bfU^\imath_n\subset \bfU_{2n}$\\ \hline
\end{tabular}
\end{center}

%The action on $V_N$...

For $*\in\{\scB,\scC\}$, let $\epsilon_*=\begin{cases}1,&\text{ if }*=\scB\\
0,&\text{ if }*=\scC\end{cases}$\\ and let 
\begin{equation}\label{Inr}
\aligned
\mathbb I(2n+\epsilon_\scB,r)&=\{(i_1,i_2,\ldots,i_r)\mid i_j\in[-n,n]\}\\
\mathbb I(2n+\epsilon_\scC,r)&=\{(i_1,i_2,\ldots,i_r)\mid i_j\in[-n, n]_{\frac12}\}.
\endaligned
\end{equation}
Both sets index bases for the {\it tensor spaces} $V_{2n+1}^{\otimes r}$ and $V_{2n}^{\otimes r}$, respectively.
 
The new symmetric notation above gives rise to an adjustment of the associated Coxeter system $(W_0,S_0)$ with $S_0=\{s_0,s_1,\ldots,s_{r-1}\}$. Here, we relabel the generators $s_i\,(1\leq i\leq r)$ considered in Subsection 5.4A by $s_{r-i}, 1\leq i\leq r$. Thus, the Dynkin diagram of type $B_r$ takes the form
\vspace{.2cm}
\begin{equation}\label{Br}
\text{
%\begin{center}
\begin{tikzpicture}[scale=1.5]
%\fill(-1,0) node {$B_m$:};
\fill (0,0) circle (1.5pt);
\fill (1,0) circle (1.5pt);
\fill (2,0) circle (1.5pt);
\fill (4,0) circle (1.5pt);
\fill (-1,0) circle (1.5pt);
\draw (-1,0.05) --
        (0,0.05);
\draw (-1,-0.05) node[below]  {$_0$} --
        (0,-0.05);
  \draw (0,0) node[below] {$_1$} --
        (1,0) node[below] {$_2$} -- (2,0)node[below] {$_3$}--(2.5,0);
\draw[style=dashed](2.5,0)--(3.5,0);
\draw (3.5,0)--(4,0) node[below] {$_{r-1}$};
%\draw (4.4,0.1)--(4.6,0);
%\draw (4.4,-0.1)--(4.6,0);
\end{tikzpicture}
%\end{center}
}
\end{equation}
Note that this labeling reverses the labeling in Examples \ref{SO_n} and  \ref{SU_n}.

  Let 
  \begin{equation}\label{H'}
  \sH':=\sH_{\up^{-1}}(W_0)\quad\text{and}\quad \boldsymbol{\sH}':=\sH'_{\mathbb Q(t)}
  \end{equation} be the Hecke algebras over $\sZ$ and $\mathbb Q(t)$, respectively  (relative to parameter $t^{-1}$) of the standard finite Coxeter system $(W_0,S_0,L)$ with constant weight function $L(s_i)=1$. Then, $\sH'$ is generated by $T_{s_0},T_{s_1},\ldots, T_{s_{r-1}}$. The subalgebra $\sH'(\fS_r)$ generated by $T_{s_1},\ldots, T_{s_{r-1}}$ is the Hecke algebra of the symmetric group $\fS_r$.

 The $\bfU_{2n+\epsilon_*}$-module $V_{2n+\epsilon_*}^{\otimes r}$ commutes with the following action by $\sH'(\fS_r)$ on the basis
$v_{\ul{i}}:=v_{i_1}\otimes\cdots\otimes v_{i_r}$, for ${\ul i}\in \mathbb I(2n+\epsilon_*,r):$
\begin{equation}\label{actionA}
v_{\ul i}T_{s_j}=\begin{cases} v_{\ul {i}s_j},&\text{ if }i_j<i_{j+1};\\
\up^{-2}v_{\ul i},&\text{ if }i_j=i_{j+1};\\
(\up^{-2}-1)v_{\ul i}+\up^{-2}v_{\ul{i}s_j},&\text{ if }i_j>i_{j+1}.
\end{cases}
\end{equation}

%$$e_{\ul i}\tilde T_{s_j}=\begin{cases} e_{\ul {i}s_j},&\text{ if }i_j<i_{j+1};\\
%\up^{-1}e_{\ul i},&\text{ if }i_j=i_{j+1};\\
%(\up^{-1}-\up)e_{\ul i}+e_{\ul(i)s_j},&\text{ if }i_j>i_{j+1},
%\end{cases}$$
%where $\wT_i=tT_i$.

 The $\bfU_{2n+\epsilon_*}$-$\bsH'(\fS_r)$-bimodule $V_{2n+\epsilon_*}^{\otimes r}$ satisfies the double centralizer property---the quantum Schur--Weyl duality.  
 
 We extend the action to the Hecke algebra $ \sH'$ of $W_0$. We first do it integrally.
  Consider the free $\sZ$-submodule of $V_{2n+\epsilon_*}^{\otimes r}$ spanned by all $v_{\ul{i}}:=v_{i_1}\otimes\cdots\otimes v_{i_r}$:
 \begin{equation}\label{Tnr}
 T(2n+\epsilon_*,r)=\text{span}_\sZ\{v_{\ul{i}}\mid{\ul i}\in \mathbb I(2n+\epsilon_*,r)\}.
 \end{equation}
  This is clearly a right $\sH'(\fS_r)$-module with the above action in \eqref{actionA}.
 Extend this action to $\sH'$ by setting (cf. \cite[(6.8), (5.2)]{BW18})
 \begin{equation}\label{actionB}v_{\ul i}T_{s_0}=\begin{cases}
v_{\ul{i}s_0}, &\text{ if }i_1>0;\\
\up^{-2}v_{\ul i},&\text{ if }i_1=0;\\
(\up^{-2}-1)v_{\ul i}+\up^{-2}v_{\ul{i}s_0},&\text{ if }i_1<0.
\end{cases}\quad\text{ for }N=2n+1
\end{equation}
\begin{equation}\label{actionC}
v_{\ul i}T_{s_0}=\begin{cases} v_{\ul{i}s_0}, &\text{ if }i_1>0;\\
(\up^{-2}-1)v_{\ul i}+\up^{-2}v_{\ul{i}s_0},&\text{ if }i_1<0.\end{cases}\quad\text{ for }N=2n.%\vspace{-1ex}
\end{equation}

For $\la\in\La_*(n,r)$, we will also relabel the components as $\la=(\la_0,\la_1,\ldots,\la_{n})$, if $\la\in\La_\jmath(n,r)=\La(n+1,r)$, and the generators for the parabolic subgroup $W_{0,\la}$ of $W_0$. See Lemma \ref{parab} below.
Let $\sD_\la$ be the shortest representatives of the right cosets of $W_{0,\la}$ in $W_0$. It is easily checked that, for ${\ul i}\in \mathbb I(2n+\epsilon_*,r)$, if ${\ul i}={\ul i}_\la d$ for some $d\in\sD_\la$, then
$v_{\ul{i}}=v_{{\ul i}_\la}T_d$.

Note that, putting $e_{\ul i}=v_{\ul{i}_\la}\wT_d$, where $\wT_d=t^{\ell(d)}T_d$, then the induced action from \eqref{actionA} and 
\eqref{actionB} (resp., \eqref{actionC}) coincides with the action in \cite[(6.8)]{BW18} (resp., \cite[(5.2)]{BW18}).

For $*\in\{\scB,\scC\}$, define {\it $i$-quantum Schur algebras} over $\sZ$:
$$U^*_{n,r}:=\End_{\sH'}\big(T(2n+\epsilon_*,r)\big),\qquad \bfU_{n,r}^*:=\End_{\boldsymbol\sH'}(V_{2n+\epsilon_*}^{\otimes r}).%\cong U_{n,r}^*\otimes\mathbb Q(t).
$$
%We will establish a $\sZ$-algebra isomorphism between $U^*_{n,r}$
% $\bfU^*_{n,r}:=\End_{\bsH'}(V_{2n+\epsilon_*}^{\otimes r})\cong U^*_{n,r}\otimes\mathbb Q(t)$ 
% and the $q$-Schur algebra $\sS^*_{t^{-1},\sZ}(n,r)$ of type $B/C$ in \eqref{Snr} in the next subsection.

The following theorem gives a new Schur--Weyl duality between $\bfU^*_n$ and $\bsH':=\mathbb Q(t)\otimes_\sZ\sH'$ for $*\in\{\jmath,\imath\}$.
% via $\sS^{\imath*}(n,r)_{\mathbb Q(t)}$.

\begin{thm}[{\cite[Ths. 5.4\&6.27]{BW18}}]\label{BW}
Maintaining the notation above and letting $*\in\{\jmath,\imath\}$, we have:
\begin{enumerate}
\item The restricted $\bfU^*_n$-action on $V_{2n+\epsilon_*}^{\otimes r}$ commutes with the $\boldsymbol{\sH}'$-action.
\item The $\bfU^*_n$-$\boldsymbol{\sH}'$-bimodule $V_{2n+\epsilon_*}^{\otimes r}$ satisfies the double centralizer property. Hence, there is a Schur--Weyl duality over $\mathbb Q(\up)$ in each of the choices of $*$.
\item In particular, there is an algebra epimorphism over $\mathbb Q(\up)$:
$$\phi_{*}:\bfU^*_n\longrightarrow \bfU^*_{n,r}:=\End_{\boldsymbol\sH'}(V_{2n+\epsilon_*}^{\otimes r}).$$
\end{enumerate}
\end{thm}
\begin{rems} 
(1) We will see below that %by specialising $\up^{-2}$ to a prime power $q$, 
$U^\jmath_{n,r}$ and $U^\imath_{n,r}$ are 
isomorphic to $\sS^\scB_{t^{-1},\sZ}(n,r)$ and $\sS^\scC_{t^{-1},\sZ}(n,r)$, respectively. Hence, they are
 Hecke endo-algebras associated with a finite orthogonal group $\O_{2r+1}(\mathbb F_q)$ or a finite symplectic group $\Sp_{2r}(\mathbb F_q)$.
 %and both are Morita equivalent to centraliser subalgebras of $A^\dagger$.
%\item By this theory developed above we have better understanding of the Hecke endomorphism algebra, such as decomposition matrix, standardly based algebra, etc..

(2) By \eqref{Ugth}, there is another quantum analog of $\mathcal U(\mathfrak{gl}_N^\theta)$, which is a tensor product of some quantum $\mathfrak{gl}_m$. The quantum Schur--Weyl duality relative to these quantum analogs and Hecke algebra of type $B$ is discussed in \cite{SS99} and \cite{Hu01}.
\end{rems}

\medskip
%\noindent
{\bf 5.4C Place permutations and the tensor space decomposition.}
 We now establish isomorphism between the Hecke endo-algebras $\sS^\scB(n,r)$ and $\sS^\scC(n,r)$ introduced in Subsection 5.4A and $U^\jmath_{n,r}$ and $U^\imath_{n,r}$ introduced in Subsection 5.4B. The isomorphisms are largely known;
 see, e.g., \cite{LL21, LNX20, LW22}. We include a proof using place permutations for completion.

We may regard an element $\ul i\in\mathbb I(2n+\epsilon_\scB,r)$ as a function 
$${\ul i}:[-r,r] \longrightarrow [-n,n],\;\; k\longmapsto i_k\text{ with }i_{-k}=\begin{cases} -i_k,&\text{ if }i_k\neq0\\
0,&\text{ if }i_k=0,\end{cases}$$
and, similarly, regard an element $\ul i\in\mathbb I(2n+\epsilon_\scC,r)$ as a function 
$${\ul i}:[-r,-1]\cup[1,r] \longrightarrow [-n,n]_{\frac12},\;\; k\longmapsto i_k\text{ with }i_{-k}=-i_k.$$

We now introduce $W_0$-actions on the sets $\mathbb I(2n+\epsilon_*,r)$, for $*\in\{\jmath,\imath\}$. Recall from \eqref{Br} that $W_0$ is the Weyl group of type $B_r$ with Coxeter generators $S_0=\{s_0,s_1,\ldots,s_{r-1}\}$, where $S_0\backslash\{s_0\}$ generates a parabolic subgroup isomorphic to the symmetric group $\fS_r$.
Note that we have relabeled the Coxeter generators $s_i, 1\leq i\leq r$ in Subsection 5.4A by $s_{r-i}, 1\leq i\leq r$. By regarding  $\fS_{2r+1}$ as the permutation group $\fS_{[-r,r]}$ of the set $[-r,r]$,  $W_0$ is the subgroup of $\fS_{[-r,r]}$ (resp., $\fS_{[-r,r]\backslash\{0\}}$, regarded as a subgroup of $\fS_{[-r,r]}$) consisting of permutations 
\begin{equation}\label{fp group}
\begin{pmatrix}-r&\cdots&-2&-1&0&1&2&\cdots &r\\
i_{-r}&\cdots&i_{-2}&i_{-1}&0&i_1&i_2&\cdots&i_r\end{pmatrix}
\end{equation}
satisfying $i_{-j}=-i_j$, for all $j\in[1,r]$ (cf. \eqref{Sr*}).
Hence, $W_0$ is generated by
$$s_0=(-1,0)(0,1)(-1,0)=(-1,1),\;\;s_i=(i,i+1)(-i,-i-1),\;\; 1\leq i\leq r-1.$$
%=\begin{pmatrix}-r&\cdots&-2&-1&0&1&2&\cdots&r\\-r&\cdots&-2&1&0&-1&2&\cdots&r\end{pmatrix},$$
%and  $s_i=(i,i+1)(-i,-i-1)$, for $i=1,2,\ldots,r-1$.

Since every element $w$ in $W_0$  satisfies $w(-j)=-w(j)$, for $j\in[-r,r]\backslash\{0\}$ and $w(0)=0$, the following action, defined by composition of functions, is well-defined
$${\ul i} .w=(i_{w(1)},i_{w(2)},\ldots,i_{w(r)}),$$
for all ${\ul i}\in\mathbb I(2n+\epsilon_*,r)$. This action generalizes the well-known {\it place permutation} from $\fS_r$ to the Weyl group $W_0$.
So, $W_0$ acts on $\mathbb I(2n+\epsilon_*,r)$.

For $\ul{i}= (i_1,i_2,\ldots,i_r)\in\mathbb I(2n+\epsilon_\scB,r)$ (here $*=\jmath$), let 
$$\lambda_k=\#\{i_j\mid 1\leq j\leq r, k=|i_j|\},\text{ for all }k\in[0,n]$$  and define
$\text{wt}(\ul{i})=\la:=(\la_0,\lambda_1,\ldots,\lambda_{n})$. 

For $\ul{i}= (i_1,i_2,\ldots,i_r)\in\mathbb I(2n+\epsilon_\scC,r)$ (here $*=\imath$), by \eqref{Inr}, $|i_k|+\frac12\in[1,n]$, for all $k\in[1,r]$. Let 
 $$\lambda_k=\#\{i_j\mid 1\leq j\leq r,k=|i_j|+\frac12\},  \text{ for all }k\in[1,n]$$ and define
$\text{wt}(\ul{i})=\la:=(\lambda_1,\ldots,\lambda_{n})$. 

Thus, for $*\in\{\scB,\scC\}$, the sets $\Lambda_*(n,r)$ in \eqref{Lanr} have the following interpretation
 $$\Lambda_*(n,r)=\{\text{wt}(\ul{i})\mid \ul{i}\in\mathbb I(2n+\epsilon_*,r)\}.$$
 For $\la\in\Lambda_\scB(n,r)$, let
$$
{\ul {i_\la}}=(\underbrace{0,\ldots,0}_{\la_0},\underbrace{1,\ldots,1}_{\la_1},\ldots,\underbrace{n,\ldots,n}_{\la_{n}})$$
and, for $\la\in\Lambda_\scC(n,r)$, let
$${\ul i_\la}=(\underbrace{1_{\frac12},\ldots,1_{\frac12}}_{\la_1},\underbrace{2_{\frac12},\ldots,2_{\frac12}}_{\la_2},\ldots,\underbrace{n_{\frac12},\ldots,n_{\frac12}}_{\la_{n}}),$$
where $a_{\frac12}=a-{\frac12}$, for all $a\in[1,n]$.
\begin{lem}\label{parab}
For $*\in\{\scB,\scC\}$ and $\la\in\Lambda_*(n,r)$, the orbit 
$${\ul {i_\la}} .W_0=\{{\ul i}\in \mathbb I(2n+\epsilon_*,r),\wt({\ul i})=\la\},$$
and $\Stab_{W_0}({\ul {i_\la}})=W_{0,\la}$ is the subgroup generated by 
\begin{itemize}
\item[(1)] $S\backslash\{s_{\la_0+\cdots+\la_i}\mid 0\leq i<n\}$, if $\la\in\Lambda_\scB(n,r)$ and
\item[(2)] $S\backslash\{s_0,s_{\la_1+\cdots+\la_i}\mid 1\leq i<n\}$, if $\la\in\Lambda_\scC(n,r)$.
\end{itemize}
\end{lem} 
\begin{proof} We first assume $\la\in\Lambda_\scB(n,r)$.
Consider the place permutation of $\fS_{[-r,r]}$ on 
$$I(2n+1,2r+1)=\{(i_{-r},\ldots,i_{-1},i_0,i_1,\ldots,i_r)\mid i_j\in[-n,n],\;\forall j\in[-r,r]\}.$$ For ${\ul i}=(i_1,i_2,\ldots,i_r)\in \mathbb I(2n+\epsilon_\scB,r)$,
let 
$$-{\ul i}^\tau\circ{\ul i}:=(i_{-r},\ldots,i_{-2},i_{-1},0,i_1,i_2,\ldots,i_r),$$
where $i_0=0$.
(Recall $i_{-j}=-i_j$.) By regarding $W_0$ as a subgroup of $\fS_{[-r,r]}$ via \eqref{fp group}, it is clear that
$\Stab_{W_0}({\ul {i}})=W_0\cap \Stab_{\fS_{[-r,r]}}(-{\ul i}^\tau\circ{\ul i})$. In particular,
$\Stab_{W_0}({\ul {i_\la}})=W_0\cap \Stab_{\fS_{[-r,r]}}(-{\ul {i_\la}}^\tau\circ{\ul {i_\la}})$.
Since 
$$\Stab_{\fS_{[-r,r]}}(-{\ul {i_\la}}^\tau\circ{\ul {i_\la}})=\fS_{\widetilde\la}\cong\fS_{\la_n}\times\cdots\times\fS_{\la_1}\times\fS_{2\la_0+1}\times\fS_{\la_1}\times\ldots\times\fS_{\la_n},$$
it follows that $\Stab_{W_0}({\ul {i_\la}})=W_0\cap\fS_{\widetilde\la}=W_{0,\la}$.

The orbit equality can be seen as follows. First, for any ${\ul i}=(i_1,i_2,\ldots,i_r)\in \mathbb I(2n+\epsilon_\scB,r)$ with $\wt({\ul i})=\la$, there is $x\in\fS_r$ such that $\ul{i_\la}x=(|i_1|, |i_2|,\ldots,|i_n|)$. Now, assume $i_{j_1},\ldots,i_{j_t}$ are negative, then 
$(|i_1|, |i_2|,\ldots,|i_n|)t_{j_1},\ldots,t_{j_t}=(i_{1},\ldots,i_{r})$, where $t_i=s_{i-1}\cdots s_1s_0s_1\cdots s_{i-1}$ is the permutation swapping $i$ and $-i$ and fixing others.
Hence, $\ul{i_\la}xt_{j_1},\ldots,t_{j_t}=(i_{1},\ldots,i_{r})$.
%By definition, $S_\la\subset \Stab_W({\ul i})$, Thus, $W_\la\subseteq \Stab_W({\ul i})$. If $w\in \Stab_W({\ul i})$
%and ${\ul i}=\ul{i_\la}$, then $w(R^\la_i)=R_i^\la$ for all $i=0,1,\ldots n$. It $\la_0\neq 0$, so $w\in\langle s_0,\fS_\la\rangle$.\

The case for $\la\in\Lambda_\scC(n,r)$ is similar.
\end{proof}
 
We now consider the $\sZ$-submodules of $T(2n+\epsilon_*,r)$ defined in \eqref{Tnr}:
$$\aligned
T(2n+\epsilon_\scB,r)_\la&=\text{span}_\sZ\{v_{\ul i}\mid {\ul i}\in\mathbb I(2n+\epsilon_\scB,r),\wt({\ul i})=\la\}=\text{span}_\sZ\{v_{\ul i_\la d}\mid d\in\sD_\la\},\\
T(2n+\epsilon_\scC,r)_\la&=\text{span}_\sZ\{v_{\ul i}\mid {\ul i}\in\mathbb I(2n+\epsilon_\scC,r),\wt({\ul i})=\la\}=\text{span}_\sZ\{v_{\ul i_\la d}\mid d\in\sD_\la\}.\\
\endaligned$$
Recall from \eqref{H'} that $\sH'=\sH_{\up^{-1}}(B_r)$.
\begin{lem}\label{Sjnr}
 For $*\in\{\jmath,\imath\}$ and any $\la\in\Lambda_*(n,r)$, $T(2n+\epsilon_*,r)_\la$ is an $\sH'$-submodule of $T(2n+\epsilon_*,r)$ and, as right $\sH'$-modules,
$$T(2n+\epsilon_*,r)_\la\cong x_\la\sH',\text{ where }x_\la=\sum_{w\in W_\la}T_w.$$
Thus, we have an algebra isomorphism%for $n+1\geq r$,
$$U^*_{n,r}\cong\End_{\sH'}\Big(\bigoplus_{\la\in\Lambda_*(n,r)}x_\la\sH'\Big)=\sS^*_{t^{-1},\sZ}(n,r).$$ 
%is Morita equivalent to the Hecke endomorphism algebra $A$ in equal parameters defined in (5.0.33), 
%while $\sS^\iC$ is Morita equivalent to the centralizer subalgebra $A^{\textsc{y}}=\End_{\sH}(\oplus_{J\subseteq S\cap\fS_r}x_J\sH)$ of $A$.
\end{lem}
\begin{proof}By definition, $\mathcal Zv_{\ul{i_\la}}$ is an $\sH'_\la$-module isomorphic to $\mathcal Z x_\la$, where $\sH'_\lambda$ is the subalgebra of $\sH'$ associated with the parabolic subgroup $W_{0,\lambda}$.
By the Frobenius reciprocity, we have an isomorphism
$$\text{Hom}_{\sH'}(x_\la\sH',T(2n+\epsilon_*,r)_\la)\cong \text{Hom}_{\sH'_\la}(\mathcal Zx_\la,T(2n+\epsilon_*,r)_\la).$$
Thus, the $\sH'_\la$-module homomorphism $\mathcal Zx_\la\to T(2n+\epsilon_*,r), x_\la\mapsto v_{\ul{i_\la}}$ induces an $\sH'$-module isomorphism from $x_\la\sH'$ to $T(2n+\epsilon_*,r)_\la$.
\end{proof}

\begin{rems} For the $*=\imath$ case, this isomorphism is also given in \cite[Prop. 4.1.4]{Grn97}. As noted in Remark \ref{Adag},
Morita equivalent versions of  $\sS^\scB(n,r)$ and $\sS^\scC(n,r)$  are considered in \cite{DS00a}.
% with equal parameters.
\end{rems}

\medskip
%\noindent
{\bf 5.4D. $i$-quantum groups and finite orthogonal/symplectic groups.} 
In order to link $i$-quantum groups with the finite orthogonal/symplectic groups discussed in Subsection 5.4A, we now recover the standard labeling used in Definition \ref{Uglmn}. The $i$-quantum groups $\bfU^*_n$ introduced in Definition \ref{FirstPre} are subalgebras of $\bfU(\mathfrak{sl}_N)$. We present them in generators and relations as a counterpart of 
$\bfU(\mathfrak{gl}_N)$.
%This requires an index shift from the labelling in  \eqref{Bbasis} and \eqref{Cbasis} and a graph automorphism for $\bfU(\mathfrak{gl}_N)$.
%Also, like the relationship between $\bfU(\mathfrak{sl}_N)$ and $\bfU(\mathfrak{gl}_N)$, we may introduce new generators $\mathbf d_j^{\pm1}$ such that $\mathbf k_i=\mathbf d_i\mathbf d_{i+1}^{-1}$ to define new $i$-quantum groups $\bfU^\jmath(n)$ and  $\bfU^\imath(n)$ generated, respectively, by $\mathbf e_i,\mathbf f_i,\mathbf d_j$ and $\mathbf e_h,\mathbf f_h, \mathbf d_k,\mathbf t$, where $i\in[1,n],j\in[1,n+1]$ and $h\in[1,n),k\in[1,n]$. 
The following definition is taken from \cite{BKLW}.

\begin{defn}\label{iQG} (1)
The algebra ${\bfU}^{\jmath}(n)$ is defined to be the associative algebra over $\mathbb Q(\up)$ generated by $e_i$, $f_i$, $d_a$, $d_{a}^{-1}$, $i\in[1,n]$, $a\in[1,n+1]$ subject to the relations (QG1)--(QG5) of Definition \ref{Uglmn} with $e_i$, $f_i$, $d_a$, $d_{a}^{-1}$ replacing $E_i$, $F_i$, $K_a$, $K_{a}^{-1}$, where, for (QG1) and (QG4), $N=n+1$ and, for (QG2), (QG3), and (QG5), $N=n$, together with the following additional relations:
\begin{itemize}
%\item[(iQG1)] $d_ad_a^{-1}=d_a^{-1}d_a=1, d_ad_b=d_bd_a$;
\item[(iQG1)] %$d_ae_jd_a^{-1}=v^{\delta_{a,j}-\delta_{a,j+1}}e_j,$
                     %$d_af_jd_a^{-1}=v^{-\delta_{a,j}+\delta_{a,j+1}}f_j$, if $a\leq n$;
%\item[(iQg2)]                     
$d_{n+1}e_jd_{n+1}^{-1}=v^{-2\delta_{n,j}}e_j,$ 
                     $d_{n+1}f_jd_{n+1}^{-1}=v^{2\delta_{n,j}}f_j$, $j\in[1,n]$;
%\item[(iQG3)] $e_if_j-f_je_i=\delta_{i,j}\frac{d_id_{i+1}^{-1}-d_i^{-1}d_{i+1}}{v-v^{-1}}$, if $i,j\neq n$;
%\item[(iQG4)] $e_ie_j=e_je_i, f_if_j=f_jf_i,$ if $|i-j|>1$;
%\item[(iQG5)] $ e_i^2e_j+e_je_i^2=[2]e_ie_je_i,$ 
% $ f_i^2f_j+f_jf_i^2=[2]f_if_jf_i$, if $|i-j|=1$;
\item[(iQG2)] $f_n^{2}e_n+e_nf_{n}^2=[2]\big(f_ne_nf_n-(vd_nd^{-1}_{n+1}+v^{-1}d_n^{-1}d_{n+1})f_n\big)$; and
\item[(iQG3)] $e_n^{2}f_n+f_ne_{n}^2=[2]\big(e_nf_ne_n-e_n(vd_nd^{-1}_{n+1}+v^{-1}d_n^{-1}d_{n+1})\big).$
\end{itemize}

(2) The algebra ${\bfU}^{\imath}(n)$ is defined to be the associative algebra over $\mathbb Q(\up)$ generated by $e_i$, $f_i$, $d_a$, $d_{a}^{-1}$, $\fkt$, $i\in[1,n)$, $a\in[1,n]$ subject to the relations (QG1)--(QG5) with $e_i$, $f_i$, $d_a$, $d_{a}^{-1}$ replacing $E_i$, $F_i$, $K_a$, $K_{a}^{-1}$ and $N=n$ 
%for $i,j\in[1,n)$, $a,b\in[1,n]$,
 together with the following relations involving $\fkt$:
\begin{itemize}
\item[(iQG1$'$)] $\fkt d_a=d_a\fkt$, $e_i \fkt=\fkt e_i$, $f_j \fkt=\fkt f_j$ for $i\neq n-1$;

\item[(iQG2$'$)] $\fkt^2 e_{n-1}+e_{n-1}\fkt^2=[2]\fkt e_{n-1} \fkt+e_{n-1}$, $e^2_{n-1}\fkt+\fkt e^2_{n-1}=[2]e_{n-1}\fkt e_{n-1}$; and
\item[(iQG3$'$)] $\fkt^2 f_{n-1}+f_{n-1}\fkt^2=[2]\fkt f_{n-1}\fkt+f_{n-1}$, $f^2_{n-1} \fkt+\fkt f^2_{n-1}=[2]f_{n-1}\fkt f_{n-1}$.

\end{itemize}

\end{defn}\index{$i$-quantum group! $\sim$ $\bfU^\jmath(n)$, $\bfU^\imath(n)$}

Like the embedding $\bfU(\mathfrak{sl}_N)$ into $\bfU(\mathfrak{gl}_N)$, we have the various algebra embeddings.

\begin{prop} \label{i-embed}
The following are $\mathbb Q(\up)$-algebra monomorphisms:
$$\eta^\jmath: \bfU^\jmath_n \lra\bfU^\jmath(n),\;\mathbf e_{n-i+\frac12}\lmt e_i,\mathbf f_{n-i+\frac12}\lmt f_i,\mathbf k_{n-i+\frac12}\lmt d_id_{i+1}^{-1}, \mathbf k_{\frac12}\lmt\up^{-1}d_nd_{n+1}^{-1},$$
for all $i\in[1,n]$ (see \cite[Remark 4.3]{BKLW}), and 
$$\eta^\imath: \bfU^\imath_n \lra\bfU^\imath(n),\;\mathbf e_{n-h}\lmt e_h,\mathbf f_{n-h}\lmt f_h,\mathbf k_{n-h}\lmt d_hd_{h+1}^{-1}, \mathbf t\lmt\fkt,$$
for all $h\in[1,n)$.
\end{prop}

Theorem \ref{BW} continues to hold for these enlarged algebras.\footnote{After an index shift similar to the one in Proposition \ref{i-embed} and an adjustment on the coalgebra structure on $\bfU(\mathfrak{gl}_N)$,  the parameters for $\bfU^*(n)$ and $\sS^*(n,r)_{\mathbb Q(\up)}$ can be made the same $\up$.}
Let $\bsS^\scB(n,r)=\sS^\scB(n,r)_{\mathbb Q(t)}$ and $\bsS^\scC(n,r)=\sS^\scC(n,r)_{\mathbb Q(t)}$.
 Thus, we have the following corollary.

\begin{cor}\label{PR}
The category $\bsS^\scB(n,r)$-{\rm mod} of finite-dimensional $\bsS^\scB(n,r)$-modules is a full subcategory of $\bfU^\jmath(n)$-{mod}; while
the category $\bsS^\scC(n,r)$-{mod} of finite-dimensional $\bsS^\scC(n,r)$-modules is a full subcategory of $\bfU^\imath(n)$-{\rm mod}.
\end{cor}

For certain classes of classical weight modules of $\bfU^*(n)$ over $\mathbb Q(\up)$ and their crystal bases, see \cite{W1, W2}.

% It would be natural to define these representations as {\it polynomial representations} of $\bfU^*(n)$. It would be also interesting to characterise polynomial representation intrinsically, say, in terms of their ``weights". 
 
 In order to link representations of finite orthogonal/symplectic groups in cross characteristics, we need to consider $i$-quantum hyperalgebras. This is done in \cite{DW22, DW23}.%For this purpose, we make the following definition.

Let
$$\left[\begin{matrix}d_i\\s\end{matrix}\right]=\left[\begin{matrix}d_i\\s\end{matrix}\right]_t=\prod_{i=1}^s\frac{d_i\up^{-i+1}-d_i^{-1}\up^{i-1)}}{\up^i-\up^{-i}}\;(s\geq1),\;\text{ and }\; \left[\begin{matrix}d_i\\0\end{matrix}\right]=1.$$

Let $U_{\sZ}^\jmath(n)$ be the $\sZ$-subalgebra of $\bfU^\jmath(n)$ generated by
$${d}_i,\left[\begin{matrix}{ d}_i\\s\end{matrix}\right], \;{d}_{n+1},\left[\begin{matrix}{d}_{n+1}\\s\end{matrix}\right]_{t^2},
\quad{e}_i^{(m)},{f}_i^{(m)}\;(1\leq i\leq n, s,m\in\mathbb N).$$
See \cite{DW22}. The definition of an integral form $U_{\sZ}^\imath(n)$ for $\bfU^\imath(n)$ is a bit tricky since the $\sZ$-subalgebra  generated by
$${d}_i,\left[\begin{matrix}{d}_i\\s\end{matrix}\right],\;\;{e}_h^{(m)},{f}_h^{(m)},\;\fkt^{(m)}\;\;(1\leq i\leq n,1\leq h\leq n-1, s,m\in\mathbb N)
$$ does not map onto $\sS^\scC(n,r)$. In fact, replacing all $\fkt^{(m)}$ by  the image of ${}^{(m)}g:=(mE_{n,n+1}^\theta)(\mathbf0)$ via a new realization of $\bfU^\imath(n)$ can fix the problem. See
\cite{DW23} for more details.
 
For any field $k$ and a specialisation $\sZ\to k$, the base-changed algebras $U_{\sZ}^\jmath(n)_{k}$ and $U_{\sZ}^\imath(n)_{k}$ are called {\it $i$-quantum hyperalgebras}.

 %for quantum linear groups are defined in term of their weights. For irreducible ones, we may simply describe them in terms their highest weights. For $i$-quantum groups, however, little is known about their representation theory.  Questions like the following need to be answered:
% \begin{itemize}
%\item[(1)] Does an $i$-quantum group have a triangular decomposition?
%\item[(2)] Does the representation category of an $i$-quantum group have a connection to a highest weight category? 
%\end{itemize}

 By Theorem \ref{BW}, Lemma \ref{Sjnr}, and results from \cite{DW22}, and \cite{DW23}, we have the following.
 
 \begin{thm}\label{integral}
(1)  There is an algebra epimorphism $\eta_r^\jmath:U_{\sZ}^\jmath(n)\to\sS^\scB(n,r)$, for all $r\geq0$. Thus, for any base change from $\sZ$ to any field $k$, it induces a surjective map. Hence, the category $\sS^\scB(n,r)_k$-{\rm mod} of finite-dimensional $\sS^\scB(n,r)_k$-modules is a full subcategory of $U_{\sZ}^{\jmath}(n)_k$-{\rm mod}.

(2) %There is an algebra epimorphism $\eta_r^\imath:U_{\sZ}^\imath(n)\to\sS^\scC(n,r)$ for all $r\geq0$. Thus, for any base change from $\sZ$ to any field $k$, it induces a surjective map. Hence, 
Similarly, with $\imath$ replacing $\jmath$, we obtain that
the category $\sS^\scC(n,r)_k$-{\rm mod} of finite-dimensional $\sS^\scC(n,r)_k$-modules is a full subcategory of $U_{\sZ}^{\imath}(n)_k$-{\rm mod}.

\end{thm}

 Since the $q$-Schur algebras of type $B$ or $C$ are not necessarily quasi-hereditary, it is unlikely that the representation categories of these $i$-quantum hyperalgebras are  highest weight categories.
However, since the decomposition matrices of $\sS^*(n,r)$, for $*\in\{\scB,\scC\}$, are unitriangular by Theorem \ref{unitri}, it 
seems to indicate that the representation category of $U_{\sZ}^{\jmath}(n)_k$ or $U_{\sZ}^{\imath}(n)_k$ has a sort of connection to a highest weight category. 

\begin{rems}(1) Lai, Nakano and Xiang considered the representation theory of $\sS^*(n,r)_k$ over a {\it field} $k$. They used a coordinate algebra approach, giving a natural polynomial representation theory as discussed in \cite{Gr80}. In particular, under a certain invertibility condition, it has been shown in \cite[Cors.~5.2.1\&6.1.1)]{LNX20} that these algebras are cellular and quasi-hereditary.

(2) Theorem \ref{integral} provides a method for investigating the representation theory of $U_{\sZ}^*(n)_k$ ($*\in\{\jmath,\imath\}$) arising from those of $\sS^*(n,r)_k$. A representation of $U_{\sZ}^*(n)_k$ obtained by inflating a representation  of $\sS^*(n,r)_k$ is also called a {\it polynomial representation} of $U_{\sZ}^*(n)_k$. The theorem establishes a link between polynomial representations of $U_{\sZ}^*(n)_k$ and representations of finite orthogonal/symplectic groups. The first step toward understanding the {\it modular} representation theory of $U_{\sZ}^*(n)_k$ would be the development of the polynomial representation theory.
\end{rems}

%Finally, by Lemma \ref{extend Iwahori} and Theorem \ref{integral}, we now may link representations of finite orthogonal/symplectic groups with those of the $i$-quantum groups $\bfU^{\jmath}(n)_k$ and $\bfU^{\imath}(n)_k$.

%Let $(K,\mathcal O, k)$ be a DVR local system and assume $q^{-\frac12}\in \mathcal O$. Let $U^{\imath*}_K$ (resp., $U^{\imath*}_{\mathcal O}$, $U^{\imath*}_k$) be the algebras obtained by specialising $t$ to $q^{-\frac12}$. Then, for $R\in\{\mathcal O,k\}$, we have algebra epimorphisms
%$$U^{\iB}_R\longrightarrow \sS^\scB(n,r)_R\cong\End_{RG(r_\scB,q)}\bigg(\bigoplus_{\la\in\Lambda_\scB(n,r)}\text{Ind}_{P(r_\scB,q)_{\la}}^{G(r_\scB,q)}1_R\bigg).$$
%Hence, we may link  representations of $U^\iB_R$ and $U^\iC_R$ with those of finite orthogonal/symplectic groups. 
%Then there is a subjectve map from $\bfU^{\imath*}_K$ to $\sS^*(n+1,r)_K$ defined by $\sH(B_r)_K$ by specialising $t$ to $q^{\frac12}$.

\section{Open problems.} We formulate several open problems inspired by this monograph.  We first recall some notation.
%\medskip

Let $G$ be a finite group of Lie type with standard finite Coxeter system $(W,S,L)$.  We maintain the notation of Section 5.2 (and so of Section 4.1, too).  Moreover,
 let $(K,\sO,k)$ be a triple as in Section 5.0 with $\sO$ sufficiently large
 to define the decomposition matrices
 $D=D_\sO$ and $D^+=D^+_\sO$ of  $A_\sO$ and $A^+_\sO$, respectively.
Recall that $r$ is the characteristic of $k$, and we consider representations of $kG$ in the cases where  $r$ does not equal the defining characteristic $p$ of $\bG$.  We also assume in this section that $r$ is good.
Using Theorem \ref{12.1a}(c), we see that the rank of $D$ is the number of modular irreducible $A_k$-modules. 

\medskip\noindent{\bf Problem 1.} What is the rank of $D$ in terms of $r$ and $W$?
\vskip.3in

The rank of $D$ always satisfies 
$$
\rank(D)\leq\rank(D^+)=|\irr(\mathbb Q W)|,$$
by Theorem \ref{12.1a}(b),(d). 
When equality holds, it also follows that $D=D^+$. In particular, $D^+$ is a submatrix of the
decomposition matrix of $G$ in this case, at least when the center of $\bG$ is connected, by Theorem \ref{12.3}. Notice that, in any given case, the rank
can be determined as the number of nonzero bilinear forms in the standardly based algebra $A_k$. See Lemma \ref{La1}.

\medskip\noindent
{\bf Problem 2.} Is the decomposition matrix $D^+=D^+_\sO$ of $A^+=A^+_\sO$ a submatrix of the decomposition matrix of $\sO G$?

\medskip 
This is related to the realization problem of the modules $X_\omega$ ($\omega\in\Omega'$) constructed in Theorem \ref{relative inj}.
A positive answer to Problem 2 would determine many decomposition numbers of $G$ associated to complex irreducible characters in the permutation module $\ind^{G}
_{B(q)}{\mathbb C}.$ These characters are all ``unipotent," and much is known about the remaining unipotent characters from Lusztig's theory of ``families" and how they are parametrized.  See \cite[12.3, 12.4, 13.2, 13.8,
13.9]{C85}.

It is a consequence of Lusztig's theory that each irreducible complex character in the ``unipotent series" lies in the same ``series" as one of the irreducible complex characters in the paragraph above.  A great deal is known about characters in a series and characters of the Weyl group, see
\cite[p. 191]{Lus93} and \cite[\S\S12.3-12.4]{C85}. 

%The connection of Deligne-Lusztig theory with cell theory is partly motivated the early formulation of Conjecture 7.1 in this monograph. 

\medskip\noindent
{\bf Problem 3.}  Assume that the center of $G$ is connected and recall that $r=$ characteristic of $k$ is good. Is  $D$ a submatrix of a square unitriangular  decomposition matrix $D^{++}=D^{++}_\sO$ (for, say, an $\sO$-algebra $A^{++}_\sO$)  with columns indexed by the set of
irreducible unipotent characters? 

\medskip

We would hope that $D^{++}$ would be the decomposition matrix arising from a stratified  $\sZ$-algebra
$A^{++}$ in the spirit of this monograph, and that it is a submatrix of the decomposition matrix of $\sO G$.

 Related to this problem is the Geck-Hiss conjecture \cite[Conj. 4.5.2]{GJ11} (see also \cite[Conj. 3.4]{GH97}) that the decomposition
matrix associated to all the unipotent characters is square unitriangular (when  $r$ is good). So, an informal version of Problem 3 might be: Could the now known validity of our Conjecture \ref{conjecture} (Theorem \ref{4.4}) above lead to progress on the Geck-Hiss conjecture? Theorem \ref{12.3} may be viewed as a step in this direction.  
Moreover, if the Geck-Hiss conjecture could be proved in this way, then 
 it would be natural to investigate if Theorem \ref{4.4} might be used, in the general case (where bad primes are allowed), to formulate some kind of approximate analog (of that conjecture).

\vskip.5in

%\newpage

% \input{Chapters/Chapter6.tex}

\appendix

\numberwithin{equation}{chapter}
\numberwithin{thm}{chapter}

\chapter[Projective coverings for Steinberg modules]{Projective coverings for Steinberg modules\\
 by Veronica Shalotenko}%\vspace{-2ex}
%\begin{center} 
%\end{center}

\medskip%\noindent
{\bf A1. Preliminaries.} 
Let $p$ be a prime number and let $\bG$ be a connected reductive algebraic group over the algebraic closure $\bar{\mathbb{F}}_p$ of the finite field $\mathbb{F}_p$. If $F$ is an endomorphism of $G$ such that some power of $F$ is a (standard) Frobenius morphism, then the fixed-point subgroup $G={\bG}^F$ of $\bG$ under the action of $F$ is a finite group of Lie type. (Finite groups of Lie type are discussed in detail in [1].) We will assume that the center of $\bG$ is connected. (We will use this assumption when describing the Gelfand-Graev module for $G$ below.) The finite group of Lie type $G$ is equipped with a BN-pair which is split in the sense of [1], \S2.5, i.e. this BN-pair satisfies the additional axioms (a) $B=UT$, where $T=B \cap N$ and $U$ is the largest normal $p$-subgroup of $B$, and (b) $\cap_{n \in N}~ nBn^{-1}=T$. Let $(W,S)$ be the Coxeter system corresponding to the BN-pair structure on $G$, where $W=N/T$. There is a length function $l$ on $W$, where for any $w \in W$, $l(w)$ is the minimal length of an expression of $w$ in the generators $s \in S$. \\

\noindent
\ul{Harish-Chandra induction and restriction.}
We will briefly describe several key aspects of modular Harish-Chandra theory; for a detailed treatment, see [5], \S4.2. In our description of Harish-Chandra induction and restriction, we will assume that $k$ is an algebraically closed field of characteristic $r \geq 0$, $r \neq p$ (thus, $k$ may be a field of characteristic 0). \\

Let $P$ be a parabolic subgroup of $G$ and let $L$ be a Levi complement in $P$; in this case, $P=U_P \rtimes L$. Let $kL$-mod denote the category of (finite-dimensional) left $kL$-modules, and let $kG$-mod denote the category of (finite-dimensional) left $kG$-modules. There is a Harish-Chandra induction\index{Harish-Chandra induction} functor
\begin{equation*}
R_{L \subseteq P}^G: kL\text{-mod} \to kG\text{-mod},
\end{equation*}
defined by $R_{L \subseteq P}^G(X)=\Ind_P^G( \tilde{X})$ for all $X \in kL$-mod, where $\tilde{X}$ denotes the inflation of $X$ from $L$ to $P$ via the surjective homomorphism $P \to L$ with kernel $U_P$. \\

There is also a Harish-Chandra restriction\index{Harish-Chandra restriction} functor
\begin{equation*}
^*{R}^G_{L \subseteq P}: kG\text{-mod} \to kL\text{-mod}.
\end{equation*}
Given a $kG$-module $Y$, $^*{R}^G_{L \subseteq P} (Y)=Y^{U_P}$ (which has the structure of a $kL$-module since $U_P$ is a normal subgroup of $P$). \\

Howlett and Lehrer [6] and Dipper and Du [3] proved that the Harish-Chandra induction and restriction functors $R_{L \subseteq P}^G$ and $^*{R}^G_{L \subseteq P}$ are independent of the choice of parabolic subgroup $P$ containing $L$. Thus, we will omit the parabolic subgroup $P$ and write $R_L^G$ and $^*{R}^G_L$ for the Harish-Chandra induction and restriction functors. \\

We will use the following properties of Harish-Chandra induction and restriction in the proof of Theorem \ref{kerfactors}. \\

\medskip\noindent
\ul{Adjointness.} 
 For
 any Levi subgroup $L$, $R_L^G$ and $^*{R}^G_L$ are exact. The functors $R_L^G$ and  $^*{R}^G_L$ are each other's two-sided adjoints. \\ 

\medskip\noindent
\ul{Mackey decomposition.}\index{Mackey decomposition}
Let $L$ be a Levi subgroup of $G$. Given an element $n \in N$ and a (left) $kL$-module $X$, we can define a $k({^{n}{L}})$-module structure on $X$ by setting $nln^{-1}.x=l.x$, for any $l \in L$ and $x \in X$. The resulting $k({^{n}{L}})$-module will be denoted by ${^{n}{X}}$. \\

Now, let $L$ and $M$ be two Levi subgroups of $G$, and let $P$ and $Q$ be parabolic subgroups containing $L$ and $M$, respectively. Let $X$ be a $kL$-module, and let $D(Q,P)$ denote a full set of $(Q,P)$-double coset representatives in $G$.  The Mackey formula provides the following direct sum decomposition of the $kM$-module $^{*}{R}^G_M(R_L^G(X))$:
$$
^{*}{R}^G_M(R_L^G(X)) \cong \underset{n \in D(Q,P)}{\bigoplus} R^M_{{^{n}{L}} \cap M} ({^{*}{R}}^{^{n}{L}}_{{^{n}{L}} \cap M}({^{n}{X}})).
$$

\medskip
%\noindent
{\bf A2. The Steinberg module and the Gelfand-Graev module.}\label{background} Let $G$ be a finite group of Lie type defined in characteristic $p$, and let $R$ be a commutative ring (with unity).

%\medskip
\noindent
\ul{The Steinberg module.} 
 We define an element $\mathfrak{e} \in RG$ by $\mathfrak{e}=\underset{w \in W}{\sum} (-1)^{l(w)} n_w \b$, where $\b=\underset{b \in B}{\sum} b$ and $n_w \in N_G(T)$ is a representative of the coset $w \in W =N/T \leq N_G(T)/T$. The left $RG$-module $\St_R=RG \mathfrak{e}$ is the Steinberg module\index{Steinberg module} for $G$ over $R$ [7]. \\

\medskip\noindent
\ul{An $r$-modular system.}
For the remainder of this appendix, $k$ will be an algebraically closed field of characteristic $r>0$, $r \neq p$. We will work with an $r$-modular system $(\sO,K,k)$, where $\sO$ is a complete discrete valuation ring with residue field $k$ and $K$ is the fraction field of $\sO$. We will assume that $\char(K)=0$. We will call an $\sO G$-module $M$ a lattice if $M$ is finitely generated and free over $\sO$. \\

Let $\St_\sO$=$\sO G \e$ be the Steinberg module for $G$ over $\sO$ ($\St_\sO$ is finitely generated and free over $\sO$, so it is a lattice). Then,  $K \otimes_\sO \St_\sO  \cong \St_K$ is an irreducible $KG$-module [7],  Theorem 3. \\

\medskip\noindent
\ul{The Gelfand-Graev module.} 
We will construct the Gelfand-Graev module for $G$ following [4], \S4. Let $\sigma: U \to K^\times$ be a group homomorphism, and define $\u_\sigma := \underset{u \in U}{\sum} \sigma(u)u \in KG$. Since $U$ is a $p$-subgroup of $G$, $r \centernot\mid |U|$ and $\sigma(u) \in \sO^\times$, for all $u \in U$, which means $\u_\sigma \in \sO G$. We have
$$\u_\sigma^2=\underset{u,u' \in U}{\sum} \Big(\sigma(u)u\Big)\Big(\sigma(u')u'\Big)=\underset{u \in U}{\sum} \underset{u' \in U}{\sum} \sigma(uu')uu'=\underset{u \in U}{\sum} \u_\sigma =|U| \u_\sigma.$$
Since $|U|$ is a unit in $\sO$, $\frac{1}{|U|} \u_\sigma$ is an idempotent in $\sO G$ and the $\sO G$-lattice $\Gamma_\sigma := \sO G \u_\sigma$ is projective. \\

We will now assume that $\sigma: U \to K^\times$ is a fixed regular character (defined in [4],  \S4.3). As above, we assume that $G=\bG^F$ and that the algebraic group ${\bG}$ has a connected center. In this case, the projective $\sO G$-lattice $\Gamma_\sigma$ is called a Gelfand-Graev module for $G$.\index{Gelfand-Graev module} Since the center of ${\bG}$ is connected, the Gelfand-Graev module is unique up to isomorphism [1], Prop. 8.1.2. We will denote $K \otimes_\sO \Gamma_\sigma$ (the Gelfand-Graev module for $G$ over $K$) by $K\Gamma_\sigma$; we will denote $k \otimes_\sO \Gamma_\sigma$ (the Gelfand-Graev module for $G$ over $k$) by $\overline{\Gamma}_\sigma$. Since $\Gamma_\sigma$ is a projective $\sO G$-module, $\overline{\Gamma}_\sigma$ is a projective $kG$-module. 

\medskip%\noindent
{\bf  A3. Properties of the Steinberg and Gelfand-Graev modules.}
%\smallskip
In the proof of Theorem \ref{kerfactors}, we will use several facts about the Steinberg and Gelfand-Graev modules which can be found in [4]. We begin with a summary of the relevant results. \\

By [4], \S4.1 ,
\begin{equation}\label{homgst}
\text{dim} \; \Hom_{KG}(K\Gamma_\sigma, \St_K)=1. \\
\end{equation}
Since $\St_K$ is irreducible, it follows that there is a surjective $KG$-module homomorphism from $K\Gamma_\sigma$ to $\St_K$. \\

A similar result holds over $\sO$. In [4], Prop. 4.2, Geck defines an $\sO G$-lattice $\rho_\sigma(\mathscr{S}_\sigma)$ in $\St_K$ (which we will denote by $\mathscr{S}_\sigma$ to simplify the notation) such that $K \otimes_\sO \mathscr{S}_\sigma \cong \St_K$. The lattice $\mathscr{S}_\sigma$ is a homomorphic image of the Gelfand-Graev lattice $\Gamma_\sigma$ [4],  Prop. 4.2. Thus, there is a surjective homomorphism $\Gamma_\sigma \twoheadrightarrow \mathscr{S}_\sigma$, which, upon reduction to $k$, yields a surjective homomorphism 
\begin{equation}\label{surjhom}
\overline{\Gamma}_\sigma \twoheadrightarrow \overline{\mathscr{S}}_\sigma
\end{equation}
from the projective $kG$-module $\overline{\Gamma}_\sigma$ to the reduction $\overline{\mathscr{S}}_\sigma$ of the lattice $\mathscr{S}_\sigma$. (By [4]. Prop. 4.7, there are many groups $G$ and fields $k$ for which $\overline{\mathscr{S}}_\sigma \cong \St_k$.) \\

Statements analogous to (\ref{homgst}) and (\ref{surjhom}) hold for any Levi subgroup $L$ of $G$. Since a Levi subgroup $L$ is itself a finite group of Lie type, there is a Steinberg module $\St_K^L$ for $L$ over $K$. By [4], \S4.3, there is also a Gelfand-Graev lattice $\Gamma_\sigma^L$ for $L$ (obtained using a restriction of the regular character $\sigma$). Thus, we have a Gelfand-Graev module $K\Gamma_\sigma^L$ for $L$ over $K$. By (\ref{homgst}), $\dim_{KL}(K\Gamma_\sigma^L,\St_K^L)=1$; by (\ref{surjhom}), there is an $\sO L$-lattice $\mathscr{S}_\sigma^L$ in $\St_K^L$ with $K \otimes_\sO \mathscr{S}_\sigma^L \cong \St_K^L$ and a surjective homomorphism $\overline{\Gamma}_\sigma^L \twoheadrightarrow \overline{\mathscr{S}}_\sigma^L$. \\

If $L$ and $M$ are Levi subgroups of $G$ with $L \subseteq M$, we have the following isomorphisms of $KL$-modules (these isomorphisms are recorded in [4], Lemma 4.4):

\begin{equation}\label{gelfandres}
^{*}{R}^M_L(K\Gamma_\sigma^M) \cong K\Gamma_\sigma^L,
\end{equation}

\begin{equation}\label{steinbergres}
^{*}{R}^M_L(\St_K^M) \cong \St_K^L, \text{ and}
\end{equation}

\begin{equation}\label{independence}
\text{If }L={^{n}{L'}}, \text{ for some } n \in N, \text{ then } ^{n}{K\Gamma_\sigma^{L'}} \cong K\Gamma_\sigma^{L} \text{ and } ^{n}{\St_K^{L'}} \cong \St_K^L. 
\end{equation}

\medskip
%\noindent
{\bf A4. Main result.}
As above, let $G$ be a finite group of Lie type defined in characteristic $p$ (arising from an algebraic group with a connected center). Let $k$ be an algebraically closed field of characteristic $r>0$, $r \neq p$, and let $(\sO,K,k)$ be an $r$-modular system, where $\sO$ is a complete discrete valuation ring and $K$ (the fraction field of $\sO$) has characteristic 0. Let $(W,S)$ be the Coxeter system associated to the $BN$-pair of $G$. Given a subset $J \subseteq S$, let $L_J$ denote the standard Levi subgroup of $G$ corresponding to $J$. By (\ref{surjhom}), for all $J \subseteq S$, there is a surjective homomorphism 
$$\overline{\Gamma}_\sigma^{L_J} \twoheadrightarrow \overline{\mathscr{S}}_\sigma^{L_J}$$
from the projective $kL_J$-module $\overline{\Gamma}_\sigma^{L_J}$ to the $kL_J$-module $\overline{\mathscr{S}}_\sigma^{L_J}$, where $\mathscr{S}_\sigma^{L_J}$ is an $\sO L_J$-lattice in $\St_K^{L_J}$ with the property that $K \otimes_\sO \mathscr{S}_\sigma^{L_J} \cong \St_K^{L_J}$. Applying the exact Harish-Chandra induction functor $R_{L_J}^G$, we obtain a surjective homomorphism 
$$R_{L_J}^G(\overline{\Gamma}_\sigma^{L_J}) \twoheadrightarrow R_{L_J}^G(\overline{\mathscr{S}}_\sigma^{L_J}).$$
 Thus, taking direct sums over all $J \subseteq S$, we obtain a surjective homomorphism
$$\beta_k: \underset{J \subseteq S}{\oplus} R_{L_J}^G(\overline{\Gamma}_\sigma^{L_J}) \twoheadrightarrow \underset{J \subseteq S}{\oplus} R_{L_J}^G(\overline{\mathscr{S}}_\sigma^{L_J})$$
from the projective $kG$-module $\underset{J \subseteq S}{\oplus} R_{L_J}^G(\overline{\Gamma}_\sigma^{L_J})$ to the $kG$-module $\underset{J \subseteq S}{\oplus} R_{L_J}^G(\overline{\mathscr{S}}_\sigma^{L_J})$. \\

We can construct an analogous homomorphism $\beta_K$ over the characteristic 0 field $K$. For any $J \subseteq S$, the surjective homomorphism $\overline{\Gamma}_\sigma^{L_J} \twoheadrightarrow \overline{\mathscr{S}}_\sigma^{L_J}$ of (\ref{surjhom}) is induced by a surjective homomorphism $\Gamma_\sigma^{L_J} \twoheadrightarrow \mathscr{S}_\sigma^{L_J}$ from the Gelfand-Graev lattice $\Gamma_\sigma^{L_J}$ to the lattice $\mathscr{S}_\sigma^{L_J}$ of $\St_K^{L_J}$. It follows that the surjective homomorphism $\beta_k$ defined above is induced by a surjective homomorphism
$$\beta: \underset{J \subseteq S}{\oplus} R_{L_J}^G(\Gamma_\sigma^{L_J}) \to R_{L_J}^G(\mathscr{S}_\sigma^{L_J})$$
of $\sO G$-lattices. Taking tensor products with $K$, we obtain a surjective homomorphism 
$$\beta_K: \underset{J \subseteq S}{\oplus} K \otimes_\sO R_{L_J}^G(\Gamma_\sigma^{L_J}) \twoheadrightarrow \underset{J \subseteq S}{\oplus} K \otimes_\sO R_{L_J}^G(\mathscr{S}_\sigma^{L_J})$$
of $KG$-modules. Since Harish-Chandra induction is ``compatible" with the $r$-modular system (See [4], \S 4), we have $K \otimes_\sO R_{L_J}^G(\Gamma_\sigma^{L_J}) \cong R_{L_J}^G(K\Gamma_\sigma^{L_J})$ and $K \otimes_\sO R_{L_J}^G(\mathscr{S}_\sigma^{L_J}) \cong R_{L_J}^G(K \otimes_\sO \mathscr{S}_\sigma^{L_J}) \cong R_{L_J}^G(\St_K^{L_J})$, for any $J \subseteq S$. Therefore, there is a surjective $KG$-module homomorphism
$$\beta_K: \underset{J \subseteq S}{\oplus} R_{L_J}^G(K\Gamma_\sigma^{L_J}) \twoheadrightarrow \underset{J \subseteq S}{\oplus} R_{L_J}^G(\St_K^{L_J}).$$

\begin{thm}\label{kerfactors}
If $\beta_K: \underset{J \subseteq S}{\oplus} R_{L_J}^G(K\Gamma_\sigma^{L_J}) \twoheadrightarrow \underset{J \subseteq S}{\oplus} R_{L_J}^G(\St_K^{L_J})$ is the surjective homomorphism described above, then the kernel of $\beta_K$ has no composition factors in common with $\underset{J \subseteq S}{\oplus} R_{L_J}^G(\St_K^{L_J})$.\footnote{A related open question is whether or not this result holds in positive characteristic.}
\end{thm}

\begin{proof}
We claim that it suffices to show that 
\begin{equation}\label{conds}
\dim \; \Hom_{KG}(\underset{J \subseteq S}{\oplus} R_{L_J}^G(K\Gamma_\sigma^J), \underset{J \subseteq S}{\oplus} R_{L_J}^G(\St_K^{L_J})) = \dim \; \Hom_{KG}(\underset{J \subseteq S}{\oplus} R_{L_J}^G(\St_K^{L_J}),\underset{J \subseteq S}{\oplus} R_{L_J}^G(\St_K^{L_J})).\footnote{Our use of this condition on the dimension of Hom spaces was inspired by a reading of [5], \S4.1. (A more general version of this condition appears in [5], Lemma 4.1.7.)} \\
\end{equation}
For suppose that the equality in (\ref{conds}) holds. Since $K$ is a field of characteristic 0, every short exact sequence of $KG$-modules splits and we have a direct sum decomposition
$$\underset{J \subseteq S}{\oplus} R_{L_J}^G(K\Gamma_\sigma^{L_J}) \cong \Ker(\beta_K) \oplus {\rm{Im}}(\beta_K) \cong \Ker(\beta_K) \oplus \Big( \underset{J \subseteq S}{\oplus} R_{L_J}^G(\St_K^{L_J})\Big)$$ 
(the homomorphism $\beta_K$ is surjective, so ${\rm{Im}}(\beta_K) \cong \underset{J \subseteq S}{\oplus} R_{L_J}^G(\St_K^{L_J})$). It follows that $ \Hom_{KG}(\underset{J \subseteq S}{\oplus} R_{L_J}^G(K\Gamma_\sigma^J), \underset{J \subseteq S}{\oplus} R_{L_J}^G(\St_K^{L_J}))$ decomposes as a direct sum of $$\Hom_{KG}(\Ker(\beta_K),  \underset{J \subseteq S}{\oplus} R_{L_J}^G(\St_K^{L_J})) \text{ and }\Hom_{KG}( \underset{J \subseteq S}{\oplus} R_{L_J}^G(\St_K^{L_J}),  \underset{J \subseteq S}{\oplus} R_{L_J}^G(\St_K^{L_J})).$$
Therefore, for equality to hold in (\ref{conds}), we must have 

$$\dim \;  \Hom_{KG}(\Ker(\beta_K),  \underset{J \subseteq S}{\oplus} R_{L_J}^G(\St_K^{L_J}))=0.$$
Since $\char(K)=0$ and all $KG$-modules are completely reducible, this is possible if and only if the $KG$-modules $\Ker(\beta_K)$ and $ \underset{J \subseteq S}{\oplus} R_{L_J}^G(\St_K^{L_J})$ have no composition factors in common. \\

We will now establish the equality
\begin{equation*}
\dim \; \Hom_{KG}(\underset{J \subseteq S}{\oplus} R_{L_J}^G(K\Gamma_\sigma^J), \underset{J \subseteq S}{\oplus} R_{L_J}^G(\St_K^{L_J})) = \dim \; \Hom_{KG}(\underset{J \subseteq S}{\oplus} R_{L_J}^G(\St_K^{L_J}),\underset{J \subseteq S}{\oplus} R_{L_J}^G(\St_K^{L_J})).
\end{equation*}
In fact, since
$$\Hom_{KG}(\underset{J \subseteq S}{\oplus} R_{L_J}^G(K\Gamma_\sigma^{L_J}), \underset{J \subseteq S}{\oplus} R_{L_J}^G(\St_K^{L_J})) \cong \underset{I,J \subseteq S}{\bigoplus} \Hom_{KG}(R_{L_I}^G(K\Gamma_\sigma^{L_I}), R_{L_J}^G(\St_K^{L_J})) \text{ and}$$
$$\Hom_{KG}(\underset{J \subseteq S}{\oplus} R_{L_J}^G(\St_K^{L_J}), \underset{J \subseteq S}{\oplus} R_{L_J}^G(\St_K^{L_J})) \cong \underset{I,J \subseteq S}{\bigoplus} \Hom_{KG}(R_{L_I}^G(\St_K^{L_I}), R_{L_J}^G(\St_K^{L_J})),$$
it is enough to show that
\begin{equation}\label{formula2}
\dim \; \Hom_{KG}(R_{L_I}^G(K\Gamma_\sigma^{L_I}), R_{L_J}^G(\St_K^{L_J})) = \dim \; \Hom_{KG}(R_{L_I}^G(\St_K^{L_I}), R_{L_J}^G(\St_K^{L_{J}})),
\end{equation}
for all $I, J \subseteq S$. \\

We will now compute $\dim \; \Hom_{KG}(R_{L_I}^G(K\Gamma_\sigma^{L_I}), R_{L_J}^G(\St_K^{L_J}))$. Since the Harish-Chandra restriction functor $^{*}{R}^G_{L_J}$ is adjoint to the Harish-Chandra induction functor $R_{L_J}^G$, we have
$$\Hom_{KG}(R_{L_I}^G(K\Gamma_\sigma^{L_I}), R_{L_J}^G(\St_K^{L_J})) \cong \Hom_{KL_J}({^{*}{R}}^G_{L_J}R_{L_I}^G( K\Gamma_\sigma^{L_I}), \St_K^{L_J}).$$
Let $P_J$ and $P_I$ denote the standard parabolic subgroups of $G$ containing $L_J$ and $L_I$, respectively, and let $D(P_J,P_I) \subseteq N$ be a full set of $(P_J,P_I)$-double coset representatives in $G$. By Mackey decomposition, we have
\begin{align*}
\Hom_{KL_J}({^{*}{R}}^G_{L_J}R_{L_I}^G( K\Gamma_\sigma^{L_I}), \St_K^{L_J}) &\cong \Hom_{KL_J}\Big( \underset{n \in D(P_J,P_I)}{\bigoplus} R_{{^{n}{L_I}} \cap L_J}^{L_J} {^{*}{R}}^{^{n}{L_I}}_{{^{n}{L_I}} \cap L_J}({^{n}{K\Gamma}}_\sigma^{L_I}),\St_K^{L_J}\Big) \\
&\cong \underset{n \in D(P_J,P_I)}{\bigoplus} \Hom_{KL_J}(R_{{^{n}{L_I}} \cap L_J}^{L_J} {^{*}{R}}^{^{n}{L_I}}_{{^{n}{L_I}} \cap L_J}({^{n}{K\Gamma}}_\sigma^{L_I}),\St_K^{L_J}).
\end{align*}
Now, by (\ref{gelfandres}) and (\ref{independence}), the $K({^{n}{L_I}} \cap L_J)$-module ${^{*}{R}}^{^{n}{L_I}}_{{^{n}{L_I}} \cap L_J}({^{n}{K\Gamma}}_\sigma^{L_I})$ is the Gelfand-Graev module $K\Gamma_\sigma^{{^{n}{L_I}} \cap L_J}$, for ${^{n}{L_I}} \cap L_J$ over $K$. Thus, continuing the computation above (and, using the adjointness of $R_{{^{n}{L_I}} \cap L_J}^{L_J}$ and $^{*}{R}_{{^{n}{L_I}} \cap L_J}^{L_J}$), we have
\begin{align*}
&\underset{n \in D(P_J,P_I)}{\bigoplus} \Hom_{KL_J}(R_{{^{n}{L_I}} \cap L_J}^{L_J} {^{*}{R}}^{^{n}{L_I}}_{{^{n}{L_I}} \cap L_J}({^{n}{K\Gamma}}_\sigma^{L_I}),\St_K^{L_J}) \\
&\cong
\underset{n \in D(P_J,P_I)}{\bigoplus} \Hom_{KL_J}(R_{{^{n}{L_I}} \cap L_J}^{L_J}(K\Gamma_\sigma^{{^{n}{L_I}} \cap L_J}),\St_K^{L_J}) \\
&\cong \underset{n \in D(P_J,P_I)}{\bigoplus} \Hom_{K({^{n}{L_I}} \cap L_J)}(K\Gamma_\sigma^{{^{n}{L_I}} \cap L_J},  {^{*}{R}}_{{^{n}{L_I}} \cap L_J}^{L_J}(\St_K^{L_J})).
\end{align*}
By (\ref{steinbergres}), ${^{*}{R}}_{{^{n}{L_I}} \cap L_J}^{L_J}(\St_K^{L_J})$ is the Steinberg module $\St_K^{{^{n}{L_I}} \cap L_J}$, for $^{n}{L_I} \cap L_J$ over $K$. Therefore,
$$\Hom_{K({^{n}{L_I}} \cap L_J)}(K\Gamma_\sigma^{{^{n}{L_I}} \cap L_J},  {^{*}{R}}_{{^{n}{L_I}} \cap L_J}^{L_J}(\St_K^{L_J})) \cong  \Hom_{K({^{n}{L_I}} \cap L_J)}(K\Gamma_\sigma^{{^{n}{L_I}} \cap L_J}, \St_K^{{^{n}{L_I}} \cap L_J}),$$
which has dimension 1 by (\ref{homgst}). Thus,
$$
\dim \; \Hom_{KG}(R_{L_I}^G(K\Gamma_\sigma^{L_I}), R_{L_J}^G(\St_K^{L_J})) = \underset{n \in D(P_J,P_I)}{\sum} 1 = |D(P_J,P_I)|.
$$

We can similarly compute $\dim \; \Hom_{KG}(R_{L_I}^G(\St_K^{L_I}), R_{L_J}^G(\St_K^{L_{J}}))$. We have
\begin{align*}
\Hom_{KG}(R_{L_I}^G(\St_K^{L_I}), R_{L_J}^G(\St_K^{L_{J}})) &\cong \Hom_{KL_J}({^{*}{R}}^G_{L_J}R_{L_I}^G( \St_K^{L_I}), \St_K^{L_J}) \\
&\cong \Hom_{KL_J}\Big( \underset{n \in D(P_J,P_I)}{\bigoplus} R_{{^{n}{L_I}} \cap L_J}^{L_J} {^{*}{R}}^{^{n}{L_I}}_{{^{n}{L_I}} \cap L_J}({^{n}{\St_K}}^{L_I}),\St_K^{L_J}\Big) \\
&\cong \underset{n \in D(P_J,P_I)}{\bigoplus} \Hom_{KL_J}(R_{{^{n}{L_I}} \cap L_J}^{L_J} {^{*}{R}}^{^{n}{L_I}}_{{^{n}{L_I}} \cap L_J}({^{n}{\St_K}}^{L_I}),\St_K^{L_J}) \\
&\cong
\underset{n \in D(P_J,P_I)}{\bigoplus} \Hom_{KL_J}(R_{{^{n}{L_I}} \cap L_J}^{L_J}(\St_K^{{^{n}{L_I}} \cap L_J}),\St_K^{L_J}) \\
&\cong \underset{n \in D(P_J,P_I)}{\bigoplus} \Hom_{K({^{n}{L_I}} \cap L_J)}(\St_K^{{^{n}{L_I}} \cap L_J},  {^{*}{R}}_{{^{n}{L_I}} \cap L_J}^{L_J}(\St_K^{L_J})) \\
&\cong \underset{n \in D(P_J,P_I)}{\bigoplus}  \Hom_{K({^{n}{L_I}} \cap L_J)}(\St_K^{{^{n}{L_I}} \cap L_J}, \St_K^{{^{n}{L_I}} \cap L_J}).
\end{align*}

Since each Steinberg module $\St_K^{{^{n}{L_I}} \cap L_J}$ is absolutely irreducible (see [4], Theorem 2.1), we have 
$$\dim \;  \Hom_{K({^{n}{L_I}} \cap L_J)}(\St_K^{{^{n}{L_I}} \cap L_J}, \St_K^{{^{n}{L_I}} \cap L_J}) =1,$$
for all $n \in D(P_J,P_I)$. Therefore, $\dim \; \Hom_{KG}(R_{L_I}^G(\St_K^{L_I}), R_{L_J}^G(\St_K^{L_{J}}))=|D(P_J,P_I)|$, proving that the equality in (\ref{formula2}) holds.
\end{proof}

\medskip\medskip
\begin{center} References\end{center}

\medskip\noindent
[1] R.W. Carter,
\textit{Finite Groups of Lie Type: Conjugacy Classes and Complex Characters}, 
Wiley, New York, 1985, reprinted 1993 as Wiley Classics Library edition.

\medskip\noindent
[2] R. Dipper,
{On quotients of Hom-functors and representations of finite general linear groups I.}
J. Algebra 130 (1990), 235--259.

\medskip\noindent
[3] R. Dipper and J. Du,
{Harish-Chandra vertices.}
J. Reine Angew. Math. 437 (1993), 101--130.

\medskip\noindent
[4] M. Geck,
\textit{On the modular composition factors of the Steinberg representation.}
J. Algebra 475 (2017), 370--391.

 \medskip\noindent
[5]  M. Geck and N. Jacon,
 \textit{Representations of Hecke Algebras at Roots of Unity.}
 Algebra and Applications, vol. 15, Springer-Verlag, 2011.

\medskip\noindent
[6] R. B. Howlett and G. I. Lehrer,
{On Harish-Chandra induction for modules of Levi subgroups.}
J. Algebra 165 (1994), 172--183. 

 \medskip\noindent
[7] R. Steinberg, {Prime power representations of finite linear groups II.}
Canad. J. Math. 9 (1957), 347--351.

\chapter{Height functions.} We collect together some terminology on height functions and related matters which
are used repeatedly. Given a non-empty finite set $\Lambda$, a {\it height function}\index{height function} on $\Lambda$ is a function
$\htt:\Lambda\to\mathbb Z$.  Often, but not always, the set $\Lambda$ indexes the irreducible modules for 
an algebra.  This is the case in the following.

\medskip
%\noindent
{\bf B1. Height filtrations.} We will consider here decreasing filtrations associated with a negative-valued height function as set in the following example.\index{height filtration}
\begin{example}\label{example1} Let $\scrK$ be a Noetherian domain with fraction field $K$, and let $H$ be a finite and  torsion-free $\scrK$-algebra. Let $\scrA$ be the additive category of $H$-finite, $\scrK$-torsion-free left $H$-modules. Assume that $H_K$ is semisimple, and let $\Lambda$ be a set indexing the distinct irreducible
left $H_K$-modules:  $\lambda\in\Lambda$ corresponds to an irreducible $H_K$-module $L_K(\lambda)$. Fix a height function $\htt$ on $\Lambda$ which takes negative values. (The assumption that $\htt$ only takes negative values
is a convenience; see Remark \ref{mystery} immediately below.) Given a  module $M\in\scrA$, $\htt$ defines a natural (decreasing) filtration $M=M_0\supseteq M_1\supseteq\cdots\supseteq
M_n=0$ of $M$ by $\scrA$-submodules: let $(M_K)_h$ be the $H_K$-submodule of $M_K$ which is the
sum of 
all irreducible submodules isomorphic to $L_K(\lambda)$ with $\htt(\lambda)< -h$. The decreasing filtration of $M$ is obtained by putting $M_h:=M\cap (M_K)_h$, for all $h$. Notice that $(M_{h-1}/M_{h})_K$ is a direct sum of irreducible modules
$L_K(\lambda)$ with $\htt(\lambda)=-h$. Also, a short exact sequence $0\to M\to N\to P\to 0$ in $\scrA$ gives, for
any integer $h$, an exact sequence $0\to M_h\to N_h\to P_h$; see \cite[Prop. 3.3(a)]{DPS17}. If $\scrE$ denotes the
category (in the sense of footnote \ref{exact sub} of Chapter 1) of those exact sequences $0\to M\to N\to P\to 0$ such that, for any $h$,  $0\to M_h\to N_h\to P_h\to 0$ {\it is} exact, then
$(\scrA,\scrE)$ is an example of an exact category. This fact may be proved in a way similar to the proof of
\cite[Th. 3.6]{DPS17}, though the latter is given in a right-module context, using superscript notation.\end{example}

 For more important examples of exact categories in this spirit, see \cite[\S3]{DPS17}
as well as earlier sections of this monograph. While $(\scrA,\scrE)$ depends on the fixed height function $\htt$, we prefer
not to include $\htt$ as part of its name. Also, replacing $\htt$ by $\htt +c$, $c\in\mathbb Z$, does not change the
filtrations on $\scrA$ or the exact category $(\scrA,\scrE)$.

\begin{rem}\label{mystery} Suppose we drop the assumption that $\htt$ takes only negative values. Then, one should replace
the filtration $M=M_0\supseteq M_1\supseteq\cdots\supseteq M_n=0$ with a filtration $M=M_{-m}\supseteq
\cdots \supseteq M_n=0$ for sufficiently large (positive) $n$ and $m$. This can be done uniformly for all $M\in\scrA$, provided all
$\htt(\lambda)$ lie in the interval $(-n,m]$.  If we modify such a height function by adding $m'$ with $m'<-m$
to all its values and reindex $M_i$, we return to the setting assumed in Example \ref{example1}, and it is evident that the two height functions lead to canonically isomorphic categories $(\scrA,\scrE)$. The extra flexibility, however, is useful: 

Suppose we take a possibly positive-valued height function $\htt:\Lambda\to\mathbb Z$, and construct $M^h$ according to the recipe
$M^h=M\cap M_K^{\leq h}$ in Theorem \ref{A1}, by giving (in the category of right $H$-modules)
a recipe for constructing an exact category through {\it increasing filtrations} 
$$\cdots\subseteq M^{h-1}\subseteq M^h\subseteq \cdots,$$
 where the irreducible (right)
$H_K$-modules appearing in $(M^h/M^{h-1})_K$ are indexed by $\lambda\in\Lambda$ with $\htt(\lambda)= h$.  Obviously, such an increasing filtration  canonically identifies (for any $M$) with a decreasing filtration 
\begin{equation}\label{latter}\cdots\subseteq M_{-(h-1)}\subseteq M_{-h}\subseteq\cdots,\end{equation}
where $M_{-h}:=M^h$ for every $h$. Moreover, a second height function can be given (assigning $\lambda\in\Lambda$ to $\htt(\lambda)-1$) so that the filtration of $M$ constructed from Example \ref{example1} is precisely that given in (\ref{latter}), which is, in turn, ``the same" as that in Theorem \ref{A1} (using $\htt$).
 This is true at least up to
the left-right-module discrepancy, but that disappears when we use $H^{\text{op}}$, instead of $H$, in one of the cases. 
\end{rem}

 The passage here from increasing to decreasing filtrations via the convention $M_{-h}=M^h$ is quite natural but does not, in general, lead to a negative-valued height function without further adjustment of the subscripts. 

The notational scheme in Example \ref{example1} has, however, been chosen to behave well under a particularly important passage from left to right modules,  namely, the passage from left to right arising from the contravariant functor $(-)^*=\Hom_\sZ(-,\sZ)$ --- the duality functor; see Section 2.2 for more details. 
%This will be explored further in Appendix C.

\medskip%\noindent
{\bf B2. Height functions compatible with a preorder.} \index{preorder}
 In practice, though some useful set $\Lambda$ may be in the background (see Remark \ref{A7}),  we also consider a set $\Omega$ with a preorder $\leq$.

 A relation $\leq$ on a set $\Omega$ that is reflexive and transitive is called a {\it preorder}. \index{preorder} We will call $\Omega=(\Omega,\leq)$ a {\it quasi-poset}\index{quasi-poset} in the book. 
 
 For such a preorder $\leq$, define an equivalence relation $\sim$ on $\Omega$ by declaring that
$\omega\sim \nu \iff \omega\leq \nu$ and $\nu\leq \omega$, and a strict version $<$ of $\leq$ by setting $\omega<\nu$ provided $\omega\leq\nu$ and $\omega\not\sim\nu$ (or $\omega\not\geq\nu$).
 
If $\overline\Omega=\Omega/\sim$ denotes the set of equivalence classes, then $\leq$ induces a partial order, denoted again by $\leq$, on $\overline\Omega$.
We  call $(\overline \Omega,\leq)$ the {\it poset associated with} $(\Omega,\leq)$.\index{quasi-poset! poset associated with $\sim$} Thus, for any $\la,\mu\in\Omega$, if $\overline\lambda$ denotes the equivalence class containing $\lambda$, then we have $\overline\omega<\overline\nu\iff\omega<\nu$.

 % That is, $\Omega$ is a quasi-poset:\index{quasi-poset} a set with a preorder (i.e., transitive and reflexive relation) $\leq$. Define an equivalence relation $\sim$ on $\Omega$ by declaring that $\lambda\sim \mu \iff \lambda\leq \mu$ and $\mu\leq \lambda$.  Given $\omega\in\Omega$, let $\bar\omega$ be the equivalence class containing $\omega$, and let $\bar\Omega$ be the set of equivalence classes. Clearly,
%$\leq$ defines a poset structure on $\bar\Omega$.\index{quasi-poset! poset associated with $\sim$} Also, for $\omega,\nu\in\Omega$, define
%$\omega<\nu$ provided $\omega\leq\nu$ and $\omega\not\sim\nu$. Thus, $\omega<\nu$ if and only
%if $\bar\omega<\bar\nu$.
%For $\lambda\in\Lambda$, $[\lambda]$ denotes
%the equivalence class containing $\lambda$. 

By definition, a height function $\htt$ on $\Omega$ is {\it compatible}
with a preorder $\leq$ on $\Omega$ provided that 

\begin{equation}\label{preorder}
\begin{cases} (1)\quad \omega<\nu\implies \htt(\omega)<\htt(\nu), \quad{\text{\rm and}} \\ 
 (2) \quad
\omega\sim\nu\implies\htt(\omega)=\htt(\nu). \end{cases}\end{equation}
Note that if $\htt$ is compatible with $\leq$, then the function $-\htt:\Omega\to\mathbb Z$
is compatible with $\leq^\op$.

Given a preorder $\leq$ on $\Omega$, there are many choices for a compatible height function $\htt$. For
example, the {\it standard height function} $\htt_{st}$ is defined by setting $\htt_{st}(\omega)$ to be the maximal length
of a proper chain $\omega_0<\omega_1<\cdots< \omega_n=\omega$. It  is compatible
with the preorder on $\Omega$. 

\begin{defn}\label{dominates}Let $\leq$ and $\leq'$ be two preorders on the same finite set $\Omega$. We say that $\leq$
{\it dominates} $\leq'$ provided $\omega\leq\nu\implies\omega\leq'\nu$, \index{preorder! dominates} for all $\omega,\nu\in\Omega$. If, in addition, $\omega<\nu\implies\omega<'\nu$, for all $\omega,\nu\in\Omega$, then we say that $\leq$ {\it strictly dominates} $\leq'$. \end{defn}
Recall that $\omega<\nu$ means that
$\omega\leq\nu$, but $\omega\not\sim\nu$. Observe that if $\leq$ also dominates (respectively, strictly dominates) $\leq'$,
then $\leq^\op$ dominates (respectively, strictly dominates) $\leq^{\prime\op}$. \index{preorder! strictly dominates}

\begin{example}\label{example2} This example explains the notation above, in which the quasi-poset $\Lambda$ has been renamed $\Omega$. Let $(W,S,L)$ be a standard finite Coxeter system.   We will make use of the (generalized) Kazhdan-Lusztig cell theory mentioned in Section 1.2. See \cite[Ch. 8]{Lus03}.  Let $\Omega$ be the set of left Kazhdan-Lusztig cells in $W$.  There is a preorder $\leq_L$ on $W$ whose equivalence classes are the left cells $\omega\in\Omega$.  Similarly, there is a preorder $\leq_{LR}$ on $W$ whose equivalence classes are the two-sided
Kazhdan-Lusztig cells. Clearly, $\leq_L$ also defines a preorder (still denoted $\leq_L$) on $\Omega$, which is
a partial order. Also, $\leq_{LR}$ defines a preorder on $\Omega$ (still denoted $\leq_{LR}$).\end{example}

\begin{prop} \label{prop1} As preorders on $\Omega$, $\leq_L$ strictly dominates $\leq_{LR}$.  \end{prop}
\begin{proof} 
This (nontrivial) result is immediate from the statement \eqref{order} in Section 1.2 which follows from (P4) and (P9) of \cite[Ch.14]{Lus03}.
Note that all conjectures (P1)--(P15) hold for the standard finite Coxeter system $(W,S,L)$.
\end{proof}
An important consequence  is that a two-sided cell module is a direct sum of left-cell modules. (See Proposition \ref{208}.)

An important property of strict dominance in general is the following consequence for height functions:

\begin{prop}\label{prop2} Suppose $\leq,\leq'$ are preorders on a set $\Omega$ and $\leq$   
strictly dominates $\leq'$.\footnote{In practice, $\Omega$ will often be the set of left cells as described above,
but, in general, it might just be a quasi-poset. It will be clear from context which is which.} If
$\frak h$ is any height function compatible with $\leq'$, then $\frak h$ is also compatible with $\leq$.
\end{prop}

\begin{proof} Let $\omega,\nu\in\Omega$.  Suppose first that $\omega\sim\nu$. Then, $\omega\sim'\nu$ just
from the dominance property. So $\htt(\omega)=\htt(\nu)$.
Next, suppose $\omega<\nu$. Then, $\omega<'\nu$ by strict dominance, and so $\htt(\omega)<
\htt(\nu).$ The proposition follows. \end{proof}.

\begin{rem}For an alternate proof of the proposition, introduce a third preorder $\leq_{\htt}$ on
$\Omega$ by declaring $\omega\leq_{\htt}\mu\iff \htt(\omega)\leq\htt(\mu)$. The compatability of
$\htt$ with $\leq'$ is equivalent to the strict dominance of $\leq_{\htt}$ by $\leq'$.   Obviously,
strict dominance is a transitive relation between preorders, and the proposition follows. (Use the fact that  there is a similar
equivalence with $\leq'$ replaced by $\leq$.)
\end{rem}

 In the context of Example \ref{example2}, we consider a height function $\htt:\Omega\to\mathbb Z$. Using
 Proposition \ref{prop1}, the following result follows immediately from Proposition \ref{prop2}.

\begin{cor}\label{discussion} Any height function $\htt$ compatible (see display \eqref{preorder}) with the preorder $\leq_{LR}$ %{\color{blue}(i.e., the a-function)} 
on the set $\Omega$ of
left cells in $W$ is compatible
with $\leq_L$.\end{cor} 

This result---now seen to be easy---plays an important role in the application of Theorem \ref{relative inj}. See, e.g., Theorem \ref{4.4}. We must use $\leq_{LR}$ or a related preorder in the statement of the latter theorem in order for it to prove the conjecture under discussion there. However, the proof of Theorem \ref{4.4}
requires properties of $q$-permutation modules $\sH x_J$ developed in Section 1.3, which depend on a height function compatible with $\leq_L$.

\begin{rem}\label{A7}
We conclude with a few words about height functions associated to the Hecke algebras of standard finite Coxeter systems. In Example \ref{example1}, the height functions arise as functions $\htt:\Lambda\to\mathbb Z$, where $\Lambda$ indexes the irreducible (left) modules (of
$H_K=\sH_K$).  In favorable circumstances,  $\htt$ is constant on indices that index irreducible modules associated to the
same two-sided cell. Thus, $\htt$ defines, in particular, a function $\htt:\Omega\to\mathbb Z$. This procedure 
can be reversed. If we have a height function $\htt:\bar\Omega\to\mathbb Z$, where $\bar\Omega$ is the set
of two-sided cells, then $\htt$ arises from functions (to which we give the same name) $\htt:\Omega\to\mathbb Z$
and $\htt:\Lambda\to\mathbb Z$. If $\htt$ is compatible with $(\Omega,\leq_{LR})$, then it is compatible with 
$(\Omega,\leq_L)$; see Corollary \ref{discussion}. This method gives a rich supply of height functions $\htt:\Omega\to\mathbb Z$ in the ``cell theory" context of Example \ref{example2} above, all compatible with the $\Lambda\to\mathbb Z$ context of Example \ref{example1}. 

 Further abusing notation, sometimes we think of {\it modules} indexed by $\Omega$ as having a height. For example, a left-cell module \index{cell module! left $\sim$, $S(\omega)$} $S(\omega)$ would 
have height $\htt_l(\omega)$, while the dual left-cell module $S_\omega$ has height $\htt_r(\omega)=-\htt_l(\omega)$. (We often denote $\htt_l$ or $\htt_r$ by simply $\htt$ if no confusion results.)\end{rem}

\chapter{Decomposition numbers and Krull dimension $\leq2$} 
Following \cite{DPS98a}, but modifying the notation slightly, a local triple $(R,Q,F)$ consists of a commutative local domain $R$, having
fraction field $Q$ and residue field $F=R/\mathfrak m$, where $\mathfrak m$ denotes the maximal ideal of $R$.  
In the modular representation theory of a finite group $G$, the local domain $R$ will usually be a discrete (sometimes complete) valuation ring (DVR), while 
$F$ and $Q$ will typically be
large enough to be splitting fields for some designated subgroups of $G$ (for example, $G$ itself, or all of its subgroups).  There are also many variations involving specific subgroups.

Given an irreducible $QG$-module $L$, it is elementary that there exists a (full) $R$-lattice $L_0$ in $L$  in the sense of Definition \ref{orders}. Because $R$ is a DVR, we can easily assume that $L_0$ is $G$-stable. 
 When $QG$ is also semisimple, a classical result
states  that there is a map
\begin{equation}\label{Groth}
\sG(QG)\longrightarrow \sG(FG)\end{equation}
from the Grothendieck group of $QG$ to that of $FG$. The map (\ref{Groth}) is called the {\it decomposition map} of the triple $(R,Q,F)$. For an irreducible $QG$-module $L$, it sends $[L]\in\sG(QG)$ to
$\sum_i m_i[Z_i]\in\sG(FG)$, where $Z_1,Z_2,\ldots$ is a listing of the irreducible $FG$-modules, and $m_i$ is the multiplicity of
$Z_i$ as a composition factor of the $FG$-module $F\otimes_RL_0$. The point is that the multiplicity $m_i$ is well-defined.  The various $m_i$ are called {\it decomposition numbers} and are often collected together to give matrix versions ({\it decomposition matrices}) of the decomposition maps. 

\medskip%\noindent
{\bf C1. Regular local rings of Krull dimension $\leq2$.}
In this monograph, we often work with the category mod-$\sH$ of finite (right) modules over the generic Hecke algebra  $\sH$; see (\ref{relations}). In this case, fix a
 triple $(R,Q,F)$, where $R$ is not a discrete valuation ring but is a regular local ring of Krull dimension $\leq 2$.
 For example, $R$ is the localization of $\sZ$ or $\sZ^\natural$ at some prime ideal; see Subsection C3 below for more details (also see the proof of \cite[Th. 19.5]{M86}). The group algebra $RG$ is replaced by the generic Hecke
algebra $\sH$ (or sometimes $\sH^\natural$). In this case, the construction of the lattice $L_0$ in the previous
paragraph is problematic. We briefly indicate how this works in the discussion below, following \cite[\S1]{DPS98a}.
We start with the following result of Auslander-Goldman:

\begin{lem} \label{AS1} \cite[p. 17]{AG60a}, \cite[Th. C.17]{DDPW08}. Let $R$ be a regular local commutative ring of Krull dimension $\leq 2$, and let $M$ be a finite $R$-module. Then, the dual $R$-module $M^*:=\Hom_R(M,R)$ is $R$-projective (and hence free, since $R$ is local). \end{lem}

The following consequence is contained in  \cite[Prop. 1.1.1 \& Cor. 1.1.2b]{DPS98a}, taking into account slight differences in notation. In particular, we have avoided the lattice terminology of \cite{AG60a}, since we use the
terminology of \cite{Reiner03}. Of course, if $R$ is a DVR and $A=RG$ for a finite group $G$, this result is classical.

\begin{thm}\label{AG2} Let  $(R,Q,F)$ be a local triple with $R$  a regular local domain of Krull dimension $\leq 2$.  Let $B$ be an algebra which is finite and free over $R$. Suppose that $\sX$ is an irreducible  $B_Q$-module. Then $\sX=X_Q$, for some $R$-projective $B$-submodule $X\subseteq \sX$.  If $Y$ is another $R$-projective $B$-submodule of $\sX$ such that $\sX=Y_Q$, then 
$X_F$ and $Y_F$ have the same composition factors over $B_F$ with the same multiplicities.\footnote{We remark that the same theorem holds without the requirement that $R$ be local, provided the condition that $X$ is ``$R$-projective'' is replaced by the condition that $X$ is ``$R$-free," with a similar replacement regarding $Y$.  The proof comes down to showing that Theorem \ref{AS1} holds without assuming that $R$ is local. This fact follows from the local version, which implies, in the non-local case, that $\Hom_R(M,R)_{\mathfrak p}$ is projective over $R({\mathfrak p})$, for every $\mathfrak p\in\Spec\,R$.}
\end{thm}

This theorem gives an evident way to define a 
decomposition map \index{decomposition map}
\begin{equation}\label{decomMap} 
\sG (B_Q) \longrightarrow \sG(B_F)\end{equation}
  entirely analogous to the map (\ref{Groth}).
This decomposition map often has natural factorizations in terms of decomposition
maps associated with DVR triples. We discuss such a factorization, leading to Theorem \ref{decompdiagram} below.

In general, define a prime ideal $\mathfrak p$ in a domain $R$ to be {\it coregular} if
$R/\mathfrak p$ is a regular ring. We think of the condition of coregularity as one that can be checked by construction. 
We give several examples later in this section.

  Let $(R,Q,F)$ be a local triple, but assume that $R$ is a regular local domain of Krull dimension 2, and again let $B$ be as in Theorem \ref{AG2}. Let $\mathfrak p$ be a height one coregular prime ideal in $R$. 
  This implies that 
   both the rings $R/\mathfrak p$ and $R_{\mathfrak p}$ are DVRs.  The fraction field of $R/\mathfrak p$ is the residue field
   $R(\mathfrak p)$ of $R_{\mathfrak p}$. The residue field of $R/\mathfrak p$ is the same as the residue field
  $F$ of $R$.    Consider the following
   commutative diagram of rings. Here $B(\mathfrak p):=B_{R(\mathfrak p)}.$   
   
 \begin{equation}\label{setup2}
\begin{CD}
B @>>> B_{\mathfrak p} @>>> {B_Q}\\
@VVV @VVV\\
B/\mathfrak pB @>>>B(\mathfrak p)\\
@VVV \\
B_F 
\end{CD}
\end{equation} 
\vspace{.1cm}

Observe that the natural ring maps associated with $B$ and the triple 
$(R,Q,F)$ can be found as compositions along the perimeter. They may be used to construct the decomposition map
$\sG(B_Q)\longrightarrow\sG(B_F)$. (One starts with a $B_Q$-module, finds in it a full $B$-module lattice, then base changes the latter to  $B_F$.)  Analogous 
 ring
maps associated with $B_{\mathfrak p}$ and the DVR triple $(R_{\mathfrak p}, Q,R(\mathfrak p))$ can be 
found as individual maps in the right-hand portion of the diagram. These can be used to construct a  decomposition map
$\sG(B_Q)\longrightarrow \sG(B(\mathfrak p))$. The ring 
maps associated with $B/\mathfrak p B$ and the DVR triple $(R/\mathfrak p, R(\mathfrak p), F)$ are similarly found at the lower left. They can be used to construct a decomposition map $\sG(B(\mathfrak p))\longrightarrow \sG(B_F)$.
Easy arguments, starting with a $B$-module and  various base changes of it, now give the following result.

\begin{thm}\label{decompdiagram} Let $(R,Q,F)$ be a regular local triple. Assume that $R$ has Krull dimension $\leq 2$, and let
$\mathfrak p\in\Spec R$ be coregular of height 1. Finally, let $B$ be an $R$-algebra as in Theorem \ref{AG2}. 
 Then, the decomposition map $\sG(B_Q)\longrightarrow \sG(B_F)$ is the composition of the decomposition maps 
\begin{equation}\label{compositionmaps} \sG(B_Q)\longrightarrow \sG(B(\mathfrak p))\longrightarrow\sG(B_F).\end{equation}
\end{thm}

The theorem is especially interesting when one of the decomposition map factors in display (\ref{compositionmaps}) induces a bijective correspondence on isomorphism classes of irreducible modules. When this occurs, the theorem gives an identification of a decomposition matrix for the remaining factor with one for the decomposition map $\sG(B_Q)\longrightarrow
\sG(B_F)$. 

The following theorem involves such ingredients and is used in Section 5.2. The proof shows that the left-hand decomposition map factor  in display (\ref{compositionmaps}) induces the desired bijective correspondence. Thus, the decomposition matrix for $\sG(B_Q)\longrightarrow \sG(B_F)$ may be identified with that for the right-hand factor $\sG(B(\mathfrak p))\longrightarrow
\sG(B_F)$ in this case.

      \begin{thm}\label{cutie}    Assume the hypotheses of Theorem \ref{decompdiagram}. Also, assume that the following two conditions hold:
   \begin{itemize}  
  \item[(1)] $B_Q$ is split, and
   \item[(2)] $B(\mathfrak p)$ is split semisimple.
  \end{itemize} 
  Then, any decomposition matrix for $B$, with respect to the regular local triple $(R,Q,F)$, agrees with one for $B/\mathfrak p B$ with respect to the DVR triple $(R/\mathfrak p, R(\mathfrak p), F)$.  \end{thm}

   The proof of Theorem \ref{cutie} is given at the end of the subsection {\bf C2}.  We also recall that the term ``split" applied to a finite (dimensional) algebra over a given base field means  that every irreducible module for the algebra is absolutely irreducible. Equivalently, the algebra mod its Jacobson radical is a direct product of matrix algebras over the base field. 
   
   We first prove some independent results
which will be useful in its proof.  The notation, especially as regards $(R,Q,F),$
is independent of that used in the statement of the theorem.  

\medskip%\noindent
{\bf C2. On semiperfect rings.}
The following discussion makes use of the theory of semiperfect rings.  Recall that a ring $E$ is semiperfect provided 
  $E/J(E)$ is semisimple (automatic in many cases) and idempotents in $E/J(E)$ ``lift" to idempotents in $E$.
  (See \cite[p. 303]{AF74}.  Here $J(E)$ denotes the Jacobson radical of $E$.) The lifting condition implies that
  any complete finite set of orthogonal idempotents in $E/J(E)$ lifts to a complete finite set of orthogonal idempotents in
  $E$. See \cite[Prop. 27.4]{AF74} and \cite[pp. 303--304]{AF74}. Another useful criterion, given in \cite[Th. 27.6]{AF74}, states that $E$ is semiperfect provided every irreducible $E$-module has a projective cover. Finally,
  \cite[Cor. 27.8]{AF74} states that $E$ is semiperfect if some Morita equivalent ring is semiperfect.
  
  \begin{prop}\label{cute2} Let $(R,Q,F)$ be a DVR triple and let $E$ be a finite and $R$-free algebra.
  Then, condition (a) below implies condition (b):
  
 \smallskip 
   (a) $E_Q$ is split semisimple over $Q$, 
   
   (b) $E$ is semiperfect as a ring.
   \end{prop}
 
   \begin{proof} See \cite[\S6, Ex. 16]{CR81}.  \end{proof}
   
   The following proposition is less trivial and could be used as a kind of converse of the proposition above. It may be viewed as a version of the Tits deformation theorem in the spirit of 
   \cite[Th. 7.4.6]{GP00}, which uses a different set of hypotheses.\footnote{The authors of \cite{GP00} only require that $R$ be integrally closed, but assume that $E_Q$ is split; in particular they do not include the semiperfect case (which holds when
   $R$ is complete). They discuss further variations of their hypotheses.  }
  
   \begin{prop}\label{cute3} Maintain the notation in the hypothesis of Proposition \ref{cute2}. 
   Assume
   \begin{itemize}
   \item[(1)] $E$ is semiperfect, and
   \item[(2)]
    $E_F$ is split semisimple. \end{itemize}
    Then, $E$ is a direct product of full matrix algebras over $R$, and $E_Q$ is split semisimple. 
   \end{prop}
 
   \begin{proof}  %Since $R$ here is a DVR, we can fix a generator $\pi$ for its maximal ideal.
   
   %First, we show that $E\cap J(E_Q)=0$; equivalently, $J(E_Q)=0$, so $E_Q$ is semisimple. Note that  $E\cap J(E_Q)$ is a nilpotent ideal in $E$, hence contained in $J(E)$. However, $J(E)\subseteq \pi E $, since $E/\pi E\cong E_F$ is semisimple by hypothesis.  Thus, $E\cap J(E_Q)\subseteq J(E)\subseteq \pi E$. But  $E/E\cap J(E_Q)$ is $R$-torsion free, hence free for the DVR $R$.  Also $E\cap J(E_Q)$ is torsion free and finitely generated over  $R$, hence free. By rank considerations show $E\cap J(E_Q)=0$. 

 %  {\color{red}Our hypotheses, together with Proposition \ref{cute2}, now  imply  that $E$ is semiperfect.} 
 
 Let $\pi$ generate the maximal ideal of $R$. Since $E/\pi E\cong E_F$ is a semisimple algebra by hypothesis, it follows that  $J(E)\subseteq \pi E $.  Conversely, 
 let $L$ be an irreducible (left) $E$-module. Thus, $L$ is an irreducible $E/\pi E=E_F$-module.  Notice that every irreducible $E$-module $L$ is finitely
  generated over $R$. Hence, $\pi L$ is a proper submodule of $L$ over $R$ or $E$, so it must be $0$. Thus, $\pi E\subseteq\text{ker}(\phi)$, for every $E$-module homomorphism $\phi:{}_EE\to L$. Hence, $\pi E\subseteq J(E)$, giving $\pi E= J(E)$. 
  
  Since $E$ is semiperfect, for the irreducible $E$-modules $L$ and $L'$, there are projective $E$-modules $P$ and $P'$ covering $L$ and $L'$, respectively. 
   
   \medskip\noindent
   {\bf Claim 1:} If $L\not\cong L'$, then $\Hom_E(P,P')=0$.
   
   \medskip
   \noindent {\it Proof of Claim 1.} We argue by contradiction. Let $0\not=f$ be in $\Hom_E(P,P')$. Since $R$ is a DVR, $\bigcap \pi^nR=0$, and it follows that
   $\bigcap\pi^nP'=0$. Hence, $f(P)\not\subseteq  \pi^nP'$, for some $n>0$. Let $n$ be the smallest non-negative integer with this property. Then, $f$ induces a nonzero map $$P/\pi P\to \pi^{n-1}P'/\pi^nP'.$$
   However, $P/\pi P\cong L$ and $\pi^{n-1}P'/\pi^nP'\cong P'/\pi P'\cong L'$. But, such a nonzero map does not exist, proving Claim 1. 
   
   \medskip\noindent
   {\bf Claim 2:} $\Hom_E(P,P)\cong R$.
   
   \medskip
   \noindent{\it Proof of Claim 2.} First, note that any non-zero $f\in\Hom_R(P,P)$ induces a nonzero map
   $P/\pi P\longrightarrow \pi^{n-1}P/\pi^nP$, if $n-1$ is chosen maximal with $f(P)\subseteq \pi^{n-1}P$. Thus, we can write $f$ as $\pi^{n-1}g$, where $g\in\Hom_R(P,P)$ and $g(P)\not\subseteq\pi P$.  (Take $g=\pi^{-n+1}f$.) Thus, $g$
   induces a nonzero $E/\pi E$-module homomorphism 
   to $P/\pi P\to P/\pi P$.  Such a map must be a scalar multiple of the identity  $1_{P/\pi P}$ since $P/\pi P\cong L$ is absolutely irreducible. Thus, $g$ belongs to $R\cdot 1_P+\pi \Hom_E(P,P)$ as does $f=\pi^{n-1}g$. However, $f$ was an arbitrary non-zero element in $\Hom_E(P,P)$. Thus, $R\cdot 1_P=\Hom_E(P,P)$ by Nakayama's Lemma. 
    This proves Claim 2. 

\medskip
 
 Next,  we show that $E$ is a direct product of full matrix algebras over $R$.  Note that it follows immediately that $E_Q$ is split semisimple, so the proposition will also follow.
 
 Since $E$ is semiperfect, the decomposition of $E_F$ into a direct sum of irreducible (left) modules 
 $$E_F=
 \bigoplus_iL_i^{\oplus n_i},$$
 where $n_i=\dim_FL_i$, gives a corresponding decomposition of $_EE$. We have 
 $$_EE=\bigoplus P_i^{n_i},$$
 where $P_i$ is the projective cover of $L_i$ as above.  Using  the (remarkable!) fact of Claim 1, 
 $$E\cong\End_E(E)^\op\cong \bigoplus_i\End_E(P_i^{\oplus n_i})^\op.$$
 By Claim 2, the latter direct sum can be written as
 \begin{equation}\label{fact}
 \bigoplus_iM_{n_i}(R).\end{equation}
  Here the direct sum gives a direct product in the category of rings or $R$-algebras.  This proves the first part of the assertion, and the rest follows.
    \end{proof}
   
   The next corollary is an obvious consequence of Proposition \ref{cute3}. Notice that each of the conclusions (a), (b), (c) of the corollary (we only need (a) later) easily reduces to the case where $E$ is a full matrix algebra over $R$.
   
   \begin{cor}\label{lastcor} Let $E$ and the DVR triple $(R,Q,F)$ satisfy the hypotheses of Proposition \ref{cute3}. 
   
   (a) The natural decomposition map
   \begin{equation}\label{decomp1}\sG(E_Q)\longrightarrow \sG(E_F)\end{equation}
   on Grothedieck groups induces 
    a natural one-to-one correspondence between the (isomorphism types of) irreducible modules for $E_Q$ and those of $E_F$.  
    
    (b) Let $M$ be any finite
   $R$-free $E$-module. Then, $M_Q$ is irreducible over $E_Q$ if and only if  $M_F$ is irreducible over $E_F$. In this case, $[M_Q]$ corresponds to $[M_F]$ in (a).
   
   (c) Finally, given any ordering of the distinct irreducible $E_Q$-modules (resp., $E_F$-modules), an ordering of the irreducible $E_F$-modules (resp., $E_Q$-modules) can be chosen so that the corresponding decomposition matrix is the identity matrix.\end{cor} 
 
 \bigskip  
 This completes our discussion of semiperfect rings. They will be used only in the proof below of 
 Theorem \ref{cutie}. 
 
 \medskip
 
 %We now return to the assumption (implicit in 
 %For the convenience of the reader, w
 We now return to the assumption, as
part of  the hypotheses shared by Theorems \ref{decompdiagram} and \ref{cutie}, that $(R,Q,F)$ is a {\it regular local triple} with $R$ of Krull dimension $\leq 2$, and $B$ is a finite and $R$-free algebra. Thus, associated to a coregular $\mathfrak p\in\Spec\,R$ having height 1, there are two local triples $(R_{\mathfrak p}, Q, R(\mathfrak p))$ and $(R/\mathfrak p, R(\mathfrak p),F)$; see display \eqref{setup2}.
 Finally, reading more from the hypothesis of
Theorem \ref{cutie}, we note that $B_Q$  is split, and
$B(\mathfrak p)$ is split semisimple. 
  
  \medskip\medskip \noindent
  {\bf Proof of Theorem \ref{cutie}.}   
 We first apply Proposition \ref{cute2} with the DVR triple $(R_{\mathfrak p}, Q, R(\mathfrak p))$
playing the role of $(R,Q,F)$ there, and $B_\mathfrak p$ playing the role of $E$. 
We use this framework to prove that $B_Q=(B_\mathfrak p)_Q$ is semisimple:

 Since $R_{\mathfrak p}$ is a DVR, we can fix a generator $\pi$ for its maximal ideal. By the semi-simplicity of $B(\mathfrak p)=B_{R({\mathfrak p})}$, we have $\pi B_{\mathfrak p}=J(B_{\mathfrak p})$ (see the first paragraph of the proof of Proposition \ref{cute3}). Thus, $B_{\mathfrak p}/\pi B_{\mathfrak p}\cong B_{R({\mathfrak p})}$ implies rank$(B_{\mathfrak p})=\dim(B_{R({\mathfrak p})})$.  We now show that $B_{\mathfrak p}\cap J(B_Q)=0$; equivalently, $J(B_Q)=0$, implying $B_Q$ is semisimple. Note that  $B_{\mathfrak p}\cap J(B_Q)$ is a nilpotent ideal in $B_{\mathfrak p}$, hence contained in $J(B_{\mathfrak p})$. 
 %However, $J(B_{\mathfrak p})\subseteq \pi B_{\mathfrak p} $, since $B_{\mathfrak p}/\pi B_{\mathfrak p}\cong B_{R({\mathfrak p})}$ is semisimple by hypothesis (and rank$(B_{\mathfrak p})=\dim(B_{R({\mathfrak p})})$).  Thus, $B_{\mathfrak p}\cap J(B_Q)\subseteq J(B_{\mathfrak p})\subseteq \pi B_{\mathfrak p}$ and
So, there is an epimorphism $B_{\mathfrak p}/B_{\mathfrak p}\cap J(B_Q)\twoheadrightarrow B_{R(\mathfrak p)}$. But,  $B_{\mathfrak p}/B_{\mathfrak p}\cap J(B_Q)$, embedded in $B_Q/J(B_Q)$, is $R_{\mathfrak p}$-torsion-free, hence free for the DVR $R_{\mathfrak p}$.  Also $B_{\mathfrak p}\cap J(B_Q)$ is torsion-free and finitely generated over  $R_{\mathfrak p}$, hence free. Rank considerations show $B_{\mathfrak p}\cap J(B_Q)=0$. 

By the hypothesis of the theorem, we conclude that $B_Q=(B_{\mathfrak p})_Q$ is split semisimple. Thus, $B_{\mathfrak p}$ is semiperfect by Proposition \ref{cute2}. We may now apply Proposition \ref{cute3}
(and Corollary \ref{lastcor}) to the DVR triples $(R_{\mathfrak p}, Q, R(\mathfrak p))$ for $E=B_\mathfrak p$
 to conclude that the decomposition matrix associated with the left decomposition map in display \eqref{compositionmaps} is an identity matrix (up to labeling of irreducible modules). Theorem \ref{cutie} now follows from Theorem \ref{decompdiagram}.    %and the hypothesis of the theorem is assumed to be satisfied. 
%  We get various identity decomposition matrices by Corollary \ref{lastcor}. These may be substituted into a matrix version of the decomposition map factorization to obtain the desired identification. We leave the remaining details to the reader. 
%  {\color{red}We need to prove that $B$ is semiperfect as a ring. Consider the $R/\mathfrak p$-algebra $\bar B:=B/{\mathfrak p}B$. Since the fraction field of $R/\mathfrak p$ is $R(\mathfrak p)$ and $B(\mathfrak p)$ is split semisimple, by Proposition \ref{cute2}, $\bar B$ is semiperfect as a ring. Thus, $\bar B/J(\bar B)$ is seimisimple and all its idempotent lift to $\bar B$. Since ${\mathfrak p}B\subseteq {\mathfrak m}B\subseteq J(B)$, it follows that $J(\bar B)= J(B)/{\mathfrak p}B$. Hence, $\bar B/J(\bar B)\cong B/J(B)$. Thus, $B/J(B)$ is semisimple and all its idempotents lift. Hence, $B$ is semisimple as a ring. }
  \qed
  
 \medskip
 \medskip
 Before proceeding, we introduce some  further terminology that will be useful in the rest of this appendix. Let
 $(R_1,Q_1,F_1)$ and $(R,Q,F)$ be two local triples. Write
  \begin{equation}\label{partial} (R,Q,F)\leq (R_1,Q_1,F_1)\end{equation}
  if there is given an inclusion $R\subseteq R_1$ which also induces both inclusions $Q\subseteq Q_1$ 
and $F\subseteq F_1$ of rings. Unlike $Q\subseteq Q_1$, the inclusion $F\subseteq F_1$ is not automatic (from $R\subseteq R_1$) but exists precisely when $\mathfrak m_1\cap R=\mathfrak m$, where $\mathfrak m_1$ and $\mathfrak m$ are the maximal ideals of $R_1$ and $R$, respectively. 
(Equivalently, $\mathfrak m\subseteq \mathfrak m_1$, since $\mathfrak m$ is maximal in $R$.)

It will also be convenient to use the following notion:

\begin{defn} \label{resfinite} Consider two local triples related as in display (\ref{partial}). 

(a) We say the triples are {\it residually equal} if the inclusion $F\subseteq F_1$ is an equality. 

(b) We say
that the local ring $R_1$ is {\it unramified} over the local ring $R$ if $\mathfrak m_1=\mathfrak m R_1$. 

\end{defn}

 \medskip%\noindent
 {\bf C3. Laurent triples.} 
  Suppose that $(\sO,K,k)$ is a DVR triple,  so that $\sO$ is a discrete valuation ring with fraction field $K$ 
 and residue field $k$. We always assume in this subsection (and in this appendix) that $K$ is a field of characteristic $0$, so that all of its extensions are separable and $\mathbb Z$ is contained in $K$. Assume also that
 $k$ has characteristic equal to  a prime integer $r>0$.\footnote{In typical applications, $\sO$ is complete and $k$ is algebraically   closed, but we do not make these assumptions here. Also, the assumption that the characteristic of $k$ is positive is a matter of notational convenience. } Form the localization
 \begin{equation}\label{O'}\sO':=\sO[t,t^{-1}]_{\mathfrak m},\end{equation}
 where $\mathfrak m$ is a maximal ideal in $\sO[t,t^{-1}]:=\sO\otimes_{\mathbb Z}\sZ$ containing the maximal ideal of $\sO$. Such an ideal $\mathfrak m$ may be constructed as the ideal generated by
 $\pi$ and $t-a$, where $\pi\in\sO$ generates the maximal ideal of $\sO$ and $a\in\sO$ does not belong to $\pi\sO$.
 (We will only consider  maximal ideals $\mathfrak m$ of this form.)
 Let $K'=K(t)$ be the fraction field of $\sO'$, and note that the residue field of the local ring $\sO'$ identifies naturally with
 $\sO[t,t^{-1}]/\mathfrak m\cong\sO/\pi\sO=k$.  (This property will result in the ``residual" equality noted in Proposition \ref{old prop 10}
 below.)

 We call $(\sO',K',k)$ the {\it Laurent triple}\index{Laurent triple} associated to the DVR triple
 $(\sO, K,k)$ and the element $a\in\sO\backslash\pi\sO$.
  It follows from \cite[Th. 19.5]{M86} that the local triple $(\sO',K',k)$ is a regular local triple in which $\sO'$ has Krull dimension at most 2 (in fact, exactly 2).  
  
  \begin{example}\label{example} When it is useful to identify $a$, we will write $^a\sO^{\prime}$ for the ring $\sO'$. We give an important example
  of this setup. 
  
   Start  with a DVR triple $(\sO, K,k)$ in which $K$ is the fraction field of $\sO$ and $k$ is the residue field of $\sO$. Now, let $a\in\sO\backslash \pi\sO$ be arbitrary, and form two discrete valuation rings $\sO'=\sO[t,t^{-1}]_{\mathfrak m'}$ and $\sO^{
  \prime\prime}=\sO[t,t^{-1}]_{\mathfrak m^{\prime\prime}}$, where
  $\mathfrak m':=(t-a,\pi)$  and $\mathfrak m^{\prime\prime}:=(t-a^n,\pi)$ are maximal ideals. Here, $n$ can be any positive integer. 
 
  We can regard $\sO^{\prime\prime}\subseteq \sO^\prime$ by identifying $t\in\sO^{\prime\prime}=\sO[t,t^{-1}]_{\mathfrak m^{\prime\prime}}$ with $t^n\in \sO'=\sO[t,t^{-1}]_{\mathfrak m^\prime}$. This also results in an
  inclusion  $K(t)\subseteq K(t)$, where $t$ on the left is identified with $t^n$ on the right.  
    \end{example}
  
  The following proposition is used in Section 5.2 in the proof of Theorem \ref{bigDecom}.

   \begin{prop}\label{old prop 10} With the above (slight abuse of) notation, 
   $$(\sO^{\prime\prime},K(t), k)\leq (\sO',K(t), k),$$
   and these triples are residually equal (as defined in Definition \ref{resfinite}(a)).
   Finally, suppose that $B$ is a finite and free $\sO^{\prime\prime}$-algebra. Assume that 
    $B_{K(t)} :=
   K(t)\otimes_{\sO^{\prime\prime}}B$ is split.  
   Then, the decomposition matrix of $B_{\sO^{\prime\prime}}$ with respect to $(\sO^{\prime\prime}, K(t),k)$ agrees with that for
   $B_{\sO'}$ with respect to $(\sO^{\prime}, K(t),k)$. 
    \end{prop}
  
  \begin{proof} The first assertion is proved
  using the discussion above the proposition: Observe that, in the notation for $\sO'$,\ $\mathfrak m^{\prime\prime}$  is generated as an ideal in $\sO^{\prime\prime}$ by $\pi$ and the element $t^n-a^n\in\sO'$. The factorization
  $$t^n-a^n=(t-a)[(t^n-a^n)/(t-a)]$$
  now shows that $\mathfrak m^{\prime\prime}\subseteq\mathfrak m'$, and the first assertion follows.
  
  The second assertion is a consequence of Proposition \ref{D9} below.
   \end{proof}

\medskip%\noindent
{\bf C4. Some uniqueness/comparison results.} 
%We begin with the following elementary result. %, whose proof is left to the reader. 
Let $(R,Q,F)$ be a local triple in which $R$ is a regular local domain of Krull dimension  $\leq 2$.  Let $B$ be
  an $R$-algebra which is finite and free over $R$.  We have seen that  decomposition maps and matrices can be associated to $B$ 
  using such a triple. The ``decomposition number" entries of the decomposition matrix are defined for a pair $(\mathfrak X,L)$ from mod-$B_Q$ and mod-$B_F$, respectively (or their left-hand analogs)  as the multiplicity
  $[X_F:L]$ of $L$ in the base change to $F$ of an $R$-free $B$-lattice $X$  with $X_Q\cong\mathfrak X$. These multiplicities do not depend on the choice of $X$ by Theorem \ref{AG2}.
  
  This provides a kind of internal uniqueness for the decomposition maps and matrices defined using the local triple
  $(R,Q,F)$. But, what if we change the triple? 
  
  \begin{prop}\label{D9}Let $(R,Q,F)\leq (R_1,Q_1,F_1)$ be regular local triples with $R,R_1$ regular of Krull dimension
  $\leq 2$.
   Let $B$ be a finite and $R$-free algebra. (These two lines are part of the hypotheses for each of the assertions (a), (b) below, and, in an evident modified form, for part (c).) 
   
   (a) Assume that 
  $B_Q$ and $B_F$ are split. Then, $B_{Q_1}$ and $B_{F_1}$ are also split and the decomposition matrices associated 
  to $B_{R_1}$ and $(R_1,Q_1,F_1)$ agree via base-change  indexing  with those associated to $B$ and $(R,Q,F)$. 
  
(b)  Similarly,  if $B_Q$ is split and the triples above are residually equal, then $B_{Q_1}$ is split, and the decomposition numbers above agree via base-change indexing.  
  
(c) More generally,  all claims above hold if the hypothesis $(R,Q,F)\leq (R_1,Q_1,F_1)$ is replaced by an injective ring homomorphism $i:R\longrightarrow R_1$ that induces  an inequality
$(i(R),i(Q), i(F))\leq (R_1,Q_1,F_1)$. (Note that the isomorphism $R\cong i(R)$ enables natural definitions of $i(Q)$ and $i(F)$ as the fraction and residue fields of $i(R)$, respectively.) 
  \end{prop}
  
  \begin{proof}We briefly sketch the proof, partly to make precise the agreement of decomposition numbers. Recall that the term ``split" applied to a finite (dimensional) algebra over a given base field means  that every irreducible module for the algebra is absolutely irreducible. %Equivalently, the algebra mod its Jacobson radical is a direct product of matrix algebras over the base field. 
  Thus, if $B_Q$ is split over $Q$, so is $B_{Q_1}=B_Q\otimes_QQ_1$ over $Q_1$. In addition, if $L, L' \ldots$ is a complete (finite) list of distinct,  irreducible $B_Q$-modules,
   then $L_{Q_1}, L'_{Q_1}, \ldots$ is a corresponding complete list for $B_{Q_1}$.
  A similar discussion applies for $F$ and $F_1$. 
  
  Let $X$ be any $R$-finite $B$-lattice, projective over $R$, with $X_Q$ irreducible. Let $L$ be any irreducible $B_F$-module. 
  Then, the decomposition number associated with the pair $(X_Q,L)$ is the multiplicity of $L$ as a $B_F$-composition factor of $X_F$. This is also the multiplicity of $L_{F_1}$ as a $B_{F_1}$-composition factor of  
  $$\begin{aligned} (X_F)_{F_1}= (X\otimes_RF)_{F_1}&\cong(X\otimes_RF)\otimes_FF_1\cong X\otimes_R(F\otimes_F F_1)
  \cr\cong X\otimes_R(R_1\otimes_{R_1}F_1)&\cong (X\otimes_RR_1)_{F_1}.\end{aligned}$$
  
  The last module on the right is the $B_{F_1}$-module obtained by base change of the $B_R$-lattice $X\otimes_RR_1$.  This lattice has irreducible base change to $B_{Q_1}$, namely,
  $$(X\otimes_RR_1)\otimes_{R_1}Q_1\cong(X\otimes_RQ)\otimes_QQ_1\cong (X_Q)_{Q_1}.$$
  So, the given decomposition number associated to the pair $(X_Q,L)$ is also the decomposition number associated to the pair $((X_Q)_{Q_1}, L_{F_1}).$ This proves all the claims of part (a). The assertions of part (b) are proved similarly. We leave the assertions of part (c) to the reader. 
  \end{proof}

  If we have a directed system of such triples, 
  such that ``eventually" $B_Q$ and $B_F$ are split,  Proposition \ref{D9}  implies that the associated decomposition matrix is uniquely determined by $B$.  This assumes, of course, that we are able to define the system of triples naturally. We have two useful types of systems in mind, one for Krull dimension $\leq 1$ (the DVR triples) and one for Krull dimension 2 (the Laurent triples). Comparisons of Laurent triples can often be reduced to considerations for underlying DVR triples, using Theorem \ref{cutie}. One important comparison, where a different approach is used, is given in Proposition \ref{old prop 10}. 
  
  We now give two useful results for constructing comparable DVR triples:% are given in Propositions \ref{D11} and \ref{D12} below.
  
 % It will be convenient to use the following result from Jacobson's book \cite[Prop. 9.4]{Jac89}, a ``lemma" to Hensel's Lemma!
  
%  \begin{prop}\label{hensel} Let $F$ be complete relative to a discrete valuation $|\,\,|$, $R$ the valuation ring of $|\,\, |$, $P$ its maximal ideal, and $\overline R=R/P$. Suppose that $f(x)$ is a monic irreducible polynomial in $R[x]$. Then the image
 % $\overline f(x)\in \overline R[x]$ is a power of an irreducible polynomial in $\overline R[x]$. \end{prop}
  
  %For some common examples of DVR triples, let $R$ be the localization of the ring $Z$ of  algebraic integers in a number field $Q$ at a maximal ideal $\mathfrak r$ containing a fixed prime integer $r$. Thus, $R=Z_\mathfrak r$. Let %$F:=R/{\mathfrak r}_{\mathfrak r}\cong Z/\mathfrak r$, a finite field of characteristic $r$. We can form the local
% triple $(R,Q,F)$. It is the localization of the ring triple $(Z,Q,F)$. We could call $(R,Q,F)$ an arithmetic local triple (and $(Z,Q,F)$ an arithmetic ring triple).  If the prime $r$ is fixed, the smallest $(\mathbb Z_\mathfrak r,\mathbb Q, \mathbb F_r)$ of these local triples $(R,Q,F)$ is especially interesting. Here $\mathfrak r$ is the principle ideal $\langle %r\rangle=r\mathbb Z$ in $\mathbb Z$. 

\begin{prop}[Boosting $Q$]\label{D11}\label{DD13}
 Let $(R,Q,F)$ be a DVR triple and let $Q_1$ be a finite extension of the field $Q$. Then, there exists a DVR triple $(R_1, Q_1, F_1)$ such that $(R,Q,F)\leq (R_1, Q_1, F_1)$. \end{prop}

\begin{proof}Let $Z_1$ be the integral closure of $R$ in $Q_1$. Then, $Z_1$ is a Dedekind domain since $R$ is a Dedekind domain; see \cite[Th. 10.7]{Jac89} or \cite[Th. 4.4]{Reiner03}. Let $\mathfrak p_1$ be a prime ideal in $Z_1$ such that $\mathfrak p_1\cap R=\mathfrak m$, where $\mathfrak m$ is the maximal ideal of $R$.  (Such a $\mathfrak p_1$ exists; see 
\cite[Lem. 2, p. 66]{M86} or \cite[Th. 5.10, p.62]{AM69}. Note that $\mathfrak p_1$  is maximal by \cite[Th. 10.3]{Jac89})
%{\color{brown}Dlete(?): Recall all nonzero prime ideals in $ Z_1$ are maximal, since $Z_1$ is a Dedekind domain; see \cite[Cor. \S11]{M86}.)} 
Set $R_1=(Z_1)_{\mathfrak p_1}$ and $F_1= Z_1/\mathfrak p_1$. Note that $R_1$ is a DVR
by \cite[Th. 9.3]{AM69} or \cite[Th. 10.6]{Jac89}. The proposition follows.
\end{proof}
 
\begin{prop}[Boosting $F$]\label{D12} Let $(R,Q,F)$ be a DVR triple with the field $Q$ of characteristic 0 and the field $F$ finite.  Let $F_1$ be a finite extension field of $F$.
Then, there is a DVR triple $(R_1,Q_1, F_1)$ such that $(R,Q,F)\leq(R_1,Q_1,F_1)$, and the local ring $R_1$ is unramified over $R$ (as defined in Definition \ref{resfinite}(b)).%immediately below display (\ref{partial})).
\end{prop}

\begin{proof}  Put $m:=|F_1|-1.$
Let $\epsilon$ be a primitive $m$th root of 1, arbitrary for the moment. Now, apply Proposition \ref{D11} to the field inclusion $Q\subseteq Q[\epsilon]$, which results in a DVR triple
$(R_1,Q_1, F_1)$, where $Q_1=Q[\epsilon]$. 
%Also, the proof of Proposition \ref{D11} gives, in the notation of that proposition, a prime ideal
%$\mathfrak  p_1$ in the ring of integers $Z_1$ of $Q_1$ such that $F_1=Z_1/{\mathfrak p_1}$.

We can interpret $Q[\epsilon]$ as $Q[x]/f(x)Q(x)$, where $f(x)$ is an irreducible $R$-integral monic polynomial depending on
$\epsilon$ and dividing  the cyclotomic polynomial $\Phi_m(x)$ (Gauss).  Let $g(x)$ be the minimal polynomial over $F$ of a generator $\theta$
 of $F_1$ with order $|F_1|-1 =m$. Thus, $\bar\Phi_m(\theta)=0$, where $\bar\Phi_m $ is the reduction of $\Phi_m$
 mod $r$, where $r$ is the characteristic of $F$ (and $r\nmid m$). We may also regard $\bar\Phi_m$ as a (polynomial) factor of $x^m-1\in F[x]$ with distinct roots in $F$. As such, it is a product of the reduction $\bar f(x)$ mod $\mathfrak m$ (where $F=R/\mathfrak m$) of the various 
 candidates for $f(x)$ dividing $\Phi_m(x)$ in $R[x]$. The reduction of one of these $f(x)$ must have $g(x)$ as a factor in  $F[x]$. 
 Choose such an $f(x)$, and identify $R[\epsilon]$ with $R[x]/f(x)R[x]$. 
 
 \medskip\noindent
 {{\bf Claim:}} $R[\epsilon]$ is integrally closed in its fraction field $Q[\epsilon]=Q[x]/f(x)Q[x]$.  
 \medskip\medskip
 
 Assuming this claim, it follows that $R[\epsilon]$ is a Dedekind domain (see \cite[Th. 10.7]{Jac89} or \cite[Th. 4.4]{Reiner03}).   Consequently, if $\mathfrak m_1$ is any one of its maximal ideals, the localization $R_1=R[\epsilon]_{\mathfrak m_1} $ is a DVR. Now, let
 $\mathfrak m_1$ be the kernel of the composite (surjective) map
 $$ R[\epsilon]=R[x]/f(x)R[x]\longrightarrow F[x]/\bar f(x)F[x] \longrightarrow F[x]/g(x)F[x]\cong F_1.$$
 
It remains to prove the Claim. 
 Let $R'$ be the integral closure of $R$ in  $Q(\epsilon)$. Clearly,
 $R[\epsilon] \subseteq R'$, and both rings have the
  same base change as $R$-modules to $Q$.  It follows that both have the 
  same rank as (free) $R$-modules. To prove that they are equal, it is sufficient, by Nakayama's lemma,
to prove $\pi R'+\Reps=R'$, where $\pi \in R$ generates its maximal ideal. 
  Since both sides contain $\pi R'$, it is sufficient to show that this equality holds if $\pi R'$ is factored out of each side.   
The resulting left-hand side $(\pi R'+\Reps)/ \pi R'$ is the image of the obvious map 
$$i: \Reps/\pi \Reps \longrightarrow   R'/\pi R',$$ which has kernel  $I/\pi \Reps,$ where 
$I:=  \Reps \cap \pi R'$
 is an ideal in $\Reps$. Some power of $I$ is contained in $\pi \Reps$, 
since the $R$-module $\pi R'/ \pi \Reps$
contains no nonzero $\pi$-torsion-free  $R$-submodule (and so is killed by some $n$th power of $\pi$ or of $\pi R')$.  Thus, $I/\pi \Reps$
   is a nilpotent 
ideal in the direct product of  fields $\Reps/\pi \Reps \cong F[x]/\bar f(x)F[x]$ and must be 
zero.  Now, we know that the map $i$ above must be
injective. However, both target and domain have the same rank as $R/\pi R$-modules. Since $R/\pi R=F$ is a field, $i$ must also be surjective. This proves the claim, and the proof of 
Proposition \ref{D12} is now complete.

\end{proof}

\begin{rems} \label{final} We conclude this appendix with two further examples of triples that are useful.

(a) {\bf Completions.} Let $(R,Q,F)$ be a regular local triple. Then, $R$ is a local ring, so let $\widehat R$ be its completion
at its maximal ideal; necessarily, $\widehat R$ is a regular local ring having the same Krull dimension as $R$. (\cite[Ch. 10]{AM69} has a general treatment of completions. See also \cite[Th. 17.5, and the proof of Th. 17.8]{M86} and the preceding discussion.) If $\widehat Q$ is the fraction field of $\widehat R$, then $(\widehat R,\widehat Q, F)$ is a
``complete regular local triple."

(b) {\bf Direct unions.} Suppose that $(R_n,Q_n,F_n)$ is a sequence of DVR triples with 
$$(R_n,Q_n, F_n)\leq (R_{n+1}, Q_{n+1},F_{n+1}), \quad n\geq 0.$$
Assume also, for sufficiently large $n$, that $R_{n+1}$ is unramified over $R_n$. 
Then, 
$$(\underset\longrightarrow\lim R_n,\underset\longrightarrow\lim\, Q_n, \underset\longrightarrow\lim F_n).$$
is a DVR triple. We leave the proof of this as an exercise, but do record the following corollary.
\end{rems}

\begin{cor}\label{corRem} Let $F$ be the algebraic closure of a prime field $\mathbb F_r$, let $R$ be the ring
of Witt vectors over $F$, and let $Q$ be the fraction field $R$. Then, $(R,Q,F)$ is a DVR triple. \end{cor}

This fact follows from Remark \ref{final}(b) and the construction of Witt vectors (see \cite[\S8.10]{Jac89}).

%\subsection*{}
%\subsection*{}
%\begin{rem} We conclude this appendix by remarking that many more DVR triples may be constructed using
%Corollary \ref{corRem} in conjunction with Proposition \ref{D11}. \end{rem}

  %For the DVR case, we begin with a prime $r$, a ring of integers $\sZ$ in a number field $Q$, a maximal ideal $\mathfrak r$ of $\sZ$ containing $r$, and we put
  %$R=\sZ_{\mathfrak r}$, $F=R/\mathfrak r$. This gives a local triple $(R,Q,F)$ with $F$ a finite field of characteristic
  %$r$. We keep $r$ fixed, and vary $Q$ and $\mathfrak r$. Take $p$ to be a prime different from $r$ but which is the natural characteristic of the fixed finite group of Lie type. 

%%%%%%%%%%%%%%%%%%

\backmatter

\printindex
\end{document}